\documentclass[11pt]{amsart}  
\usepackage{amsbsy,amsmath,amsthm,amssymb,color,verbatim,hyperref}
\usepackage{ulem}
\usepackage{enumitem,longtable}
\usepackage{MnSymbol,wasysym}
\usepackage{subfig}
\RequirePackage{etex}
\usepackage{fullpage,cancel}
\usepackage{float}
\usepackage{wrapfig}
\usepackage[dvipsnames]{xcolor}
\usepackage{graphicx}
\usepackage{multicol}
\usepackage{multirow}
\usepackage{mathrsfs}
\usepackage{longtable}
\usepackage{makecell}%To keep spacing of text in tables
\setcellgapes{4pt}%parameter for the spacing
\usepackage{supertabular}
\usepackage{tabu}
\usepackage{tikz}
\usepackage{tkz-graph}
\tikzstyle{vertex}=[circle, draw, inner sep=0pt, minimum size=4pt]
\newcommand{\vertex}{\node[vertex]}
\tikzstyle{vtx}=[circle, draw, inner sep=0pt, minimum size=8pt]
\newcommand{\vtx}{\node[vtx]}
\definecolor{darkgreen}{cmyk}{.9,0,.9,.2}
\definecolor{midgray}{gray}{0.60}
\definecolor{lightgray}{gray}{0.90}
\definecolor{lmgray}{gray}{0.70}
\definecolor{one}{RGB}{203,152,203}
\definecolor{two}{RGB}{254,218,152}
\definecolor{three}{RGB}{152,152,253}
\definecolor{four}{RGB}{253,152,152}
\definecolor{five}{RGB}{152,203,152}
\definecolor{six}{RGB}{152,152,152}
\definecolor{seven}{RGB}{203,203,152}
\definecolor{eight}{RGB}{254,202,184}
\definecolor{nine}{RGB}{254,238,152}
\definecolor{ten}{RGB}{240,160,176}
\definecolor{eleven}{RGB}{211,152,236}
\definecolor{twelve}{RGB}{152,152,203}
\definecolor{thirteen}{RGB}{223,156,121}
\definecolor{fourteen}{RGB}{121,90,223}
\definecolor{fifteen}{RGB}{152,181,90}
\definecolor{sixteen}{RGB}{192,90,152}
\definecolor{seventeen}{RGB}{203,142,90}
\definecolor{eighteen}{RGB}{90,121,90}
\definecolor{nineteen}{RGB}{192,176,90}
\definecolor{twenty}{RGB}{223,172,122}
\definecolor{twenty-one}{RGB}{227,141,173}
\definecolor{twenty-two}{RGB}{239,150,114}
\definecolor{twenty-three}{RGB}{143,95,156}
\definecolor{twenty-four}{RGB}{150,90,205}
\definecolor{twenty-five}{RGB}{174,119,133}
\definecolor{twenty-six}{RGB}{133,144,72}
\definecolor{twenty-seven}{RGB}{174,53,133}
\definecolor{twenty-eight}{RGB}{90,108,53}
\definecolor{twenty-nine}{RGB}{155,158,53}
\definecolor{thirty}{RGB}{202,113,77}
\definecolor{thirty-one}{RGB}{143,89,119}
\definecolor{thirty-two}{RGB}{144,56,177}
\definecolor{thirty-three}{RGB}{191,104,154}
\definecolor{thirty-four}{RGB}{209,122,136}
\definecolor{thirty-five}{RGB}{223,135,85}
\definecolor{thirty-six}{RGB}{166,104,90}
\definecolor{thirty-seven}{RGB}{205,71,79}
\definecolor{thirty-eight}{RGB}{104,122,79}
\definecolor{thirty-nine}{RGB}{79,86,42}
\definecolor{forty}{RGB}{155,150,31}
\definecolor{forty-one}{RGB}{180,102,55}
\definecolor{forty-two}{RGB}{120,67,96}
\definecolor{forty-three}{RGB}{144,53,155}
\definecolor{forty-four}{RGB}{187,83,137}
\definecolor{forty-five}{RGB}{180,93,133}
\definecolor{forty-six}{RGB}{209,100,113}
\definecolor{forty-seven}{RGB}{201,112,85}
\definecolor{forty-eight}{RGB}{155,82,79}
\definecolor{forty-nine}{RGB}{122,93,47}
\definecolor{fifty}{RGB}{62,72,47}
\definecolor{fifty-one}{RGB}{148,137,24}
\definecolor{fifty-two}{RGB}{180,97,42}
\definecolor{fifty-three}{RGB}{107,60,83}
\definecolor{fifty-four}{RGB}{130,39,141}
\definecolor{fifty-five}{RGB}{187,82,124}
\definecolor{fifty-six}{RGB}{181,77,124}
\definecolor{fifty-seven}{RGB}{180,80,119}
\definecolor{fifty-eight}{RGB}{196,87,113}
\definecolor{fifty-nine}{RGB}{194,99,79}
\definecolor{sixty}{RGB}{173,93,47}
\definecolor{onectwo}{RGB}{165,42,42}
\definecolor{twoctwo}{RGB}{255,165,0}
\definecolor{threectwo}{RGB}{255,255,0}
\definecolor{fourctwo}{RGB}{127,0,0}
\definecolor{fivectwo}{RGB}{0,255,255}
\definecolor{sixctwo}{RGB}{128,128,128}
\definecolor{sevenctwo}{RGB}{255,0,0}
\definecolor{eightctwo}{RGB}{127,255,212}
\definecolor{ninectwo}{RGB}{220,20,60}
\definecolor{tenctwo}{RGB}{255,192,203}
\definecolor{elevenctwo}{RGB}{50,205,50}
\definecolor{twelvectwo}{RGB}{173,255,47}
\definecolor{thirteenctwo}{RGB}{0,0,255}
\definecolor{fourteenctwo}{RGB}{255,0,255}
\definecolor{fifteenctwo}{RGB}{255,215,0}
\definecolor{sixteenctwo}{RGB}{30,144,255}
\definecolor{seventeenctwo}{RGB}{63,224,208}
\definecolor{eighteenctwo}{RGB}{238,130,238}
\definecolor{nineteenctwo}{RGB}{0,127,0}
\definecolor{twentyctwo}{RGB}{255,140,0}
\definecolor{twenty-onectwo}{RGB}{0,0,127}
\definecolor{twenty-twoctwo}{RGB}{102,205,170}
\definecolor{twenty-threectwo}{RGB}{123,104,238}
\definecolor{twenty-fourctwo}{RGB}{255,69,0}
\definecolor{onedtwo}{RGB}{255,0,0}
\definecolor{twodtwo}{RGB}{0,0,255}
\definecolor{threedtwo}{RGB}{0,128,0}
\definecolor{fourdtwo}{RGB}{165,42,42}
\usepackage{colortbl}
\def\l{\lambda}
\def\r{\rho}
\def\NN{{\mathbb N}}
\def\ZZ{{\mathbb Z}}
\def\QQ{{\mathbb Q}}
\def\RR{{\mathbb R}}
\def\FF{{\mathbb F}}
\def\CC{{\mathbb C}}
\allowdisplaybreaks
\def\B{\mathcal{B}}
\def\C{\mathcal{C}}
\def\A{\mathcal{A}}
\def\D{\mathcal{D}}
\def\a{\alpha}
\def\w{\varpi}
\def\al{\alpha}
\def\fg{\mathfrak{g}}
\def\fh{\mathfrak{h}}
\def\la{\lambda}
\def\be{\beta}
\def\P{\mathscr{P}}

\def\hroot{\tilde{\alpha}}
\numberwithin{equation}{section}

\usepackage{relsize}
\makeatletter
\newtheorem*{rep@theorem}{\rep@title}
\newcommand{\newreptheorem}[2]{%
\newenvironment{rep#1}[1]{%
 \def\rep@title{#2 \ref{##1}}%
 \begin{rep@theorem}}%
 {\end{rep@theorem}}}
\makeatother

\makeatletter
\newtheorem*{rep@conjecture}{\rep@title}
\newcommand{\newrepconjecture}[2]{%
\newenvironment{rep#1}[1]{%
 \def\rep@title{#2 \ref{##1}}%
 \begin{rep@conjecture}}%
 {\end{rep@conjecture}}}
\makeatother

\makeatletter
\newcommand{\addresseshere}{%
  \enddoc@text\let\enddoc@text\relax
}
\makeatother

\theoremstyle{definition}
\newtheorem{definition}{Definition}[section]

\newtheorem{theorem}{Theorem}[section]
\newreptheorem{theorem}{Theorem}

\newtheorem{lemma}{Lemma}[section] 
 
\newrepconjecture{conjecture}{Conjecture}
\newtheorem{question}{Question}[section]

\title{Visualizing the support of Kostant's weight multiplicity formula for the rank two Lie algebras}

\author{Pamela E. Harris}
\address{Department of Mathematics and Statistics, Williams College, United States}
\email{\textcolor{blue}{\href{mailto:peh2@williams.edu}{peh2@williams.edu}}}

\author{Marissa Loving}
\address{Department of Mathematics, University of Illinois at Urbana-Champaign, United States}
\email{\textcolor{blue}{\href{mailto:mloving2@illinois.edu}{mloving2@illinois.edu}}}

\author{Juan Ramirez}
\address{Department of Mathematics, University of Houston, United States}
\email{\textcolor{blue}{\href{mailto:jjramirez8@uh.edu}{jjramirez8@uh.edu}}}

\author{Joseph Rennie}
\address{Department of Mathematics, University of Illinois at Urbana-Champaign, United States}
\email{\textcolor{blue}{\href{mailto:rennie2@illinois.edu}{rennie2@illinois.edu}}}

\author{Gordon Rojas Kirby}
\address{Department of Mathematics, University of California Santa Barbara, United States}
\email{\textcolor{blue}{\href{mailto:gkirby@math.ucsb.edu}{gkirby@math.ucsb.edu}}}

\author{Eduardo Torres Davila}
\address{Department of Mathematics, San Diego State University, United States}
\email{\textcolor{blue}{\href{mailto:etorresdavila@sdsu.edu}{etorresdavila@sdsu.edu}}}

\author{Fabrice O. Ulysse}
\address{Department of Mathematics, Cornell University, United States}
\email{\textcolor{blue}{\href{mailto:fou3@cornell.edu}{fou3@cornell.edu}}}


\begin{document}

\maketitle

\begin{abstract}
The multiplicity of a weight in a finite-dimensional irreducible representation of a simple Lie algebra $\mathfrak{g}$ can be computed via Kostant's weight multiplicity formula. 
This formula consists of an alternating sum over the Weyl group (a finite group) and involves a partition function known as Kostant's partition function. 
Motivated by the observation that, in practice, most terms in the sum are zero, our main results describe the elements of the Weyl alternation sets. The Weyl alternation sets are subsets of the Weyl group which contributes nontrivially to the multiplicity of a weight in a highest weight representation of the Lie algebras 
$\mathfrak{so}_4(\mathbb{C})$, $\mathfrak{so}_5(\mathbb{C})$,  $\mathfrak{sp}_4(\mathbb{C})$, and the exceptional Lie algebra $\mathfrak{g}_2$. 
By taking a geometric approach, we extend the work of Harris, Lescinsky, and Mabie on
$\mathfrak{sl}_3(\mathbb{C})$, to provide visualizations of these Weyl alternation sets for all pairs of integral weights $\lambda$ and $\mu$ of the Lie algebras considered. \end{abstract}

\section{Introduction}
Throughout we let $\mathfrak{g}$ be a simple Lie algebra and  $\mathfrak{h}$ be a Cartan subalgebra of $\mathfrak{g}$. We let $\Phi$ denote the set of roots corresponding to $(\mathfrak{g},\mathfrak{h})$, $\Phi^+\subseteq\Phi$ is the set of positive roots, and $\Delta\subseteq\Phi^+$ is the set of simple roots. We denote the Weyl group by $W$, and recall that it is generated by reflections orthogonal to the simple roots. For any $\sigma \in W$ we let $\ell(\sigma)$ denote the length of $\sigma$, which represents the minimum nonnegative integer $k$ such that $\sigma$ is a product of $k$ reflections.

In the representation theory of simple Lie algebras it is of interest to compute the multiplicity of a weight $\mu$ in a finite-dimensional complex irreducible representation of the Lie algebra $\mathfrak{g}$. We recall that the multiplicity of a weight $\mu$ is the dimension of a vector subspace associated to $\mu$, which is called a weight space. The theorem of the highest weight states that every irreducible (complex) representation of $\mathfrak{g}$ is a highest weight representation with highest weight $\lambda$. Such a representation is denoted by $L(\lambda)$, and if $\mu$ is a weight of $L(\lambda)$ then one can compute the multiplicity of this weight using Kostant's weight multiplicity formula, which is defined in~\cite{kostant} as:
\begin{equation}
m( \lambda, \mu) = \sum_{ \sigma \in W} (-1)^{ \ell (\sigma)}   \wp ( \sigma ( \lambda + \rho ) - ( \mu + \rho)),
\label{eq:KWMF} 
\end{equation}
where $\rho$ is equal to half the sum of the positive roots, and $\wp:\mathfrak{h}^*\to\mathbb{N}$ denotes Kostant's partition function, where $\wp(\xi)$  gives the  number of ways of expressing the weight $\xi$ as a nonnegative integral sum of positive roots.  

Although Kostant's formula can be used to compute weight multiplicities, its implementation is very difficult given that the order of the Weyl group grows factorially in the rank of the Lie algebra. Moreover, closed formulas for Kostant's partition function are not known in much generality. However,  
one way to address these challenges is to determine the elements of the Weyl group that contribute nonzero partition function values, and eliminate the numerous terms appearing in the multiplicity formula that contribute trivially to Equation \eqref{eq:KWMF}. 
This requires one to determine when $ \sigma( \lambda + \rho) - (\mu + \rho)$ can be expressed as a nonnegative integral linear combination of the simple roots of $\mathfrak{g}$.
Thus, we are interested in determining the \textit{Weyl alternation set}, denoted $\mathcal{A} (\lambda, \mu)$, which is the set of $\sigma\in W$  satisfying $\wp(\sigma( \lambda + \rho) - (\mu + \rho)) > 0.$

Previous work concerning Weyl alternation sets includes determining and enumerating $\A(\hroot,0)$ when $\hroot$ is the highest root of the simple Lie algebras \cite{harristhesis,hiw} and $\A(\lambda,0)$ when  $\lambda$ is the sum of the simple roots of the classical Lie algebras \cite{CHI}. 
In these cases the size of the Weyl alternation sets are given by constant coefficient homogeneous relations. In particular, for the Lie algebras $\mathfrak{sl}_{r+1}(\mathbb{C})$ and $\mathfrak{so}_{2r+1}(\mathbb{C})$ the size of the sets $\A(\sum_{\alpha\in\Delta}\alpha,0)$ are given by Fibonacci numbers, while in  $\mathfrak{sp}_{2r}(\mathbb{C})$ and $\mathfrak{so}_{2r}(\mathbb{C})$, the analogous sets have cardinalities given by a multiple of the Lucas numbers. More recently, Harris,  Rahmoeller, Schneider, and  Simpson determine Weyl alternation sets for pairs of weights whose multiplicity is equal to one \cite{HRSS}. 

In this work, we present the Weyl alternation sets $\A(\lambda,\mu)$ for  integral weight $\lambda$ and dominant integral weight $\mu$ of the Lie algebras $\mathfrak{so}_4(\mathbb{C})$, $\mathfrak{so}_5(\mathbb{C})$, $\mathfrak{sp}_4(\mathbb{C})$, and the exceptional Lie algebra $\mathfrak{g}_2$ (Theorems \ref{thm:maind2}, \ref{thm:mainb2}, \ref{thm:mainc2}, and \ref{thm:maing2}, respectively). These results generalize \cite[Theorem 2.3.1, 2.4.1, 2.5.1, and 2.6.1]{harristhesis} which only present the Weyl alternation sets $\A(\lambda,0)$, where $\lambda$ is an integral weight.
\begin{figure}[H]
\centering
 \resizebox{!}{3in}{%
\input{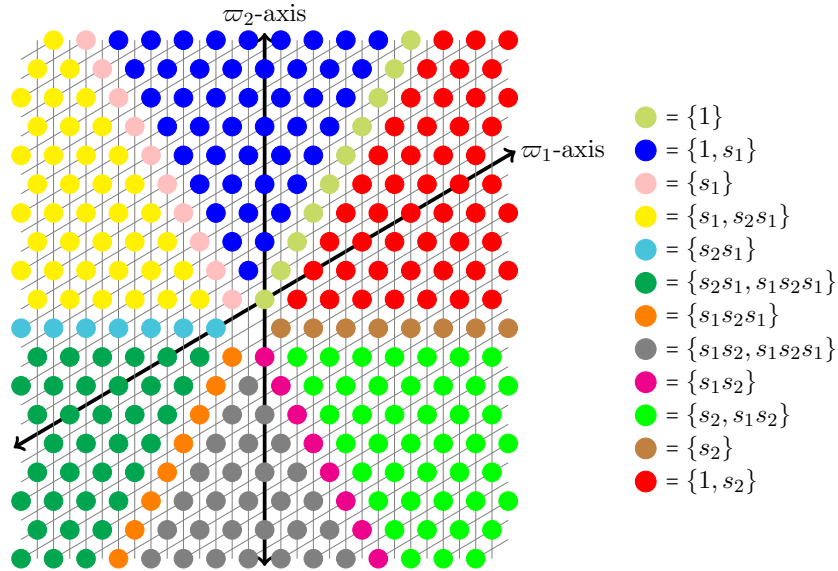}
 }%
\caption{Weyl alternation diagram for $\mathfrak{sl}_3(\mathbb{C})$ with $\mu=0$. Reproduced from \cite{HLM}.}
\label{fig:HLM}
\end{figure}
More recently, Harris, Lescinsky, and Mabie considered the Lie algebra $\mathfrak{sl}_3(\mathbb{C})$ and took a geometric approach to determining the elements of the Weyl group that contribute nontrivially to $m(\lambda,\mu)$ by varying the weights $\lambda$ and $\mu$ over the fundamental weight lattice \cite{HLM}. This involved computing the  \textit{Weyl alternation diagram} (associated to the weight $\mu$), introduced in \cite[Section 2.7]{harristhesis}. These  diagrams provide a visualization of the Weyl alternation sets by associating all integral linear combinations of the fundamental weights with lattice points and encoding the elements in Weyl alternation sets via coloring a corresponding subset of weights on the fundamental lattice which share a common Weyl alternation set. For example, Figures \ref{fig:HLM} and \ref{fig:mu=0}, first appearing in \cite[Figures 2.16 - 2.20]{harristhesis}, present the Weyl alternation diagrams of $\mu=0$ for each rank 2 Lie algebra. 
In this paper, we present the Weyl alternation diagrams when $\mu$ is a dominant integral weight of the Lie algebras $\mathfrak{so}_4(\mathbb{C})$, $\mathfrak{so}_5(\mathbb{C})$, $\mathfrak{sp}_4(\mathbb{C})$, and the exceptional Lie algebra $\mathfrak{g}_2$ . 

This paper is organized as follows: Section \ref{sec:background} provides the background and definitions necessary to make our approach precise. Section \ref{sec:altsets} provides the proofs of our main results describing the Weyl alternation sets (Theorem \ref{thm:mainb2}, \ref{thm:mainc2}, \ref{thm:maind2}, and \ref{thm:maing2}), and Section \ref{sec:diagrams} provides the construction of the Weyl alternation diagrams. We end with Section \ref{sec:open} where we provide a direction for future research.
\begin{figure}[H]%
    \centering
    \subfloat[$B_2$]{{\includegraphics[width=6.5cm, height=6.5cm]{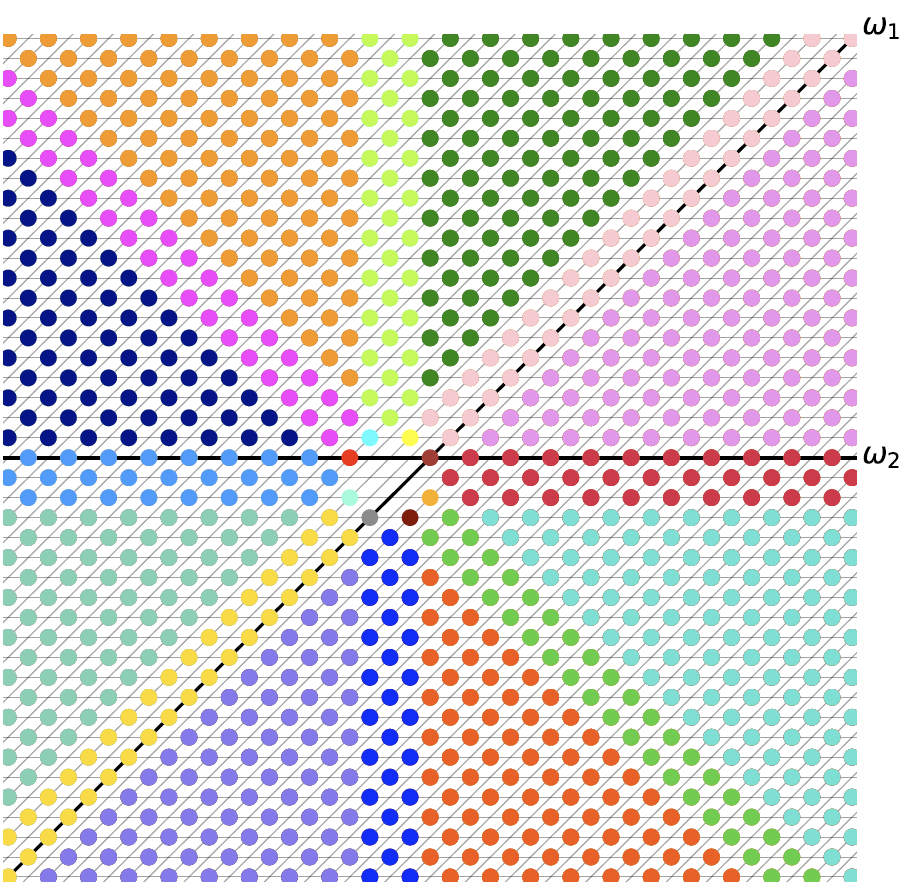} }
    \label{subfig:b2_n0_m0}
    }
    \qquad
    \subfloat[$C_2$]{{\includegraphics[width=6.5cm, height=6.5cm]{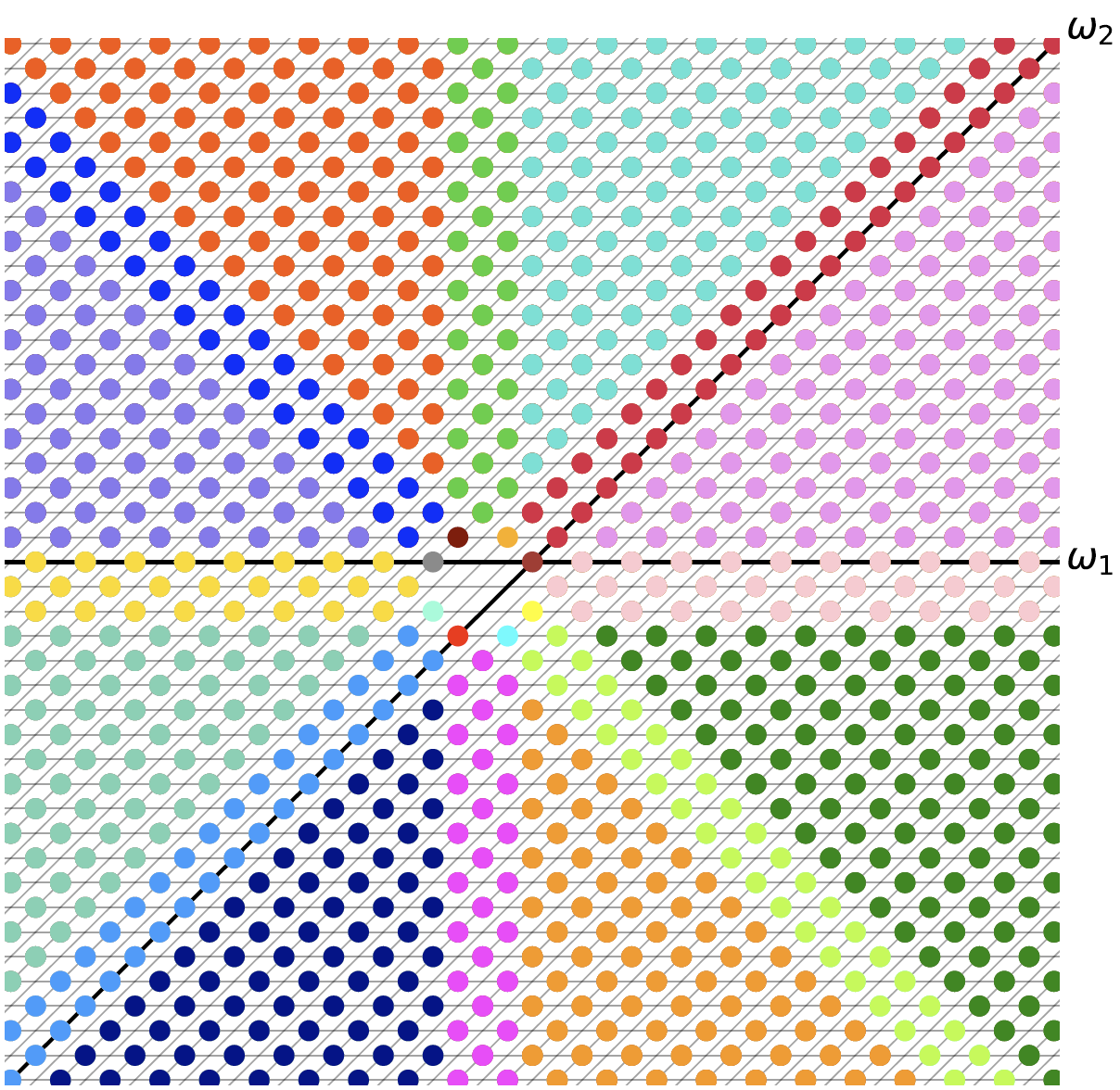} }
    \label{subfig:c2_n0_m0}
    }\\
    \subfloat[$D_2$]{{\includegraphics[width=6.5cm, height=6.5cm]{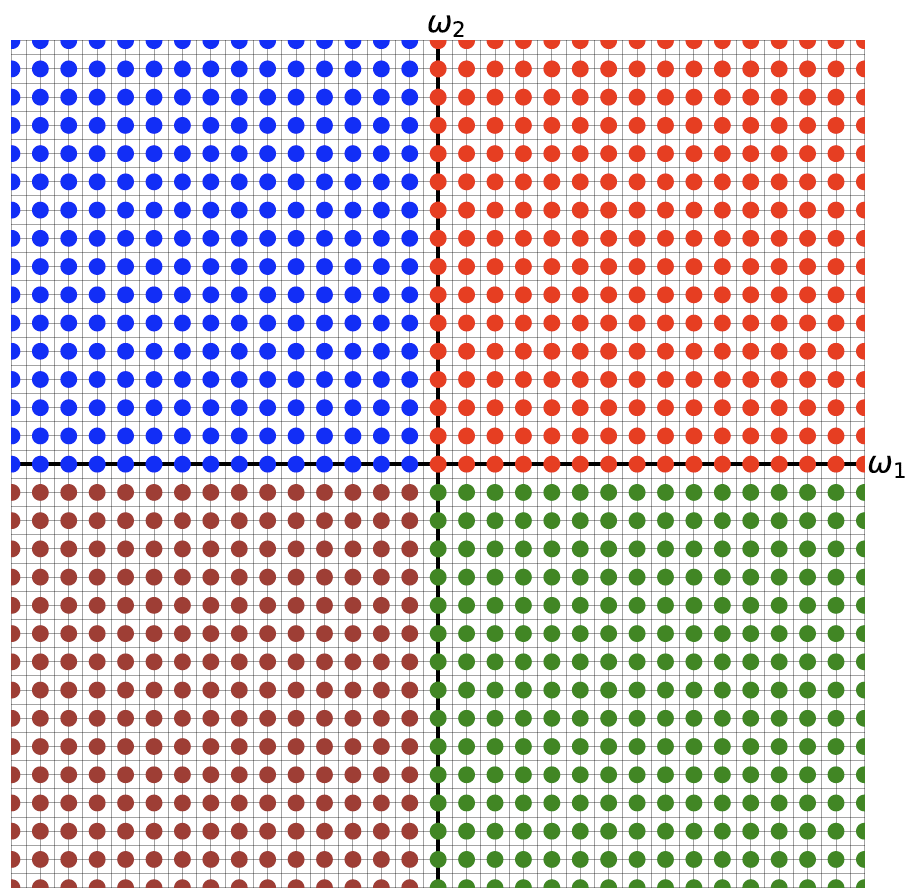} }
    \label{subfig:d2_n0_m0}
    }
    \qquad
    \subfloat[$G_2$]{{\includegraphics[width=6.5cm, height=6.5
    cm]{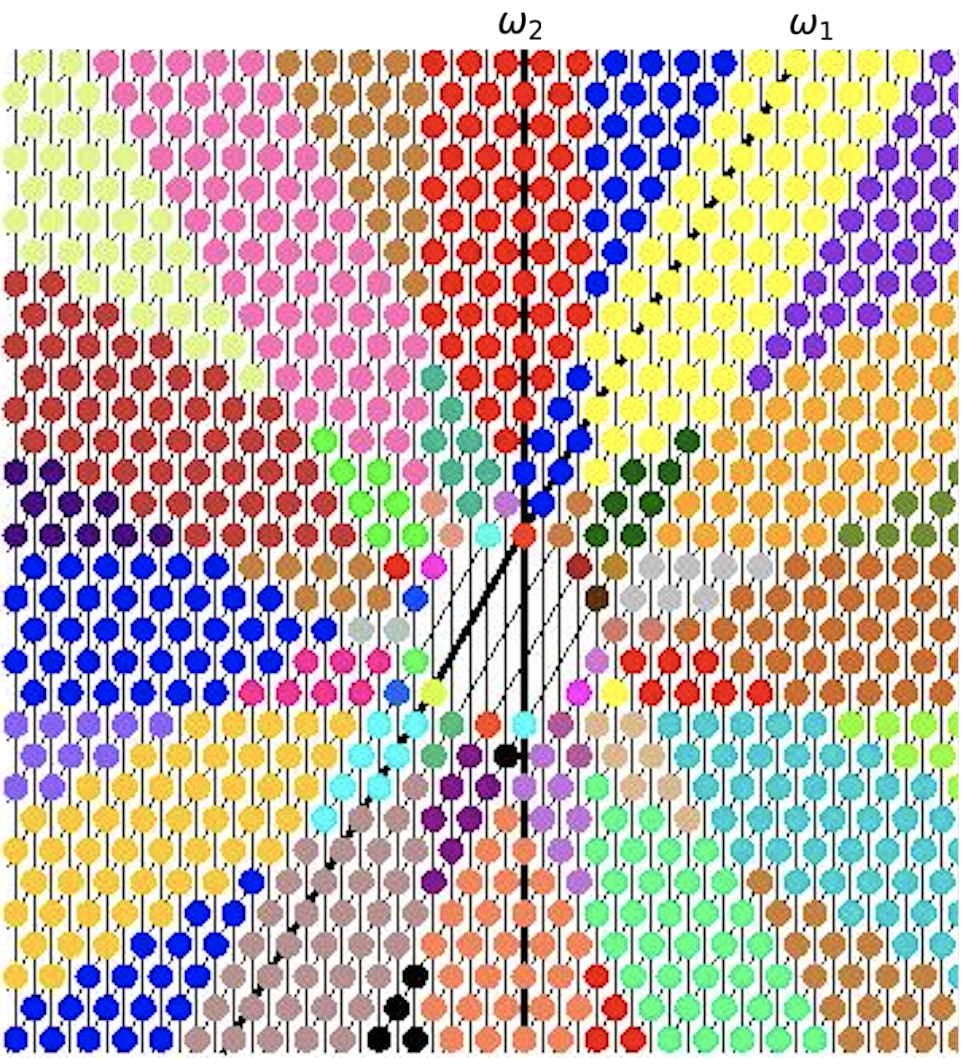}}
    \label{subfig:g2_n0_m0}
    }
    \caption{Weyl alternation diagrams for the Lie algebras of type $B_2$, $C_2$, $D_2$, and $G_2$, where $\mu = 0$. Reproduced from \cite{harristhesis}.}
    \label{fig:mu=0}
\end{figure}

\section{Background}
\label{sec:background}
In this section we provide the background necessary to make our approach precise. We use the notation and definitions of \cite{GW} and for the technical background on the representation theory of Lie algebras, as it relates to computing weight multiplicities, we refer the interested reader to \cite{harris}. 

For any $k\geq 2$, let $M_k(\mathbb{C})$ denote the set of $k\times k$ matrices with complex valued entries. Then the finite-dimensional (classical) Lie algebras are classified into families in the following way.\begin{itemize}
\item Type $A_r$ $(r\geq 1)$: $\mathfrak{sl}_{r+1}(\mathbb{C})=\{X\in M_{r+1}(\mathbb{C}): Trace(X)=0\}$.
\vspace{.2cm}
\item Type $B_r$ $(r\geq 2)$: $\mathfrak{so}_{2r+1}(\mathbb{C})=\{X\in M_{2r+1}(\mathbb{C}): X^t=-X\}$.
\vspace{.2cm}
\item Type $C_r$ $(r\geq 3)$: $\mathfrak{sp}_{2r}(\mathbb{C})=\{X\in M_{2r}(\mathbb{C}): X^tJ=-JX\}$, where $J=\begin{bmatrix}
  0 & I \\
  -I & 0
 \end{bmatrix}$ with $I$ the $r\times r$ identity matrix.
\vspace{.2cm}
\item Type $D_r$ $(r\geq 4)$: $\mathfrak{so}_{2r}(\mathbb{C})=\{X\in M_{2r}(\mathbb{C}): X^t=-X\}$.
\end{itemize}
In these cases, if $X,Y$ are elements of the Lie algebra, then the Lie bracket is defined by
\[
[X,Y]=XY-YX.
\]

The main definition we need in our work is the following. 
\begin{definition}\label{def:Weylaltset}
Given a pair of weights $\lambda$ and $\mu$ we define the \textit{Weyl alternation set} associated to $\l$ and $\mu$, denoted $\A(\l, \mu)$, as the set of Weyl group elements $\sigma \in W$ such that $\wp(\sigma(\l + \r)-(\mu + \r)) > 0$.
\end{definition}

We now give the necessary background for each specific Lie algebra we consider.

\subsection{Lie algebra of type \texorpdfstring{$B_2$}{B2}}
The Lie algebra of type $B_2$ is $\mathfrak{so}_{5}(\mathbb{C})$.
Let $\Delta = \{\al_1, \al_2\}$ be the set of simple roots and $\Phi^{+} = \{\al_1, \al_2, \al_1+\al_2, \al_1+2\al_2\}$ be the set of positive roots. Then the fundamental weights are
\[
    \w_1 = \al_1 + \al_2 \mbox{ and }
    \w_2 = \frac{1}{2}\al_1 + \al_2.
\]
As a result, $\rho = \w_1 + \w_2 =\frac{1}{2}\sum_{\alpha\in\Phi^+}\alpha= \frac{3}{2}\al_1+2\al_2$. 
% We provide a geometric interpretation of the Weyl group in Figure \ref{fig:root_systemb2}. 
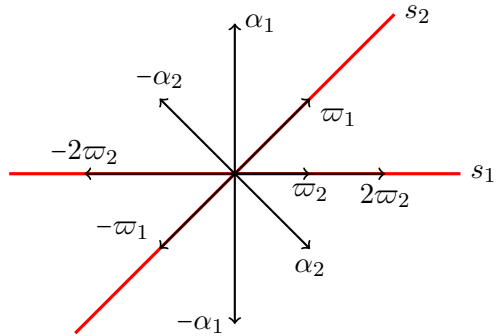
\begin{figure}[ht]
\begin{center}
\begin{tikzpicture}
\draw[color=red,very thick] (180:3) to (0:3);
\draw[color=red,very thick] (225:3) to (45:3);
\foreach \i in {0,1,2,3}{
\draw[solid, thick, ->] (0:0) to ({90*\i}:2);
\draw[solid, thick, ->] (0:0) to ({45+90*\i}: sqrt 2);
}
\draw [thick, ->] (0:0) to (0:1);
\node[anchor=north] at (0:1) {$\w_2$};
\node [anchor=north] at (0:2) {$2\w_2$};
\node[anchor=north west] at (45:sqrt 2) {$\w_1$};
\node[anchor=west] at (90: 2) {$\al_1$};
\node[anchor=south] at (135: sqrt 2) {$-\al_2$};
\node[anchor=south] at (180: 2) {$-2\w_2$};
\node[anchor=south east] at (225: sqrt 2) {$-\w_1$};
\node[anchor=east] at (270: 2) {$-\al_1$};
\node[anchor=north] at (270+45: sqrt 2) {$\al_2$};
\node[anchor=west] at (45:3) {$s_2$};
\node[anchor=west] at (0:3){$s_1$};
\end{tikzpicture}
 \caption{The root system of $B_2$.}
    \label{fig:root_systemb2}
\end{center}
\end{figure}
The Weyl group is generated by $s_1$ and $s_2$, where $s_1$ is the reflection through the hyperplane perpendicular to $\al_1$, and $s_2$ is the reflection through the hyperplane perpendicular to $\al_2$. Figure \ref{fig:root_systemb2} provides a geometric interpretation of the Weyl group; the roots are depicted by vectors, and the two hyperplanes $s_1$ and $s_2$ are colored red. Table \ref{tab:WeylB2} shows how each element of the Weyl group acts on the simple roots. Moreover, the action of the reflections on the simple roots and fundamental weights is given by 
\begin{align*}
    s_i(\al_j) = \begin{cases}
    -\al_j &\mbox{ if } i = j\\ 
    \al_j+i\al_i &\mbox{ if } i \neq j 
    \end{cases}\qquad\mbox{and}\qquad&
    s_i(\w_j) = \begin{cases}
    \w_j-\al_j &\mbox{ if } i = j\\
    \w_j &\mbox{ if } i \neq j.
    \end{cases}
\end{align*}

\begin{table}[H]
    \centering
\begin{tabular}{|c|c|c|c|c|c|c|c|c|c|}\hline
    $\sigma\in W$ & 1 & $s_1$ & $s_2$ &$s_2s_1$ & $s_1s_2$ & $s_1s_2s_1$ &$s_2s_1s_2$ & $s_2s_1s_2s_1$ & $s_1s_2s_1s_2$\\ \hline
    $\sigma(\al_1)$ & $\al_1$ & $-\al_1$ & $2\w_2$ &$-2\w_2$ &$2\w_2$ &$-2\w_2$ & $\al_1$ & $-\al_1$ & $-\al_1$\\ \hline
    $\sigma(\al_2)$ & $\al_2$ &$\w_1$ & $-\al_2$ &$\w_1$ & $-\w_1$ & $\al_2$ & $-\w_1$ & $-\al_2$ & $-\al_2$\\ \hline
\end{tabular}  
\caption{The elements of the Weyl group of $B_2$ and their action on the simple roots.}
\label{tab:WeylB2}
\end{table}

Lastly, we remark that the Weyl group of $B_2$ is isomorphic to the dihedral group of order 8.%, also known as the rigid motions of a square.

\subsection{Lie algebra of type \texorpdfstring{$C_2$}{C2}}
The Lie algebra of type $C_2$ is $\mathfrak{sp}_{4}(\mathbb{C})$. For the Lie algebra of type $C_2$ we have $\Delta = \{\al_1, \al_2\}$, $\Phi^+ = \{\al_1, \al_2, \al_1+\al_2, 2\al_1+\al_2\}$, and the fundamental weights are
\[    \w_1 = \al_1 + \frac{1}{2}\al_2 \qquad\mbox{and}\qquad
    \w_2 = \al_1+\al_2.\]
As a result, $\rho = \w_1 + \w_2 = 2\al_1+\frac{3\al_2}{2}$.
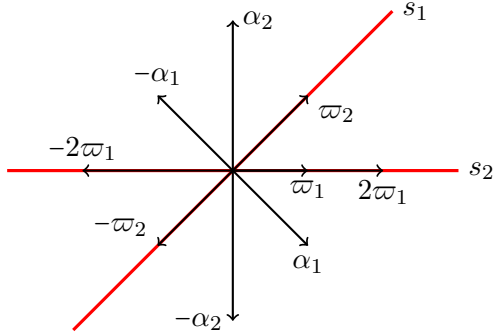
\begin{figure}[ht]
\begin{center}
\begin{tikzpicture}
\draw[color=red,very thick] (180:3) to (0:3);
%\draw[thick] (0:0) to (270)
\draw[color=red,very thick] (225:3) to (45:3);
\foreach \i in {0,1,2,3}{
\draw[thick, ->] (0:0) to ({90*\i}:2);
\draw[thick, ->] (0:0) to ({45+90*\i}: sqrt 2);
}
\draw [thick, ->] (0:0) to (0:1);
\node[anchor=north] at (0:1) {$\w_1$};
\node [anchor=north] at (0:2) {$2\w_1$};
\node[anchor=north west] at (45:sqrt 2) {$\w_2$};
\node[anchor=west] at (90: 2) {$\al_2$};
\node[anchor=south] at (135: sqrt 2) {$-\al_1$};
\node[anchor=south] at (180: 2) {$-2\w_1$};
\node[anchor=south east] at (225: sqrt 2) {$-\w_2$};
\node[anchor=east] at (270: 2) {$-\al_2$};
\node[anchor=north] at (270+45: sqrt 2) {$\al_1$};
\node[anchor=west] at (45:3) {$s_1$};
\node[anchor=west] at (0:3){$s_2$};
\end{tikzpicture}
 \caption{The root system of $C_2$.}
    \label{fig:root_systemc2}
\end{center}
\end{figure}

The Weyl group is generated by $s_1$ and $s_2$, where $s_1$ is the reflection through the hyperplane perpendicular to $\al_1$, and $s_2$ is the reflection through the hyperplane perpendicular to $\al_2$. Figure \ref{fig:root_systemc2} provides a geometric interpretation of the Weyl group; the roots are depicted by vectors, and the two hyperplanes $s_1$ and $s_2$ are colored red. Table \ref{tab:WeylC2} shows how each element of the Weyl group acts on the simple roots. Moreover, the action of the reflections $s_1$ and $s_2$ on the simple roots and fundamental weights is given by 
\begin{align*}
    s_i(\al_j) = \begin{cases}
    -\al_j &\mbox{ if } i = j\\ 
    \al_j+j\al_i &\mbox{ if } i \neq j
    \end{cases}\qquad&\mbox{and}\qquad
    s_i(\w_j) = \begin{cases}
    \w_j-\al_j &\mbox{ if } i = j\\
    \w_j &\mbox{ if } i \neq j.\\
    \end{cases}
\end{align*}
\begin{table}[H]
    \centering
\begin{tabular}{|c|c|c|c|c|c|c|c|c|c|}\hline
    $\sigma\in W$ & 1 & $s_1$ & $s_2$ &$s_2s_1$ & $s_1s_2$ & $s_1s_2s_1$ &$s_2s_1s_2$ & $s_2s_1s_2s_1$ & $s_1s_2s_1s_2$\\ \hline 
    $\sigma(\al_1)$ & $\al_1$ & $-\al_1$ & $\w_2$ &$-\w_2$ &$\w_2$ &$-\w_2$ & $\al_1$ & $-\al_1$ & $-\al_1$\\ \hline
    $\sigma(\al_2)$ & $\al_2$ &$2\w_1$ & $-\al_2$ &$2\w_1$ & $-2\w_1$ & $\al_2$ & $-2\w_1$ & $-\al_2$ & $-\al_2$\\ \hline
\end{tabular}
\caption{The elements of the Weyl group of $C_2$ and their action on the simple roots.}\label{tab:WeylC2}
\end{table}
Lastly, we remark that the Weyl group of $C_2$ is isomorphic to the dihedral group of order 8.%, also known as the rigid motions of a square.
\subsection{Lie algebra of type \texorpdfstring{$D_2$}{D2}}
The Lie algebra of type $D_2$ is $\mathfrak{so}_{4}(\mathbb{C})$. For the Lie algebra of type  $D_2$, $\Delta=\Phi^{+} = \{\al_1, \al_2\}$ and the fundamental weights are
\[
    \w_1 = \tfrac{1}{2}\al_1\qquad\mbox{and}\qquad
    \w_2 = \tfrac{1}{2}\al_2.
\]
As a result, $\rho = \w_1 + \w_2 = \tfrac{1}{2}\al_1+\tfrac{1}{2}\al_2$.
\begin{figure}[ht]
\begin{center}
\begin{tikzpicture}[scale=.8]
\node [anchor=south] at (0:2) {$\al_1$};
\node[anchor=west] at (90: 2) {$\al_2$};
\node[anchor=south] at (180: 2) {$-\al_1$};
\node[anchor=east] at (270: 2) {$-\al_2$};
\node[anchor=west] at (0:3){$s_2$};
\node[anchor=west] at (90:3){$s_1$};
\node[anchor=south] at (0:1){$\w_1$};
\node[anchor=west] at (90:1){$\w_2$};
\draw[red, very thick] (0:-3) to (0:3);
\draw[red, very thick] (90:-3) to (90:3);
\draw[thick, ->] (0:0) to (0:2);
\draw[thick, ->] (0:0) to (90:2);
\draw[thick, ->] (0:0) to (0:-2);
\draw[thick, ->] (0:0) to (90:-2);
\draw[thick, ->] (0:0) to (0:1);
\draw[thick, ->] (0:0) to (90:1);
\end{tikzpicture}
 \caption{The root system of $D_2$.}
    \label{fig:root_systemd2}
\end{center}
\end{figure}
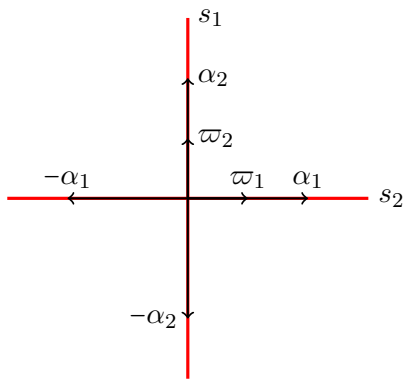

The Weyl group is generated by $s_1$ and $s_2$, where $s_1$ is the reflection through the hyperplane perpendicular to $\al_1$, and $s_2$ is the reflection through the hyperplane perpendicular to $\al_2$. Figure \ref{fig:root_systemd2} provides a geometric interpretation of the Weyl group; the roots are depicted by vectors, and the two hyperplanes $s_1$ and $s_2$ are colored red. Table \ref{tab:WeylD2} shows how each element of the Weyl group acts on the simple roots. Moreover, the action of the reflections on the simple roots and fundamental weights is given by
\begin{align*}
    s_i(\al_j) = \begin{cases}
    -\al_j &\mbox{ if } i = j\\ 
    \al_j &\mbox{ if } i \neq j
    \end{cases}\qquad&\mbox{and}\qquad
    s_i(\w_j) = \begin{cases}
    \w_j-\al_j &\mbox{ if } i = j\\
    \w_j &\mbox{ if } i \neq j.\\
    \end{cases}
\end{align*}
\begin{table}[H]
\centering
\begin{tabular}{|c|c|c|c|c|}\hline
    $\sigma \in W$ & 1 & $s_1$ & $s_2$ &$s_2s_1$ \\ \hline
    $\sigma(\al_1)$ & $\al_1$ & $-\al_1$ & $\al_1$ &$-\al_1$ \\ \hline
    $\sigma(\al_2)$ & $\al_2$ &$\al_2$ & $-\al_2$ &$-\al_2$ \\ \hline
\end{tabular}
\caption{The elements of the Weyl group of $D_2$ and their action on the simple roots.}\label{tab:WeylD2}
\end{table}
Lastly, we remark that the Weyl group of $D_2$ is isomorphic to the Klein-four group.

\subsection{Lie algebra of type \texorpdfstring{$G_2$}{G2}}
For the exceptional Lie algebra of type $G_2$, \[\Delta = \{\al_1, \al_2\}, \Phi^{+} = \{\al_1, \al_2, \al_1+\al_2, 2\al_1+\al_2, 3\al_1+\al_2, 3\al_1+2\al_2 \}.\] To simplify notation we let
\[
    \be_1 = \al_1 + \al_2,\;
    \be_2 = 3\al_1 + 2\al_2,\;
    \be_3 = 2\al_1 + \al_2, \text{ and } 
    \be_4 = 3\al_1 + \al_2. 
    \]
Then the fundamental weights are
\[
    \w_1 = 2\al_1 + \al_2 \qquad\mbox{and}\qquad
    \w_2 = 3\al_1 + 2\al_2.
\]
As a result, $\rho = \w_1 + \w_2 = 5\al_1+3\al_2$.
\begin{figure}[ht]
\centering
\begin{tikzpicture}
\draw[red, very thick] (90:-3) to (90:3);
\draw[red, very thick] (60:-3) to (60:3);
\foreach \i in {0,1,2}{
\draw[thick, ->] (0:0) to ({30+120*\i}:2);
\draw[thick, ->] (0:0) to ({90+120*\i}:2);
\draw[thick, ->] (0:0) to ({60+120*\i}:1.11803);
\draw[thick, ->] (0:0) to ({120*(\i+1)}:1.11803);
}
\node [anchor=north west] at (0:1.11803) {$\a_1$};
\node [anchor=south] at (150:2) {$\a_2$};
\node [anchor=south] at (90:3) {$s_1$};
\node [anchor=south] at (0:-1.11803) {$-\a_1$};
\node [anchor=north] at (150:-2) {$-\a_2$};
\node [anchor=south west] at (60:3) {$s_2$};
\node [anchor=south west] at (120:1.11803) {$\be_1$};
\node [anchor=north] at (120:-1.11803) {$-\be_1$};
\node [anchor=north west] at (30:2) {$\be_4$};
\node [anchor=south east] at (30:-2) {$-\be_4$};
\node [anchor=west] at (90:2) {$\be_2$};
\node [anchor=east] at (90:-2) {$-\be_2$};
\node [anchor=west] at (60:1.11803) {$\be_3$};
\node [anchor=east] at (60:-1.11803) {$-\be_3$};
\end{tikzpicture}
 \caption{The root system of $G_2$.}
    \label{fig:root_systemg2}
\end{figure}
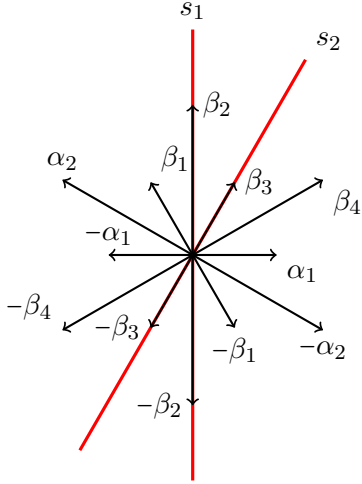
In Figure \ref{fig:root_systemg2}, the roots are depicted by vectors, and the two red lines denoted by $s_1$ and $s_2$ are the  hyperplanes orthogonal to the simple roots $\alpha_1$ and $\alpha_2$, respectively, which define the generators of the Weyl group. Table \ref{tab:WeylG2} shows how each element of the Weyl group acts on the simple roots. Moreover, the action of the reflections on the simple roots and fundamental weights is given by
%-------
%------------
\begin{align*}
    s_i(\al_j) = \begin{cases}
    -\al_j &\mbox{ if } i = j\\ 
    \w_j-\al_j &\mbox{ if } i \neq j
    \end{cases}\qquad&\mbox{and}\qquad
    s_i(\w_j) = \begin{cases}
    \w_j-\al_j &\mbox{ if } i = j\\
    \w_j &\mbox{ if } i \neq j.\\
    \end{cases}
\end{align*}
\begin{table}[H]
\centering
\resizebox{.9\textwidth}{!}{%
\begin{tabular}{|c|c|c|c|c|c|c|c|c|c|c|c|c|} 
 \hline
 $\sigma \in W$ & 1 & $s_1$ & $s_2$ & $s_2 s_1$ & $s_1 s_2$ & $s_1 s_2 s_1$ & $s_2s_1 s_2$ & $(s_2s_1)^2$ & $(s_1s_2)^2$ & $s_1(s_2s_1)^2$ & $s_2(s_1s_2)^2$ & $(s_2s_1)^3$\\ 
 \hline
 $\sigma(\al_1)$ & $\al_1$ & -$\al_1$ & $\be_1$  & -$\be_1$ & $\be_3$ & -$\be_3$ & $\be_3$ & -$\be_3$ & $\be_1$ & $-\be_1$ & $\al_1$ & $-\al_1$ \\ 
 \hline
 $\sigma(\al_2)$ & $\al_2$ & $\be_4$ & -$\al_2$ & $\be_2$ & -$\be_4$ & $\be_2$ & -$\be_2$ & $\be_4$ & -$\be_2$ & $\al_2$ & $-\be_4$ & $-\al_2$\\ 
 \hline
\end{tabular}}\\
\caption{The elements of the Weyl group of $G_2$ and their action on the simple roots.}\label{tab:WeylG2}
\end{table}
%-------
Lastly, we remark that the Weyl group of $G_2$ is isomorphic to the dihedral group of order 12.%, also known as the rigid motions of a hexagon.

\section{Weyl alternation sets}\label{sec:altsets}
We  begin by recalling that the only portion of the fundamental weight lattice of a Lie algebra that encodes finite-dimensional irreducible representations is the nonnegative quadrant, which is called the dominant Weyl chamber \cite[Chapter 3]{GW}. However, in what follows, we consider the entire fundamental weight lattice. Doing so fully illustrates the symmetry of the Weyl group's action as a reflection group on the weights of the Lie algebras. 
In the following sections we describe the Weyl alternation sets of the Lie algebras of type $B_2$, $C_2$, $D_2$, and $G_2$. 
We have also color-coded each of the Weyl alternation set conditions. Each color corresponds to a region of the Lie algebra's Weyl alternation diagram which we describe fully in Section \ref{sec:diagrams}.

We remark that to give a complete description of the Weyl alternation sets, $\A(\lambda, \mu)$, we must consider each $\sigma\in W$ and determine when $\sigma(\lambda+\rho)-\rho-\mu$ has coefficients in $\mathbb{N}$ when expressed as a linear combination of the simple roots.
Namely, for every $\sigma\in W$, we want to find conditions such that $\sigma(\lambda+\rho)-(\mu+\rho) \in \NN{\a_1} \oplus \NN{\a_2}$. When restricting $\lambda,\mu$ to be in the dominant chamber, the Weyl alternation sets $\A(\lambda, \mu)$ allow us to reduce the number of terms in the computation of weight multiplicities.

\subsection{Lie algebra of type \texorpdfstring{$B_2$}{B2}}
We begin by establishing a divisibility condition necessary for a weight in the fundamental weight lattice to be in the root lattice.
\begin{lemma}\label{ref:divlemB2}
Let $\lambda = c_1\w_1+c_2\w_2$ with $c_1, c_2 \in \ZZ$. Then $\lambda \in \ZZ{\a_1} \oplus \ZZ{\a_2}$ if and only if $c_2$ is divisible by $2$.
\end{lemma}

\begin{proof}
Observe that $\lambda = c_1\w_1+c_2\w_2 = (c_1+\tfrac{c_2}{2})\al_1+(c_1+c_2)\al_2$. Since $c_1,c_2 \in \ZZ$, then $c_1+c_2\in\ZZ$. Also, $\tfrac{c_2}{2}+c_1=a$ for some $a \in \ZZ$ if and only if $c_2 = 2(a-c_1)$. Hence, $c_2$ is divisible by 2.
\end{proof}
\begin{table}[H]
    \centering
\begin{tabular}{|l|l|}\hline
    $J_1: \tfrac{c_2-m}{2}+c_1-n \geq 0 $&$J_2: c_1+c_2-n-m \geq 0 $\\\hline
    $J_3: \tfrac{c_2-m}{2}-n-1 \geq 0 $ & $J_4:c_1-n-m-1  \geq 0 $\\\hline
    $J_5: -c_1-n-m-3  \geq0 $ & $J_6: \tfrac{-c_2-m}{2}-n-2 \geq 0 $\\\hline
    $J_7: \tfrac{-c_2-m}{2}-c_1-n-3 \geq 0 $&$J_8: -c_1-c_2-n-m-4 \geq 0 $\\\hline
\end{tabular}
\caption{Conditions for Theorem \ref{thm:mainb2}.}
    \label{tab:b2}
\end{table}

We are now ready to describe the Weyl alternation sets for the Lie algebra of type $B_2$.
\begin{theorem}[Weyl alternation sets of  $\mathfrak{so}_{5}(\mathbb{C})$]\label{thm:mainb2}
Let $\lambda = c_1\w_1+c_2\w_2$ for some $c_1,c_2 \in \ZZ$ with $2|c_2$. Fix $\mu = n\w_1+m\w_2$ with $n,m \in \mathbb{N}$ and $2|m.$ To simplify notation we use inequalities in Table \ref{tab:b2} and let $\neg$ be the negation of an inequality. Then 
\begin{align*}
\A(\l, \mu) &=\begin{cases}    %1 element only
\tikz\draw[onectwo, fill] (0,0) circle (2.5pt); \quad
      \{1\} &\mbox{if } J_1, J_2, \neg J_3, \mbox{ and } \neg J_4 ,\\
      %----------------------------------------------------------
      \tikz\draw[twoctwo, fill] (0,0) circle (2.5pt); \quad
      \{s_1\} &\mbox{if } J_2, J_3, \neg J_1, \mbox{ and } \neg J_5 ,\\
      %----------------------------------------------------------
      \tikz\draw[threectwo, fill] (0,0) circle (2.5pt); \quad
      \{s_2\} &\mbox{if } J_1, J_4, \neg J_2, \mbox{ and } \neg J_6 ,\\
      %----------------------------------------------------------
      \tikz\draw[fourctwo, fill] (0,0) circle (2.5pt); \quad
      \{s_2s_1\} &\mbox{if } J_3, J_5, \neg J_2, \mbox{ and } \neg J_7 ,\\
      %----------------------------------------------------------
      \tikz\draw[fivectwo, fill] (0,0) circle (2.5pt); \quad
      \{s_1s_2\} &\mbox{if } J_6, J_4, \neg J_8, \mbox{ and } \neg J_1 ,\\
      %----------------------------------------------------------
      \tikz\draw[sixctwo, fill] (0,0) circle (2.5pt); \quad
      \{s_1s_2s_1\} &\mbox{if } J_7, J_5, \neg J_8, \mbox{ and } \neg J_3 ,\\
      %----------------------------------------------------------
      \tikz\draw[sevenctwo, fill] (0,0) circle (2.5pt); \quad
      \{s_2s_1s_2\} &\mbox{if } J_6, J_8, \neg J_4, \mbox{ and } \neg J_7 ,\\
      %----------------------------------------------------------
      \tikz\draw[eightctwo, fill] (0,0) circle (2.5pt); \quad
      \{(s_2s_1)^2\} &\mbox{if } J_7, J_8, \neg J_5, \mbox{ and } \neg J_6 ,\\
      %-------------------
      %2 elements
      \tikz\draw[ninectwo, fill] (0,0) circle (2.5pt); \quad
      \{1, s_1\} &\mbox{if } J_1, J_3, \neg J_4, \mbox{ and } \neg J_5 ,\\
      %----------------------------------------------------------
      \tikz\draw[tenctwo, fill] (0,0) circle (2.5pt); \quad
      \{1, s_2\} &\mbox{if } J_2, J_4, \neg J_3, \mbox{ and } \neg J_6 ,\\
      %----------------------------------------------------------
      \tikz\draw[elevenctwo, fill] (0,0) circle (2.5pt); \quad
      \{s_1, s_2s_1\} &\mbox{if } J_2, J_5, \neg J_1, \mbox{ and } \neg J_7 ,\\
      %----------------------------------------------------------
      \tikz\draw[twelvectwo, fill] (0,0) circle (2.5pt); \quad
      \{s_2, s_1s_2\} &\mbox{if } J_1, J_6, \neg J_2, \mbox{ and } \neg J_8 ,\\
      %----------------------------------------------------------
      \tikz\draw[thirteenctwo, fill] (0,0) circle (2.5pt); \quad
      \{s_2s_1, s_1s_2s_1\} &\mbox{if } J_3, J_7, \neg J_2, \mbox{ and } \neg J_8 ,\\
      %----------------------------------------------------------
      \tikz\draw[fourteenctwo, fill] (0,0) circle (2.5pt); \quad
      \{s_1s_2, s_2s_1s_2\} &\mbox{if } J_4, J_8, \neg J_1, \mbox{ and } \neg J_7 ,\\
      %----------------------------------------------------------
      \tikz\draw[fifteenctwo, fill] (0,0) circle (2.5pt); \quad
      \{s_1s_2s_1, (s_2s_1)^2\} &\mbox{if } J_5, J_8, \neg J_6, \mbox{ and } \neg J_3 ,\\
      %----------------------------------------------------------
      \tikz\draw[sixteenctwo, fill] (0,0) circle (2.5pt); \quad
      \{s_2s_1s_2, (s_2s_1)^2\} &\mbox{if } J_6, J_7, \neg J_5, \mbox{ and } \neg J_4 ,\\
      %-----------------
      %3 elements
      \tikz\draw[seventeenctwo, fill] (0,0) circle (2.5pt); \quad
      \{1, s_1, s_2s_1\} &\mbox{if } J_3, J_4, \neg J_5, \mbox{ and } \neg J_6 ,\\
      %----------------------------------------------------------
      \tikz\draw[eighteenctwo, fill] (0,0) circle (2.5pt); \quad
      \{1, s_1, s_2\} &\mbox{if } J_1, J_5, \neg J_4, \mbox{ and } \neg J_7 ,\\
      %----------------------------------------------------------
      \tikz\draw[nineteenctwo, fill] (0,0) circle (2.5pt); \quad
      \{1, s_2, s_1s_2\} &\mbox{if } J_2, J_6, \neg J_3, \mbox{ and } \neg J_8 ,\\
      %----------------------------------------------------------
   \tikz\draw[twentyctwo, fill] (0,0) circle (2.5pt); \quad  \{s_2, s_1s_2, s_2s_1s_2\} 
       &\mbox{if } J_7, J_2, \neg J_8, \mbox{ and } \neg J_1 ,\\
      %----------------------------------------------------------
      \tikz\draw[twenty-onectwo, fill] (0,0) circle (2.5pt); \quad
  \{s_1s_2, s_2s_1s_2, (s_2s_1)^2\}    &\mbox{if } J_1, J_8, \neg J_2, \mbox{ and } \neg J_7 ,\\
      %----------------------------------------------------------
      \tikz\draw[twenty-twoctwo, fill] (0,0) circle (2.5pt); \quad
     \{s_2s_1s_2, (s_2s_1)^2, s_1s_2s_1\}  &\mbox{if } J_3, J_8, \neg J_2, \mbox{ and } \neg J_6 ,\\
      %----------------------------------------------------------
      \tikz\draw[twenty-threectwo, fill] (0,0) circle (2.5pt); \quad
      \{s_2s_1, s_1s_2s_1, (s_2s_1)^2\}  &\mbox{if } J_7, J_4, \neg J_5, \mbox{ and } \neg J_1 ,\\
      %----------------------------------------------------------
      \tikz\draw[twenty-fourctwo, fill] (0,0) circle (2.5pt); \quad
     \{s_1, s_2s_1, s_1s_2s_1\}  &\mbox{if } J_6, J_5, \neg J_4, \mbox{ and } \neg J_3 ,\mbox{ and}\\
      \phantom{\tikz\draw[twenty-fourctwo, fill] (0,0) circle (2.5pt);}\quad \emptyset & \mbox{otherwise}.
\end{cases}
\end{align*}
\end{theorem}

\begin{proof}
Let $\l = c_1\w_1+c_2\w_2$ for some $c_1, c_2 \in \ZZ$ and $\mu = n\w_1+m\w_2$ with $n,m \in \mathbb{N}$. Then
\begin{align*}
    1(\l + \rho) - (\mu + \rho) &= (\tfrac{c_2-m}{2}+c_1-n)\al_1+(c_1+c_2-n-m)\al_2,\\
    s_1(\l + \rho) - (\mu + \rho) &= (\tfrac{c_2-m}{2}-n-1)\al_1 + (c_1+c_2-n-m)\al_2,\\
    s_2(\l + \rho) - (\mu + \rho) &= (\tfrac{c_2-m}{2}+c_1-n)\al_1+(c_1-n-m-1)\al_2,\\
    s_2s_1(\l + \rho) - (\mu + \rho) &= (\tfrac{c_2-m}{2}-n-1)\a_1 + (-c_1-n-m-3)\a_2,\\
    s_1s_2(\l + \rho) - (\mu + \rho) &= (\tfrac{-c_2-m}{2}-n-2)\a_1 + (c_1-n-m-1)\a_2,\\
    s_1(s_2s_1)(\l + \rho) - (\mu + \rho) &= (\tfrac{-c_2-m}{2}-c_1-n-3)\al_1 + (-c_1-n-m-3)\al_2,\\
    s_2(s_1s_2)(\l + \rho) - (\mu + \rho) &= (\tfrac{-c_2-m}{2}-n-2)\al_1+(-c_1-c_2-n-m-4)\al_2,\\
    (s_2s_1)^2(\l + \rho) - (\mu + \rho) &= (\tfrac{-c_2-m}{2}-c_1-n-3)\al_1+(-c_1-c_2-n-m-4)\al_2.
\end{align*}
From the equations above and the definition of a Weyl alternation set, it follows that 
\begin{align*}
    1 &\in \A(\l,\mu) \Leftrightarrow \tfrac{c_2-m}{2}+c_1-n \geq 0 \mbox{ and } c_1+c_2-n-m \geq 0,\\
    s_1 &\in \A(\l,\mu) \Leftrightarrow \tfrac{c_2-m}{2}-n-1 \geq 0 \mbox{ and } c_1+c_2-n-m \geq 0, \\ 
    s_2 &\in \A(\l,\mu) \Leftrightarrow \tfrac{c_2-m}{2}+c_1-n \geq 0 \mbox{ and } c_1-n-m-1 \geq 0,\\
    s_2s_1 &\in \A(\l,\mu) \Leftrightarrow \tfrac{c_2-m}{2}-n-1 \geq 0 \mbox{ and } -c_1-n-m-3 \geq 0,\\
    s_1s_2 &\in \A(\l,\mu) \Leftrightarrow \tfrac{-c_2-m}{2}-n-2 \geq 0 \mbox{ and } c_1-n-m-1 \geq 0,\\
    s_1(s_2s_1) &\in \A(\l,\mu) \Leftrightarrow \tfrac{-c_2-m}{2}-c_1-n-3 \geq 0 \mbox{ and } -c_1-n-m-3 \geq 0,\\
    s_2(s_1s_2) &\in \A(\l,\mu) \Leftrightarrow \tfrac{-c_2-m}{2}-n-2 \geq 0 \mbox{ and } -c_1-c_2-n-m-4 \geq 0,\\
    (s_2s_1)^2 &\in \A(\l,\mu) \Leftrightarrow \tfrac{-c_2-m}{2}-c_1-n-3\geq 0 \mbox{ and } -c_1-c_2-n-m-4 \geq 0.
\end{align*}
By Lemma \ref{ref:divlemB2},
intersecting these solution sets on the lattice $\ZZ{\w_1} \oplus 2\ZZ{\w_2}$ produces the desired results.
\end{proof}

% Now let us consider the Lie algebra of type $C_2$.

\subsection{Lie algebra of type \texorpdfstring{$C_2$}{C2}}
We begin by establishing a divisibility condition necessary for a weight in the fundamental weight lattice to be in the root lattice.
\begin{lemma}\label{ref:divlemC2}
Let $\lambda = c_1\w_1+c_2\w_2$ with $c_1, c_2 \in \ZZ$. Then $\lambda \in \ZZ{\a_1} \oplus \ZZ{\a_2}$ if and only if $2|c_1$.
\end{lemma}
\begin{proof}
Observe that $\lambda = c_1\w_1+c_2\w_2 = (c_1+c_2)\al_1+(\tfrac{c_1}{2}+c_2)\al_2$. Therefore, since $c_1, c_2 \in \ZZ$, then $(c_1+c_2) \in \ZZ$. Similarly,  $(\tfrac{c_1}{2}+c_2) = a$ for some $a \in \ZZ$ if and only if $c_1 = 2(a-c_2)$. Thus, $2|c_1$.
\end{proof}
\begin{table}[H]
    \centering
\begin{tabular}{|l|l|}\hline
    $L_1: c_1+c_2-n-m \geq 0 $&$L_2: \tfrac{c_1-n}{2}+c_2-m \geq 0 $\\\hline
    $L_3: c_2-n-m-1\geq 0 $ & $L_4: \tfrac{c_1-n}{2}-m-1 \geq 0 $\\\hline
    $L_5: \tfrac{-c_1-n}{2}-m-2\geq0 $ & $L_6: -c_2 - n - m - 3 \geq 0 $\\\hline
    $L_7:-c_1 - c_2 - n - m - 4\geq 0 $&$L_8: \tfrac{-c_1-n}{2}-c_2 - m - 3 \geq 0 $\\\hline
\end{tabular}
\caption{Conditions for Theorem \ref{thm:mainc2}.}
    \label{tab:c2}
\end{table}
We are now ready to describe the Weyl alternation sets for the Lie algebra of type $C_2$.
\begin{theorem}[Weyl alternation sets of  $\mathfrak{sp}_{4}(\mathbb{C})$]\label{thm:mainc2}
Let $\lambda = c_1\w_1+c_2\w_2$ for some $c_1,c_2 \in \ZZ$ with $2|c_1$. Fix $\mu = n\w_1+m\w_2$ with $n,m \in \NN$ and $2|n.$ To simplify notation we use inequalities in Table \ref{tab:c2} and let $\neg$ be the negation of an inequality. Then 
\begin{align*}
\A(\l, \mu) &=\begin{cases}    %1 element only
     \tikz\draw[onectwo, fill] (0,0) circle (2.5pt) 
; \quad \{1\} &\mbox{if } L_1, L_2, \neg L_4, \mbox{ and } \neg L_3 ,\\
%-------------------------------------------
   \tikz\draw[twoctwo, fill] (0,0) circle (2.5pt) 
;\quad    \{s_1\} &\mbox{if } L_3, L_2, \neg L_5, \mbox{ and } \neg L_1 ,\\
%-------------------------------------------
   \tikz\draw[threectwo, fill] (0,0) circle (2.5pt) 
;\quad   \{s_2\} &\mbox{if } L_1, L_4, \neg L_2, \mbox{ and } \neg L_6 ,\\
      %-------------------------------------------
      \tikz\draw[fourctwo, fill] (0,0) circle (2.5pt) 
;\quad \{s_2s_1\} &\mbox{if } L_3, L_5, \neg L_2, \mbox{ and } \neg L_7 ,\\
      %-------------------------------------------
    \tikz\draw[fivectwo, fill] (0,0) circle (2.5pt) 
;\quad  \{s_1s_2\} &\mbox{if } L_6, L_4, \neg L_8, \mbox{ and } \neg L_1 ,\\
      %-------------------------------------------
    \tikz\draw[sixctwo, fill] (0,0) circle (2.5pt) 
;\quad  \{s_1s_2s_1\} &\mbox{if } L_7, L_5, \neg L_8, \mbox{ and } \neg L_3 ,\\
      %-------------------------------------------
      \tikz\draw[sevenctwo, fill] (0,0) circle (2.5pt) 
;\quad \{s_2s_1s_2\} &\mbox{if } L_6, L_8, \neg L_4, \mbox{ and } \neg L_7 ,\\
      %-------------------------------------------
    \tikz\draw[eightctwo, fill] (0,0) circle (2.5pt) 
;\quad  \{(s_2s_1)^2\} &\mbox{if } L_7, L_8, \neg L_5, \mbox{ and } \neg L_6 ,\\
      %-------------------
      %2 elements
    \tikz\draw[ninectwo, fill] (0,0) circle (2.5pt) 
;\quad  \{1, s_1\} &\mbox{if } L_1, L_3, \neg L_4, \mbox{ and } \neg L_5 ,\\
    \tikz\draw[tenctwo, fill] (0,0) circle (2.5pt) 
;\quad  \{1, s_2\} &\mbox{if } L_2, L_4, \neg L_3, \mbox{ and } \neg L_6 ,\\
    \tikz\draw[elevenctwo, fill] (0,0) circle (2.5pt) 
;\quad  \{s_1, s_2s_1\} &\mbox{if } L_2, L_5, \neg L_1, \mbox{ and } \neg L_7 ,\\
    \tikz\draw[twelvectwo, fill] (0,0) circle (2.5pt) 
;\quad  \{s_2, s_1s_2\} &\mbox{if } L_1, L_6, \neg L_2, \mbox{ and } \neg L_8 ,\\
    \tikz\draw[thirteenctwo, fill] (0,0) circle (2.5pt) 
;\quad  \{s_2s_1, s_1s_2s_1\} &\mbox{if } L_3, L_7, \neg L_2, \mbox{ and } \neg L_8 ,\\
    \tikz\draw[fourteenctwo, fill] (0,0) circle (2.5pt) 
;\quad  \{s_1s_2, s_2s_1s_2\} &\mbox{if } L_4, L_8, \neg L_1, \mbox{ and } \neg L_7 ,\\
     \tikz\draw[fifteenctwo, fill] (0,0) circle (2.5pt) 
;\quad \{s_1s_2s_1, (s_2s_1)^2\} &\mbox{if } L_5, L_8, \neg L_6, \mbox{ and } \neg L_3 ,\\
     \tikz\draw[sixteenctwo, fill] (0,0) circle (2.5pt) 
;\quad \{s_2s_1s_2, (s_2s_1)^2\} &\mbox{if } L_6, L_7, \neg L_5, \mbox{ and } \neg L_4 ,\\
      %-----------------
      %3 elements
     \tikz\draw[seventeenctwo, fill] (0,0) circle (2.5pt) 
;\quad \{1, s_1, s_2s_1\} &\mbox{if } L_1, L_5, \neg L_4, \mbox{ and } \neg L_7 ,\\
     \tikz\draw[eighteenctwo, fill] (0,0) circle (2.5pt) 
;\quad \{1, s_1, s_2\} &\mbox{if } L_3, L_4, \neg L_5, \mbox{ and } \neg L_6 ,\\
      \tikz\draw[nineteenctwo, fill] (0,0) circle (2.5pt) 
;\quad \{1, s_2, s_1s_2\} &\mbox{if } L_2, L_6, \neg L_3, \mbox{ and } \neg L_8 ,\\
     \tikz\draw[twentyctwo, fill] (0,0) circle (2.5pt) 
;\quad  \{s_2, s_1s_2, s_2s_1s_2\} &\mbox{if } L_1, L_8, \neg L_2, \mbox{ and } \neg L_7 ,\\
     \tikz\draw[twenty-onectwo, fill] (0,0) circle (2.5pt) 
;\quad \{s_1s_2, s_2s_1s_2, (s_2s_1)^2\} &\mbox{if } L_7, L_4, \neg L_5, \mbox{ and } \neg L_1 ,\\
     \tikz\draw[twenty-twoctwo, fill] (0,0) circle (2.5pt) 
;\quad \{s_2s_1s_2, (s_2s_1)^2, s_1s_2s_1\} &\mbox{if } L_6, L_5, \neg L_4, \mbox{ and } \neg L_3 ,\\
     \tikz\draw[twenty-threectwo, fill] (0,0) circle (2.5pt) 
;\quad \{(s_2s_1)^2, s_1s_2s_1, s_2s_1\} &\mbox{if } L_3, L_8, \neg L_2, \mbox{ and } \neg L_6 ,\\
     \tikz\draw[twenty-fourctwo, fill] (0,0) circle (2.5pt)
;\quad \{s_1, s_2s_1, s_1s_2s_1\} &\mbox{if } L_7, L_2, \neg L_8, \mbox{ and } \neg L_1 ,\mbox{ and}\\
      \phantom{\tikz\draw[twenty-fourctwo, fill] (0,0) circle (2.5pt);}\quad \emptyset & \mbox{otherwise}.
\end{cases}
\end{align*}
\end{theorem}
\begin{proof}
Let $\l = c_1\w_1+c_2\w_2$ for some $c_1, c_2 \in \ZZ$ and $\mu = n\w_1+m\w_2$ with $n,m \in \mathbb{N}$.
\begin{align*}
    1(\l + \rho) - (\mu + \rho) &= (c_1+c_2-n-m)\a_1 + (\tfrac{c_1-n}{2}+c_2-m)\a_2,\\
    s_1(\l + \rho) - (\mu + \rho) &= (c_2-n-m-1)\al_1 + (\tfrac{c_1-n}{2}+c_2-m)\al_2,\\
    s_2(\l + \rho) - (\mu + \rho) &= (c_1+c_2-n-m)\al_1+(\tfrac{c_1-n}{2}-m-1)\al_2,\\
    s_2s_1(\l + \rho) - (\mu + \rho) &= (c_2-n-m-1)\a_1 + (\tfrac{-c_1-n}{2}-m-2)\a_2,\\
    s_1s_2(\l + \rho) - (\mu + \rho) &= (-c_2-n-m-3)\a_1 + (\tfrac{c_1-n}{2}-m-1)\a_2,\\
    s_1(s_2s_1)(\l + \rho) - (\mu + \rho) &= (-c_1-c_2-n-m-4)\al_1 + (\tfrac{-c_1-n}{2}-m-2)\al_2,\\
    s_2(s_1s_2)(\l + \rho) - (\mu + \rho) &= (-c_2 -n -m - 3)\al_1+(\tfrac{-c_1-n}{2}-c_2-m - 3)\al_2,\\
    (s_2s_1)^2(\l + \rho) - (\mu + \rho) &= (-c_1 - c_2 - n -m - 4)\al_1+(\tfrac{-c_1 - n}{2} - c_2 - m - 3)\al_2.
\end{align*}
From the equations above and the definition of a Weyl alternation set, it follows that 
\begin{align*}
    1 &\in \A(\l,\mu) \Leftrightarrow c_1+c_2-n-m \geq 0 \mbox{ and } \tfrac{c_1-n}{2}+c_2-m \geq 0,\\
    s_1 &\in \A(\l,\mu) \Leftrightarrow c_2-n-m-1 \geq 0 \mbox{ and } \tfrac{c_1-n}{2}+c_2-m \geq 0, \\ 
    s_2 &\in \A(\l,\mu) \Leftrightarrow c_1+c_2-n-m \geq 0 \mbox{ and } \tfrac{c_1-n}{2}-m-1 \geq 0,\\
    s_2s_1 &\in \A(\l,\mu) \Leftrightarrow c_2-n-m-1 \geq 0 \mbox{ and } \tfrac{-c_1-n}{2}-m-2 \geq 0,\\
    s_1s_2 &\in \A(\l,\mu) \Leftrightarrow -c_2-n-m-3 \geq 0 \mbox{ and } \tfrac{c_1-n}{2}-m-1 \geq 0,\\
    s_1(s_2s_1) &\in \A(\l,\mu) \Leftrightarrow -c_1 - c_2 - n -m - 4 \geq 0 \mbox{ and } \tfrac{-c_1-n}{2}-m-2 \geq 0,\\
    s_2(s_1s_2) &\in \A(\l,\mu) \Leftrightarrow -c_2-n-m-3\geq 0 \mbox{ and } \tfrac{-c_1 - n}{2} - c_2 - m - 3 \geq 0,\\
    (s_2s_1)^2 &\in \A(\l,\mu) \Leftrightarrow -c_1 - c_2 - n -m - 4\geq 0 \mbox{ and } \tfrac{-c_1 - n}{2} - c_2 - m - 3 \geq 0.
\end{align*}
By Lemma \ref{ref:divlemC2},
intersecting these solution sets on the lattice $2\ZZ{\w_1} \oplus \ZZ{\w_2}$ produces the desired results.
\end{proof}

\subsection{Lie algebra of type \texorpdfstring{$D_2$}{D2}}
We begin by establishing a divisibility condition necessary for a weight in the fundamental weight lattice to be in the root lattice.
\begin{lemma}\label{lemma:divd2}
Let $\lambda = c_1\w_1+c_2\w_2$ with $c_1,c_2 \in \ZZ$. Then $\lambda \in \ZZ{\a_1} \oplus \ZZ{\a_2}$ if and only if $2|c_1$ and $2|c_2$.
\end{lemma}
\begin{proof}
Observe that $\lambda = c_1\w_1+c_2\w_2 = (\tfrac{c_1}{2})\al_1+(\tfrac{c_2}{2})\al_2$. Therefore, $\lambda \in \ZZ{\a_1} \oplus \ZZ{\a_2}$ precisely when $(\tfrac{c_1}{2})=x$ for some x $\in \ZZ$ and $(\tfrac{c_2}{2})=y$ for some y $\in \ZZ$. Hence, $2|c_1$ and $2|c_2$.
\end{proof}
We are now ready to describe the Weyl alternation sets for the Lie algebra of type $D_2$.
\begin{theorem}[Weyl alternation sets of  $\mathfrak{so}_{4}(\mathbb{C})$]\label{thm:maind2}
Let $\lambda = c_1\w_1+c_2\w_2$ for some $c_1,c_2 \in 2\ZZ$ and fix $\mu = n\w_1+m\w_2$ with $n,m \in 2\mathbb{N}$. Then 
\begin{align*}
\A(\l, \mu) &=\begin{cases}    %1 element only
      \tikz\draw[onedtwo, fill] (0,0) circle (2.5pt); \quad \{1\} &\mbox{if } \tfrac{c_1-n}{2} \geq 0,\mbox{ and } \tfrac{c_2-m}{2} \geq 0,\\
      %-------------------------------------------------------
      \tikz\draw[twodtwo, fill] (0,0) circle (2.5pt); \quad \{s_1\} &\mbox{if } \tfrac{-c_1-n-2}{2} \geq 0 \mbox{ and } \tfrac{c_2-m}{2} \geq 0,\\
      %-----------------------------------
       \tikz\draw[threedtwo, fill] (0,0) circle (2.5pt); \quad\{s_2\} &\mbox{if } \tfrac{c_1-n}{2} \geq 0, \mbox{ and } \tfrac{-c_2-m-2}{2} \geq 0,\\
      %----------------------------------
      \tikz\draw[fourdtwo, fill] (0,0) circle (2.5pt); \quad \{s_2s_1\} &\mbox{if } \tfrac{-c_1-n-2}{2} \geq 0, \mbox{ and } \tfrac{-c_2-m-2}{2} \geq 0 \mbox{, and}\\
      \phantom{\tikz\draw[twenty-fourctwo, fill] (0,0) circle (2.5pt);} \quad \emptyset &\mbox{otherwise}.
\end{cases}
\end{align*}
\end{theorem}
\begin{proof}
Let $\l = c_1\w_1+c_2\w_2$ for some $c_1, c_2 \in \ZZ$ and $\mu = n\w_1+m\w_2$ with $n,m \in \mathbb{N}$. Thus
\begin{align*}
    1(\l + \rho) - (\mu + \rho) &= (\tfrac{c_1-n}{2})\a_1 + (\tfrac{c_2-m}{2})\a_2,\\
    s_1(\l + \rho) - (\mu + \rho) &= (\tfrac{-c_1-n-2}{2})\a_1 + (\tfrac{c_2-m}{2})\a_2,\\
    s_2(\l + \rho) - (\mu + \rho) &= (\tfrac{c_1-n}{2})\a_1 + (\tfrac{-c_2-m-2}{2})\a_2,\\
    s_2s_1(\l + \rho) - (\mu + \rho) &= (\tfrac{-c_1-n-2}{2})\a_1 + (\tfrac{-c_2-m-2}{2})\a_2.
\end{align*}
From the equations above and the definition of a Weyl alternation set, it follows that 
\begin{align*}
    1 &\in \A(\l,\mu) \Leftrightarrow \tfrac{c_1-n}{2} \geq 0 \mbox{ and } \tfrac{c_1-m}{2} \geq 0,\\
    s_1 &\in \A(\l,\mu) \Leftrightarrow \tfrac{-c_1-n-2}{2} \geq 0 \mbox{ and } \tfrac{c_1-m}{2} \geq 0, \\ 
    s_2 &\in \A(\l,\mu) \Leftrightarrow \tfrac{c_1-n}{2} \geq 0 \mbox{ and }  \tfrac{-c_2-m-2}{2}  \geq 0,\\ 
    s_2s_1 &\in \A(\l,\mu) \Leftrightarrow \tfrac{-c_1-n-2}{2} \geq 0 \mbox{ and }  \tfrac{-c_2-m-2}{2}\geq 0.  
\end{align*}
By Lemma \ref{lemma:divd2}, intersecting these solution sets on the lattice $2\ZZ{\w_1} \oplus 2\ZZ{\w_2}$ produces the desired results.
\end{proof}

\subsection{Lie algebra of type \texorpdfstring{$G_2$}{G2}}
Given that in the Lie algebra of type $G_2$ the root lattice and fundamental weight lattice are equivalent, we do not have a divisibility condition in this section. 
\begin{table}[H]
    \centering
\begin{tabular}{|l|l|l|}\hline
    $K_1: c_1-n \geq 0$&$K_2: c_2-m \geq 0$&$K_3: -c_1+3c_2-n -1 \geq 0 $\\\hline
    $K_4: c_1-c_2-m-1 \geq 0$&$K_5: 2c_1 - 3c_2 - n - 4 \geq 0 $&$K_6: -c_1 + 2c_2 - m - 2\geq0$\\\hline
    $K_7: -c_1 -n - 10 \geq 0$&$K_8: -c_2 - m -6\geq 0$&$K_9: c_1 - 3c_2 - n - 9 \geq 0$\\\hline
    $K_{10}:-c_1 + c_2 - m - 5 \geq 0$&$K_{11}:-2c_1 + 3c_2 - n - 6 \geq 0$&$K_{12}:c_1 - 2c_2 - m - 4 \geq 0$\\\hline
\end{tabular}
\caption{Conditions for Theorem \ref{thm:maing2}}
    \label{tab:g2}
\end{table}

We are now ready to describe the Weyl alternation sets for the Lie algebra of type $G_2$.
\begin{theorem}[Weyl alternation sets of  $\mathfrak{g}_{2}$]
\label{thm:maing2}
Let $\l = c_1\al_1+c_2\al_2$ for some $c_1, c_2 \in \ZZ$ and fix $\mu = n\al_1+m\al_2$ with $n,m \in \mathbb{N}$. To simplify notation we let $\neg$ be the negation of an inequality in Table \ref{tab:g2}. Then
\begin{longtable}{p{4in}p{4in}}
% \begin{tabular}
\tikz\draw[one, fill] (0,0) circle (2.5pt); $\A(\l, \mu) =\{1\}$ &if $K_1$, $K_2$, $\neg K_3$, and $\neg K_4$,
\\%-------------------------------------------
\tikz\draw[two, fill] (0,0) circle (2.5pt) 
; $\A(\l, \mu) =\{s_1\}$ &if  $K_3$, $K_2$, $\neg K_6$, and $\neg K_1$,
\\%-------------------------------------------
\tikz\draw[three, fill] (0,0) circle (2.5pt);
$\A(\l, \mu) =\{s_2\}$ &if  $K_1$, $K_4$, $\neg K_2$,  and  $\neg K_5$,
\\%-------------------------------------------
\tikz\draw[four, fill] (0,0) circle (2.5pt);
$\A(\l, \mu) =\{s_1s_2\}$ &if $K_5$, $K_4$,  $\neg K_{12}$,  and $ \neg K_1$,
\\%-------------------------------------------
\tikz\draw[five, fill] (0,0) circle (2.5pt);
$\A(\l, \mu) =\{s_2s_1\}$ &if $K_3$, $K_6$,   $\neg K_{11}$,  and  $\neg K_2$,
\\%-------------------------------------------
\tikz\draw[six, fill] (0,0) circle (2.5pt);
$\A(\l, \mu) =\{s_1s_2s_1\}$ &if $K_{11}$, $K_6$,  $\neg K_{10}$,  and $\neg K_3$,
\\%-------------------------------------------
\tikz\draw[seven, fill] (0,0) circle (2.5pt) 
; $\A(\l, \mu) =\{s_2s_1s_2\}$ &if $K_5$, $K_{12}$,  $\neg K_4$,  and $\neg K_9$, 
\\%-------------------------------------------
\tikz\draw[eight, fill] (0,0) circle (2.5pt) 
; $\A(\l, \mu) =\{(s_1s_2)^2\}$ &if $K_9$, $K_{12}$,  $\neg K_8$,  and $\neg K_5$,
\\%-------------------------------------------
\tikz\draw[nine, fill] (0,0) circle (2.5pt) 
; $\A(\l, \mu) =\{(s_2s_1)^2\}$ &if $K_{11}$, $K_{10}$, $\neg K_6$,  and $\neg K_7$,
\\%-------------------------------------------
\tikz\draw[ten, fill] (0,0) circle (2.5pt) 
; $\A(\l, \mu) =\{s_1(s_2s_1)^2\}$ &if $K_7$, $K_{10}$, $\neg K_8$,  and $\neg K_{11}$,
\\%-------------------------------------------
\tikz\draw[eleven, fill] (0,0) circle (2.5pt);
$\A(\l, \mu) =\{s_2(s_1s_2)^2\}$ &if $K_9$, $K_8$, $\neg K_{12}$,  and $\neg K_7$,
\\%-------------------------------------------
\tikz\draw[twelve, fill] (0,0) circle (2.5pt) 
; $\A(\l, \mu)=\{(s_1s_2)^3\}$ &if $K_7$, $K_8$, $\neg K_{10}$,  and $\neg K_9$,
\\%-------------------------------------------
%2 elements
\tikz\draw[thirteen, fill] (0,0) circle (2.5pt) 
; $\A(\l, \mu) =\{1, s_1\}$ &if $K_1$, $K_3$, $\neg K_6$,  and $\neg K_4$,
\\%-------------------------------------------
\tikz\draw[fourteen, fill] (0,0) circle (2.5pt) 
; $\A(\l, \mu) =\{1,s_2\}$ &if $K_2$, $K_4$, $\neg K_3$,  and $\neg K_5$,
\\%-------------------------------------------
\tikz\draw[fifteen, fill] (0,0) circle (2.5pt) 
; $\A(\l, \mu) =\{s_1,s_2s_1\}$ &if $K_2$, $K_6$, $\neg K_1$,  and $\neg K_{11}$,
\\%-------------------------------------------
\tikz\draw[sixteen, fill] (0,0) circle (2.5pt) 
; $\A(\l, \mu) =\{s_2, s_1s_2\}$ &if $K_1$, $K_5$, $\neg K_2$,  and $\neg K_{12}$,
\\%-------------------------------------------
\tikz\draw[seventeen, fill] (0,0) circle (2.5pt) 
; $\A(\l, \mu) =\{s_1s_2, s_2s_1s_2\}$ &if $K_4$, $K_{12}$, $\neg K_1$,  and $\neg K_9$,
\\%-------------------------------------------
\tikz\draw[eighteen, fill] (0,0) circle (2.5pt) 
; $\A(\l, \mu) =\{s_2s_1, s_1s_2s_1\}$ &if $K_3$, $K_{11}$, $\neg K_2$,  and $\neg K_{10}$,
\\%-------------------------------------------
\tikz\draw[nineteen, fill] (0,0) circle (2.5pt) 
; $\A(\l, \mu) =\{s_1s_2s_1, (s_2s_1)^2\}$ &if $K_6$, $K_{10}$, $\neg K_3$,  and $\neg K_7$,
\\%-------------------------------------------
\tikz\draw[twenty, fill] (0,0) circle (2.5pt) 
; $\A(\l, \mu) =\{s_2s_1s_2, (s_1s_2)^2\}$ &if $K_5$, $K_9$, $\neg K_4$,  and $\neg K_8$,
\\%-------------------------------------------
\tikz\draw[twenty-one, fill] (0,0) circle (2.5pt) 
; $\A(\l, \mu) =\{(s_1s_2)^2, s_2(s_1s_2)^2\}$ &if  $K_{12}$, $K_8$, $\neg K_5$,  and $\neg K_7$,
\\%-------------------------------------------
\tikz\draw[twenty-two, fill] (0,0) circle (2.5pt) 
; $\A(\l, \mu) =\{(s_2s_1)^2, s_1(s_2s_1)^2\}$ &if $K_{11}$, $K_7$, $\neg K_6$,  and $\neg K_8$,
\\%-------------------------------------------
\tikz\draw[twenty-three, fill] (0,0) circle (2.5pt) 
; $\A(\l, \mu) =\{s_1(s_2s_1)^2, (s_1s_2)^3\}$ &if  $K_{10}$, $K_8$, $\neg K_{11}$,  and $\neg K_9$,
\\%-------------------------------------------
\tikz\draw[twenty-four, fill] (0,0) circle (2.5pt) 
; $\A(\l, \mu) =\{s_2(s_1s_2)^2, (s_1s_2)^3\}$ &if $K_9$, $K_7$, $\neg K_{12}$,  and $\neg K_{10}$,
\\%------------------------------------------------
%3 elements
\tikz\draw[twenty-five, fill] (0,0) circle (2.5pt) 
; $\A(\l, \mu) =\{1,s_1,s_2\}$ &if $K_3$, $K_4$, $\neg K_6$, and $\neg K_5$,
\\%------------------------------------------------
\tikz\draw[twenty-six, fill] (0,0) circle (2.5pt) 
; $\A(\l, \mu) =\{1,s_1, s_2s_1\}$ &if $K_1$, $K_6$, $\neg K_4$, and $\neg K_{11}$,
\\%------------------------------------------------
\tikz\draw[twenty-seven, fill] (0,0) circle (2.5pt) 
; $\A(\l, \mu) =\{1, s_2, s_1s_2\}$ &if $K_2$, $K_5$, $\neg K_3$,  and $\neg K_{12}$, 
\\%------------------------------------------------
\tikz\draw[twenty-eight, fill] (0,0) circle (2.5pt) 
; $\A(\l, \mu) =\{s_1,s_2s_1,s_1s_2s_1\}$ &if $K_2$, $K_{11}$, $\neg K_{10}$,  and $\neg K_1$,
\\%------------------------------------------------
\tikz\draw[twenty-nine, fill] (0,0) circle (2.5pt) 
; $\A(\l, \mu) =\{s_2s_1, (s_2s_1)^2, s_1s_2s_1\}$ &if $K_3$, $K_{10}$, $\neg K_2$, and $\neg K_7$,
\\%------------------------------------------------
\tikz\draw[thirty, fill] (0,0) circle (2.5pt) 
; $\A(\l, \mu) =\{(s_2s_1)^2, s_1s_2s_1, s_1(s_2s_1)^2\}$ &if $K_7$, $K_6$, $\neg K_8$,  and $\neg K_3$,
\\%------------------------------------------------
\tikz\draw[thirty-one, fill] (0,0) circle (2.5pt) 
; $\A(\l, \mu) =\{(s_2s_1)^2, (s_1s_2)^3, s_1(s_2s_1)^2\}$ &if$K_{11}$, $K_8$, $\neg K_6$,  and $\neg K_9$,
\\%------------------------------------------------
\tikz\draw[thirty-two, fill] (0,0) circle (2.5pt) 
; $\A(\l, \mu) =\{(s_1s_2)^3, s_1(s_2s_1)^2, s_2(s_1s_2)^2\}$ &if $K_9$, $K_{10}$, $\neg K_{12}$,  and $\neg K_{11}$,
\\%------------------------------------------------
\tikz\draw[thirty-three, fill] (0,0) circle (2.5pt) 
; $\A(\l, \mu) =\{(s_1s_2)^3, s_2(s_1s_2)^2, (s_1s_2)^2\}$ &if $K_7$, $K_{12}$, $\neg K_{10}$,  and $\neg K_5$,
\\%------------------------------------------------
\tikz\draw[thirty-four, fill] (0,0) circle (2.5pt) 
; $\A(\l, \mu) =\{(s_1s_2)^2, s_2(s_1s_2)^2, s_2s_1s_2\}$ &if $K_8$, $K_5$, $\neg K_4$, and $\neg K_7$ ,
\\%------------------------------------------------
\tikz\draw[thirty-five, fill] (0,0) circle (2.5pt) 
; $\A(\l, \mu) =\{(s_1s_2)^2, s_1s_2, s_2s_1s_2\}$ &if $K_9$, $K_4$, $\neg K_8$, and $\neg K_1$,
\\%------------------------------------------------
\tikz\draw[thirty-six, fill] (0,0) circle (2.5pt) 
; $\A(\l, \mu) =\{s_2, s_1s_2,s_2s_1s_2\}$ &if $K_1$, $K_{12}$, $\neg K_9$,  and $\neg K_2$,
 \\%-----------------------------------------------
%     %4 elements
\tikz\draw[thirty-seven, fill] (0,0) circle (2.5pt) 
; $\A(\l, \mu) =\{1,s_1,s_2,s_1s_2\}$ & if $K_3$, $K_5$, $\neg K_6$, and $\neg K_{12}$, 
\\%-----------------------------------------------
\tikz\draw[thirty-eight, fill] (0,0) circle (2.5pt); 
$\A(\l, \mu) =\{1,s_1,s_2,s_2s_1\}$ &if $K_4$, $K_6$, $\neg K_5$,  and $\neg K_{11}$,
\\%-----------------------------------------------
\tikz\draw[thirty-nine, fill] (0,0) circle (2.5pt) 
; $\A(\l, \mu) =\{1,s_1,s_2s_1,s_1s_2s_1\}$ & if $K_1$, $K_{11}$, $\neg K_4$,  and $\neg K_{10}$,
\\%-----------------------------------------------
\tikz\draw[forty, fill] (0,0) circle (2.5pt) 
; $\A(\l, \mu) =\{s_1,s_2s_1, (s_2s_1)^2, s_1s_2s_1\}$ &if $K_2$, $K_{10}$, $\neg K_1$, and $\neg K_7$,
\\%-----------------------------------------------
\tikz\draw[forty-one, fill] (0,0) circle (2.5pt) 
; $\A(\l, \mu) =\{s_2s_1,(s_2s_1)^2, s_1s_2s_1, s_1(s_2s_1)^2\}$ &if $K_3$, $K_7$, $\neg K_2$, and $\neg K_8$,
\\%-----------------------------------------------
\tikz\draw[forty-two, fill] (0,0) circle (2.5pt) 
; $\A(\l, \mu) =\{ (s_2s_1)^2, (s_1s_2)^3,s_1s_2s_1,s_1(s_2s_1)^2\}$ &if $K_8$, $K_6$, $\neg K_9$,  and $\neg K_3$,
\\%-----------------------------------------------
\tikz\draw[forty-three, fill] (0,0) circle (2.5pt) 
; $\A(\l, \mu) =\{(s_2s_1)^2, (s_1s_2)^3, s_1(s_2s_1)^2, s_2(s_1s_2)^2\}$ &if $K_{11}$, $K_9$, $\neg K_6$,  and $\neg K_{12}$,
\\%-----------------------------------------------
\tikz\draw[forty-four, fill] (0,0) circle (2.5pt) 
; $\A(\l, \mu) =\{(s_1s_2)^3, (s_1s_2)^2, s_1(s_2s_1)^2, s_2(s_1s_2)^2\}$ &if $K_{10}$, $K_{12}$, $\neg K_{11}$,  and $\neg K_5$,
\\%-----------------------------------------------
\tikz\draw[forty-five, fill] (0,0) circle (2.5pt) 
; $\A(\l, \mu) =\{(s_1s_2)^3, (s_1s_2)^2, s_2(s_1s_2)^2, s_2s_1s_2\}$ &if $K_7$, $K_5$, $\neg K_{10}$, and $\neg K_4$,
\\%-----------------------------------------------
\tikz\draw[forty-six, fill] (0,0) circle (2.5pt) 
; $\A(\l, \mu) =\{(s_1s_2)^2, s_1s_2, s_2(s_1s_2)^2, s_2s_1s_2\}$ &if $K_8$, $K_4$, $\neg K_7$, and $\neg K_1$,
\\%-----------------------------------------------
\tikz\draw[forty-seven, fill] (0,0) circle (2.5pt) 
; $\A(\l, \mu) =\{s_2, (s_1s_2)^2, s_1s_2, s_2s_1s_2\}$ &if $K_1$, $K_9$, $\neg K_2$, and $\neg K_8$,
\\%-----------------------------------------------
\tikz\draw[forty-eight, fill] (0,0) circle (2.5pt) 
; $\A(\l, \mu) =\{1, s_2, s_1s_2, s_2s_1s_2\}$ &if $K_2$, $K_{12}$, $\neg K_3$,  and $\neg K_9$,
 \\%--------------------------------------------------------
% %5 elements
\tikz\draw[forty-nine, fill] (0,0) circle (2.5pt) 
; $\A(\l, \mu) =\{1,s_1,s_2,s_2s_1,s_1s_2\}$ &if $K_6$, $K_5$, $\neg K_{11}$,  and $\neg K_{12}$,
\\%--------------------------------------------------------
\tikz\draw[fifty, fill] (0,0) circle (2.5pt) 
; $\A(\l, \mu) =\{1, s_1, s_2, s_2s_1, s_1s_2s_1\}$ &if $K_4$, $K_{11}$, $\neg K_{10}$,  and $\neg K_5$,
\\%--------------------------------------------------------
\tikz\draw[fifty-one, fill] (0,0) circle (2.5pt) 
; $\A(\l, \mu) =\{1, s_1, s_2s_1, (s_2s_1)^2, s_1s_2s_1\}$ &if $K_1$, $K_{10}$, $\neg K_4$, and $\neg K_7$,
\\%--------------------------------------------------------
\tikz\draw[fifty-two, fill] (0,0) circle (2.5pt) 
; $\A(\l, \mu) =\{s_1, s_2s_1, (s_2s_1)^2, s_1s_2s_1, s_1(s_2s_1)^2\}$ &if $K_2$, $K_7$, $\neg K_1$,  and $\neg K_8$,
\\%--------------------------------------------------------
\tikz\draw[fifty-three, fill] (0,0) circle (2.5pt) 
; $\A(\l, \mu) =\{s_2s_1, (s_2s_1)^2, (s_1s_2)^3, s_1s_2s_1, s_1(s_2s_1)^2\}$ &if $K_3$, $K_8$, $\neg K_2$, and  $\neg K_9$,
\\%--------------------------------------------------------
\tikz\draw[fifty-four, fill] (0,0) circle (2.5pt) 
; $\A(\l, \mu) =\{(s_2s_1)^2, (s_1s_2)^3, s_1s_2s_1, s_1(s_2s_1)^2, s_2(s_1s_2)^2\}$ &if $K_6$, $K_9$, $\neg K_3$,  and $\neg K_{12}$,
\\%--------------------------------------------------------
\tikz\draw[fifty-five, fill] (0,0) circle (2.5pt) 
; $\A(\l, \mu) =\{(s_2s_1)^2, (s_1s_2)^3, (s_1s_2)^2, s_1(s_2s_1)^2, s_2(s_1s_2)^2\}$ &if $K_{11}$, $K_{12}$, $\neg K_6$,  and $\neg K_5$,
\\%--------------------------------------------------------
\tikz\draw[fifty-six, fill] (0,0) circle (2.5pt) 
; $\A(\l, \mu) =\{(s_1s_2)^3, (s_1s_2)^2, s_1(s_2s_1)^2, s_2(s_1s_2)^2, s_2s_1s_2\}$ &if $K_{10}$, $K_5$, $\neg K_{11}$, and $\neg K_4$,
\\%--------------------------------------------------------
\tikz\draw[fifty-seven, fill] (0,0) circle (2.5pt) 
; $\A(\l, \mu) =\{(s_1s_2)^3, (s_1s_2)^2, s_1s_2, s_2(s_1s_2)^2, s_2s_1s_2\}$ &if $K_4$, $K_7$, $\neg K_{10}$, and $\neg K_1$,
\\%--------------------------------------------------------
\tikz\draw[fifty-eight, fill] (0,0) circle (2.5pt);
$\A(\l, \mu) =\{s_2, (s_1s_2)^2, s_1s_2, s_2(s_1s_2)^2, s_2s_1s_2\}$ &if $K_1$, $K_8$, $\neg K_2$, and $\neg K_7$,
\\%--------------------------------------------------------
\tikz\draw[fifty-nine, fill] (0,0) circle (2.5pt);
 $\A(\l, \mu) =\{1, s_2, (s_1s_2)^2, s_1s_2, s_2s_1s_2\}$ &if $K_9$, $K_2$, $\neg K_8$, and $\neg K_3$,
\\%--------------------------------------------------------
\tikz\draw[sixty, fill] (0,0) circle (2.5pt);
$\A(\l, \mu)=\{1, s_1, s_2, s_1s_2, s_2s_1s_2\}$ &if $K_{12}$, $K_3$, $\neg K_6$, and $\neg K_9$,  and
\\%--------------------------------------------------------
\phantom{\tikz\draw[twenty-fourctwo, fill] (0,0) circle (2.5pt);} $\A(\l, \mu)\,=\,\emptyset$ & otherwise.
% \end{tabular}
\end{longtable}
\end{theorem}
\begin{proof}
Let $\l = c_1\al_1+c_2\al_2$ for some $c_1, c_2 \in \ZZ$ and $\mu = n\al_1+m\al_2$ with $n,m \in \mathbb{N}$. Thus
\begin{align*}
    1(\l + \rho) - (\mu + \rho) &= (c_1 - n)\a_1 + (c_2 - m)\a_2,\\
    s_1(\l + \rho) - (\mu + \rho) &= (-c_1+3c_2-n-1)\al_1 + (c_2-m)\al_2,\\
    s_2(\l + \rho) - (\mu + \rho) &= (c_1-n)\al_1+(c_1-c_2-m-1)\al_2,\\
    s_2s_1(\l + \rho) - (\mu + \rho) &= (-c_1 + 3c_2 - n - 1)\a_1 + (-c_1 + 2c_2 - m - 2)\a_2,\\
    s_1s_2(\l + \rho) - (\mu + \rho) &= (2c_1 -3c_2 - n - 4)\a_1 + (c_1 - c_2 - m - 1)\a_2,\\
    s_1(s_2s_1)(\l + \rho) - (\mu + \rho) &= (-2c_1+3c_2-n-6)\al_1 + (-c_1+2c_2-m-2)\al_2,\\
    s_2(s_1s_2)(\l + \rho) - (\mu + \rho) &= (2c_1-3c_2-n-4)\al_1+(c_1-2c_2-m-4)\al_2,\\
    (s_2s_1)^2(\l + \rho) - (\mu + \rho) &= (-2c_1+3c_2-n-6)\al_1+(-c_1+c_2-m-5)\al_2,\\
    (s_1s_2)^2(\l + \rho) - (\mu + \rho) &= (c_1-3c_2-n-9)\al_1+(c_1-2c_2-m-4)\al_2,\\
    s_1(s_2s_1)^2(\l + \rho) - (\mu + \rho) &= (-c_1 - n - 10)\al_1+(-c_1 + c_2-m-5)\al_2,\\
    s_2(s_1s_2)^2(\l + \rho) - (\mu + \rho) &= (c_1-3c_2-n-9)\al_1+(-c_2-m-6)\al_2,\\
    (s_1s_2)^3(\lambda + \rho) - (\mu + \rho) &= (-c_1-n-10)\al_1 + (-c_2-m-6)\al_2.
\end{align*}
From the equations above and Definition \ref{def:Weylaltset}, it follows that 
\begin{align*}
    1 &\in \A(\l,\mu) \Leftrightarrow c_1 - n \geq 0 \mbox{ and } c_2 - m \geq 0,\\
    s_1 &\in \A(\l,\mu) \Leftrightarrow -c_1 + 3c_2 - n - 1 \geq 0 \mbox{ and } c_2 - m \geq 0, \\ 
    s_2 &\in \A(\l,\mu) \Leftrightarrow c_1-n \geq 0 \mbox{ and } c_1-c_2-m-1 \geq 0,\\
    s_2s_1 &\in \A(\l,\mu) \Leftrightarrow -c_1 + 3c_2 - n - 1 \geq 0 \mbox{ and } -c_1 + 2c_2 - m - 2 \geq 0,\\
    s_1s_2 &\in \A(\l,\mu) \Leftrightarrow 2c_1 -3c_2 - n - 4 \geq 0 \mbox{ and } c_1 - c_2 - m - 1 \geq 0,\\
    s_1(s_2s_1) &\in \A(\l,\mu) \Leftrightarrow -2c_1+3c_2-n-6 \geq 0 \mbox{ and } -c_1+2c_2-m-2 \geq 0,\\
    s_2(s_1s_2) &\in \A(\l,\mu) \Leftrightarrow 2c_1-3c_2-n-4 \geq 0 \mbox{ and } c_1-2c_2-m-4 \geq 0,\\
    (s_2s_1)^2 &\in \A(\l,\mu) \Leftrightarrow -2c_1+3c_2-n-6\geq 0 \mbox{ and } -c_1+c_2-m-5 \geq 0,\\
    (s_1s_2)^2 &\in \A(\l,\mu) \Leftrightarrow c_1-3c_2-n-9 \geq 0 \mbox{ and } c_1-2c_2-m-4 \geq 0,\\
    s_1(s_2s_1)^2 &\in \A(\l,\mu) \Leftrightarrow -c_1 - n - 10 \geq 0 \mbox{ and } -c_1 + c_2-m-5 \geq 0,\\
    s_2(s_1s_2)^2 &\in \A(\l,\mu) \Leftrightarrow c_1-3c_2-n-9 \geq 0 \mbox{ and } -c_2-m-6 \geq 0, \\
    (s_1s_2)^3 &\in \A(\l,\mu) \Leftrightarrow -c_1-n-10 \geq 0 \mbox{ and } -c_2-m-6 \geq 0.
\end{align*}
Intersecting these solution sets on the $\ZZ{\al_1} \oplus \ZZ{\al_2}$ lattice produces the desired results.
\end{proof}

\section{Weyl alternation diagrams}\label{sec:diagrams}
Theorems \ref{thm:mainb2}, \ref{thm:mainc2}, \ref{thm:maind2}, and \ref{thm:maing2} establish inequalities which describe when certain elements of the Weyl group appear in the Weyl alternation set $\A(\lambda,\mu)$.
From a geometric point of view, we are finding the regions of a plane in which $\sigma(\lambda + \rho)-(\mu+\rho)$ is a nonnegative integral linear combination of the simple roots. 
\begin{figure}[H]%
    \centering
    \subfloat[ $B_2$]{{\includegraphics[width=1.5in]{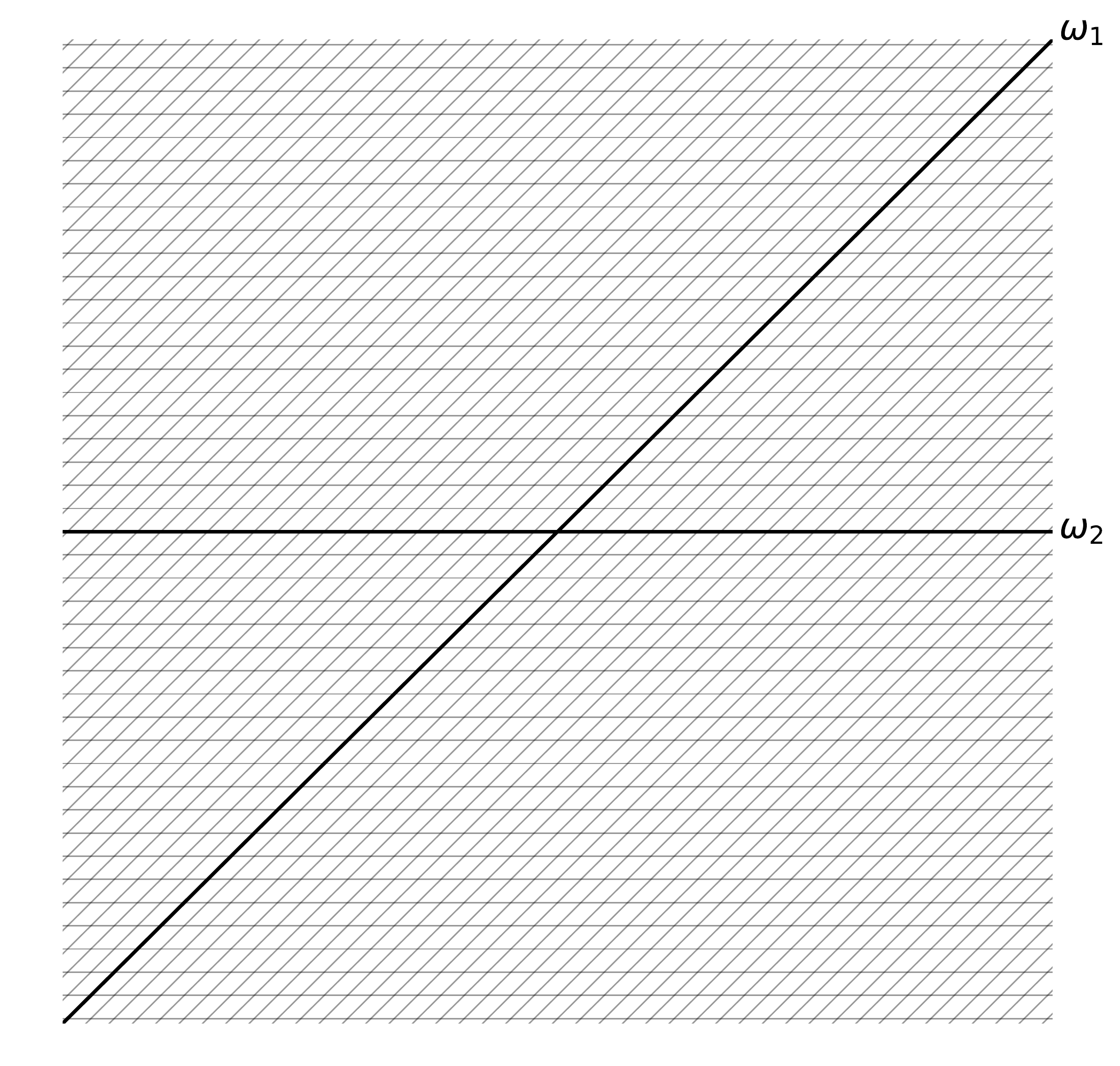}}
    \label{subfig:b2_grid}
    }%
    \hfill
    \subfloat[$C_2$]{{\includegraphics[width=1.5in]{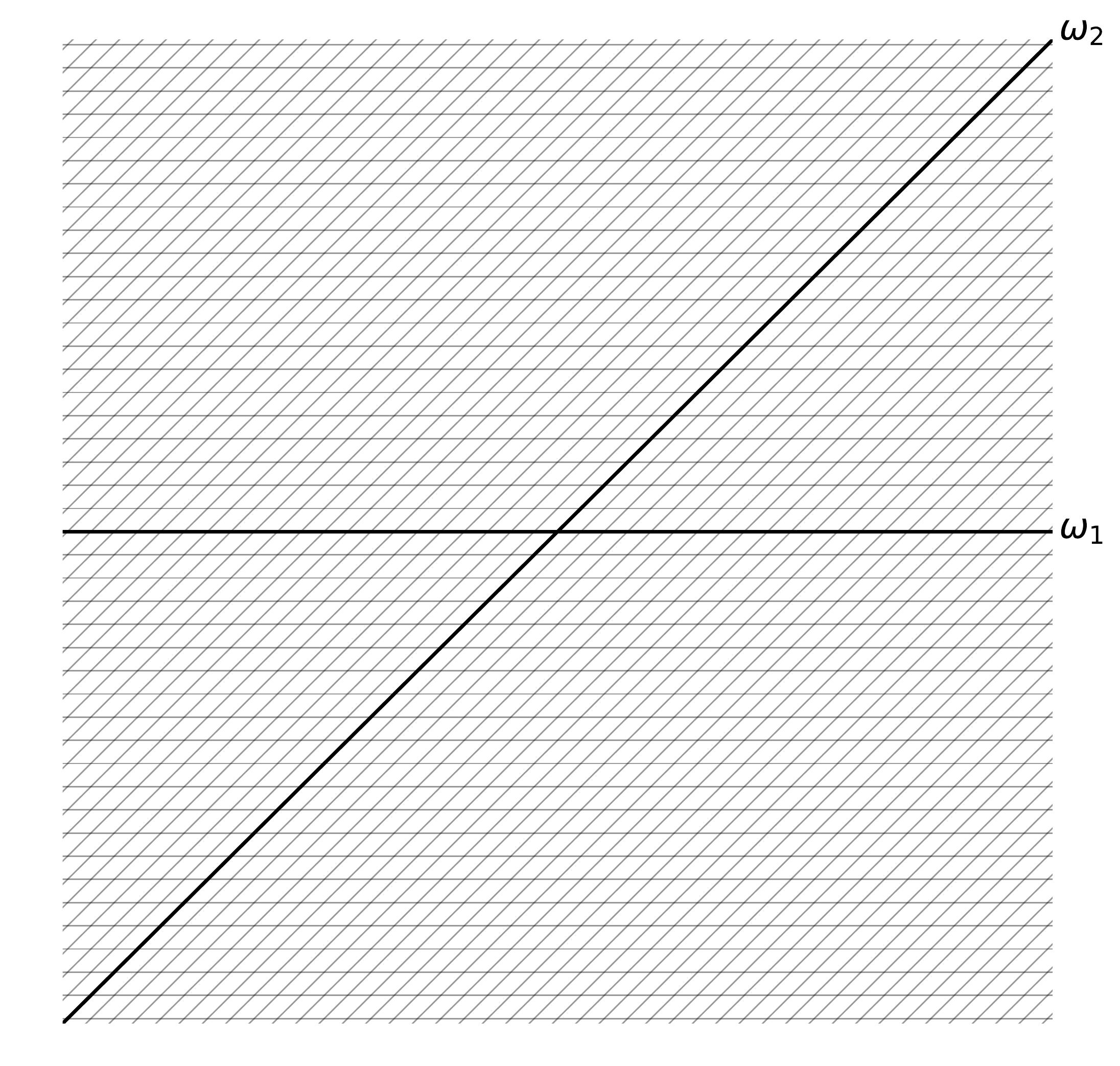} }
    \label{subfig:c2_grid}
    }
    \hfill
    \subfloat[$D_2$]{{\includegraphics[width=1.5in]{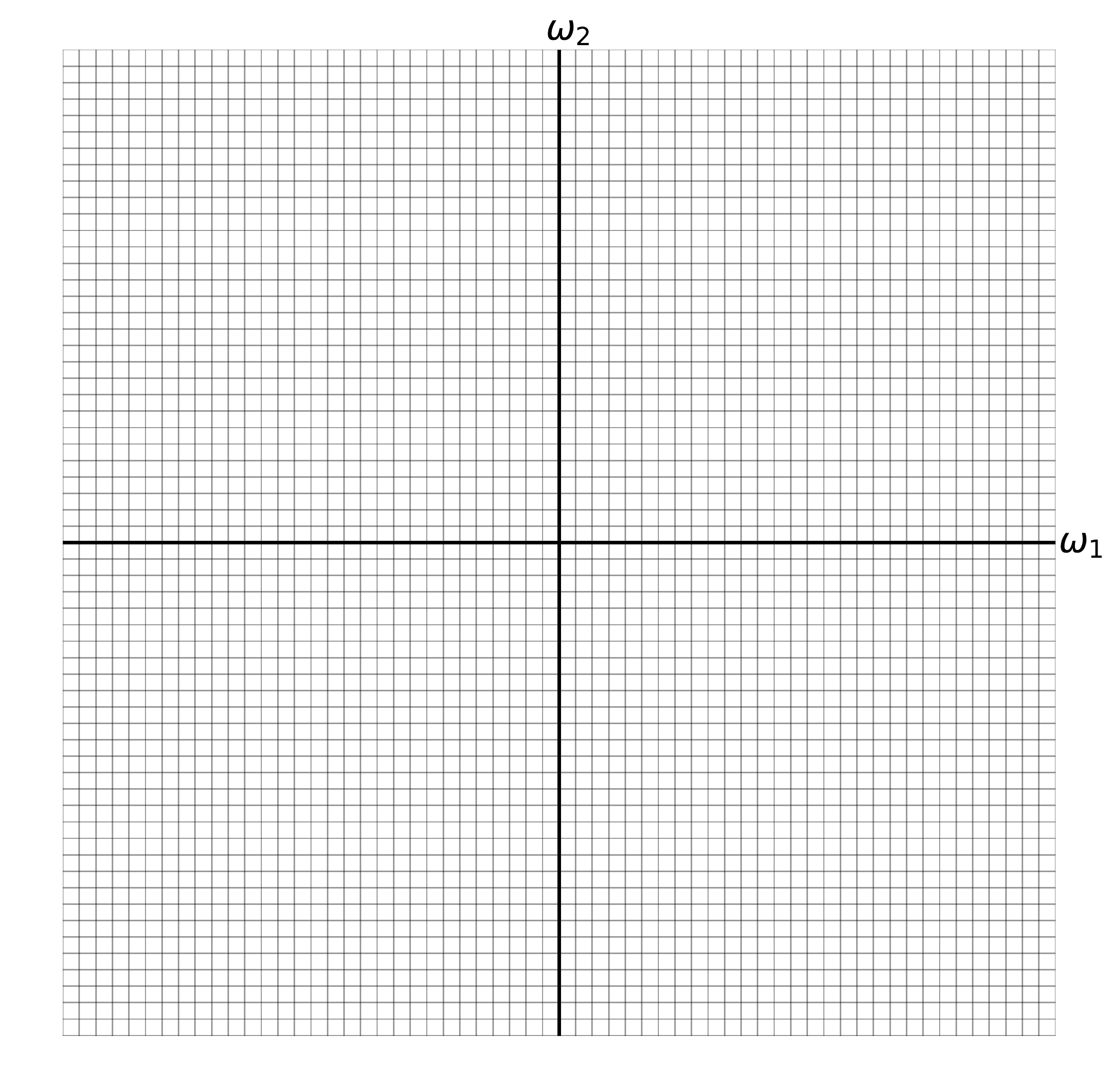}}
    \label{subfig:d2_grid}
    }%
    \hfill
    \subfloat[$G_2$]{{\includegraphics[width=1.5in]{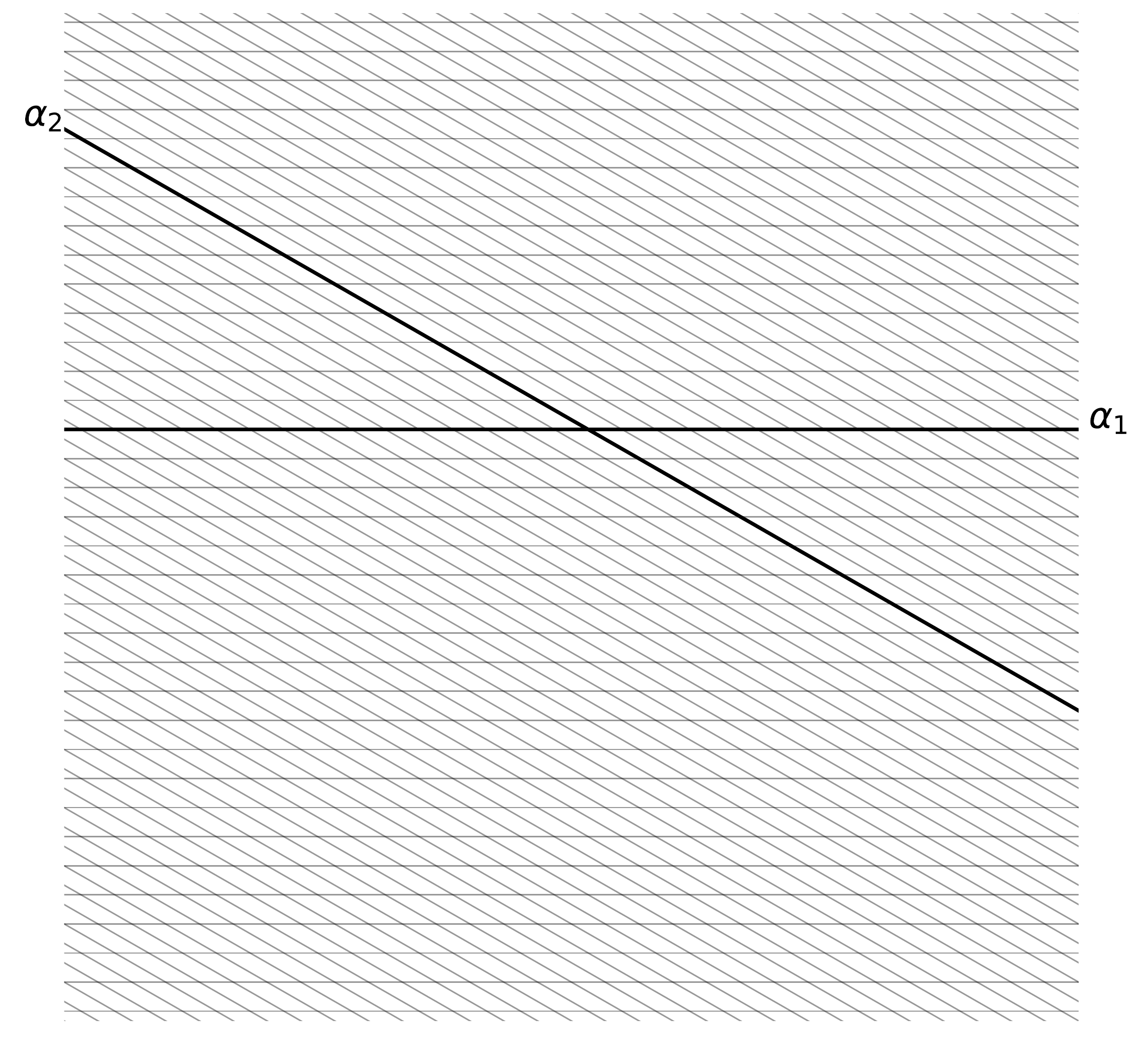} }
    \label{subfig:g2_grid}
    }%
    \caption{The fundamental weight lattice of Lie algebras of type $B_2$, $C_2$, $D_2$, and $G_2$.}
    \label{fig:grids}
\end{figure}
The construction of Weyl alternation diagrams entails graphing the corresponding linear inequalities found in Section \ref{sec:altsets} on the fundamental weight lattice of each respective Lie algebra we have considered. These lattices are  illustrated in 
Figure \ref{fig:grids}. Thus, the figures are created in the same way one would shade the solution sets of linear inequalities in $\mathbb{R}^2$. We then use the assigned distinct colors for each non-empty Weyl alternation set (i.e. solution set) from Theorems \ref{thm:mainb2}, \ref{thm:mainc2}, \ref{thm:maind2}, and \ref{thm:maing2} to precisely describe these regions. This allows us to present a multicolored diagram which provides a visual representation of the support of Kostant's partition function for Lie algebras of type $B_2$, $C_2$, $D_2$, and $G_2$. 
Before we present the Weyl alternation diagrams we define a particular subset of the diagram, which allows us to highlight some of the symmetry within the diagram. This definition first appeared in \cite{HLM}.
\begin{definition}
An \textit{empty region} on the lattice $\mathbb{Z}{\w_1} \oplus \mathbb{Z}{\w_2}$ is a set of lattice points such that every point $(\lambda,\mu)$ satisfies $A(\lambda,\mu) = \emptyset$.
\end{definition}
%---------------------------------------------------------------------------------------------------
% Section B_2
\subsection{Lie algebra of type \texorpdfstring{$B_2$}{B2}}
For each $\sigma\in W$, we plot the conditions in Table~\ref{tab:WeylB2} by placing a solid colored dot on the integral weights for which $\sigma(\lambda+\rho)-(\mu+\rho) \in \mathbb{N}\a_1 \oplus \mathbb{N}\a_2$. In Figure \ref{fig:b2_single_elements}, we let $\mu=0$ and present the corresponding region for each Weyl group element. Note that changing $\mu$ will only translate the solution sets. In what follows we describe how the Weyl diagrams change as we vary the weight $\mu$.
\begin{figure}[H]%
    \centering
    \subfloat[$\sigma = 1$]{{\includegraphics[width=1.55in]{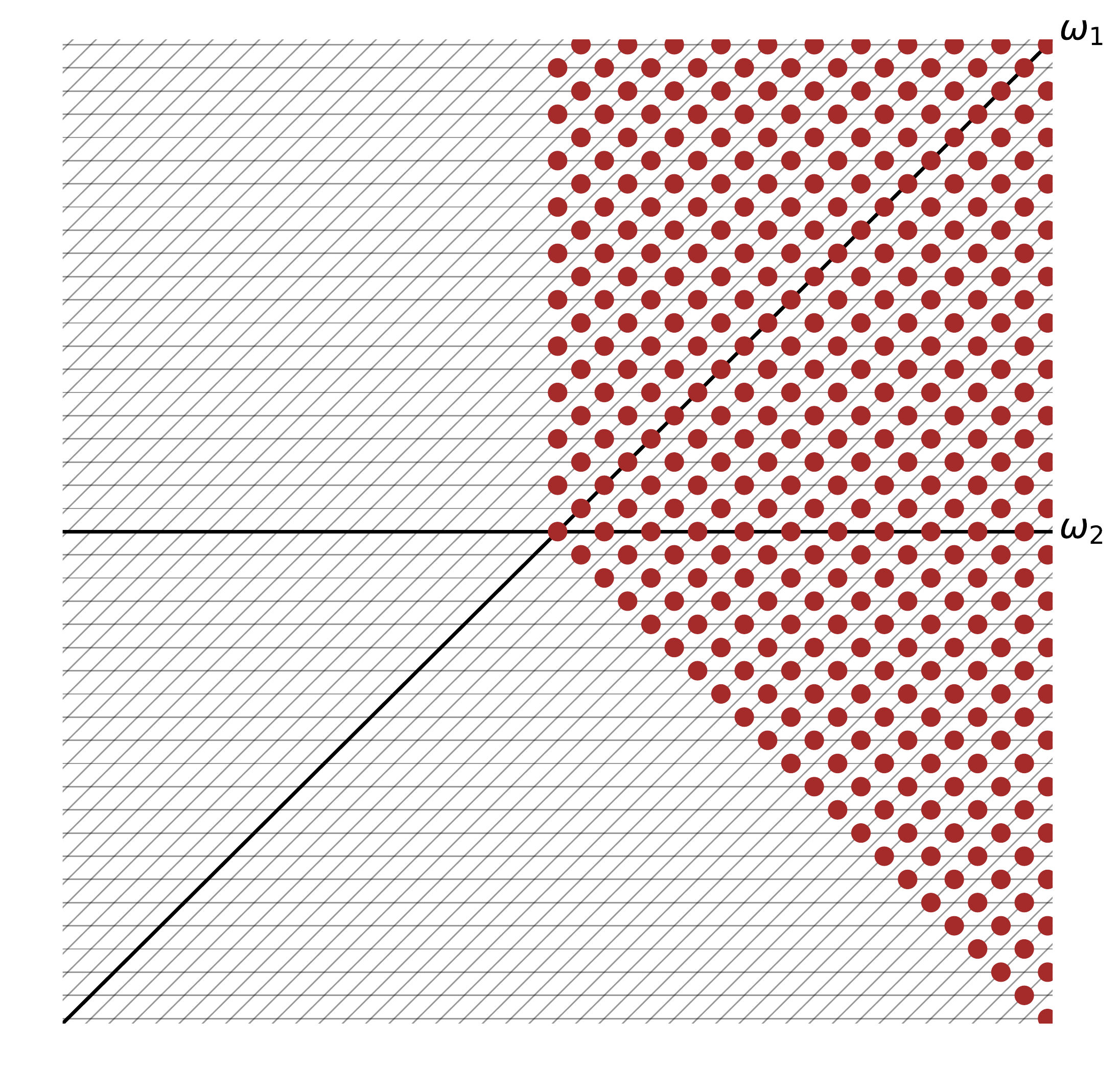}}}
    \hfill
    \subfloat[$\sigma = s_1$]{{\includegraphics[width=1.55in]{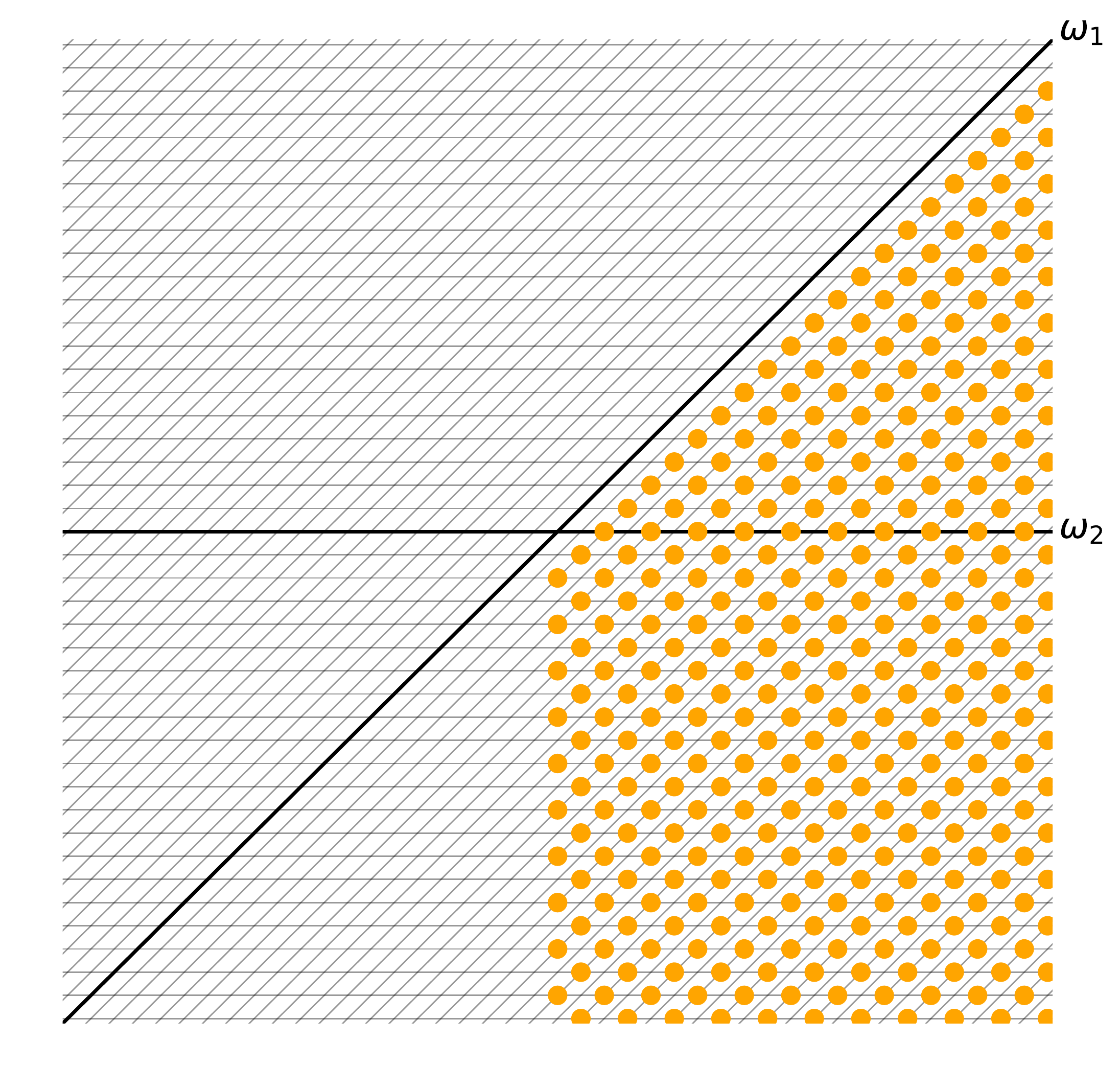} }}
    \hfill
    \subfloat[$\sigma = s_2$]{{\includegraphics[width=1.55in]{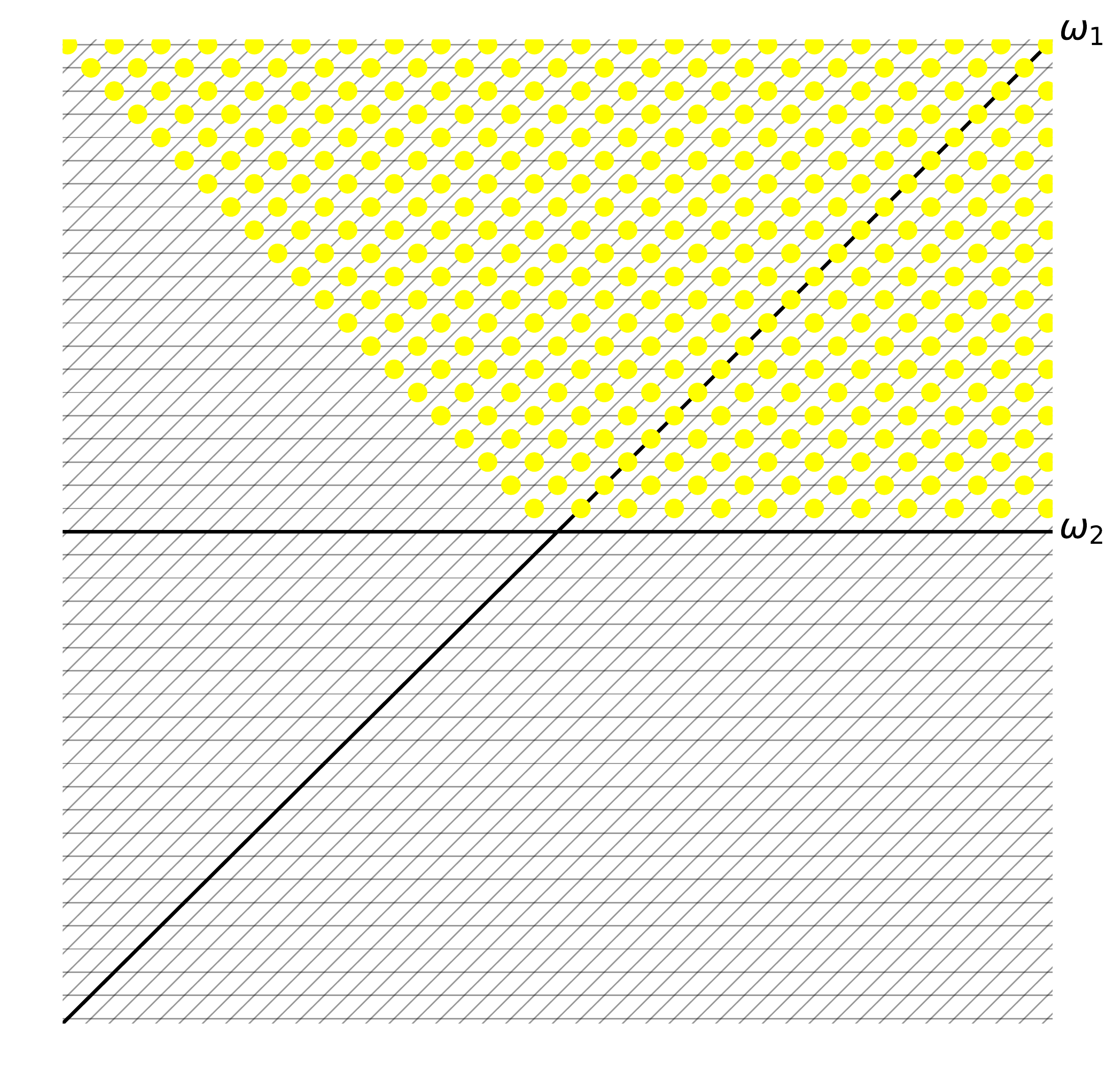} }}
    \hfill
    \subfloat[$\sigma = s_2s_1$]{{\includegraphics[width=1.55in]{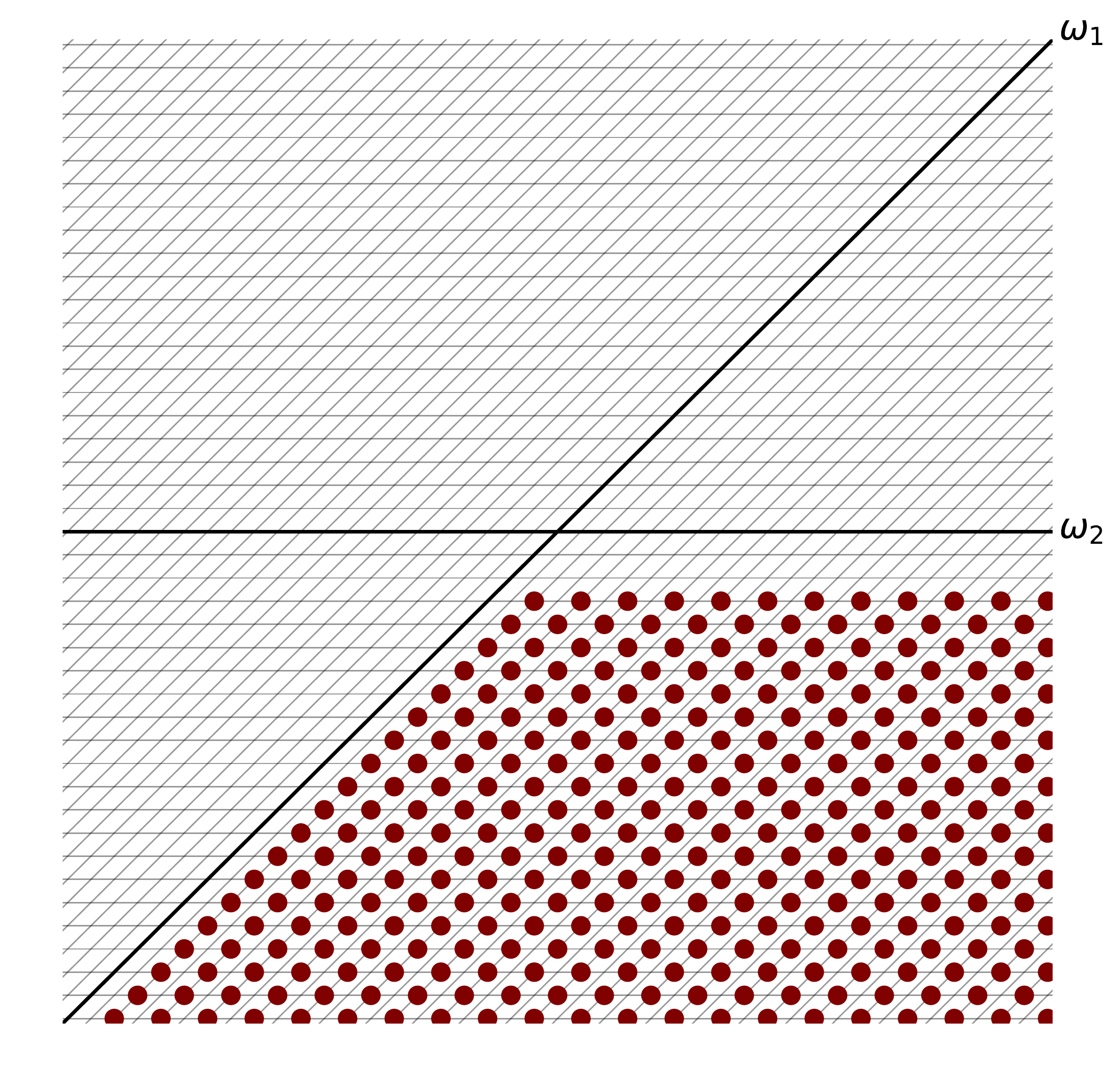} }}
    \\
    \subfloat[$\sigma = s_1s_2$]{{\includegraphics[width=1.55in]{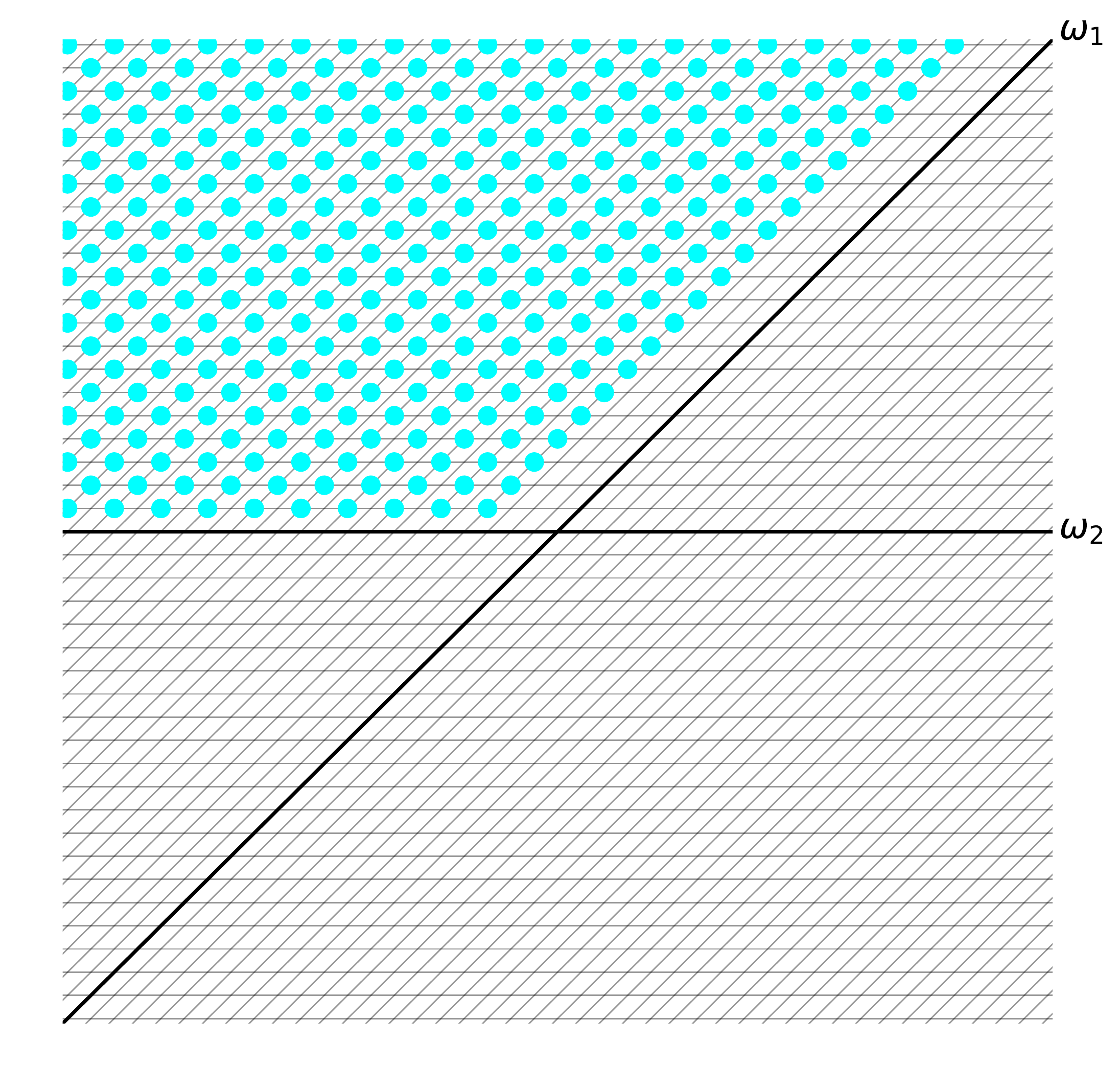} }}
    \hfill
    \subfloat[$\sigma = s_1s_2s_1$]{{\includegraphics[width=1.55in]{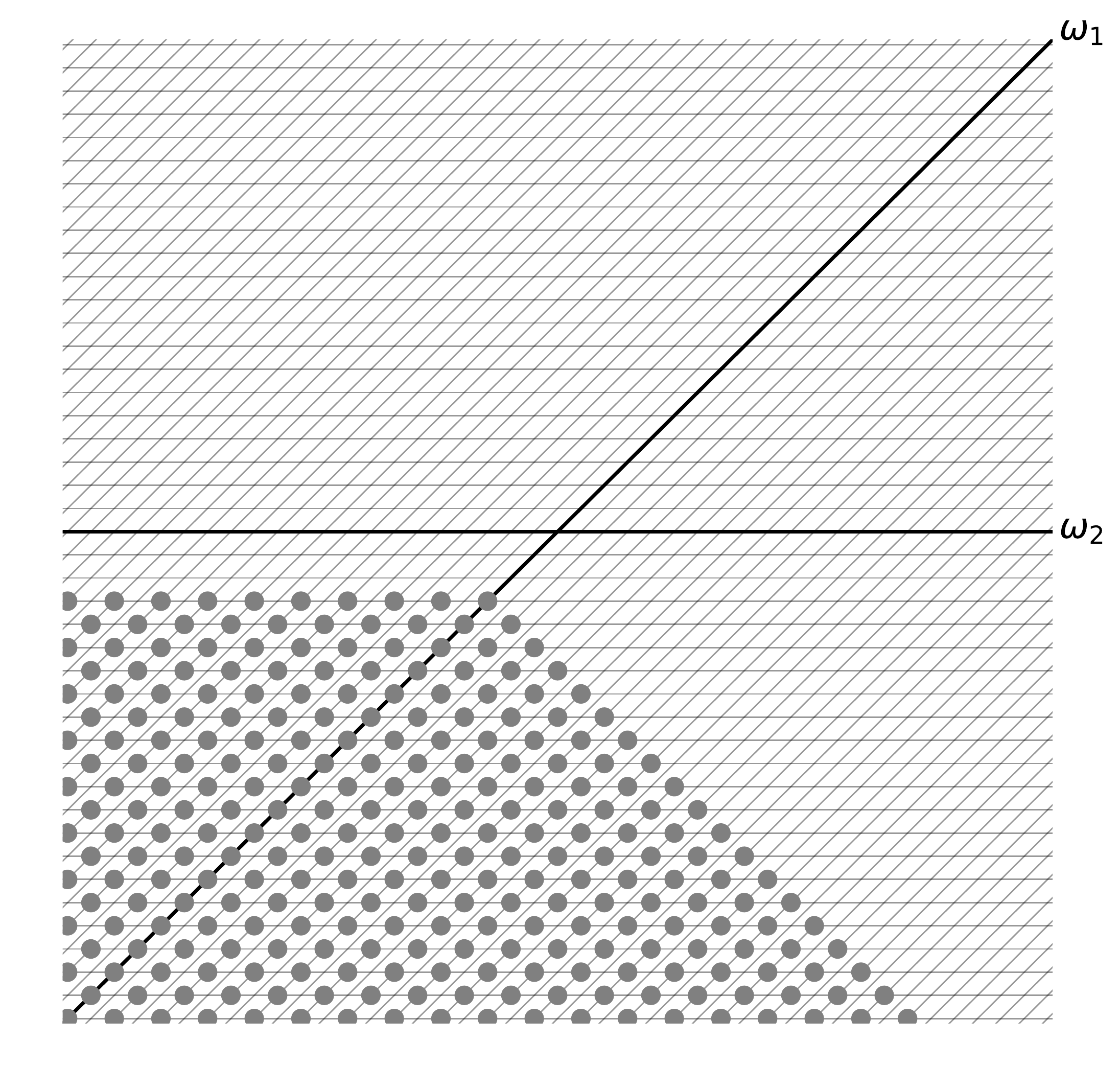} }}
    \hfill
    \subfloat[$\sigma = s_2s_1s_2$]{{\includegraphics[width=1.55in]{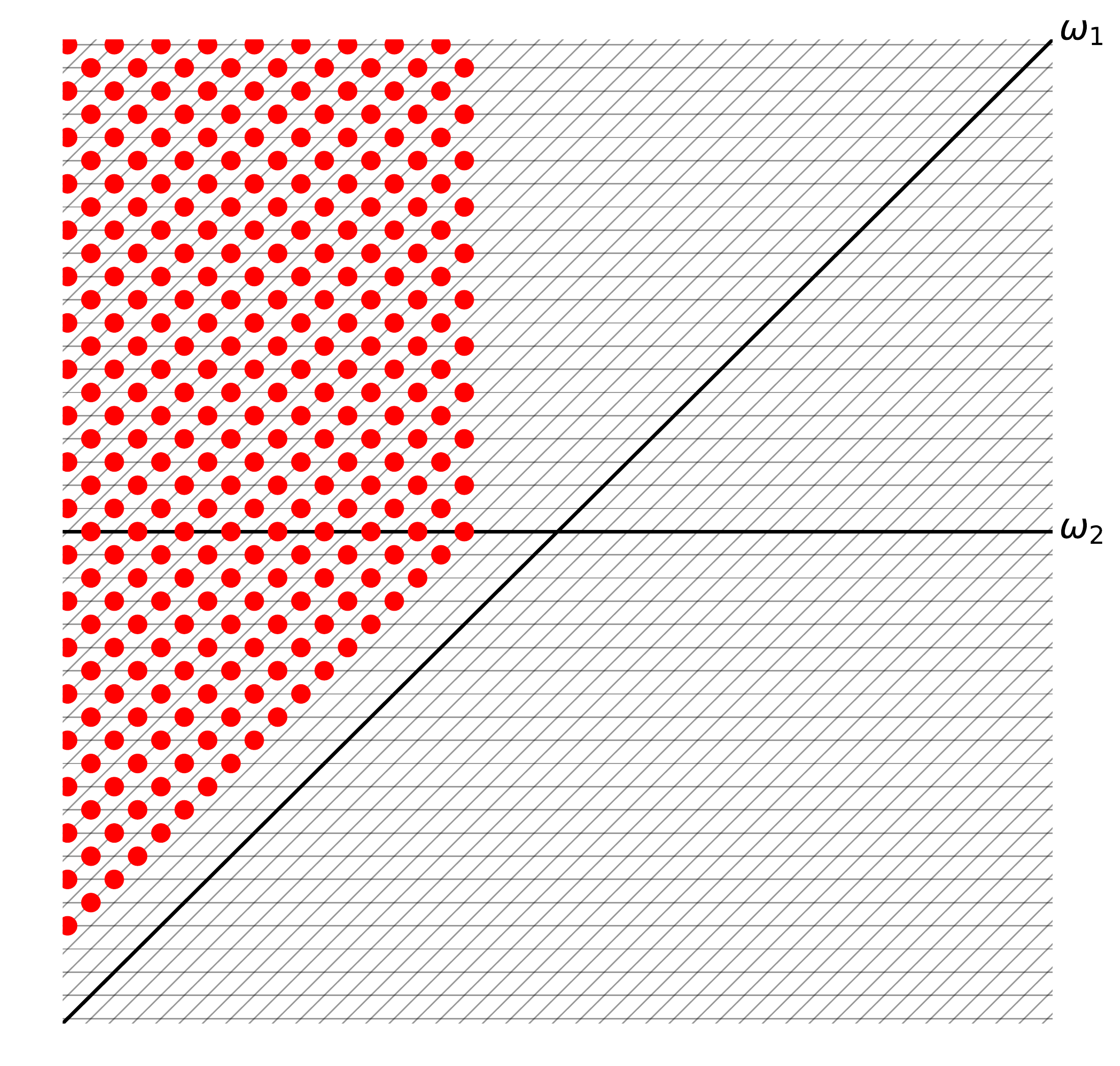} }}
    \hfill
    \subfloat[$\sigma = (s_2s_1)^2$]{{\includegraphics[width=1.55in]{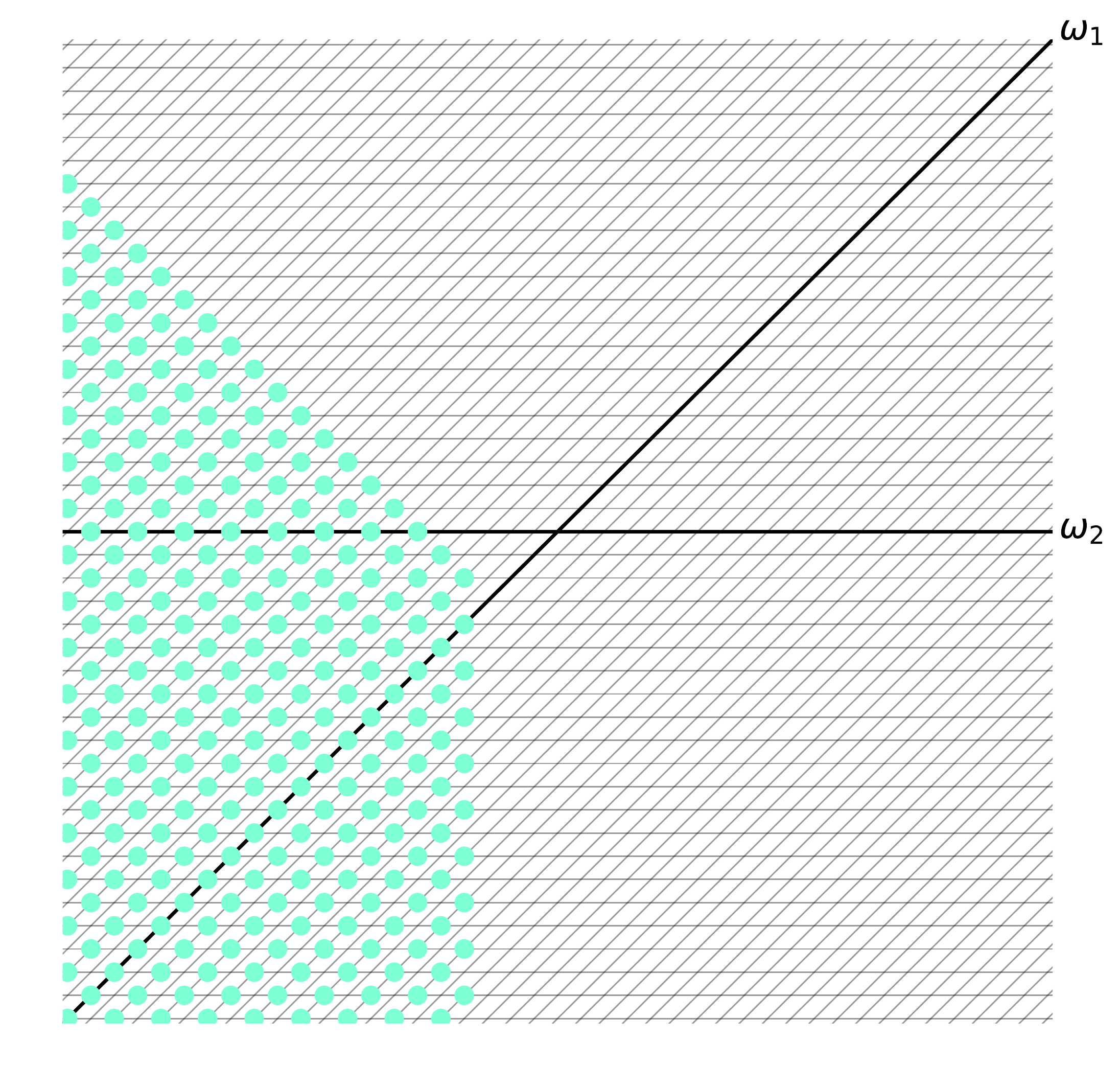} }}
    \caption{Solution sets to linear inequalities corresponding to the Lie algebra of type $B_2$.}
    \label{fig:b2_single_elements}
\end{figure}

% Section B_2 Case mu = n\w_1
\subsubsection{\normalfont{\textbf{Case}} \texorpdfstring{$\mu = n\w_1$}{mu equals n omega 1}}
Figures \ref{subfig:b2_n1}-\ref{subfig:b2_n4} illustrate the Weyl alternation diagrams for $\mu = n\w_1$ such that $n = 1,2,3,4$. We observe that the empty region takes the shape of a square oriented so that a vertex points up. We also note that as $n$ increases from $1$ to $4$, the length of the edges of the square also increases.
\begin{figure}[H]%
    \centering
    \subfloat[$\mu = \w_1$]{{\includegraphics[width=1.5in]{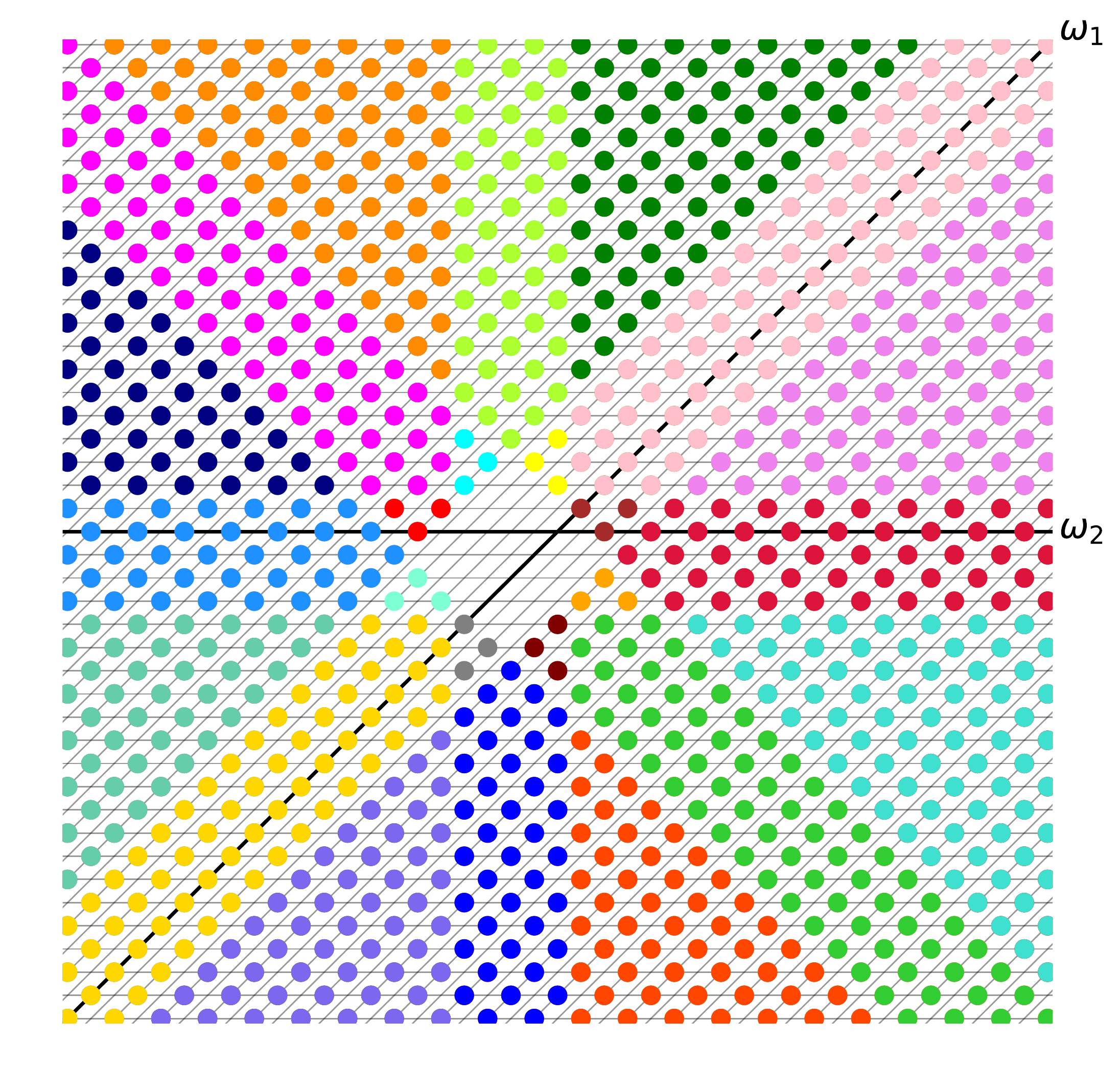} }
    \label{subfig:b2_n1}
    }
    \hfill
    \subfloat[$\mu = 2\w_1$]{{\includegraphics[width=1.5in]{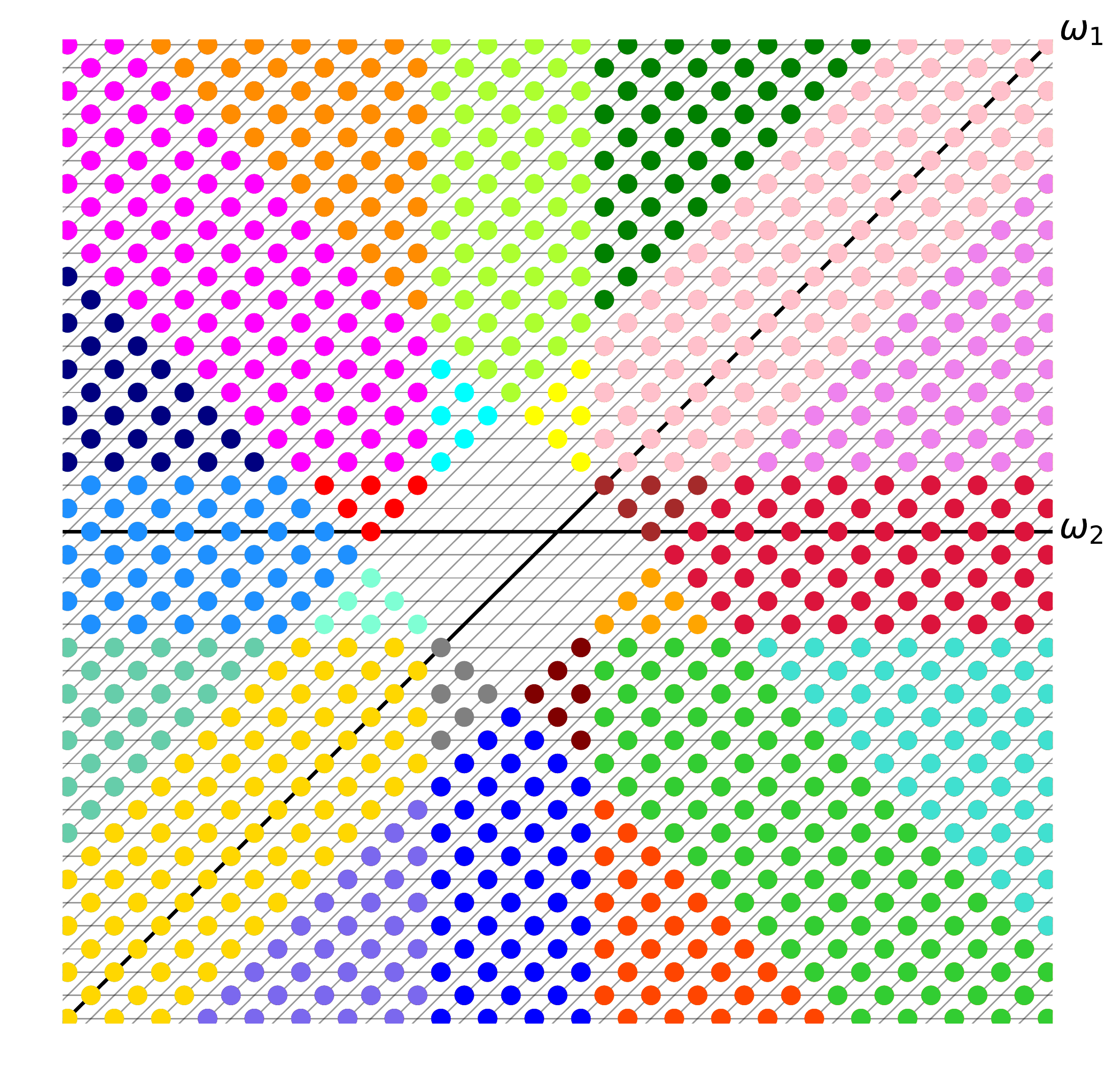} }
    \label{subfig:b2_n2}
    }
    \hfill
    \subfloat[$\mu = 3\w_1$]{{\includegraphics[width=1.5in]{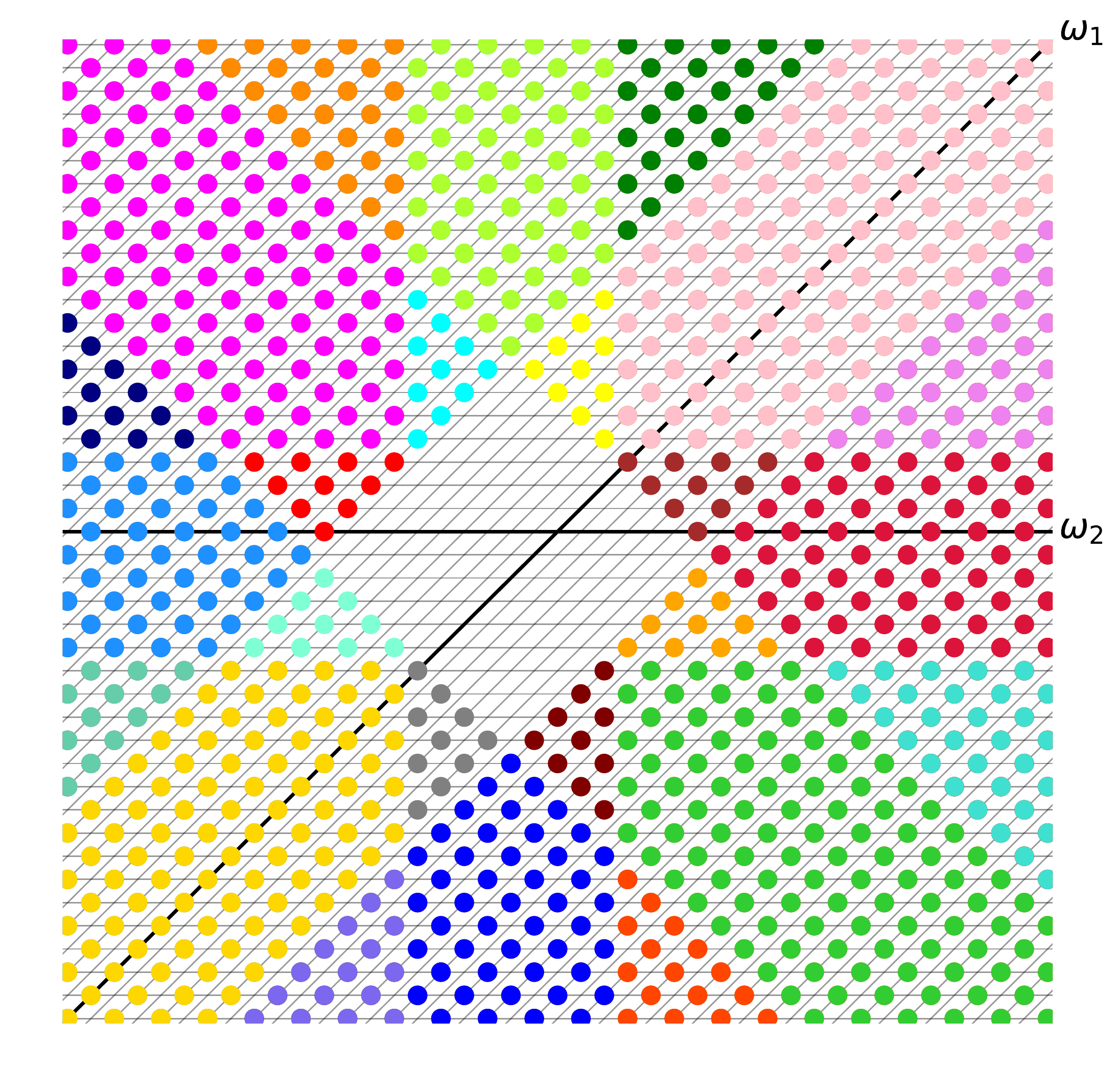} }
    \label{subfig:b2_n3}
    }
    \hfill
    \subfloat[$\mu = 4\w_1$]{{\includegraphics[width=1.5in]{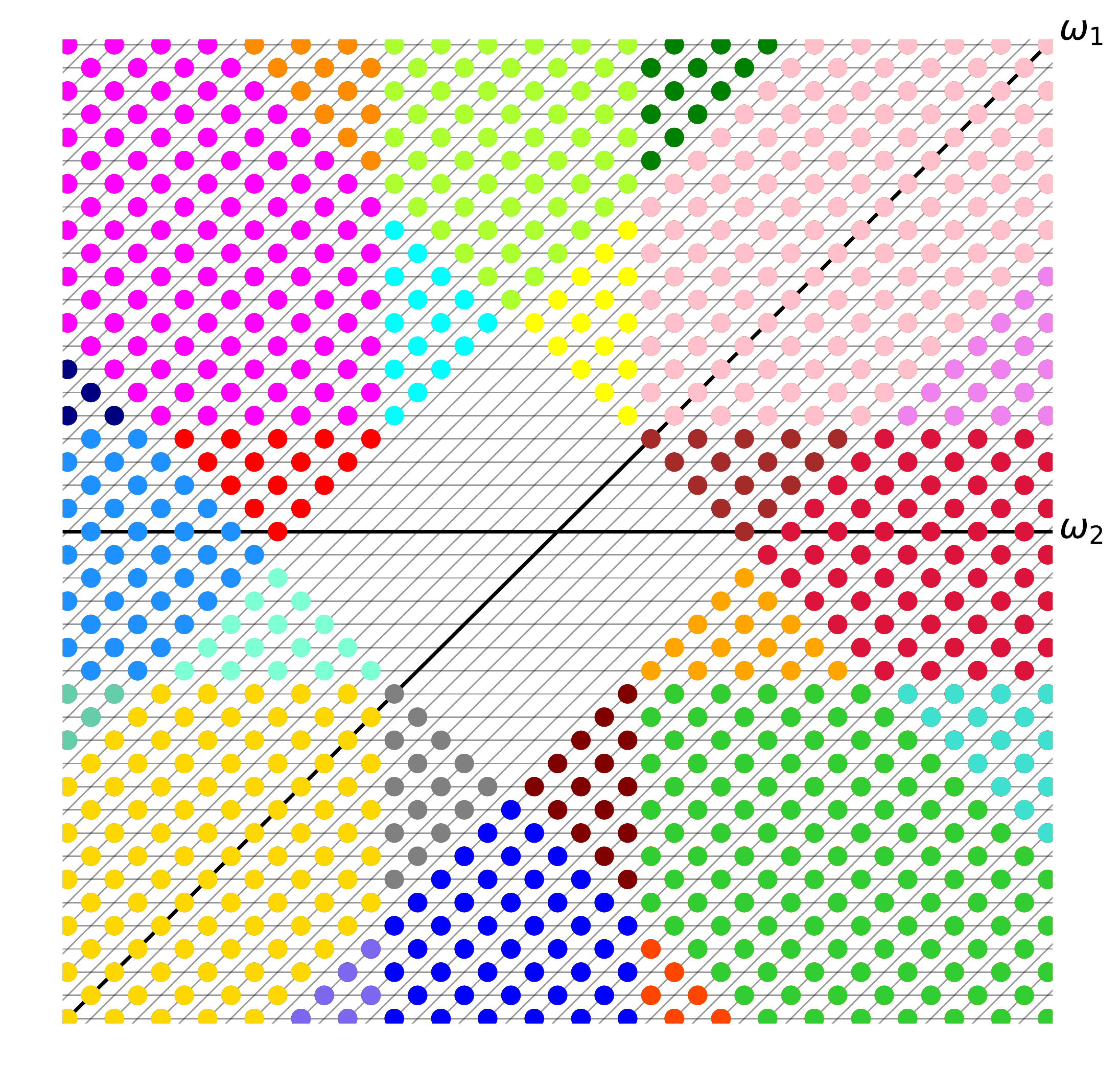} }
    \label{subfig:b2_n4}
    }
    \caption{Weyl alternation diagrams for the Lie algebra of type $B_2$ with $\mu=n\w_1$.}
    \label{fig:b2_mu_n}
\end{figure}

% Section B_2 Case mu = m\w_2
\subsubsection{\normalfont{\textbf{Case}} \texorpdfstring{$\mu = m\w_2$}{mu equals m omega 2}} 
Figures \ref{subfig:b2_m2}-\ref{subfig:b2_m8} illustrate the Weyl alternation diagrams for $\mu = m\w_2$ such that $m = 2,4,6,8$. We observe that the empty region is in the shape of a square oriented so that an edge is on top. We also note that as $m$ increases from $2$ to $8$, the length of the edges of the square in the center also increases.
\begin{figure}[H]%
    \centering
    \subfloat[$\mu = 2\w_2$]{{\includegraphics[width=1.5in]{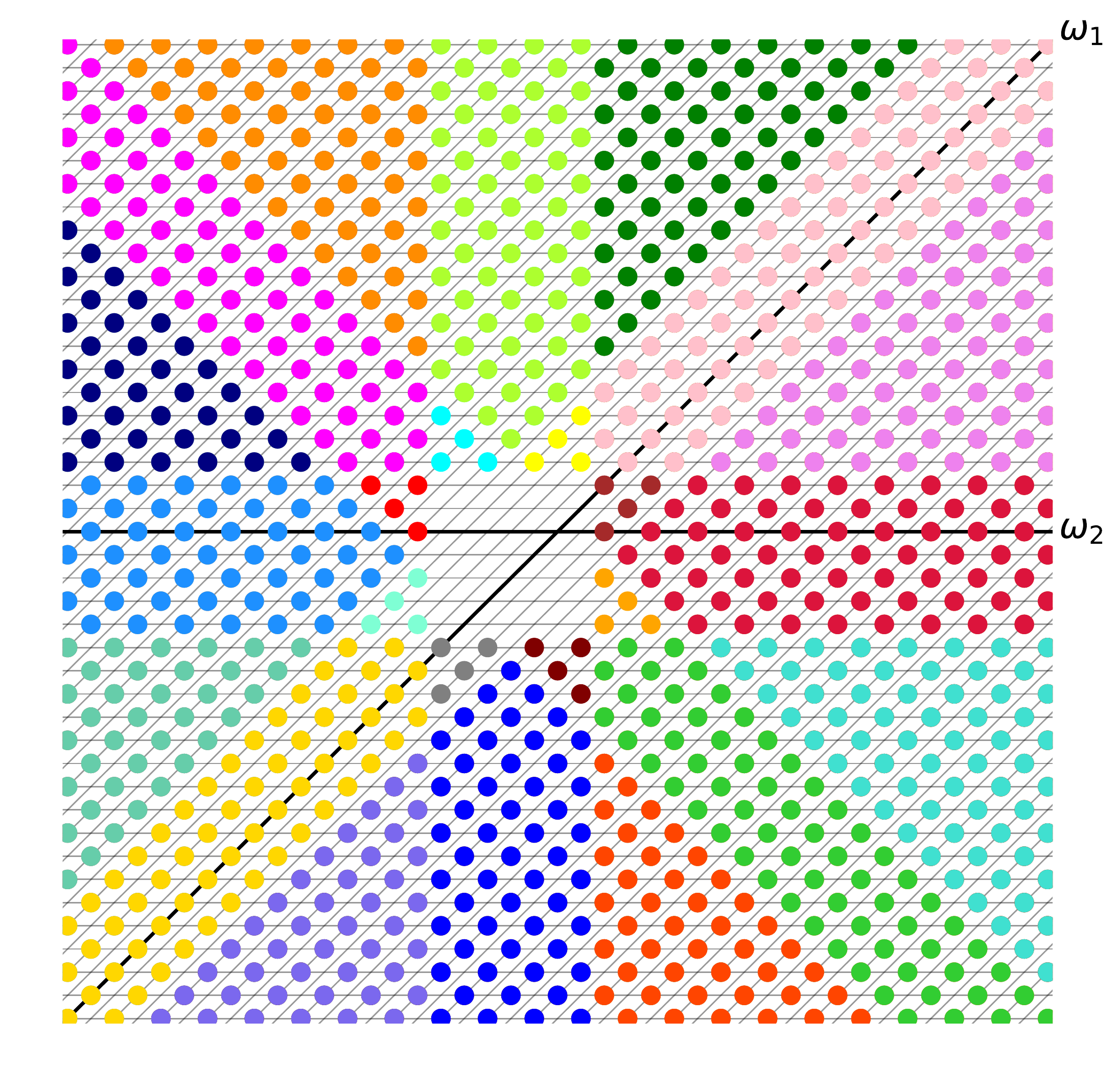}}
    \label{subfig:b2_m2}
    }%
    \hfill
    \subfloat[$\mu = 4\w_2$]{{\includegraphics[width=1.5in]{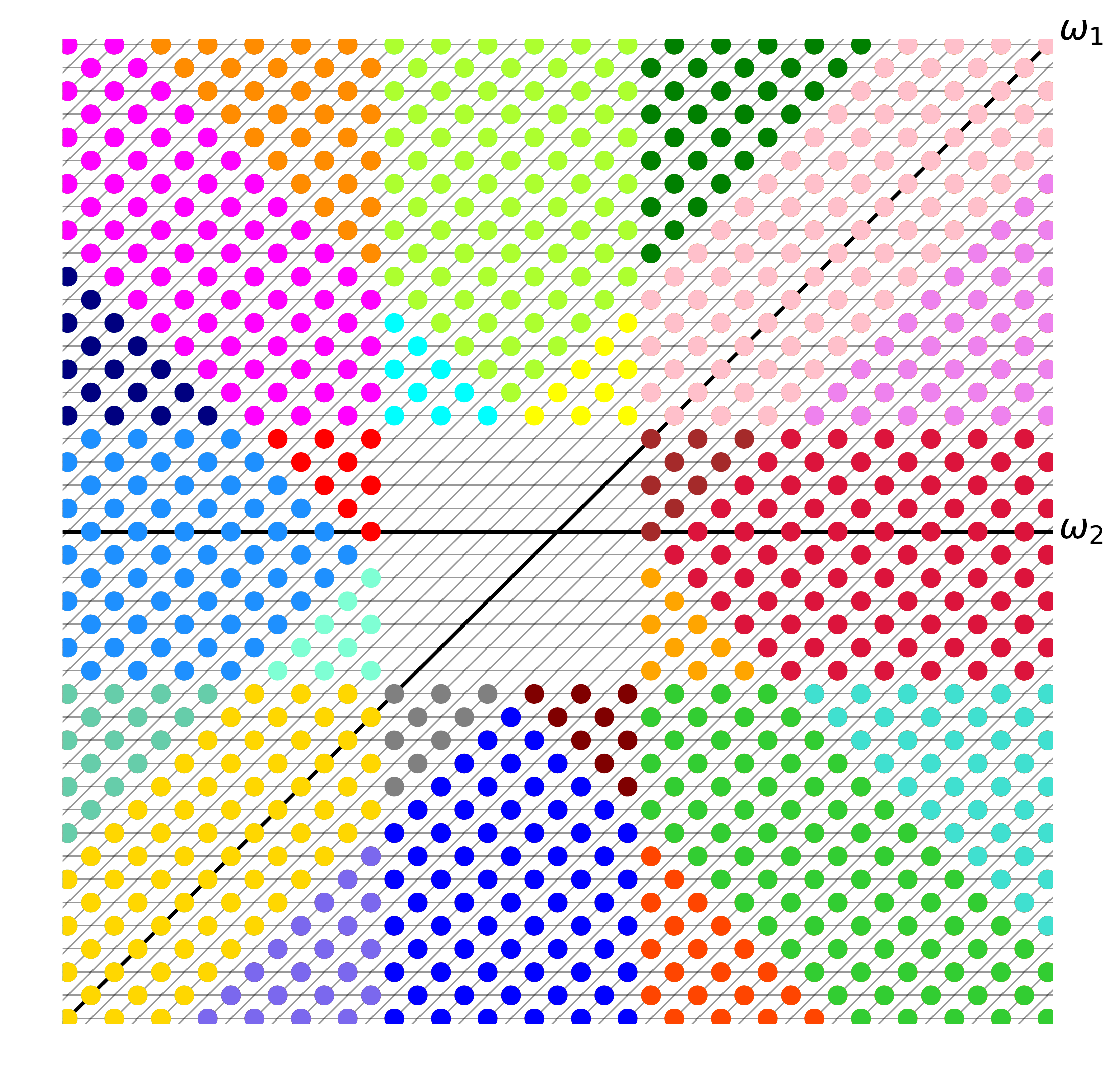} }
    \label{subfig:b2_m4}
    }
    \hfill
    \subfloat[$\mu = 6\w_2$]{{\includegraphics[width=1.5in]{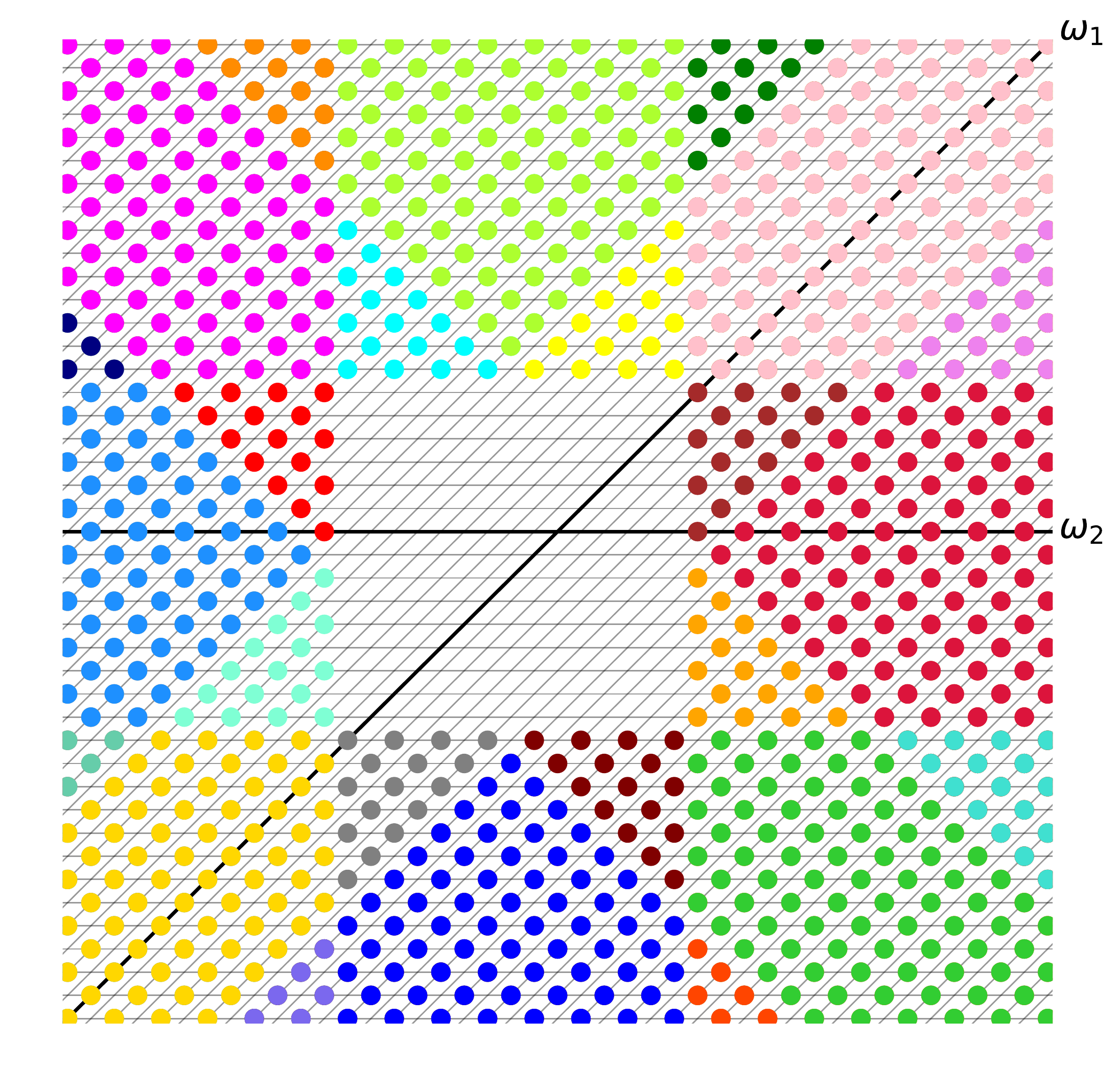}}
    \label{subfig:b2_m6}
    }%
    \hfill
    \subfloat[$\mu = 8\w_2$]{{\includegraphics[width=1.5in]{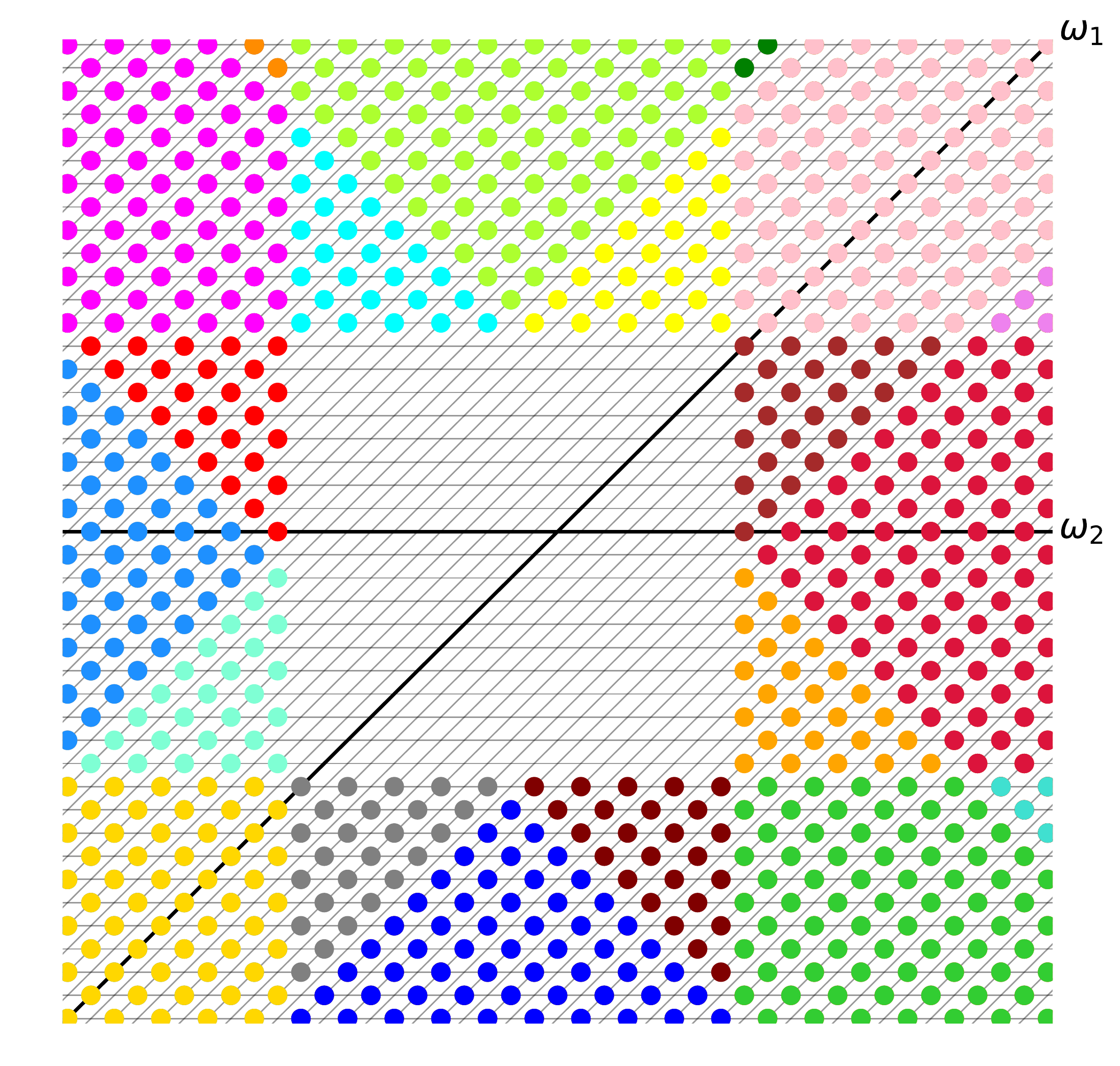} }
    \label{subfig:b2_m8}
    }
    \caption{Weyl alternation diagrams for the Lie algebra of type $B_2$ with $\mu=m\w_2$.}
    \label{fig:b2_mu_m}
\end{figure}
% Section B_2 Case mu = n\w_1 + m\w_2
\subsubsection{\normalfont{\textbf{Case}} \texorpdfstring{$\mu = n\w_1 + m\w_2$}{mu equals n omega 1 plus m omega 2}}
Figures \ref{subfig:b2_n1m2}-\ref{subfig:b2_n3m2} illustrate the Weyl alternation diagrams for $\mu = n\w_1 + m\w_2$. We observe that the empty region is in the shape of an 8-pointed star. Additionally, as $n$ and $m$ both increase so does the size of the star in the center.
\begin{figure}[H]%
    \centering
    \subfloat[$\mu = \w_1 + 2\w_2$]{{\includegraphics[width=1.5in]{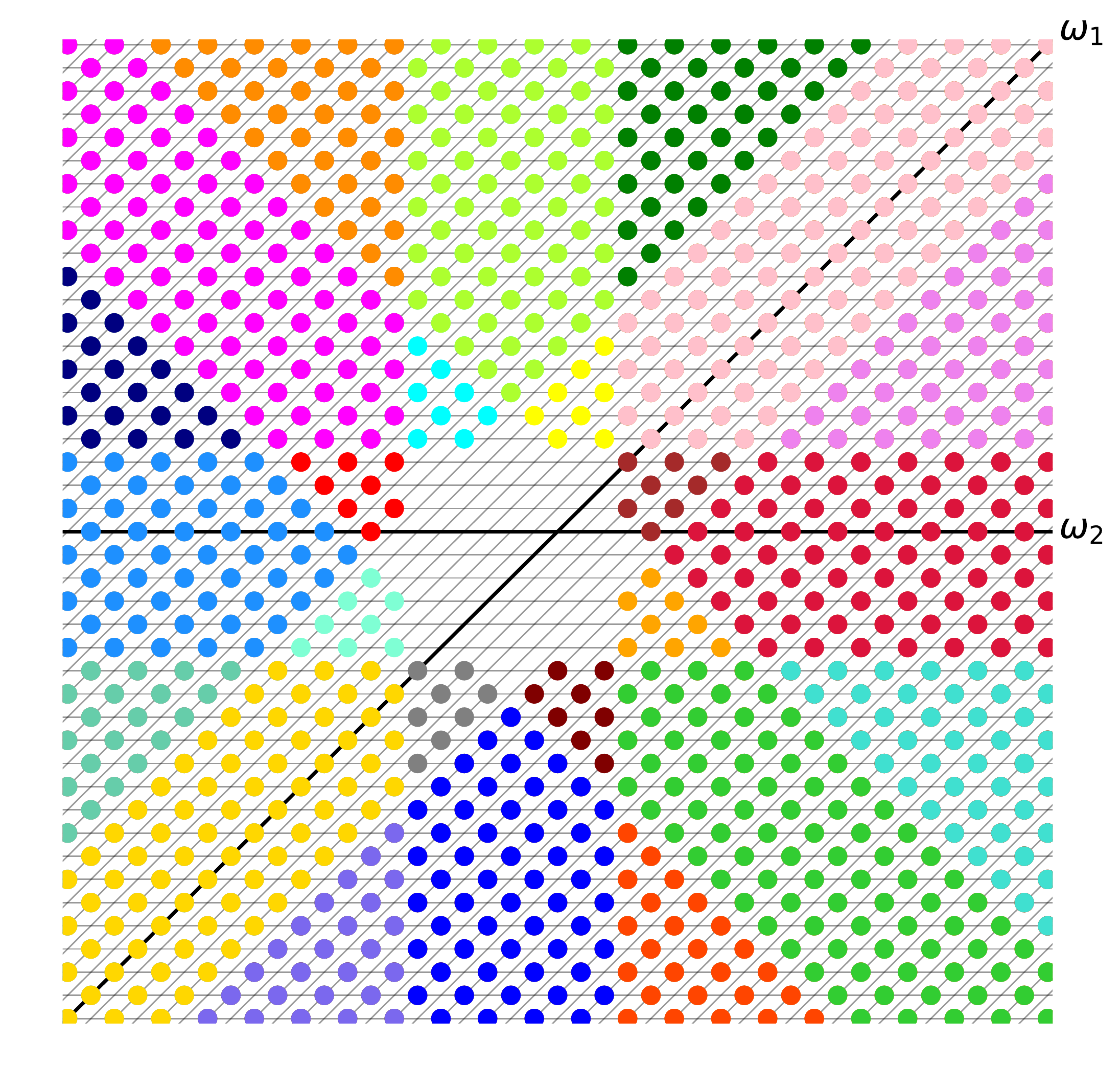}}
    \label{subfig:b2_n1m2}
    }
    \hfill
    \subfloat[$\mu = 2\w_1 + 2\w_2$]{{\includegraphics[width=1.5in]{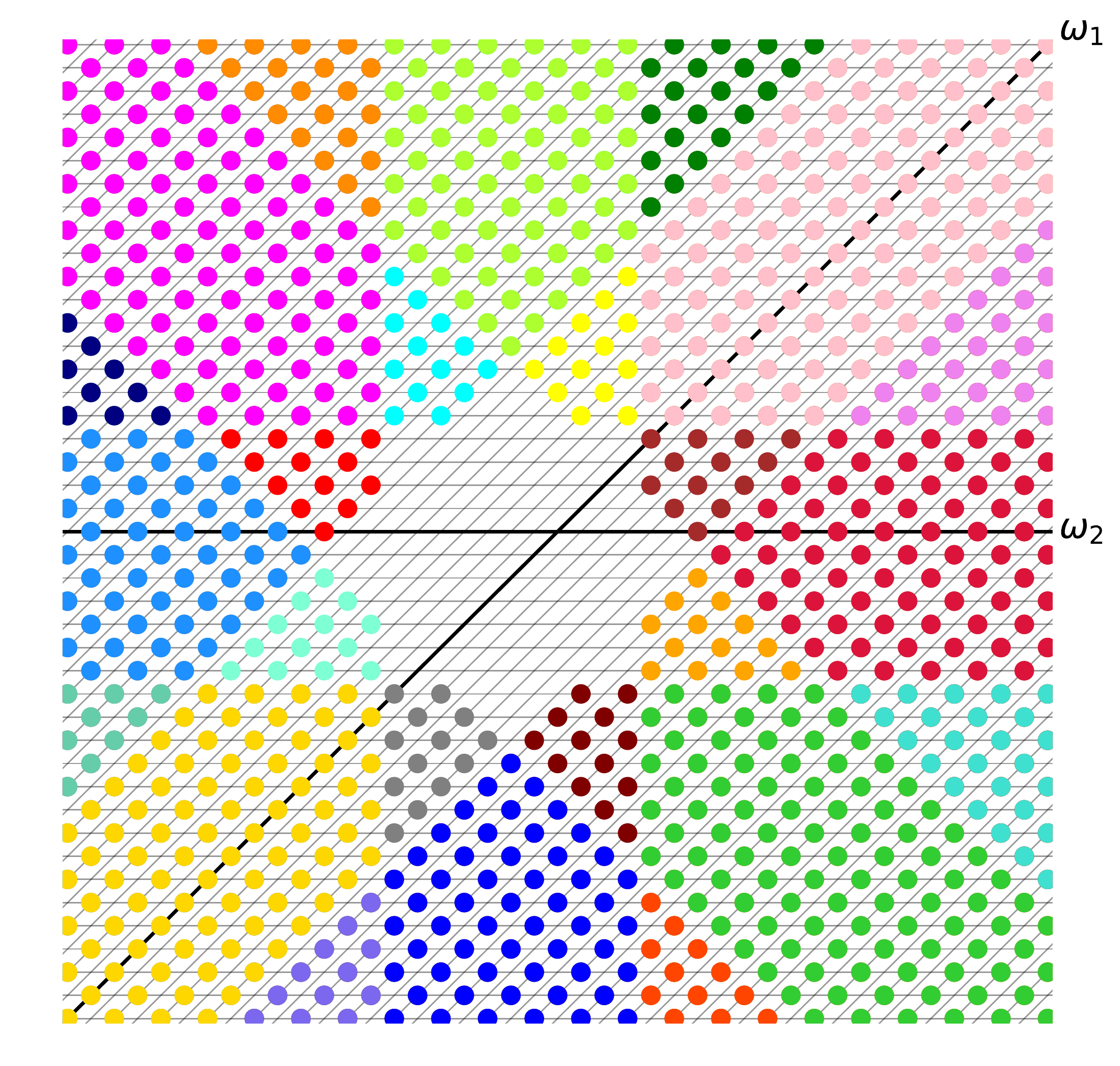} }
    \label{subfig:b2_n2m2}
    }
    \hfill
    \subfloat[$\mu = 2\w_1+4\w_2$]{{\includegraphics[width=1.5in]{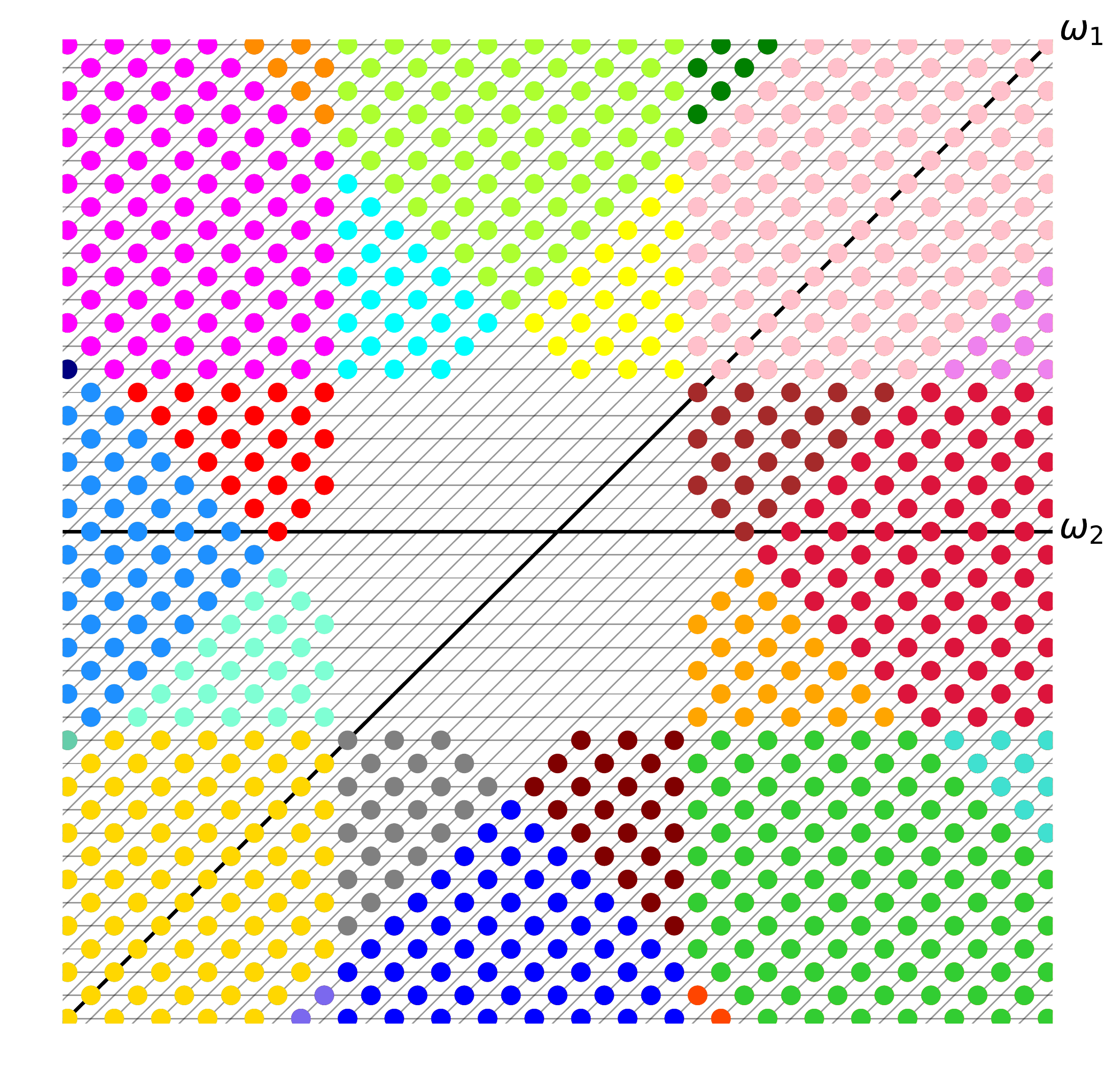}}
    \label{subfig:b2_n2m4}
    }
    \hfill
    \subfloat[$\mu = 3\w_1 + 2\w_2$]{{\includegraphics[width=1.5in]{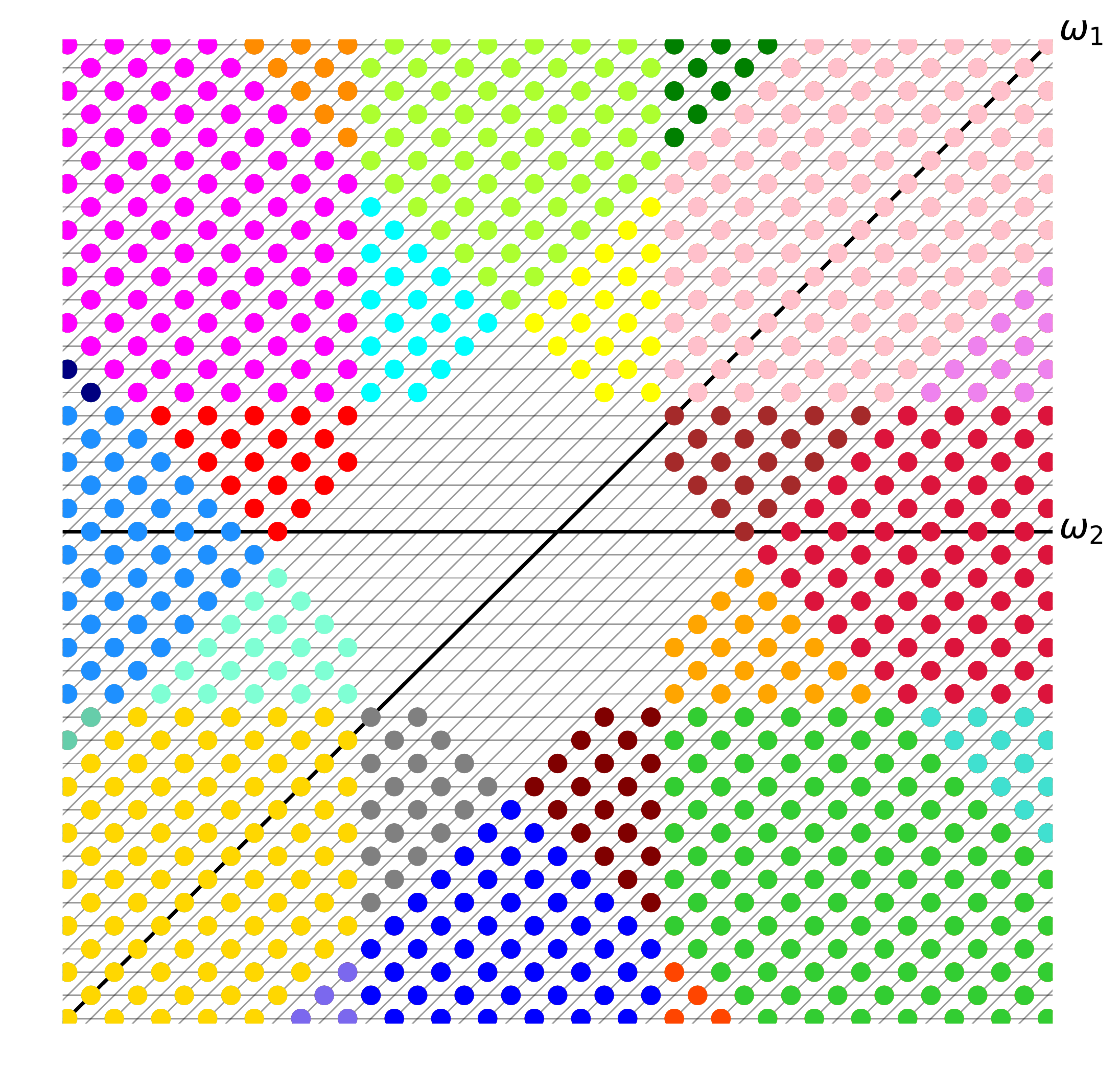}}
    \label{subfig:b2_n3m2}
    }
    \caption{Weyl alternation diagrams for the Lie algebra of type $B_2$ with \mbox{$\mu=n\w_1+m\w_2$.}}
    \label{fig:b2_mu_nm}
\end{figure}
% Inequalities for b2
\begin{figure}[H]%
\centering
\resizebox{2in}{!}{
\includegraphics{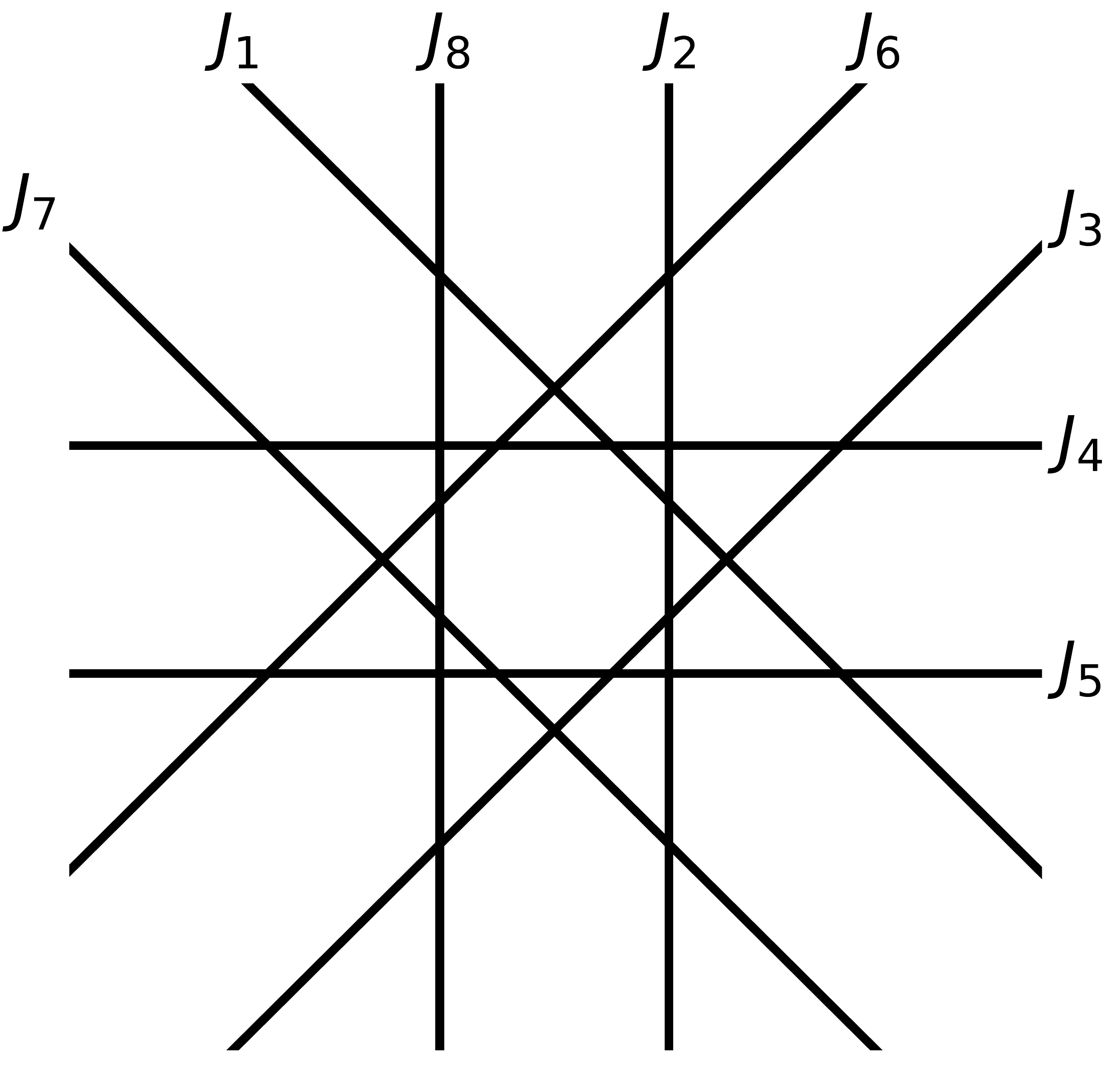}
}
\caption{Set of linear inequalities determining the boundaries of the Weyl alternation sets for the Lie algebra of type $B_2$.}
\label{fig:b2_ineqs}
\end{figure}

To explain the shapes that form in the empty region of each Weyl alternation diagram for the Lie algebra of type $B_2$ we turn to Figure \ref{fig:b2_ineqs}. From Theorem \ref{thm:mainb2}, we notice that all inequalities depend on $n$ and $m$ which simply shift the inequalities, but never change the direction of the line. This means that Figure \ref{fig:b2_ineqs} is a good representation of the inequalities formed.

\begin{figure}[H]%
\centering
\subfloat[Edge on top.]{
\includegraphics[width=4cm]{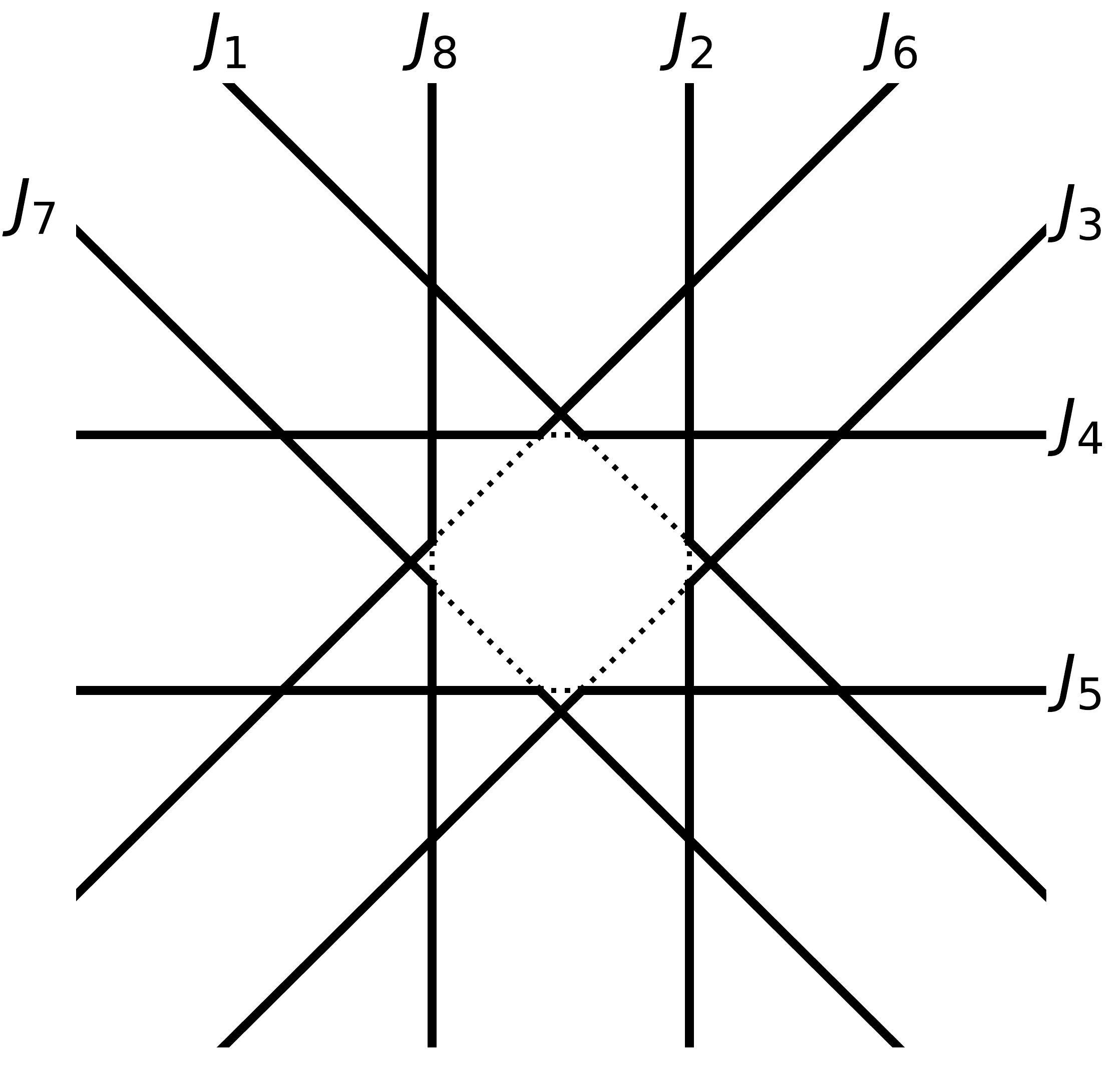}
\label{subfig:b2_edge_case}
}
\quad
\subfloat[Vertex on top.]{
\includegraphics[width=4cm]{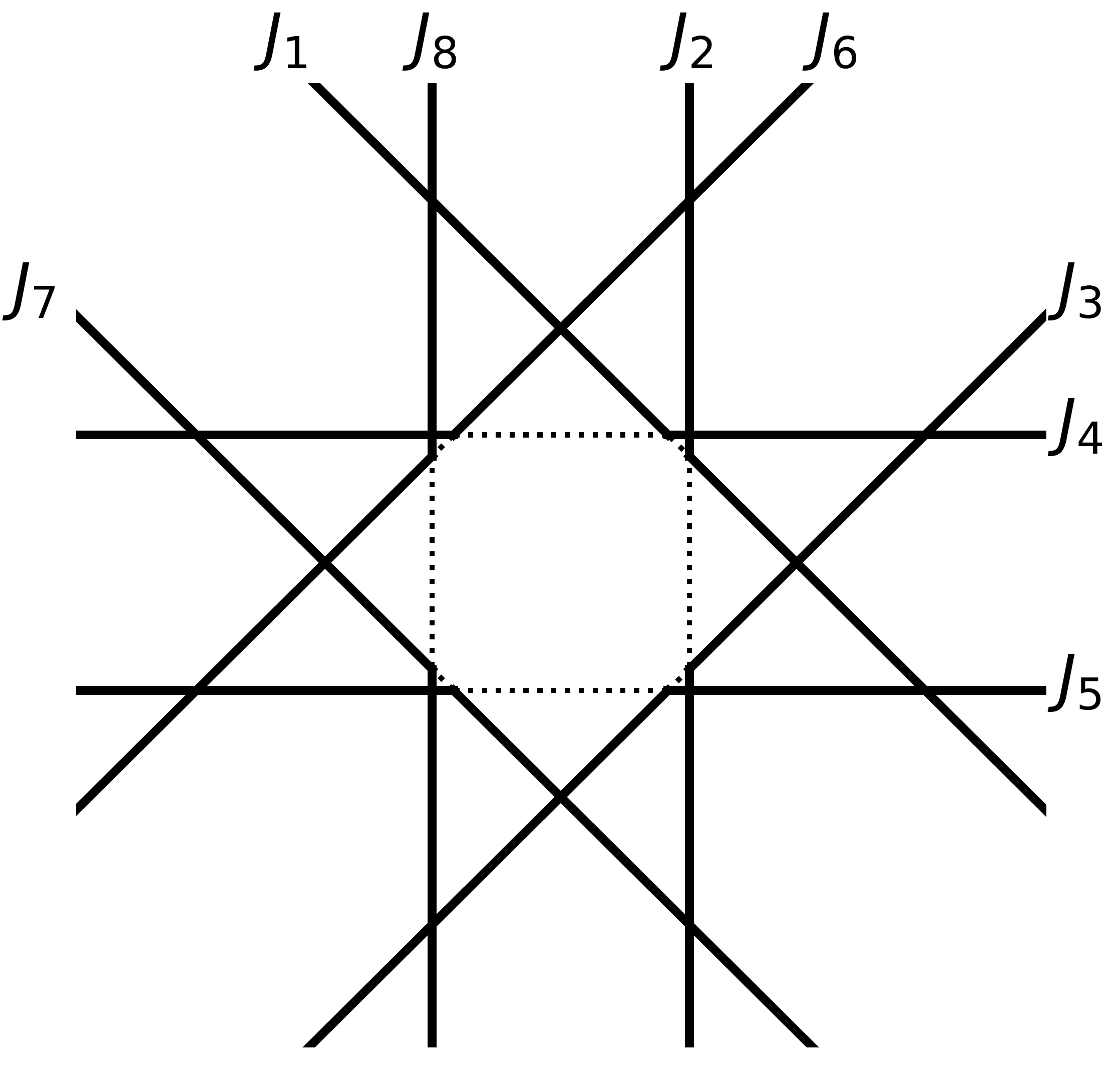}
\label{subfig:b2_vertex_case}
}
\quad
\subfloat[8-pointed star.]{
\includegraphics[width=4cm]{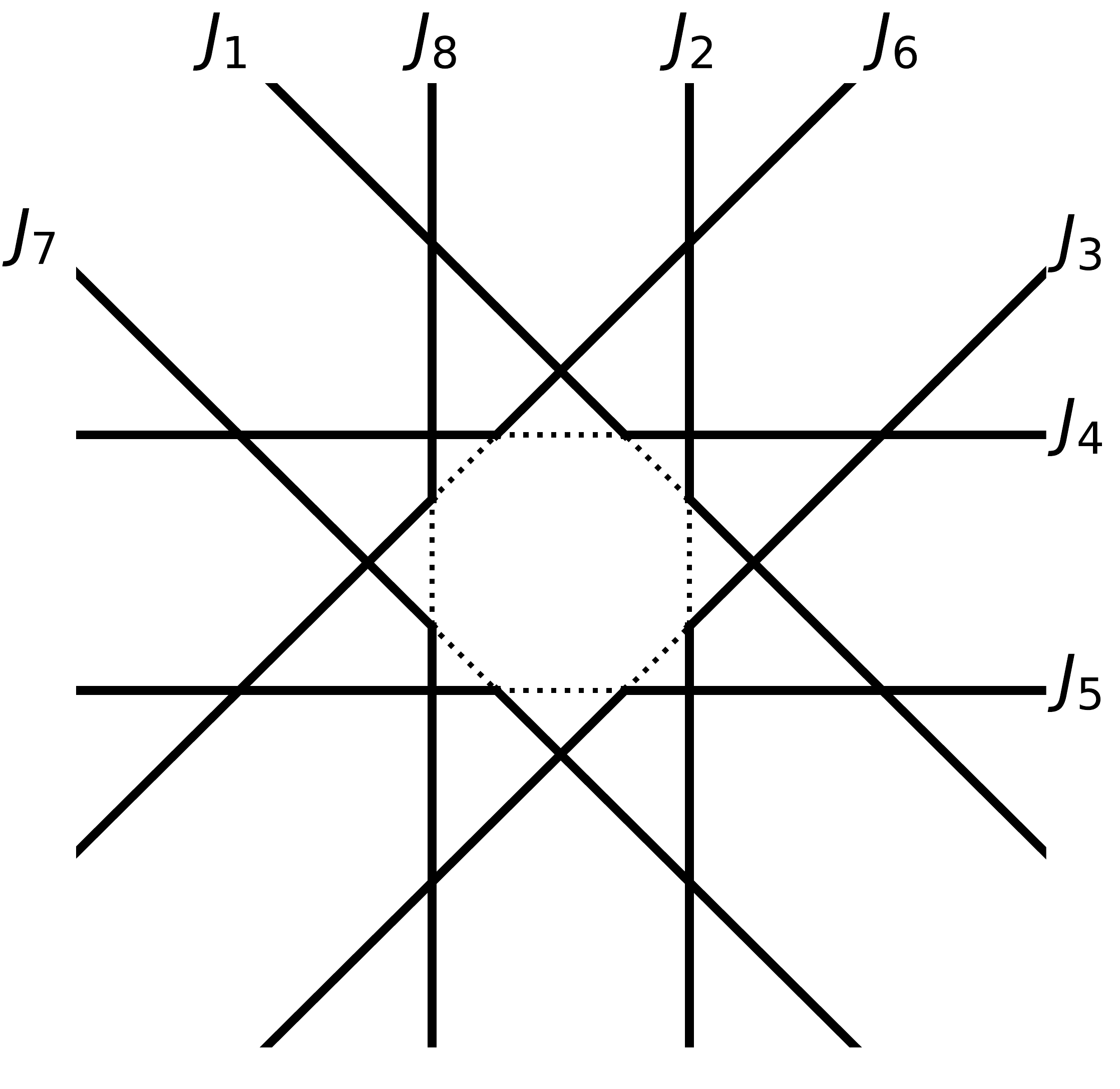}
\label{subfig:b2_star_case}
}
\caption{Different formations of the empty region for the Lie algebra of type $B_2$.}
\label{fig:b2_cases}
\end{figure}

From Figure \ref{subfig:b2_edge_case}, observe that the empty region becomes a square with an edge on top if and only if inequalities $J_1$ and $J_6$ intersect at or below $J_4$ and also the inequalities $J_2$ and $J_4$ intersect strictly above $J_1$. This occurs exactly when $\mu = m\w_2$ such that $m \in 2\NN$. Notice, in Figure \ref{subfig:b2_edge_case}, inequalities $J_1$ and $J_6$ intersect at a point where the divisibility condition is not satisfied, thus the intersection occurs at $J_4$ as desired. Similarly, the empty region becomes a square with a vertex pointing up, depicted in Figure \ref{subfig:b2_vertex_case}, when the inequalities $J_1$ and $J_6$ intersect strictly above $J_4$ and also the inequalities $J_2$ and $J_4$ intersect at or below $J_1$. This occurs exactly when $\mu = n\w_1$ such that $n \in \NN.$ Notice, in Figure \ref{subfig:b2_vertex_case}, inequalities $J_2$ and $J_4$ intersect at a point where the divisibility condition is not satisfied, thus the intersection occurs at $J_1$ as desired. The empty region takes the shape of an 8-pointed star if and only if the inequalities $J_1$ and $J_6$ intersect strictly above $J_4$ and the inequalities $J_2$ and $J_4$ intersect strictly above $J_1$. This is depicted in Figure \ref{subfig:b2_star_case}. Namely, this occurs when $\mu = n\w_1+m\w_2$ such that $n,m\in\NN$ and $2|m$.
%---------------------------------------------------------------------------------------------------
% Section C_2
\subsection{Lie algebra of type \texorpdfstring{$C_2$}{C2}}
For each $\sigma\in W$, we plot the conditions in Table~\ref{tab:WeylC2} by placing a solid colored dot on the integral weights for which $\sigma(\lambda+\rho)-(\mu+\rho) \in \mathbb{N}\a_1 \oplus \mathbb{N}\a_2$. In Figure \ref{fig:c2_single_elements}, we let $\mu=0$ and present the corresponding region for each Weyl group element. Note that changing $\mu$ will only translate the solution sets. In what follows we describe how the Weyl diagrams change as we vary the weight $\mu$.

\begin{figure}[H]%
    \centering
    \subfloat[$\sigma = 1$]{{\includegraphics[width=1.5in]{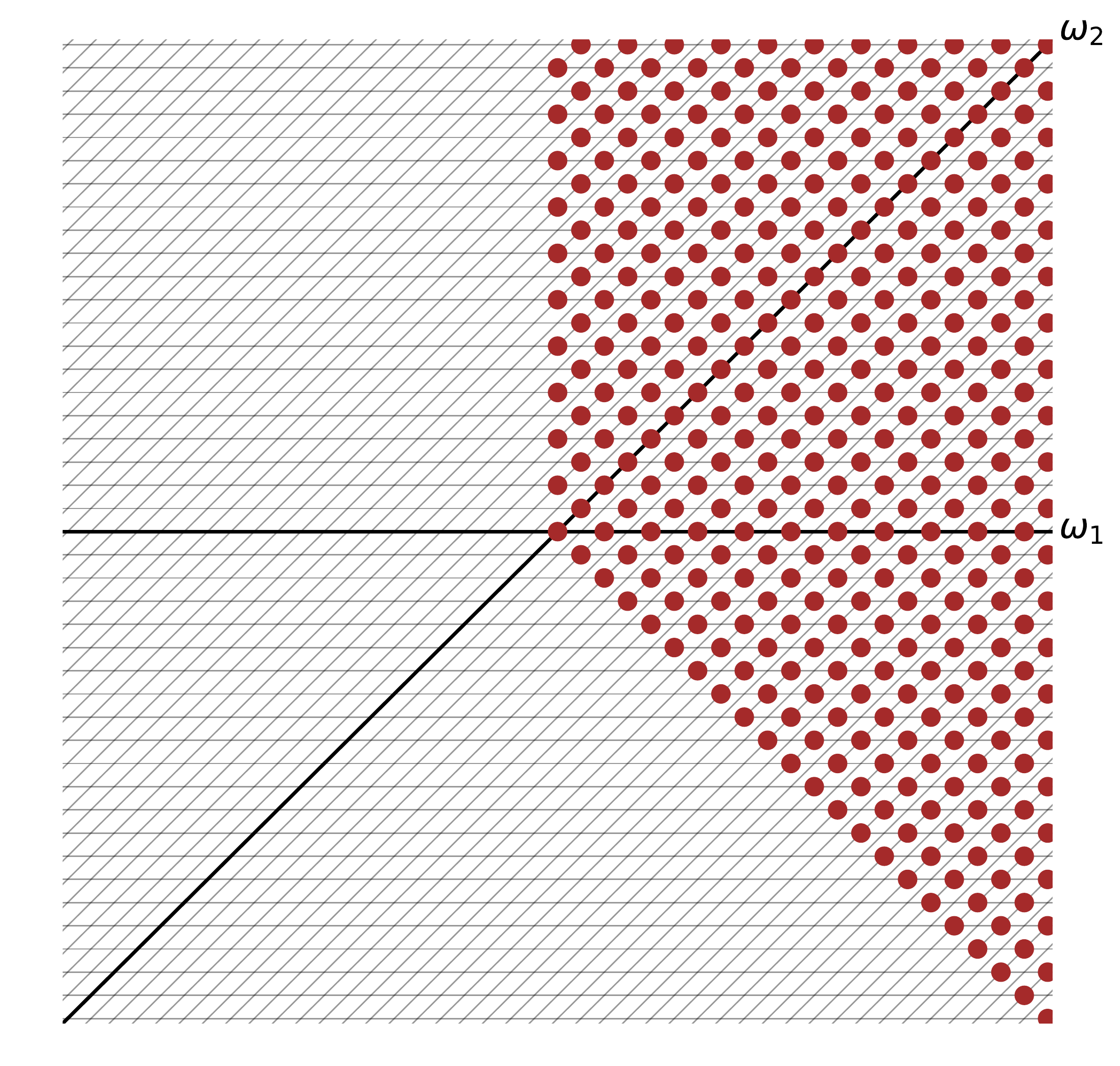}}}
    \hfill
    \subfloat[$\sigma = s_1$]{{\includegraphics[width=1.5in]{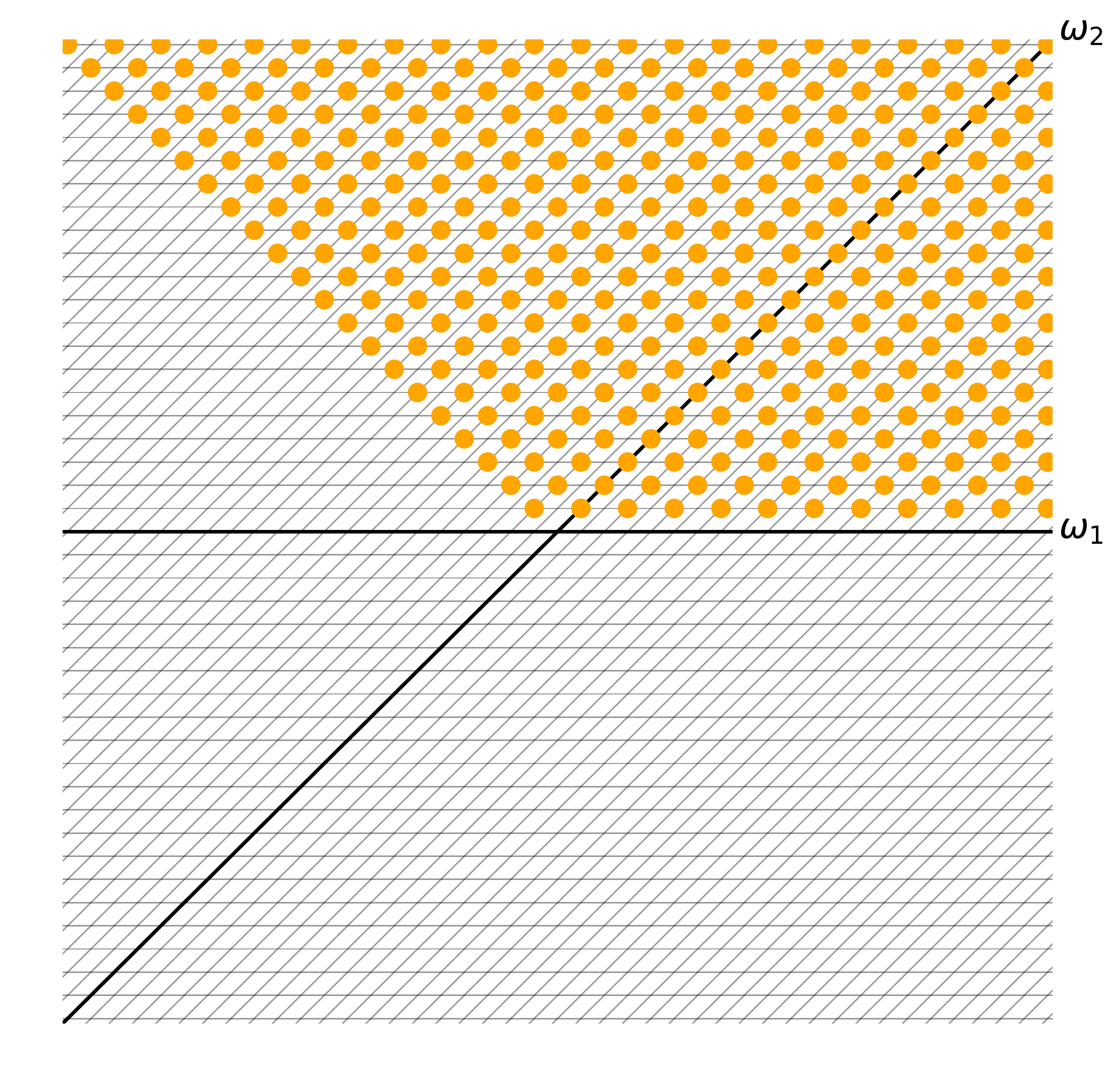} }}
    \hfill
    \subfloat[$\sigma = s_2$]{{\includegraphics[width=1.5in]{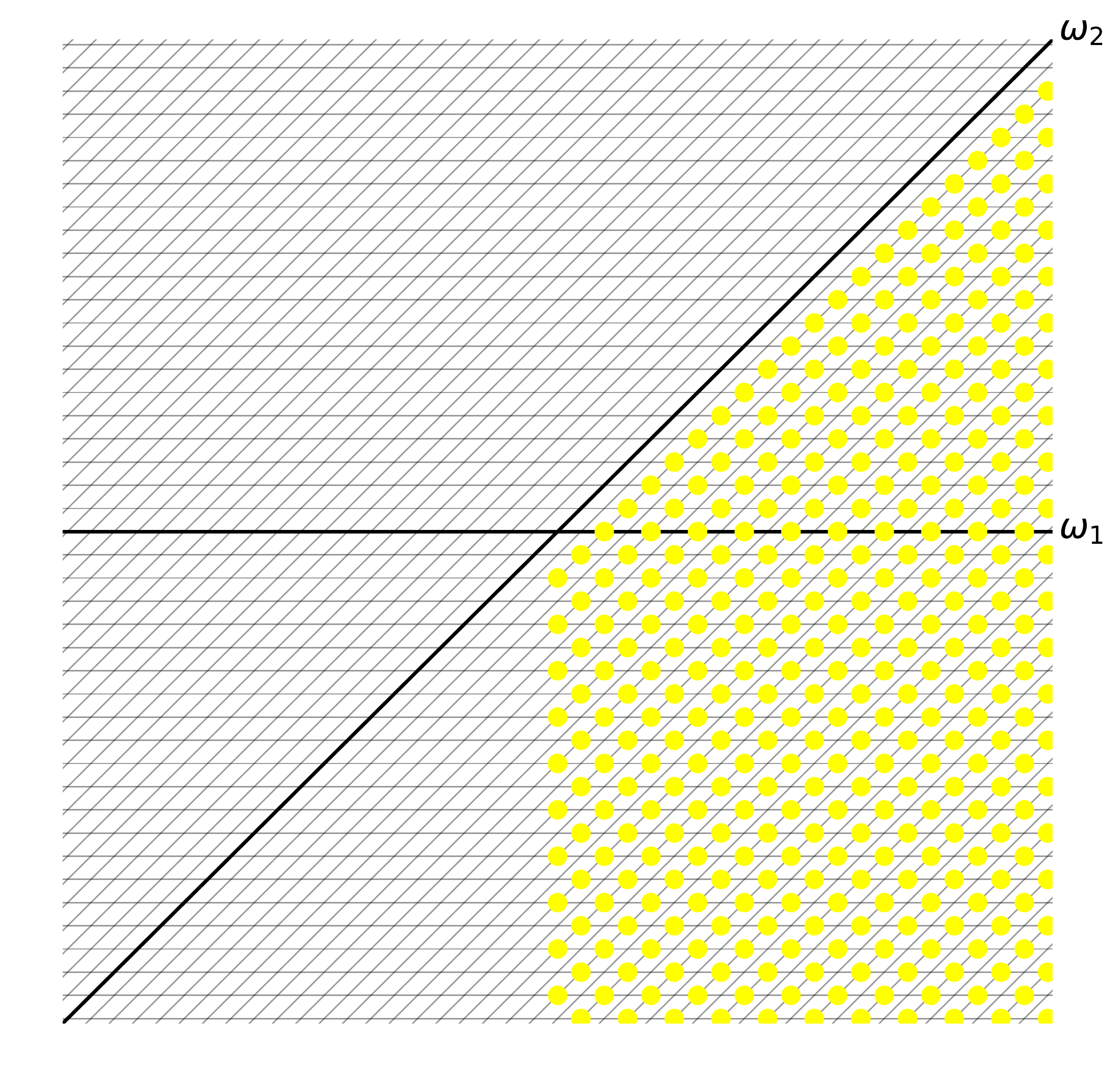} }}
    \hfill
    \subfloat[$\sigma = s_2s_1$]{{\includegraphics[width=1.5in]{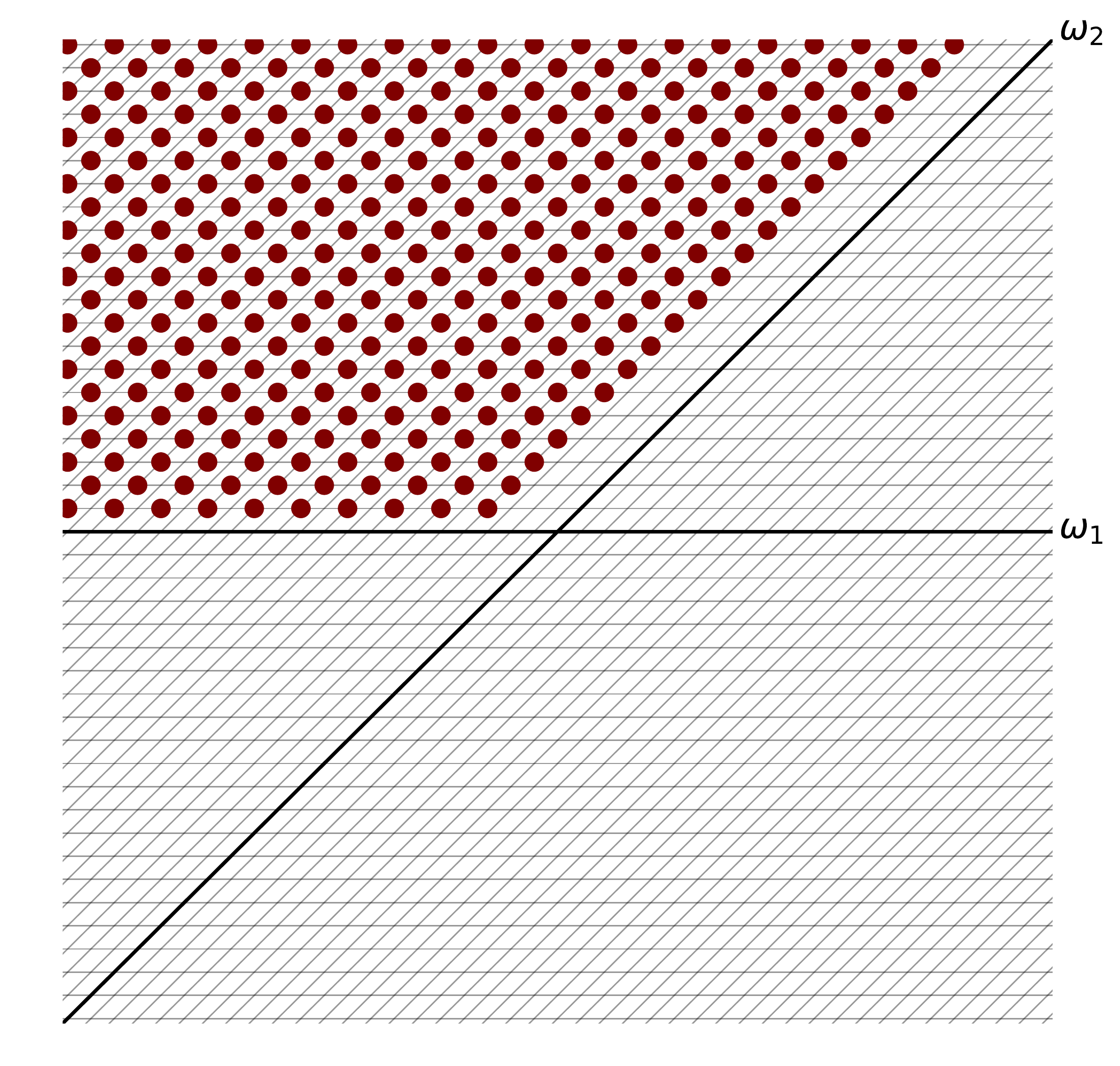} }}\\
    \subfloat[$\sigma = s_1s_2$]{{\includegraphics[width=1.5in]{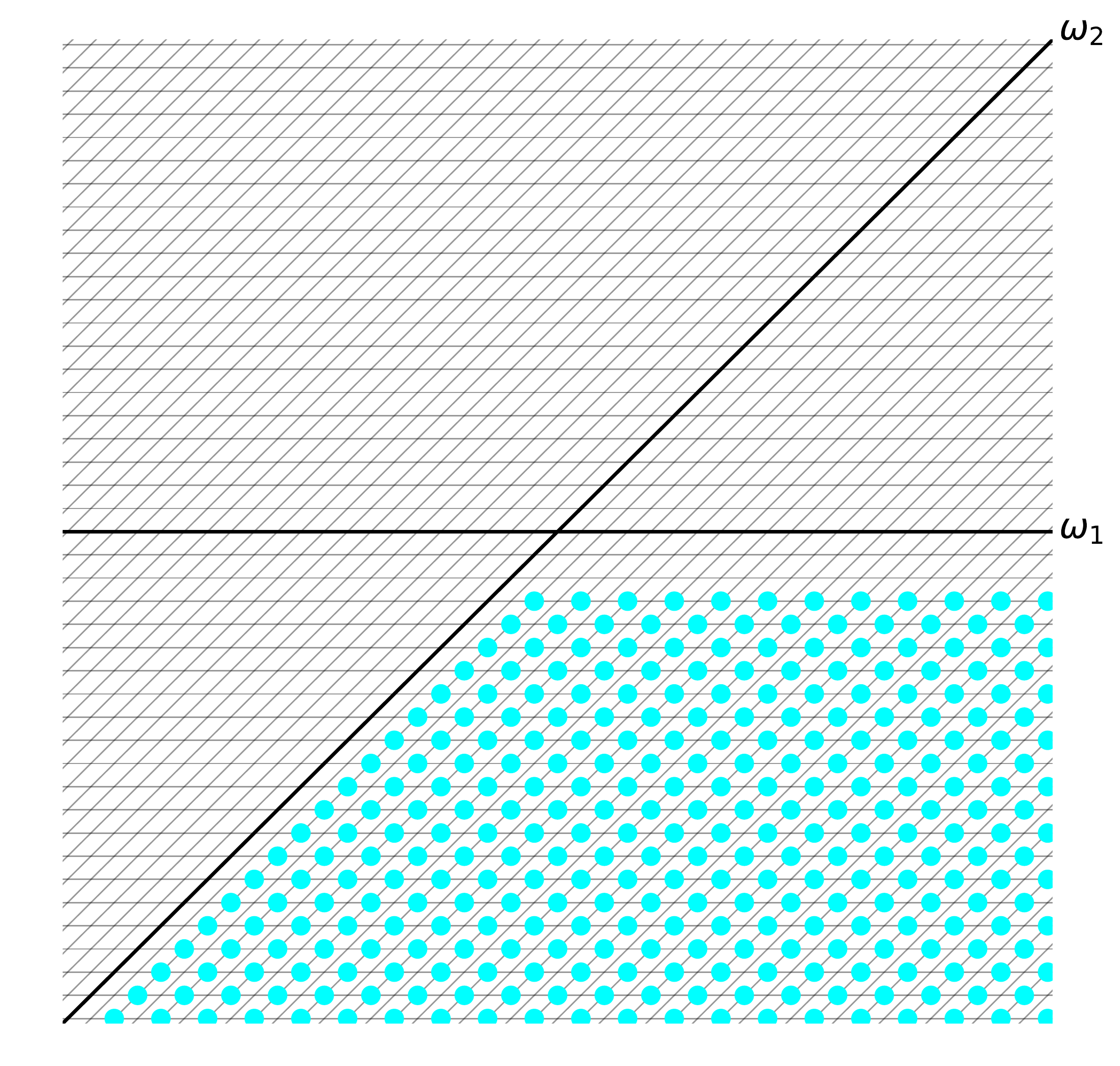} }}
    \hfill
    \subfloat[$\sigma = s_1s_2s_1$]{{\includegraphics[width=1.5in]{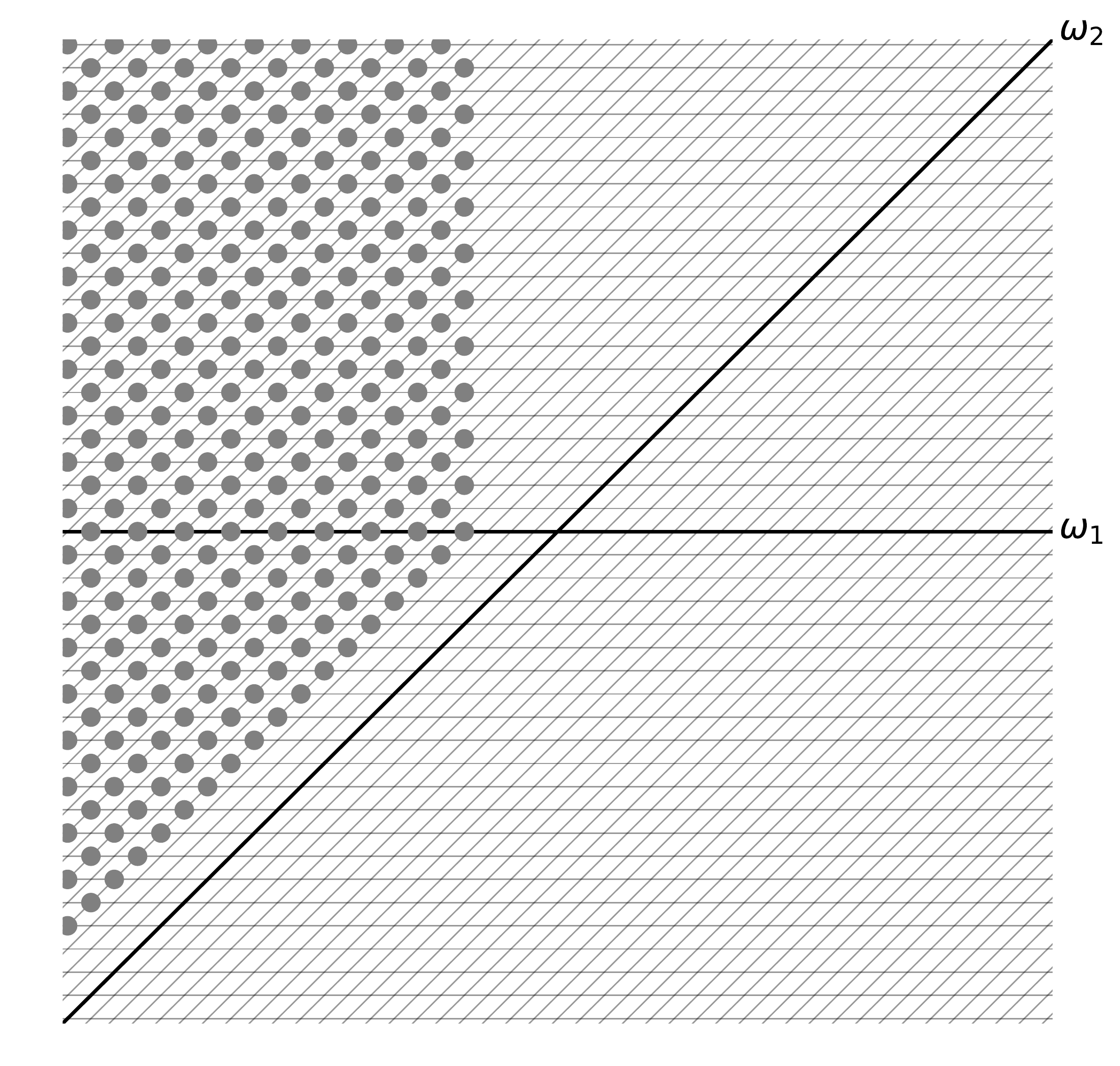} }}
    \hfill
    \subfloat[$\sigma = s_2s_1s_2$]{{\includegraphics[width=1.5in]{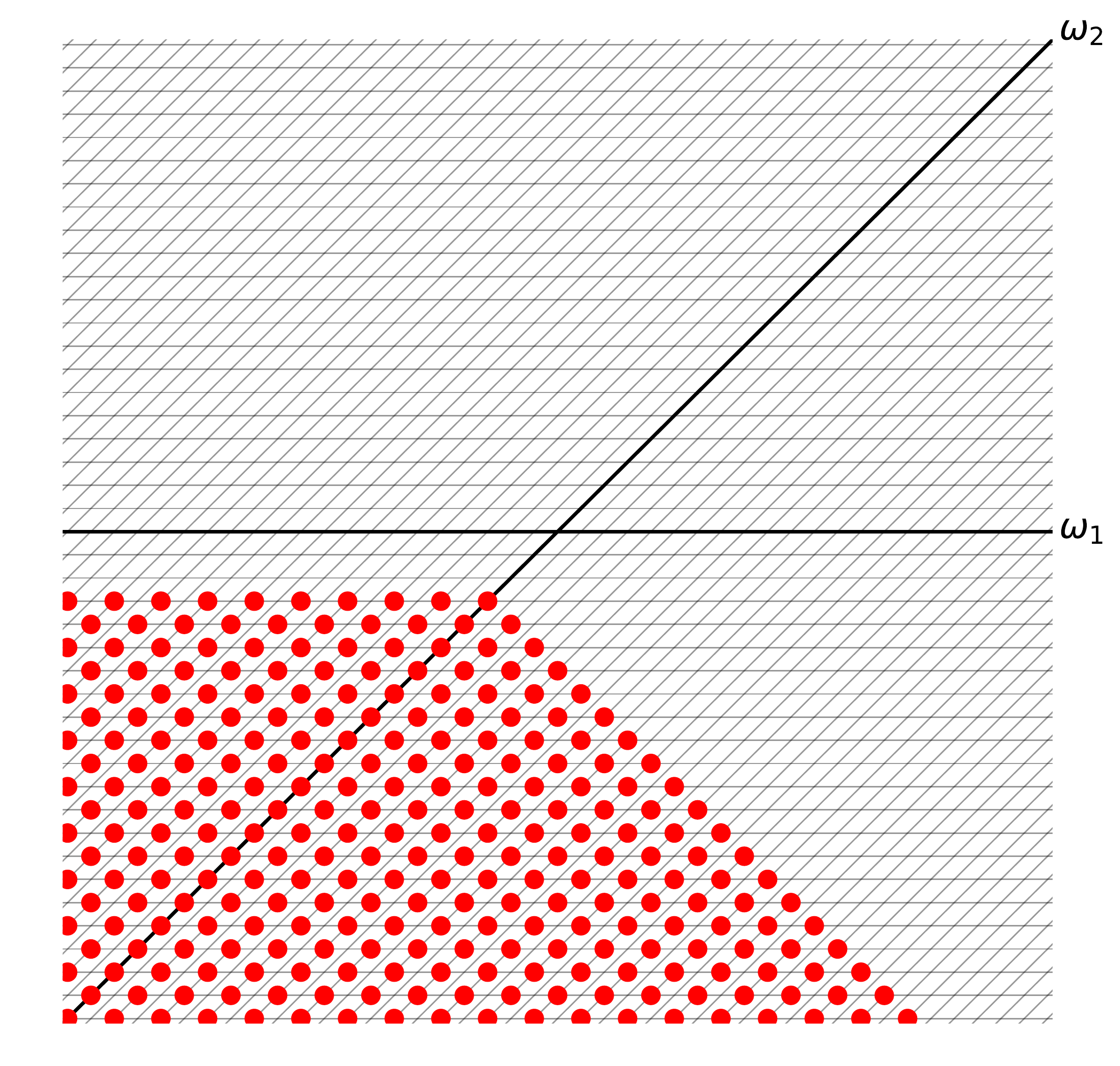} }}
    \hfill
    \subfloat[$\sigma = (s_2s_1)^2$]{{\includegraphics[width=1.5in]{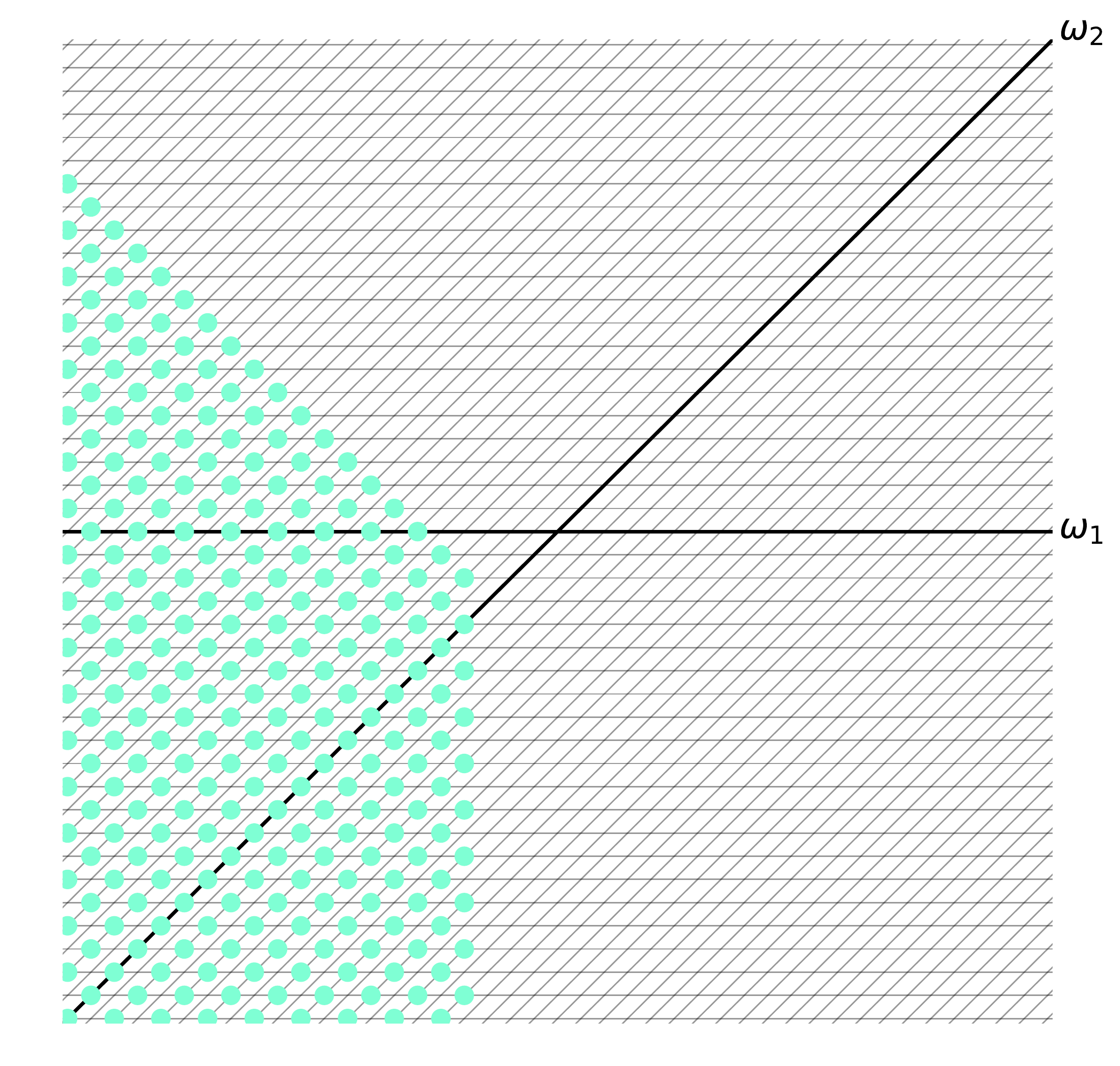} }}
    \caption{Solution sets to linear inequalities corresponding to the Lie algebra of type $C_2$.}
    \label{fig:c2_single_elements}
\end{figure}

% Section C_2 mu = n\w_1
\subsubsection{\normalfont{\textbf{Case}} \texorpdfstring{$\mu = n\w_1$}{mu equals n omega 1}}
Figures \ref{subfig:c2_n2}-\ref{subfig:c2_n8} illustrate the Weyl alternation diagrams for $\mu = n\w_1$ such that $n = 2,4,6,8$. We observe that the empty region is in the shape of a square with an edge on top. We also note that as $n$ increases from $2$ to $8$, the length of the edges of the square in the center also increases.
\begin{figure}[H]%
    \centering
    \subfloat[$\mu = 2\w_1$]{{\includegraphics[width=1.5in]{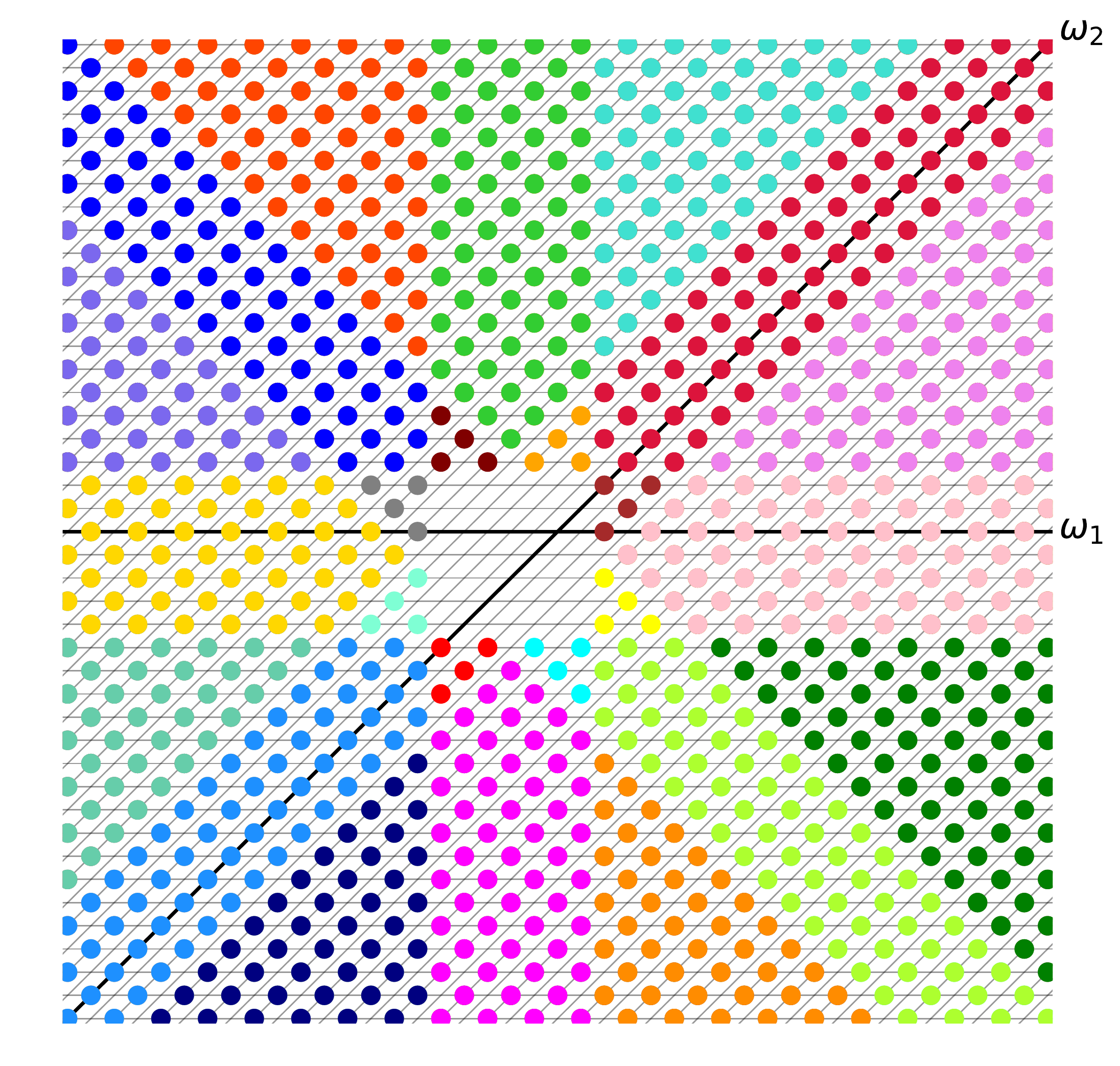} }
    \label{subfig:c2_n2}
    }
    \hfill
    \subfloat[$\mu = 4\w_1$]{{\includegraphics[width=1.5in]{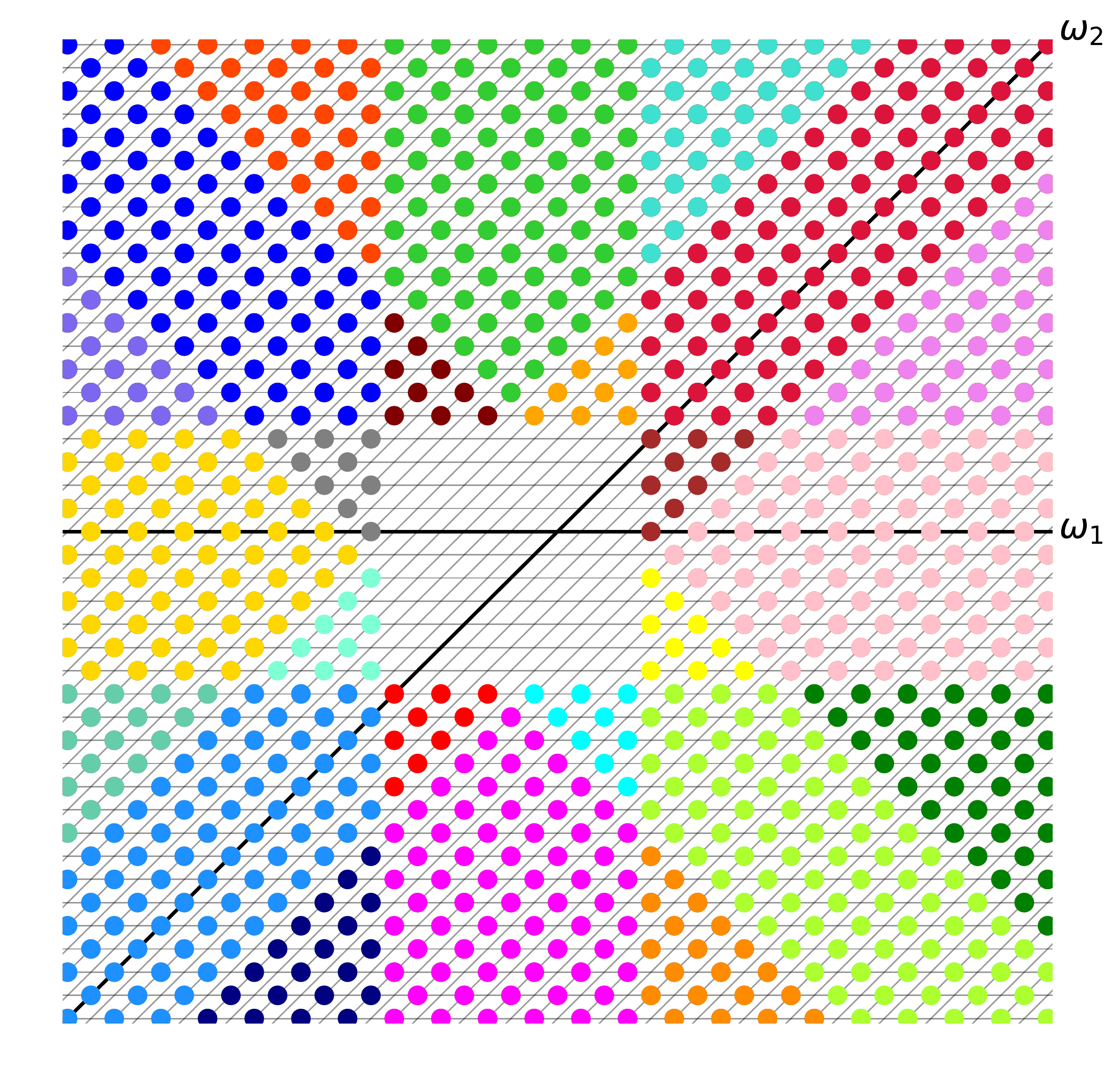} }
    \label{subfig:c2_n4}
    }
    \hfill
    \subfloat[$\mu = 6\w_1$]{{\includegraphics[width=1.5in]{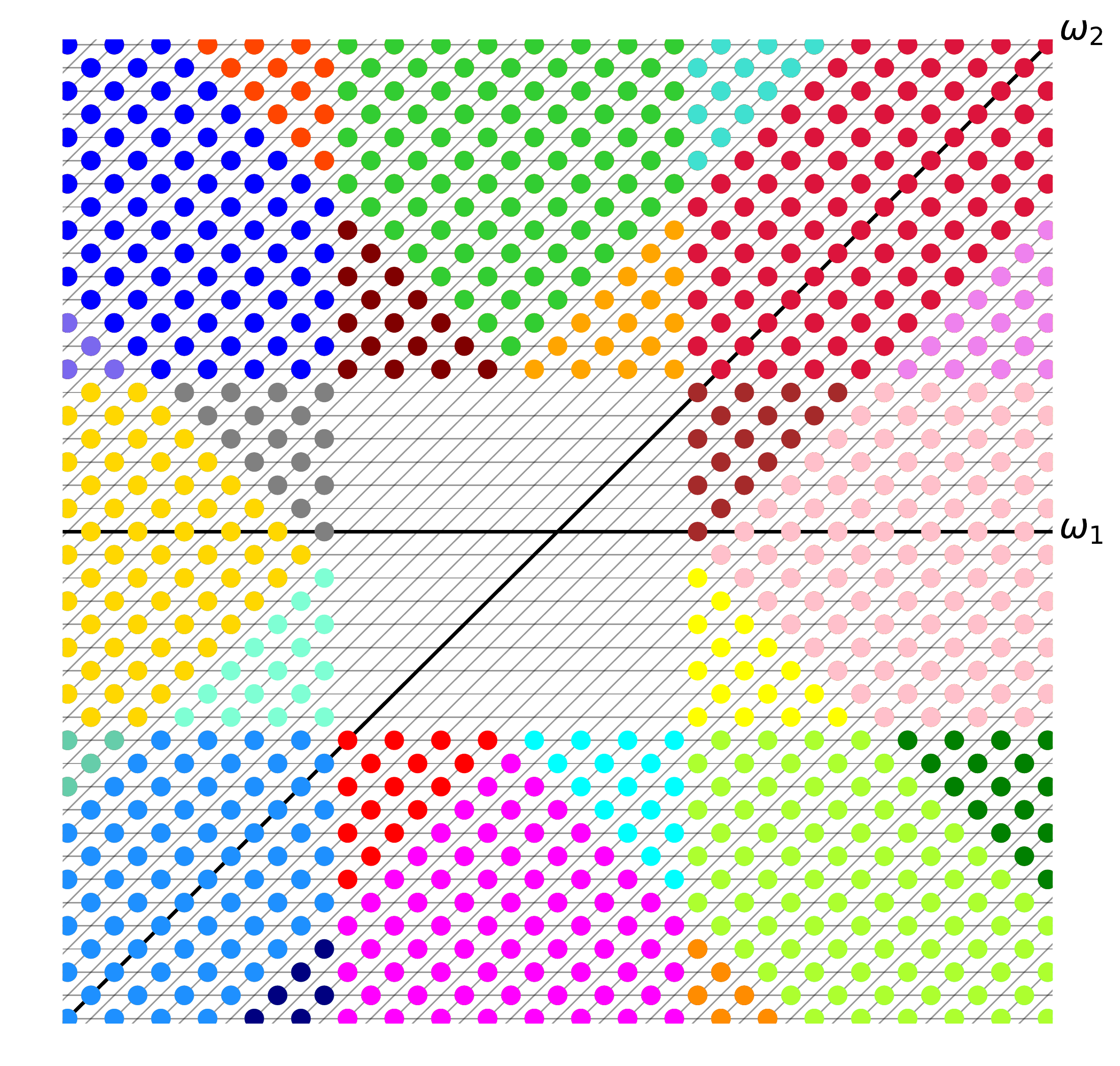} }
    \label{subfig:c2_n6}
    }
    \hfill
    \subfloat[$\mu = 8\w_1$]{{\includegraphics[width=1.5in]{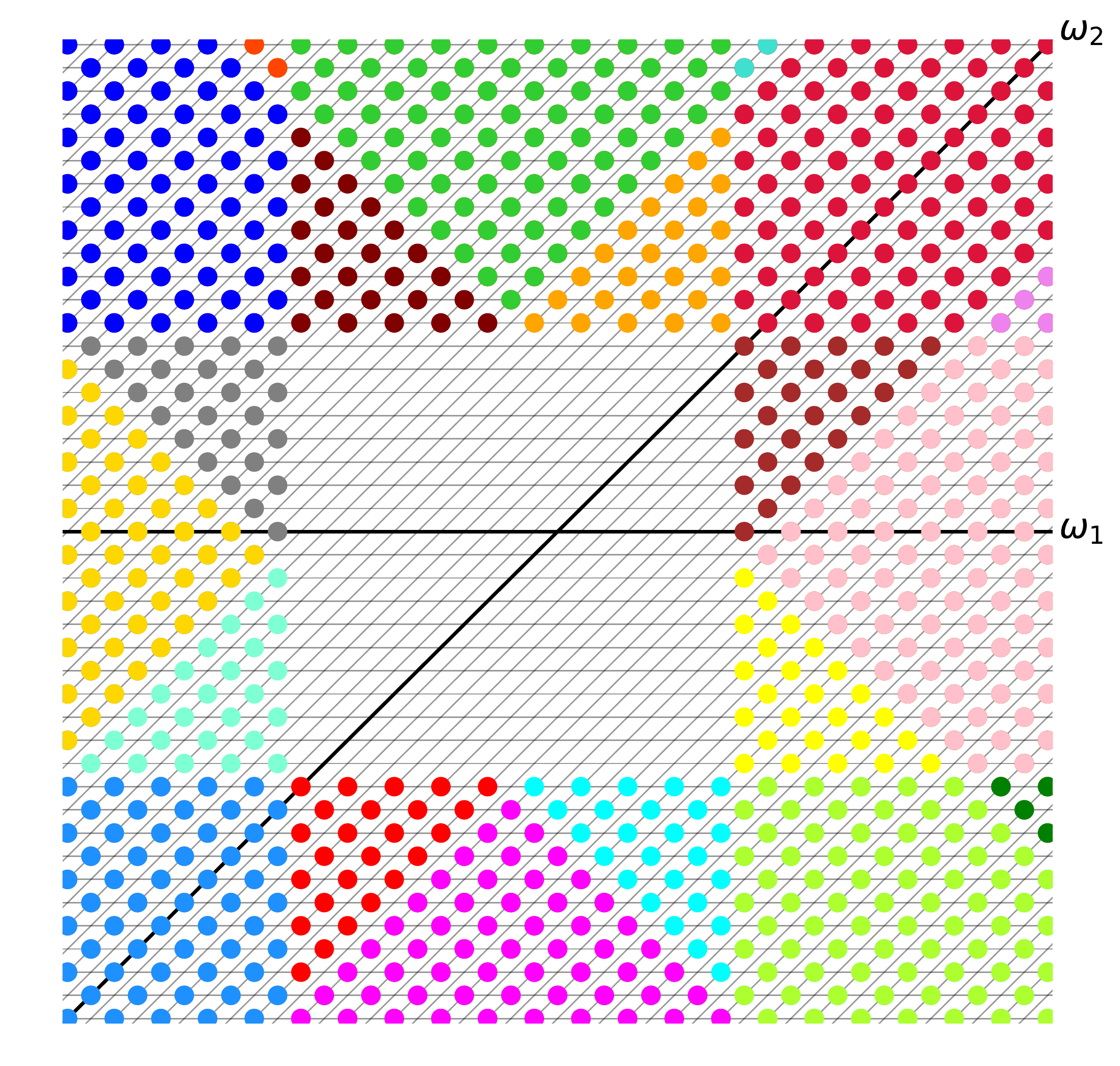} }
    \label{subfig:c2_n8}
    }
    \caption{Weyl alternation diagrams for the Lie algebra of type $C_2$ with $\mu=n\w_1$.}
    \label{fig:c2_mu_n}
\end{figure}
\subsubsection{\normalfont{\textbf{Case}} \texorpdfstring{$\mu = m\w_2$}{mu equals m omega 2}}
Figures \ref{subfig:c2_m1}-\ref{subfig:c2_m4} illustrate the Weyl alternation diagrams for $\mu = m\w_2$ such that $m = 1,2,3,4$. We observe that the empty region is in the shape of a square with a vertex on top. We also note that as $m$ increases from $1$ to $4$, the length of the edges of the square in the center also increases.
\begin{figure}[H]%
    \centering
    \subfloat[$\mu = \w_2$]{{\includegraphics[width=1.5in]{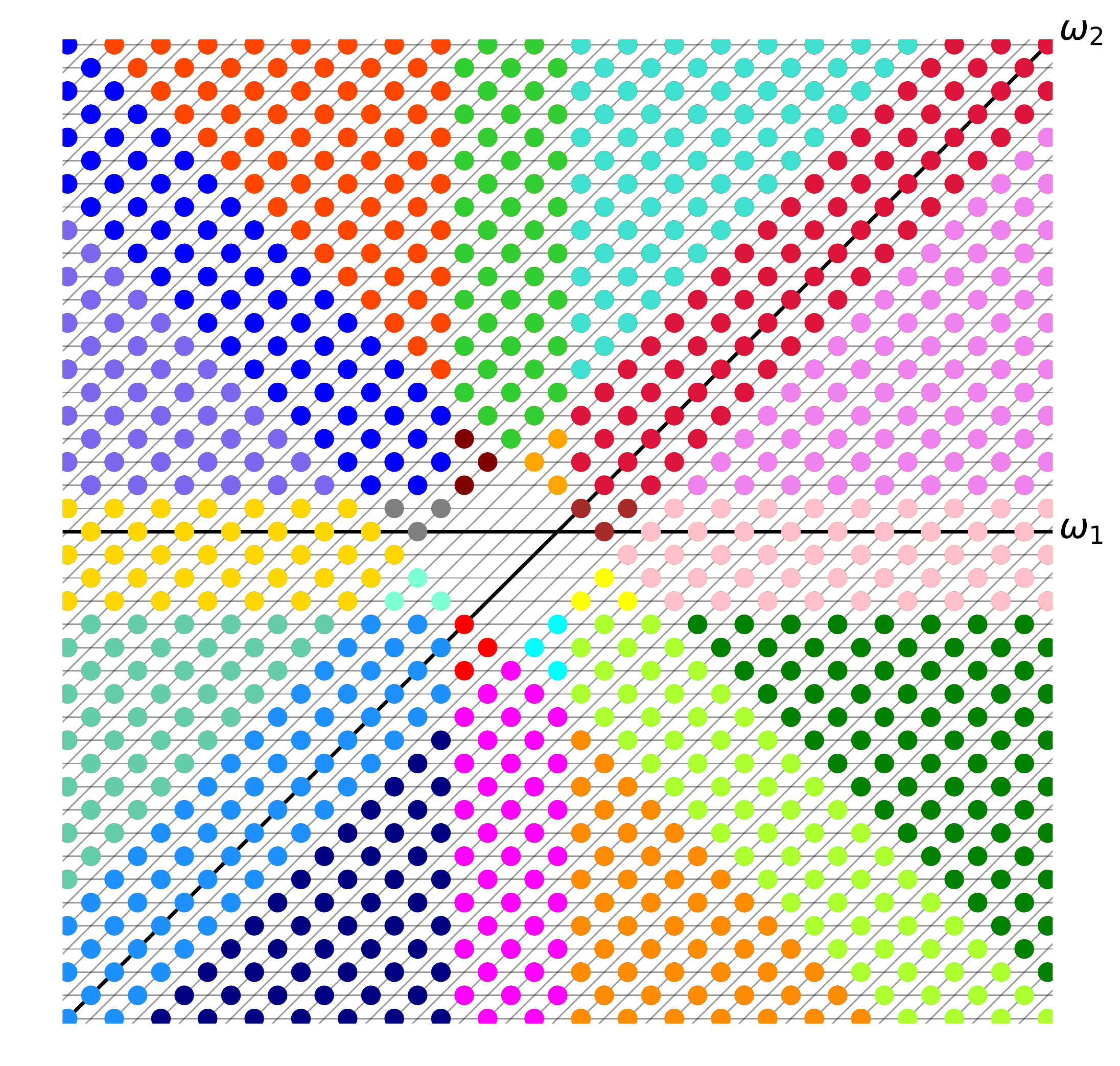}}
    \label{subfig:c2_m1}
    }%
    \hfill
    \subfloat[$\mu = 2\w_2$]{{\includegraphics[width=1.5in]{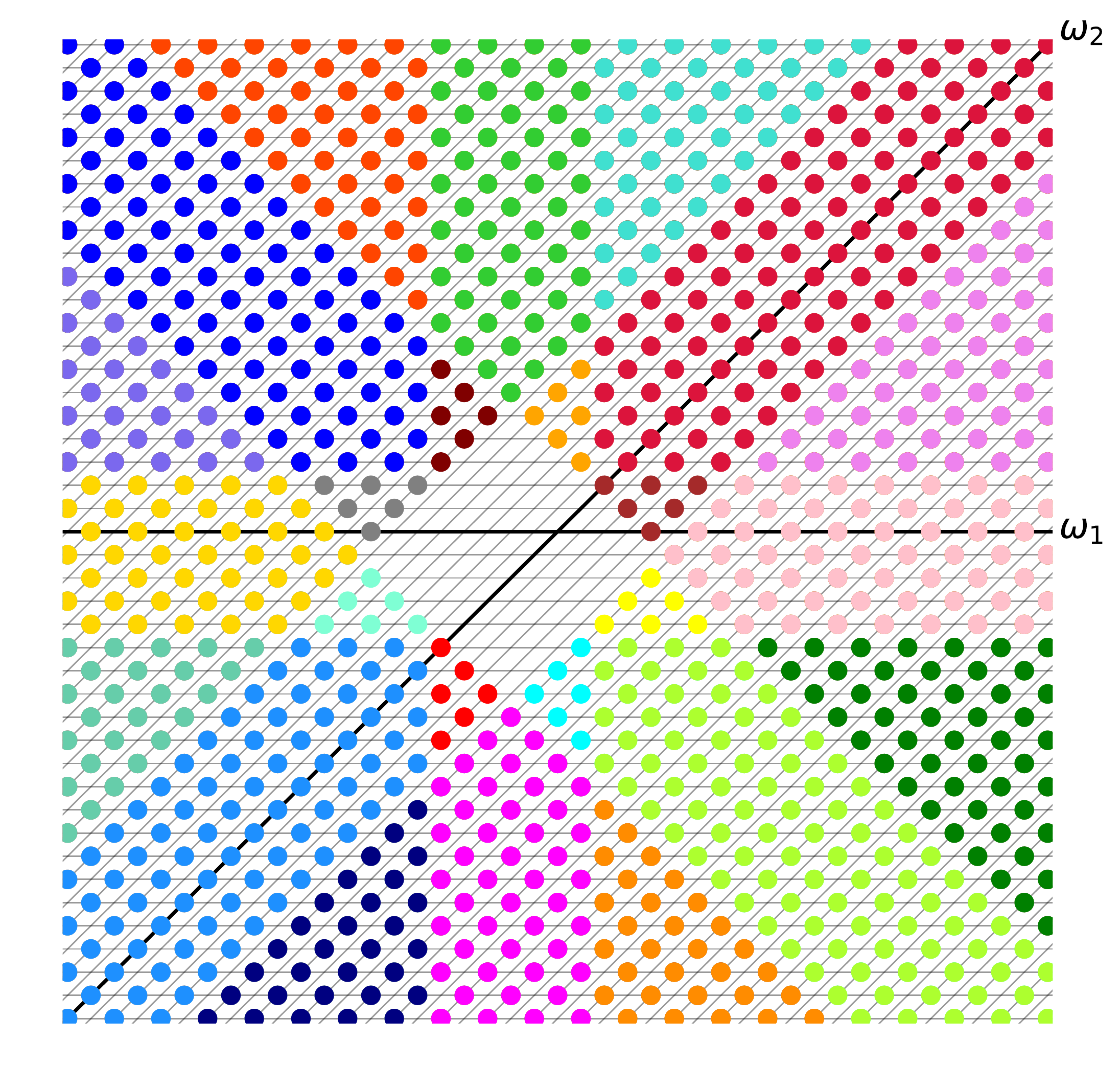} }
    \label{subfig:c2_m2}
    }
    \hfill
    \subfloat[$\mu = 3\w_2$]{{\includegraphics[width=1.5in]{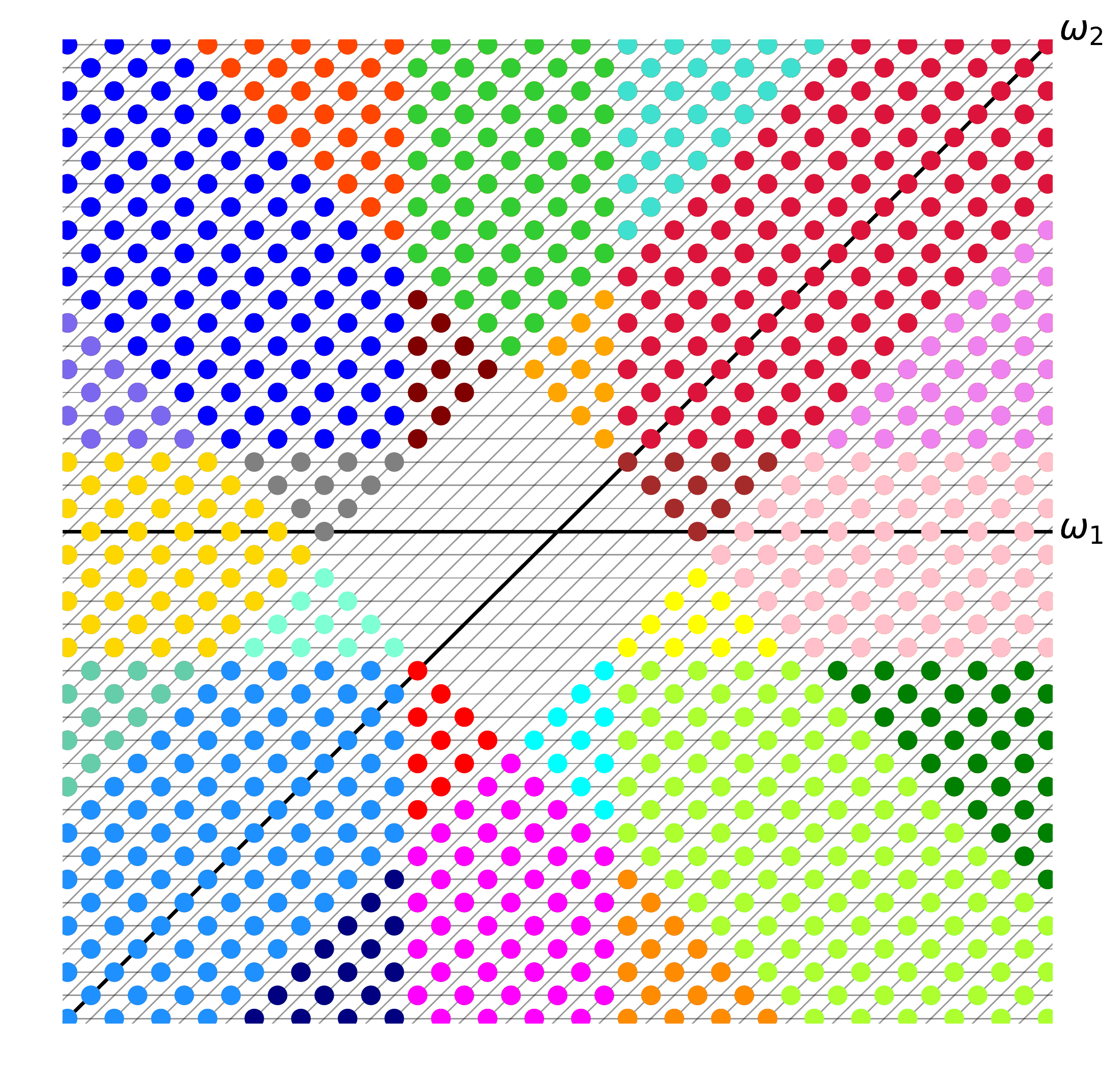}}
    \label{subfig:c2_m3}
    }%
    \hfill
    \subfloat[$\mu = 4\w_2$]{{\includegraphics[width=1.5in]{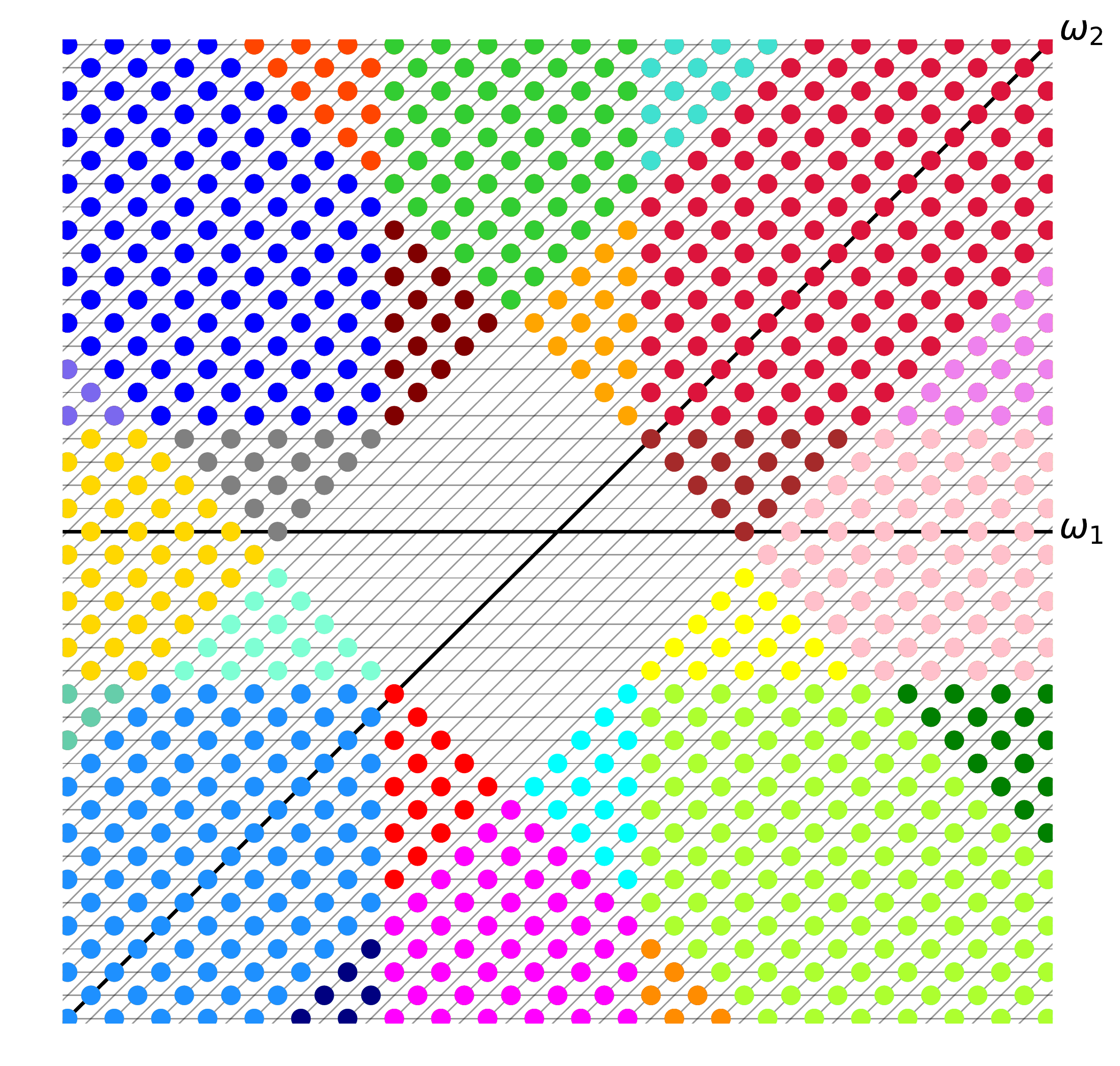} }
    \label{subfig:c2_m4}
    }
    \caption{Weyl alternation diagrams for the Lie algebra of type $C_2$ with $\mu=m\w_2$.}
    \label{fig:c2_mu_m}
\end{figure}
\subsubsection{\normalfont{\textbf{Case}} \texorpdfstring{$\mu = n\w_1 + m\w_2$}{mu equals n omega 1 plus m omega 2}}
Figures \ref{subfig:c2_n2m1}-\ref{subfig:c2_n2m3} illustrate the Weyl alternation diagrams for $\mu = n\w_1 + m\w_2$ such that $\mu$ is a positive integral linear combination of the fundamental weights. We observe that the empty region is in the shape of an 8-pointed star. Additionally, as $n$ and $m$ both increase, so does the size of the star in the center.
\begin{figure}[H]%
    \centering
    \subfloat[$\mu = 2\w_1 + \w_2$]{{\includegraphics[width=1.5in]{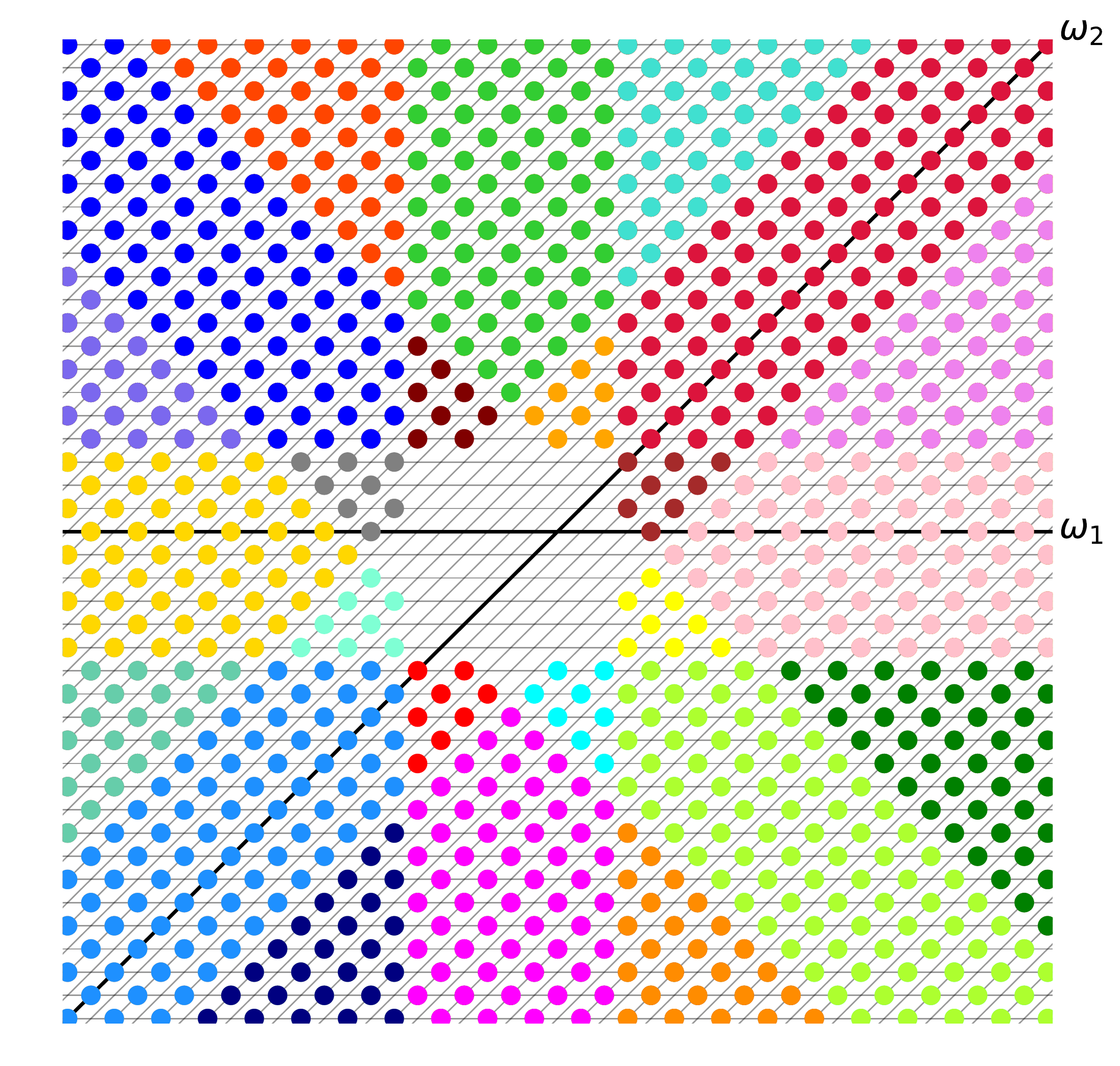}}
    \label{subfig:c2_n2m1}
    }
    \hfill
    \subfloat[$\mu = 2\w_1 + 2\w_2$]{{\includegraphics[width=1.5in]{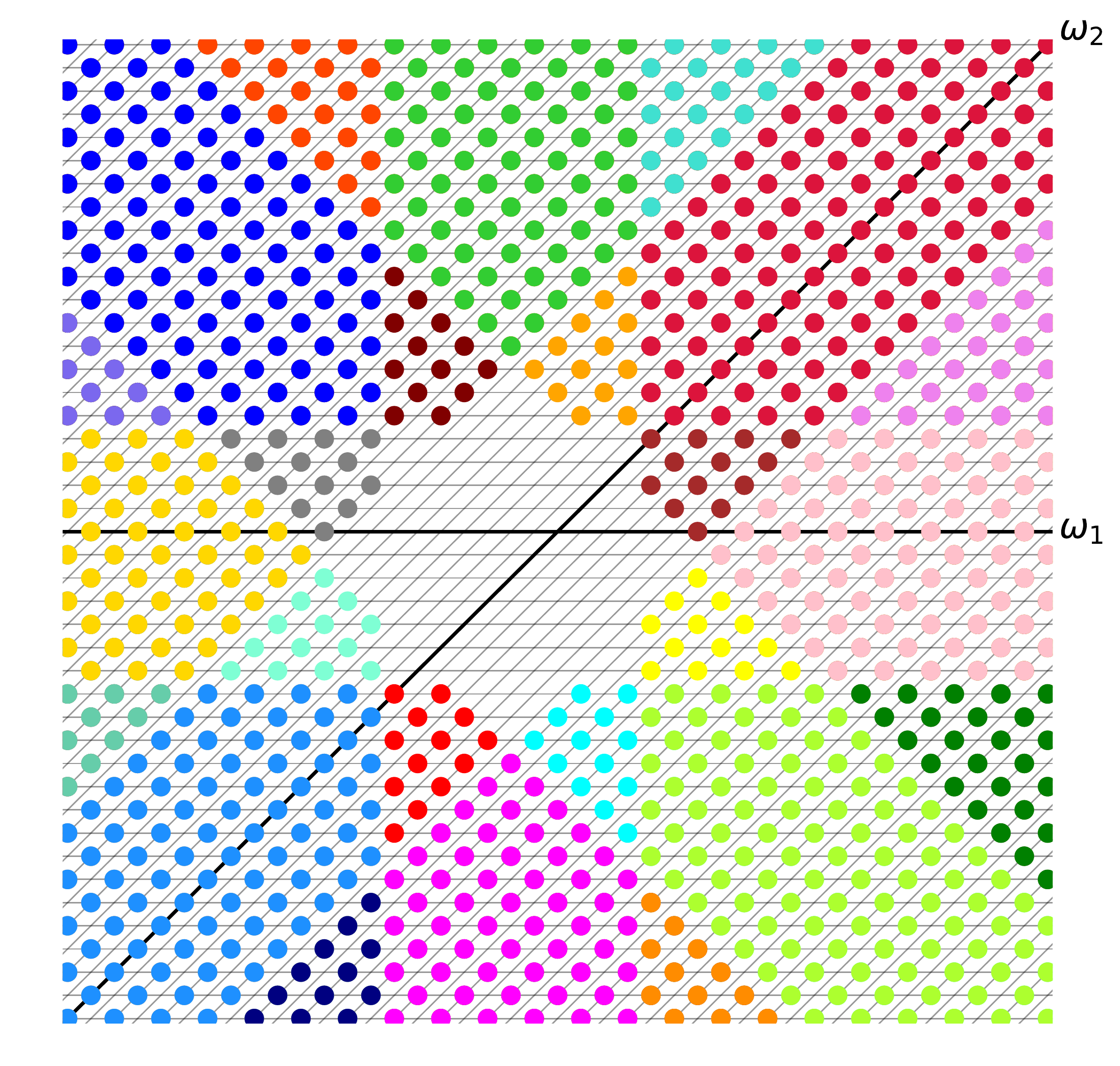} }
    \label{subfig:c2_n2m2}
    }
    \hfill
    \subfloat[$\mu = 4\w_1+2\w_2$]{{\includegraphics[width=1.5in]{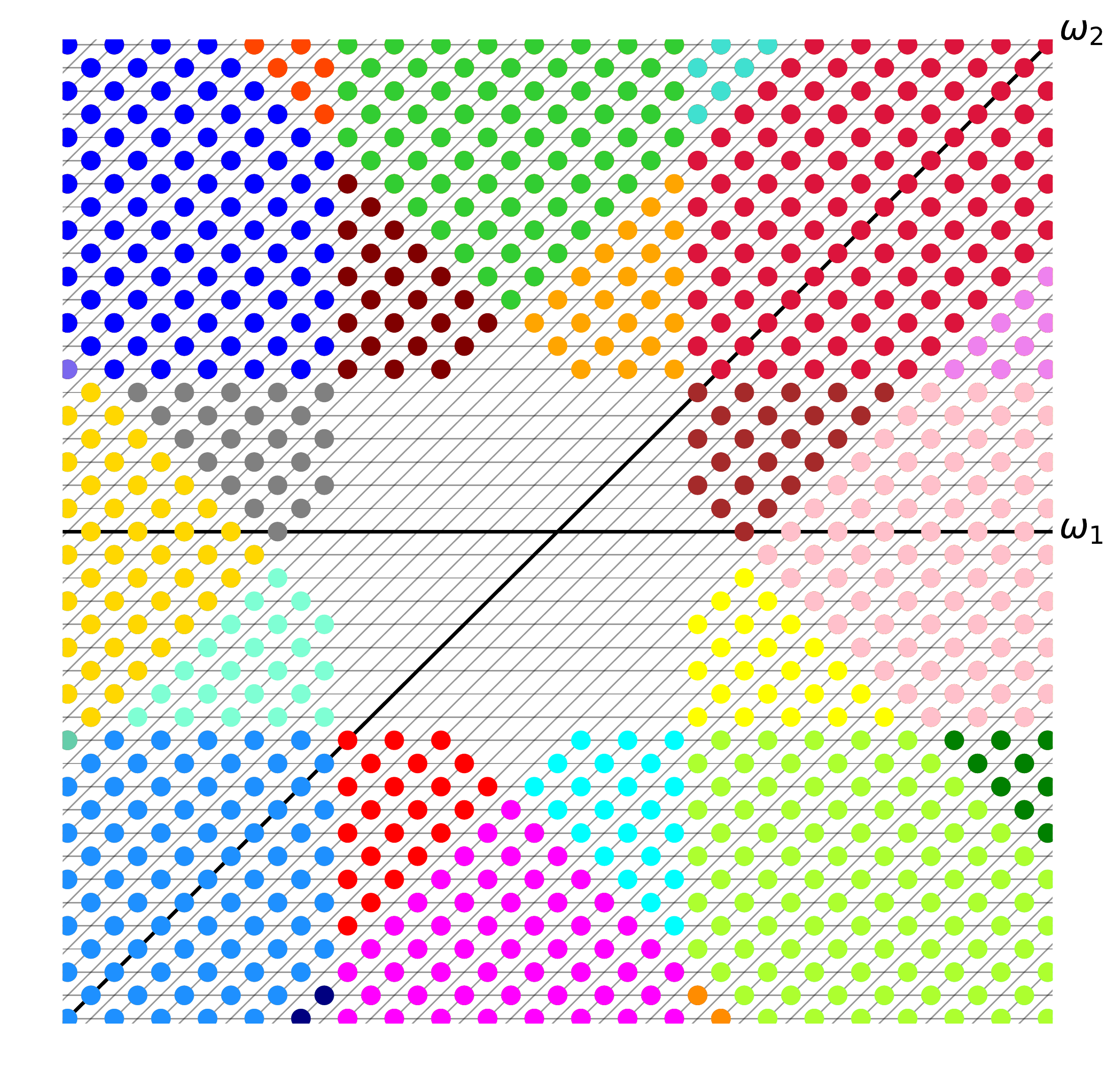}}
    \label{subfig:c2_n4m2}
    }%This one is done
    \hfill
    \subfloat[$\mu = 2\w_1+3\w_2$]{{\includegraphics[width=1.5in]{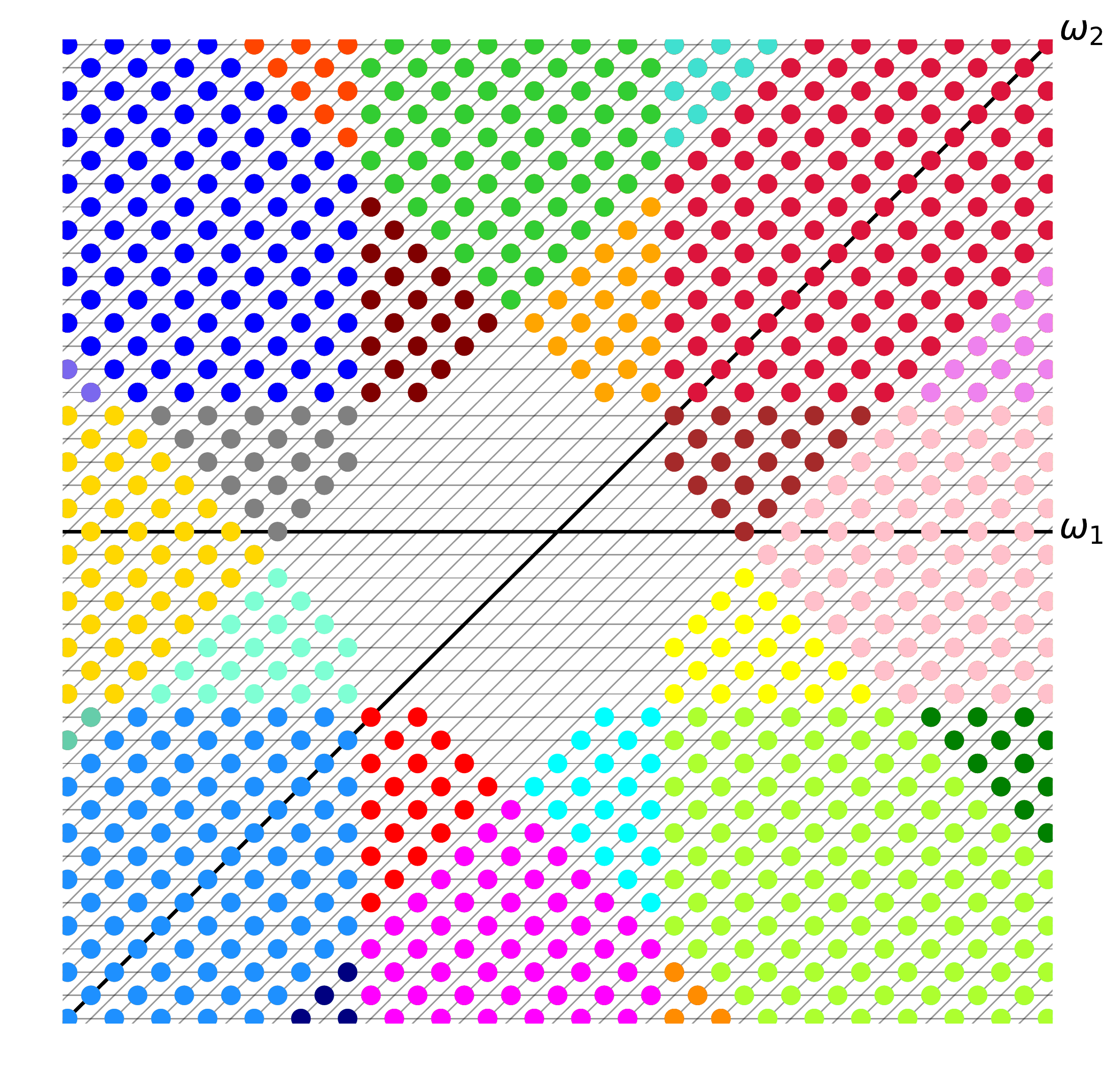}}
    \label{subfig:c2_n2m3}
    }%This one is done
    \caption{Weyl alternation diagrams for the Lie algebra of type $C_2$ with \mbox{$\mu=n\w_1+m\w_2$.}}
    \label{fig:c2_mu_nm}
\end{figure}
% Inequalities for c2
 
\begin{figure}[H]%
\centering
\resizebox{2in}{!}{
\includegraphics{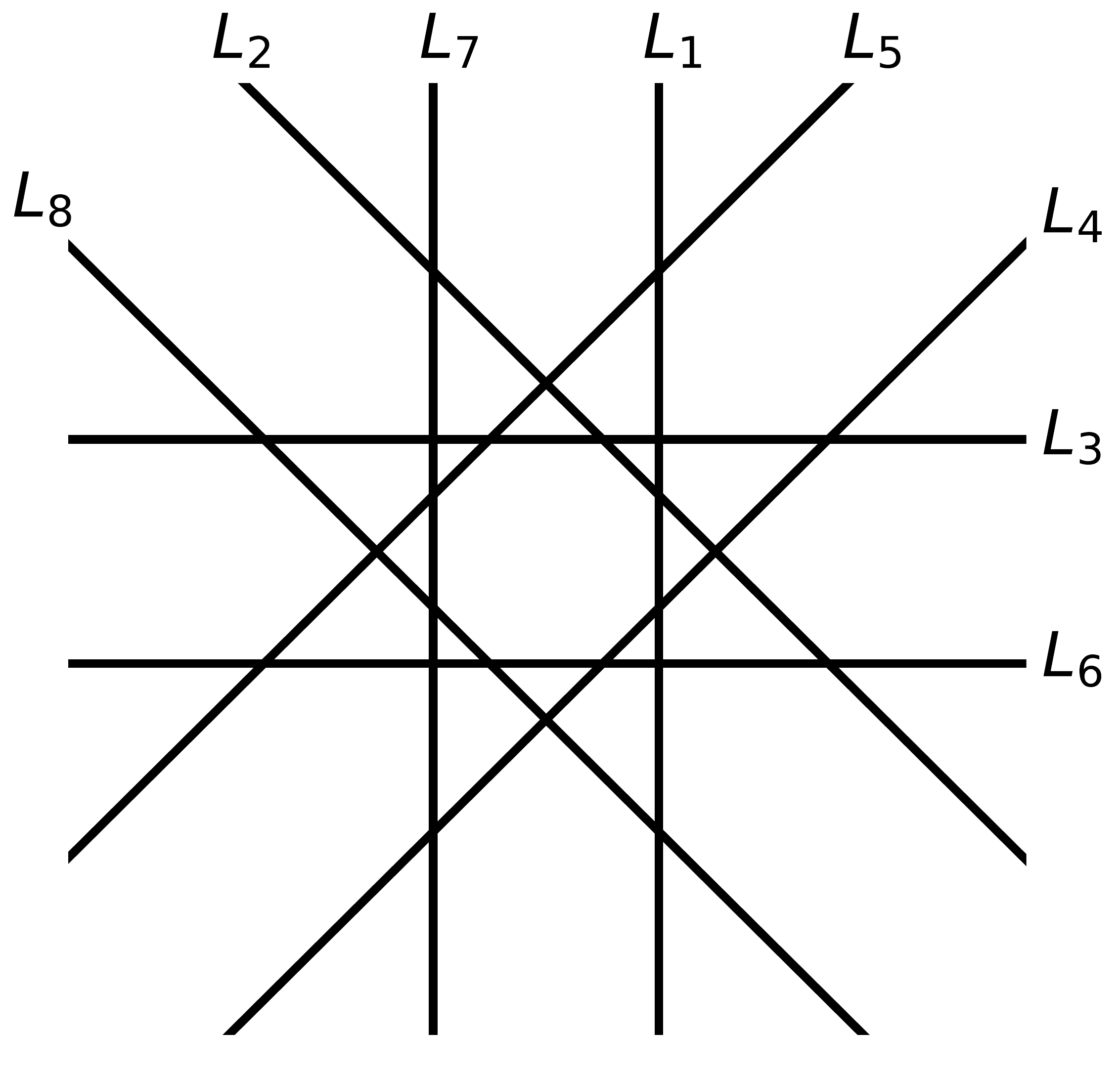}
}
\caption{Set of linear inequalities determining the boundaries of the Weyl alternation sets for the Lie algebra of type $C_2$.}
\label{fig:c2_ineqs}
\end{figure}

To explain the shapes that form in the empty region of each Weyl alternation diagram for the Lie algebra of type $C_2$ we turn to Figure \ref{fig:c2_ineqs}. From Theorem \ref{thm:mainc2}, we notice that all inequalities depend on $n$ and $m$ which simply shift the inequalities, but never change the direction of the line. This means that Figure \ref{fig:c2_ineqs} is a good representation of the inequalities formed.

%-------------------

\begin{figure}[H]%
\centering
\subfloat[Edge on top.]{
\includegraphics[width=4cm]{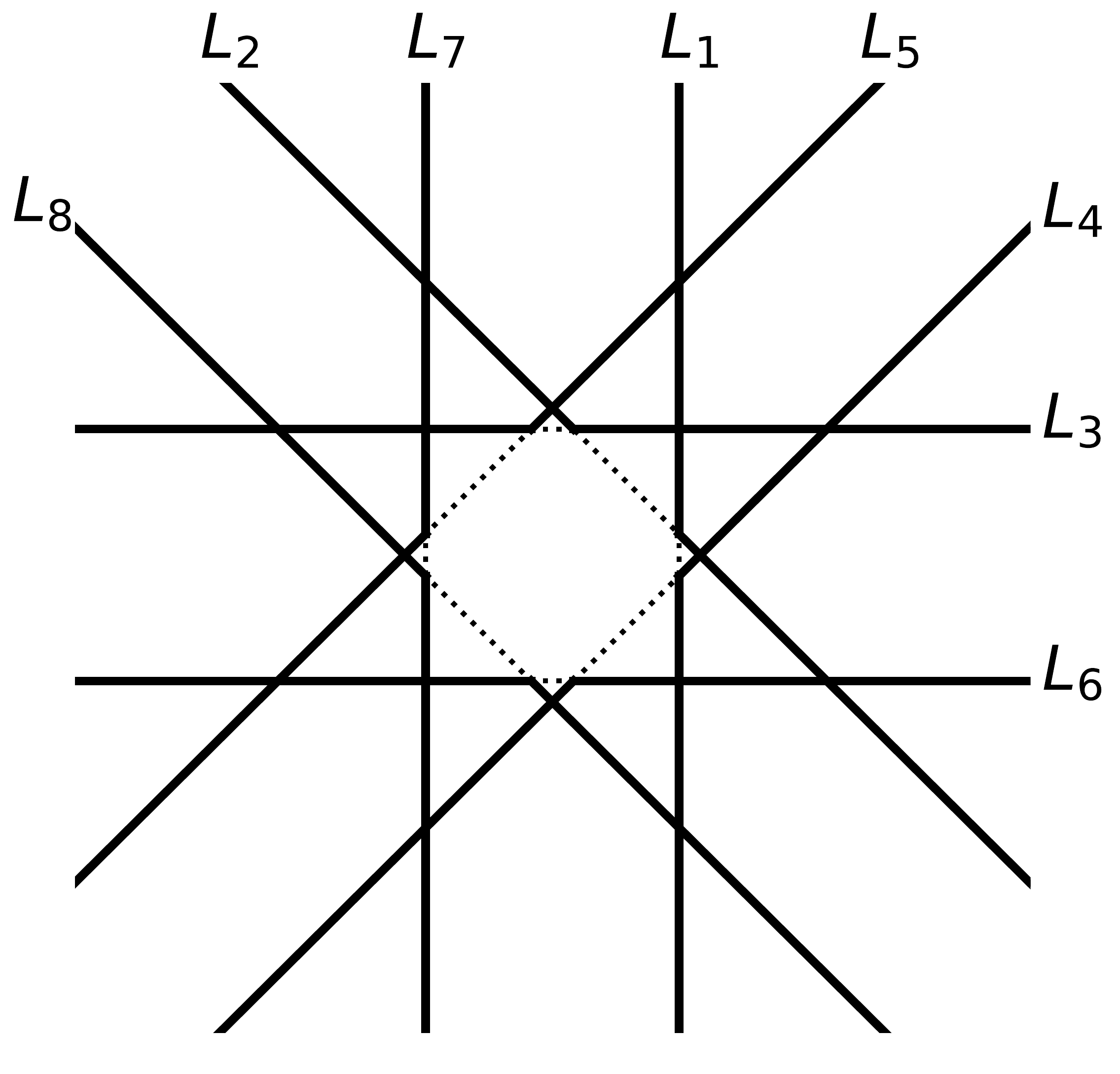}\label{subfig:c2_edgeontop}
}
\quad
\subfloat[Vertex on top.]{
\includegraphics[width=4cm]{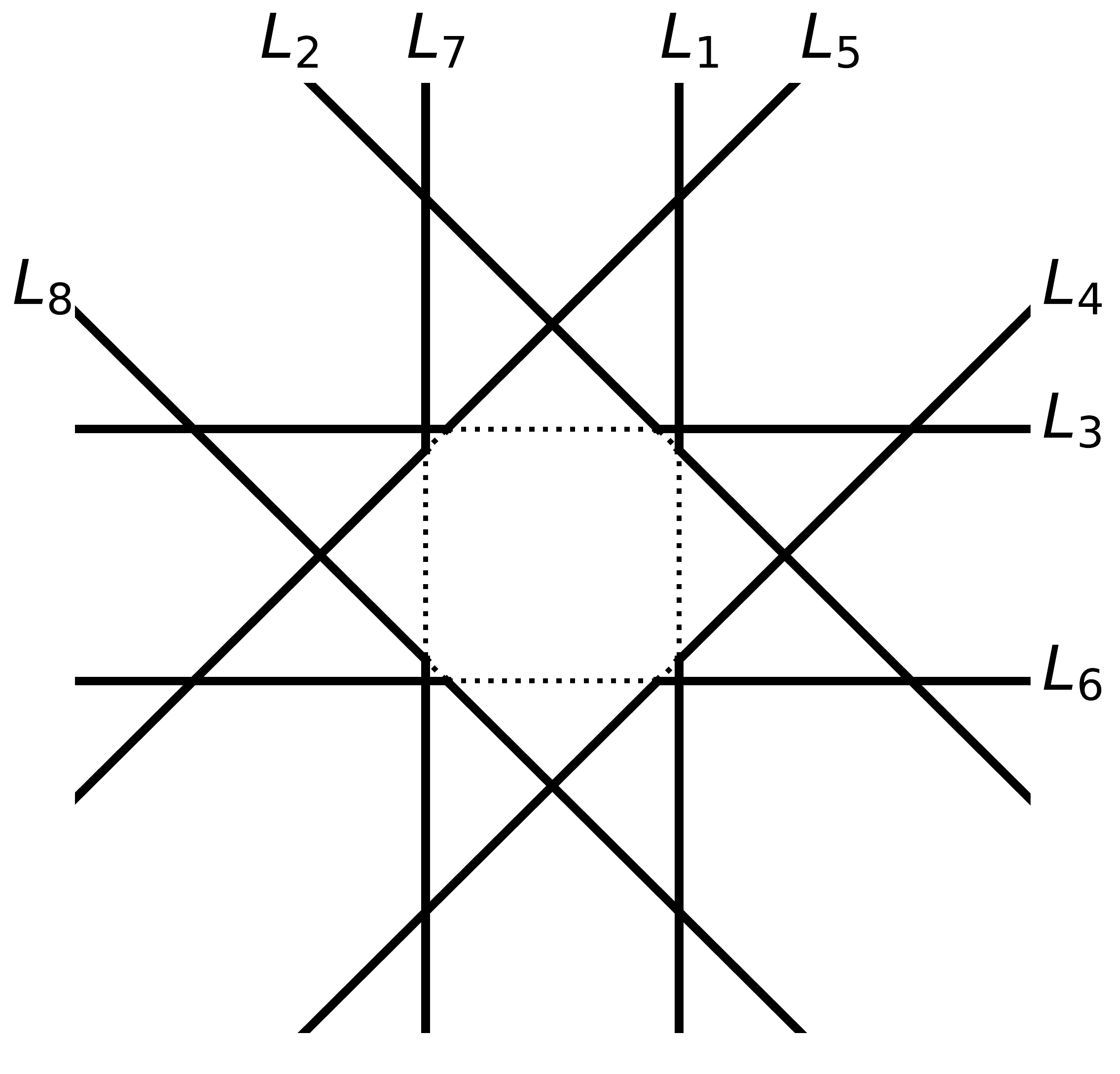}\label{subfig:c2_vertexontop}
}
\quad
\subfloat[8-pointed star.]{
\includegraphics[width=4cm]{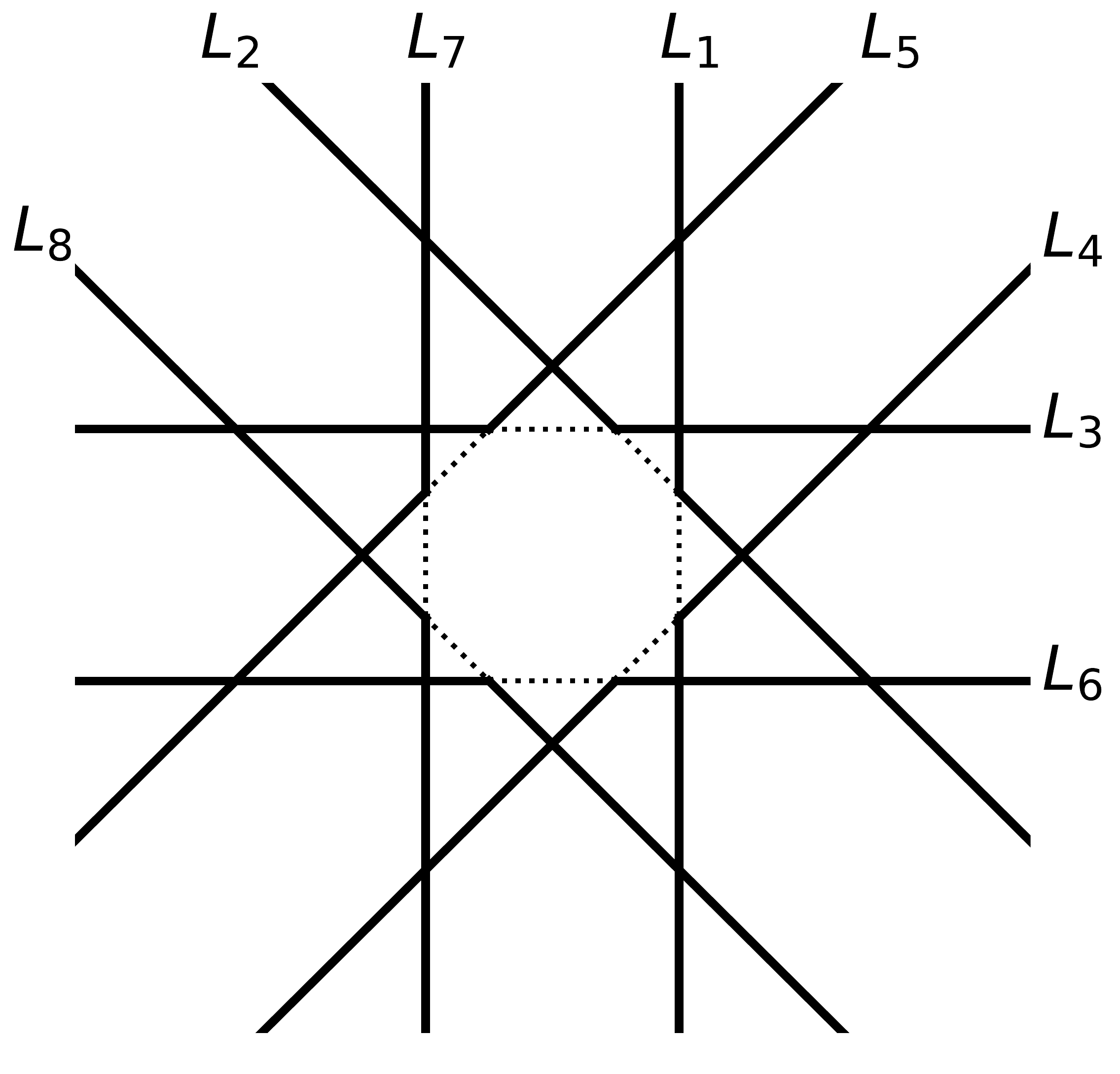}\label{subfig:c2_8pointedstar}
}
\caption{Different formations of the empty region for the Lie algebra of type $C_2$.}
\label{fig:c2_cases}
\end{figure}

From Figure \ref{subfig:c2_edgeontop}, observe that the empty region becomes a square with an edge on top if and only if the inequalities $L_2$ and $L_5$ intersect at or below $L_3$ and the inequalities $L_1$ and $L_3$ intersect strictly above $L_2$. This occurs exactly when $\mu = n\w_1$ such that $n \in 2\NN$. Notice, in Figure \ref{subfig:c2_edgeontop}, inequalities $L_2$ and $L_5$ intersect at a point where the divisibility condition is not satisfied, thus the intersection occurs at $L_3$ as desired. Similarly, the empty region becomes a square with a vertex pointing up, as depicted in Figure \ref{subfig:c2_vertexontop}, when the inequalities $L_2$ and $L_5$ intersect strictly above $L_3$ and the inequalities $L_1$ and $L_3$ intersect at or below $L_2$. This occurs exactly when $\mu = m\w_2$ such that $m \in \NN.$ Notice, in Figure \ref{subfig:c2_vertexontop}, inequalities $L_1$ and $L_3$ intersect at a point where the divisibility condition is not satisfied, thus the intersection occurs at $L_2$ as desired. The empty region takes the shape of an 8-pointed star if and only if the inequalities $L_2$ and $L_5$ intersect strictly above $L_3$ and the inequalities $L_1$ and $L_3$ intersect strictly above $L_2$. This is depicted in Figure \ref{subfig:c2_8pointedstar}. Namely, this occurs when $\mu = n\w_1+m\w_2$ such that $n,m\in\NN$ and $2|n$.
%---------------------------------------------------------------------------------------------------------
% Section D_2
\subsection{Lie algebra of type \texorpdfstring{$D_2$}{D2}}
For each $\sigma\in W$, we plot the conditions in Table~\ref{tab:WeylD2} by placing a solid colored dot on the integral weights for which $\sigma(\lambda+\rho)-(\mu+\rho) \in \mathbb{N}\a_1 \oplus \mathbb{N}\a_2$. In Figure \ref{fig:d2_single_elements}, we let $\mu=0$ and present the corresponding region for each Weyl group element. Note that changing $\mu$ will only translate the solution sets. In what follows we describe how the Weyl diagrams change as we vary the weight $\mu$.
\begin{figure}[H]%
    \centering
    \subfloat[$\sigma = 1$]{{\includegraphics[width=1.5in]{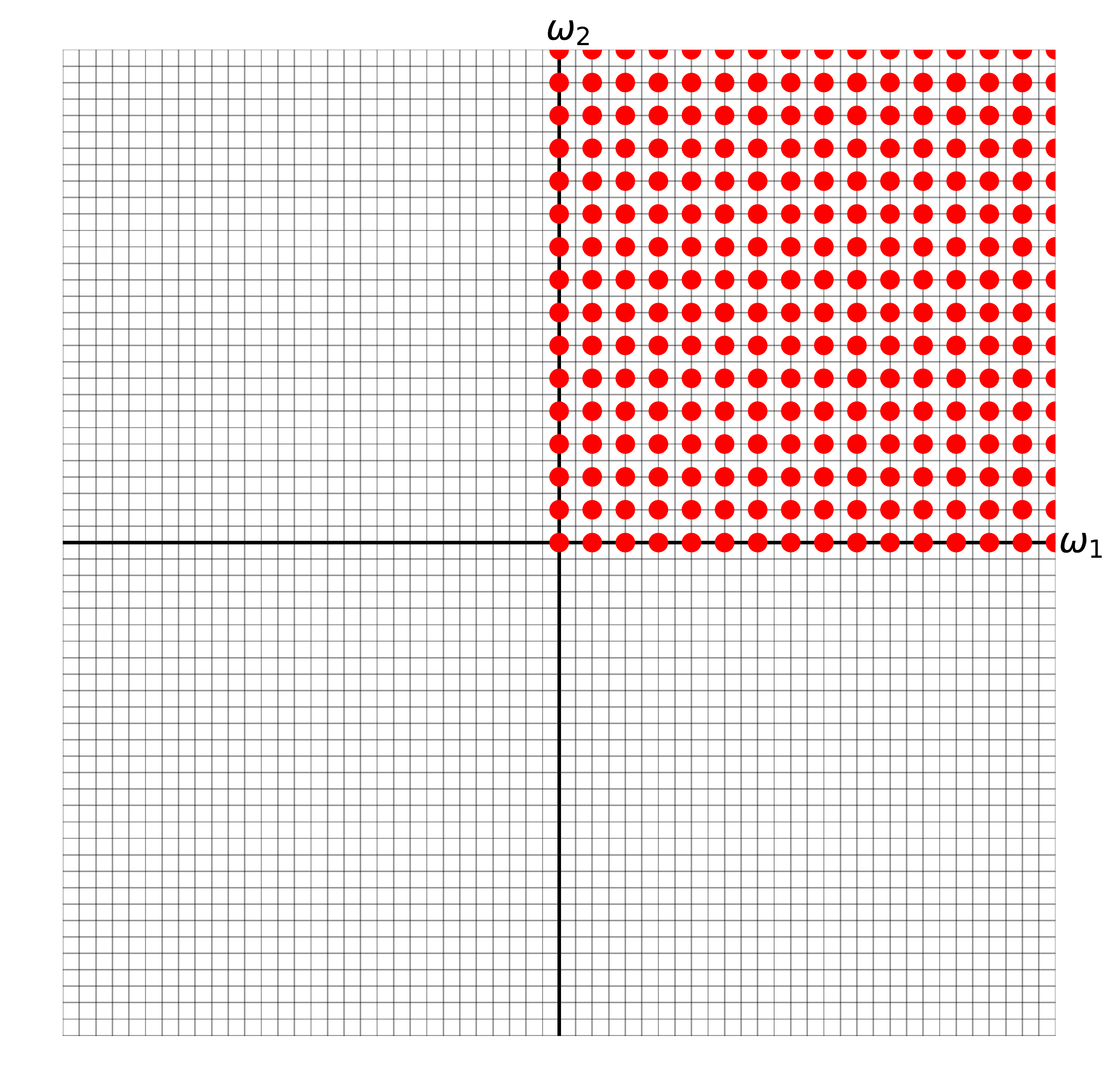}}}
    \hfill
    \subfloat[$\sigma = s_1$]{{\includegraphics[width=1.5in]{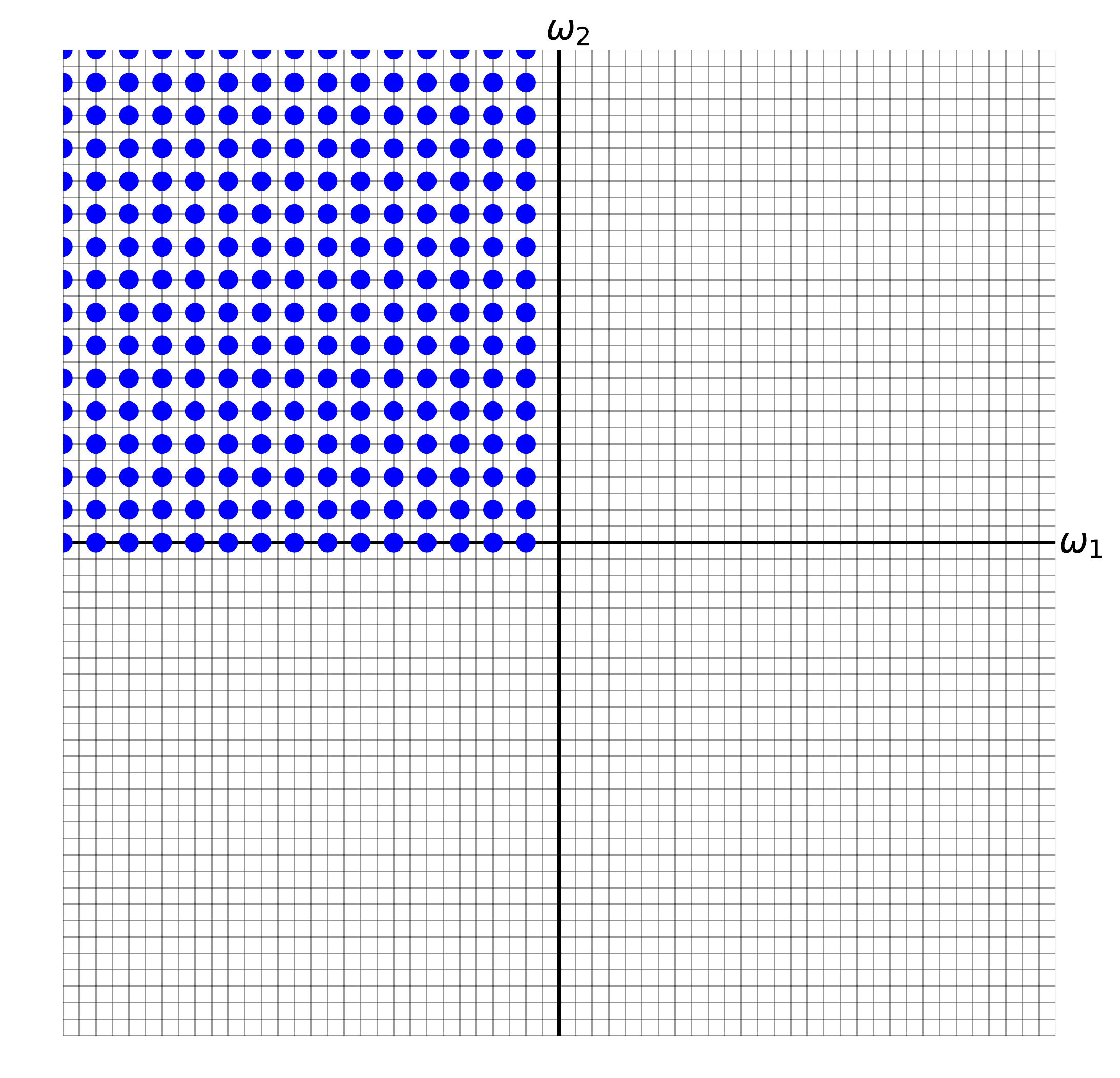} }}
    \hfill
    \subfloat[$\sigma = s_2$]{{\includegraphics[width=1.5in]{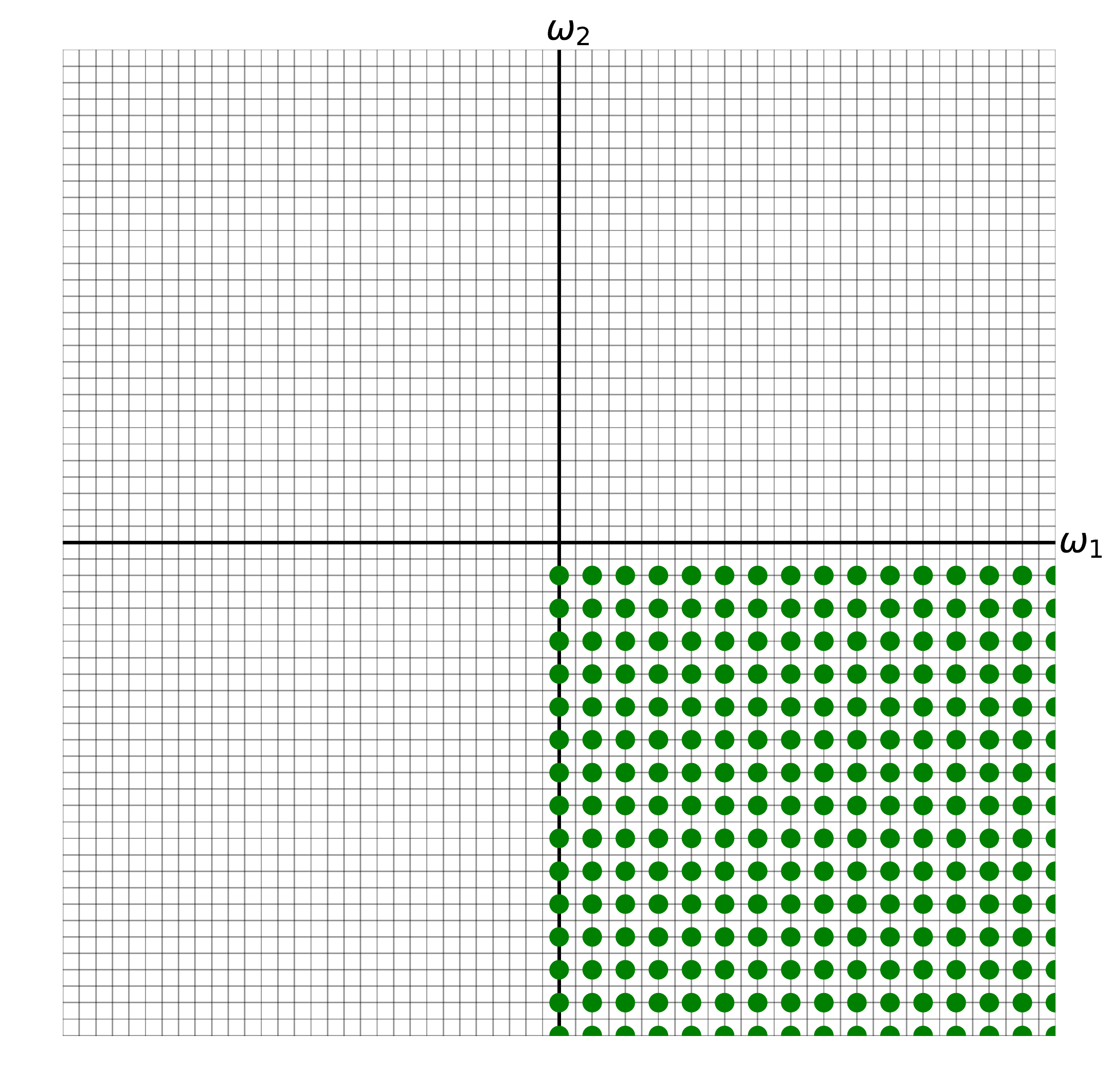} }}
    \hfill
    \subfloat[$\sigma = s_2s_1$]{{\includegraphics[width=1.5in]{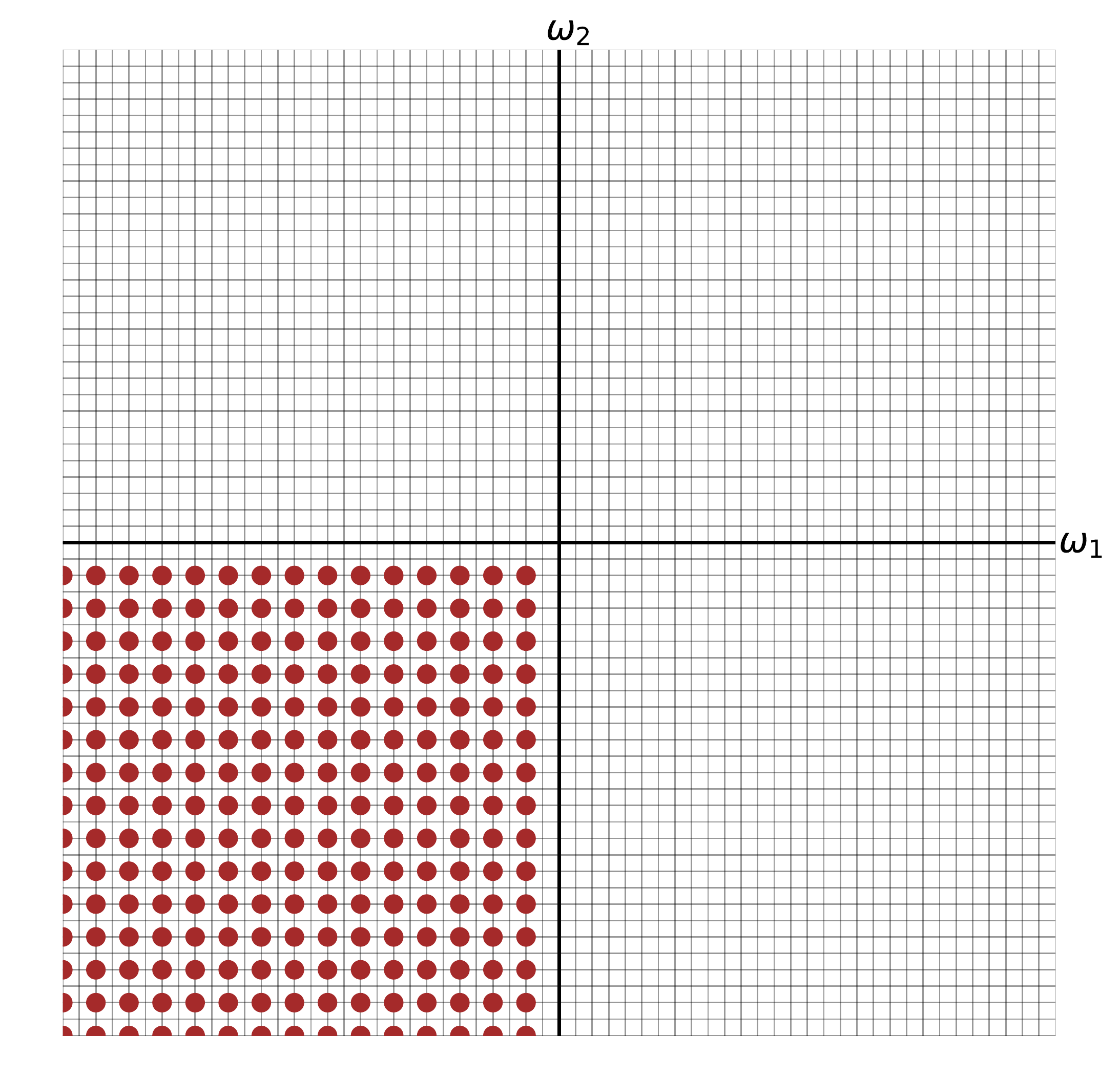} }}
    \caption{Solution sets to linear inequalities corresponding to the Lie algebra of type $D_2$.}
    \label{fig:d2_single_elements}
\end{figure}
\subsubsection{\normalfont{\textbf{Case}} \texorpdfstring{$\mu = n\w_1$}{mu equals n omega 1}}
Figures \ref{subfig:d2_n2}-\ref{subfig:d2_n8} illustrate the Weyl alternation diagrams for $\mu = n\w_1$ such that $n = 2,4,6, 8$. We observe that the empty region is a vertical rectangle that stretches infinitely. We also note that as $n$ increases from $2$ to $8$, the width of the rectangle in the center also increases. 

\begin{figure}[H]%
    \centering
    \subfloat[$\mu = 2\w_1$]{{\includegraphics[width=1.5in]{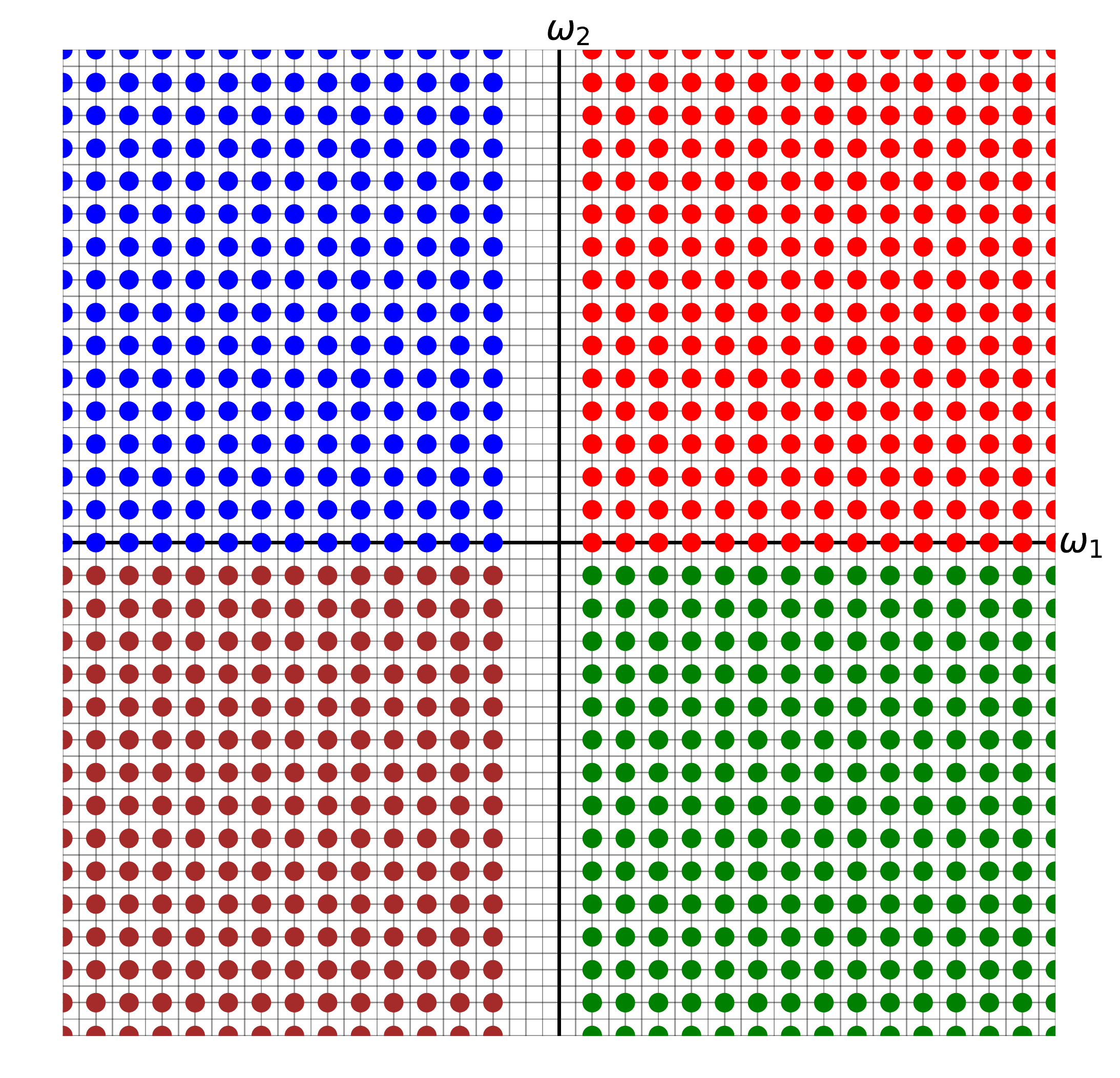} }
    \label{subfig:d2_n2}
    }
    \hfill
    \subfloat[$\mu = 4\w_1$]{{\includegraphics[width=1.5in]{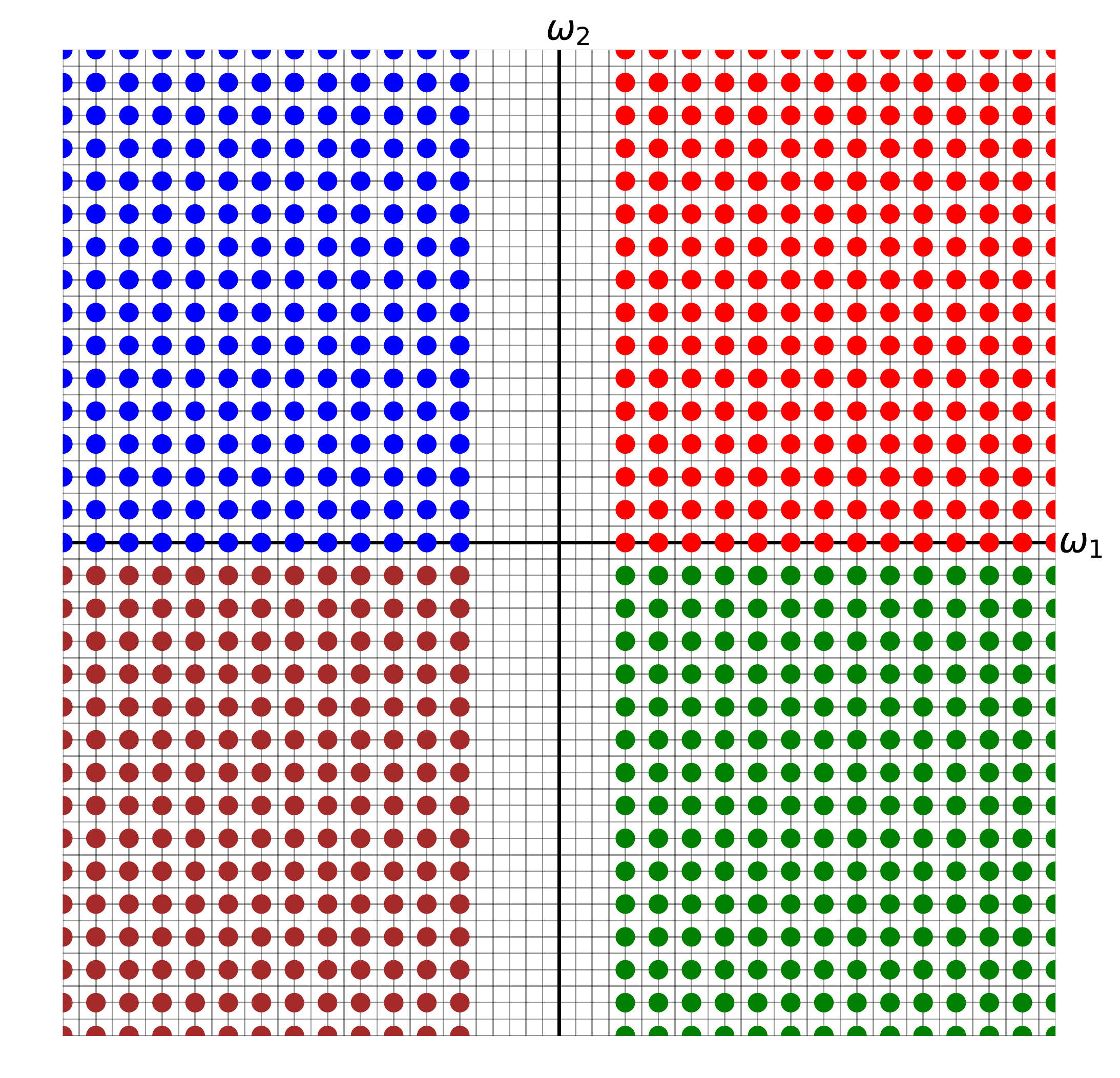} }
    \label{subfig:d2_n4}
    }
    \hfill
    \subfloat[$\mu = 6\w_1$]{{\includegraphics[width=1.5in]{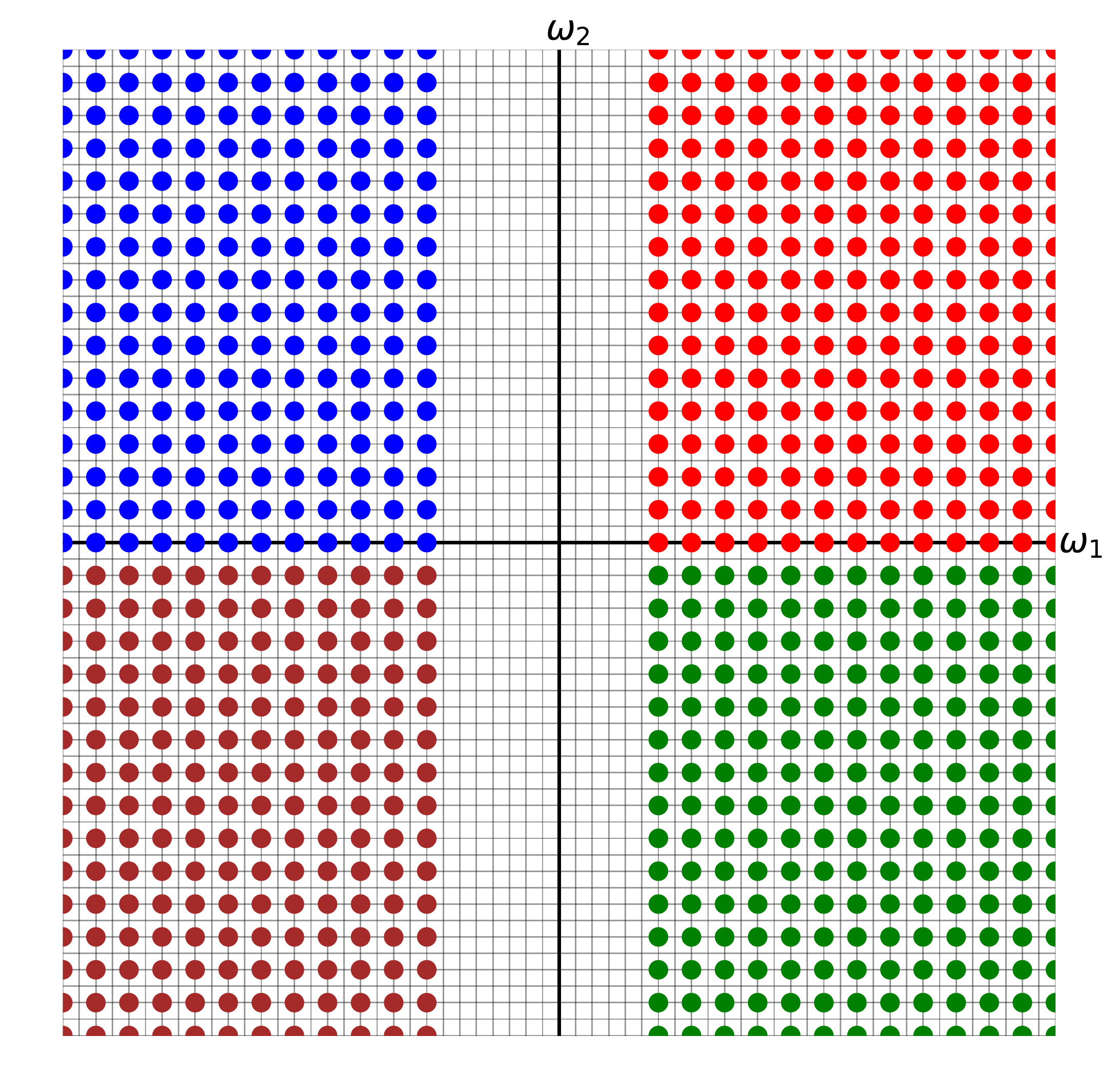} }
    \label{subfig:d2_n6}
    }
    \hfill
    \subfloat[$\mu = 8\w_1$]{{\includegraphics[width=1.5in]{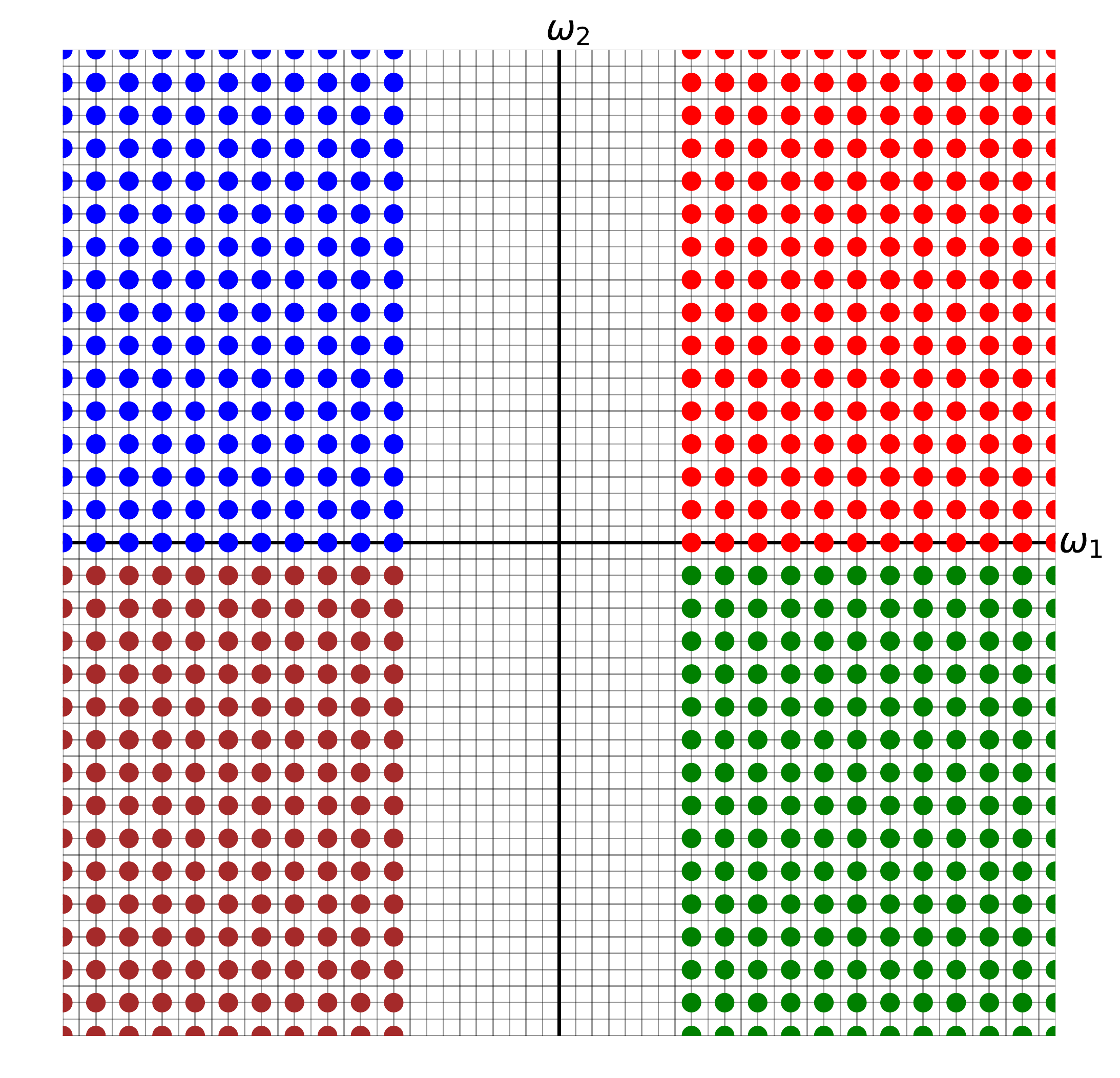} }
    \label{subfig:d2_n8}
    }
    \caption{Weyl alternation diagrams for the Lie algebra of type $D_2$ with $\mu=n\w_1$.}
    \label{fig:d2_mu_n}
\end{figure}
\subsubsection{\normalfont{\textbf{Case}} \texorpdfstring{$\mu = m\w_2$}{mu equals m omega 2}}
Figures \ref{subfig:d2_m2}-\ref{subfig:d2_m8} illustrate the Weyl alternation diagrams for $\mu = m\w_2$ such that $m = 2,4,6, 8$. We observe that the empty region is a horizontal rectangle that stretches out infinitely. We also note that as $m$ increases from $2$ to $8$, the width of the rectangle in the center also increases. 
\begin{figure}[H]%
    \centering
    \subfloat[$\mu = 2\w_2$]{{\includegraphics[width=1.5in]{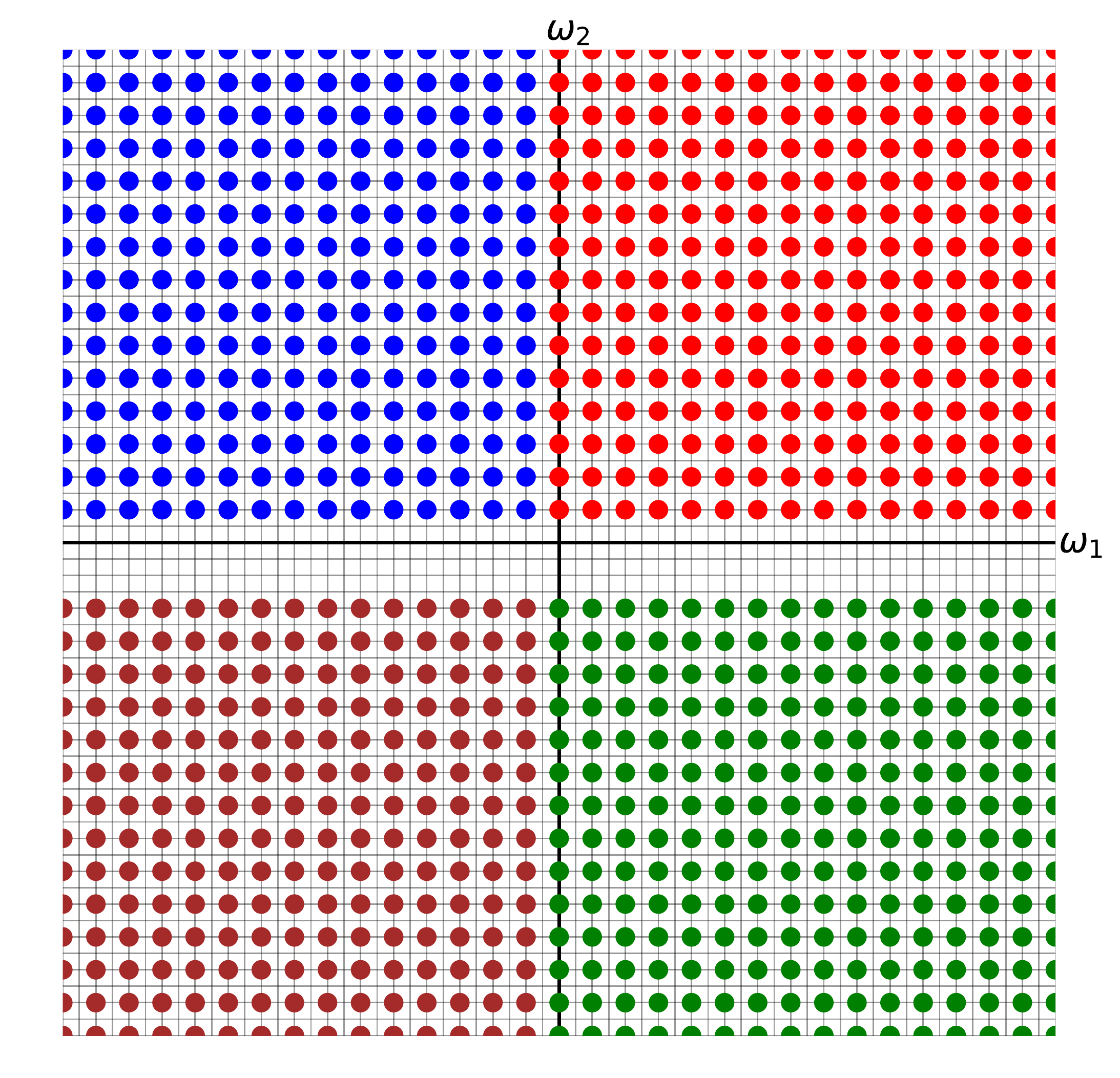} }
    \label{subfig:d2_m2}
    }
    \hfill
    \subfloat[$\mu = 4\w_2$]{{\includegraphics[width=1.5in]{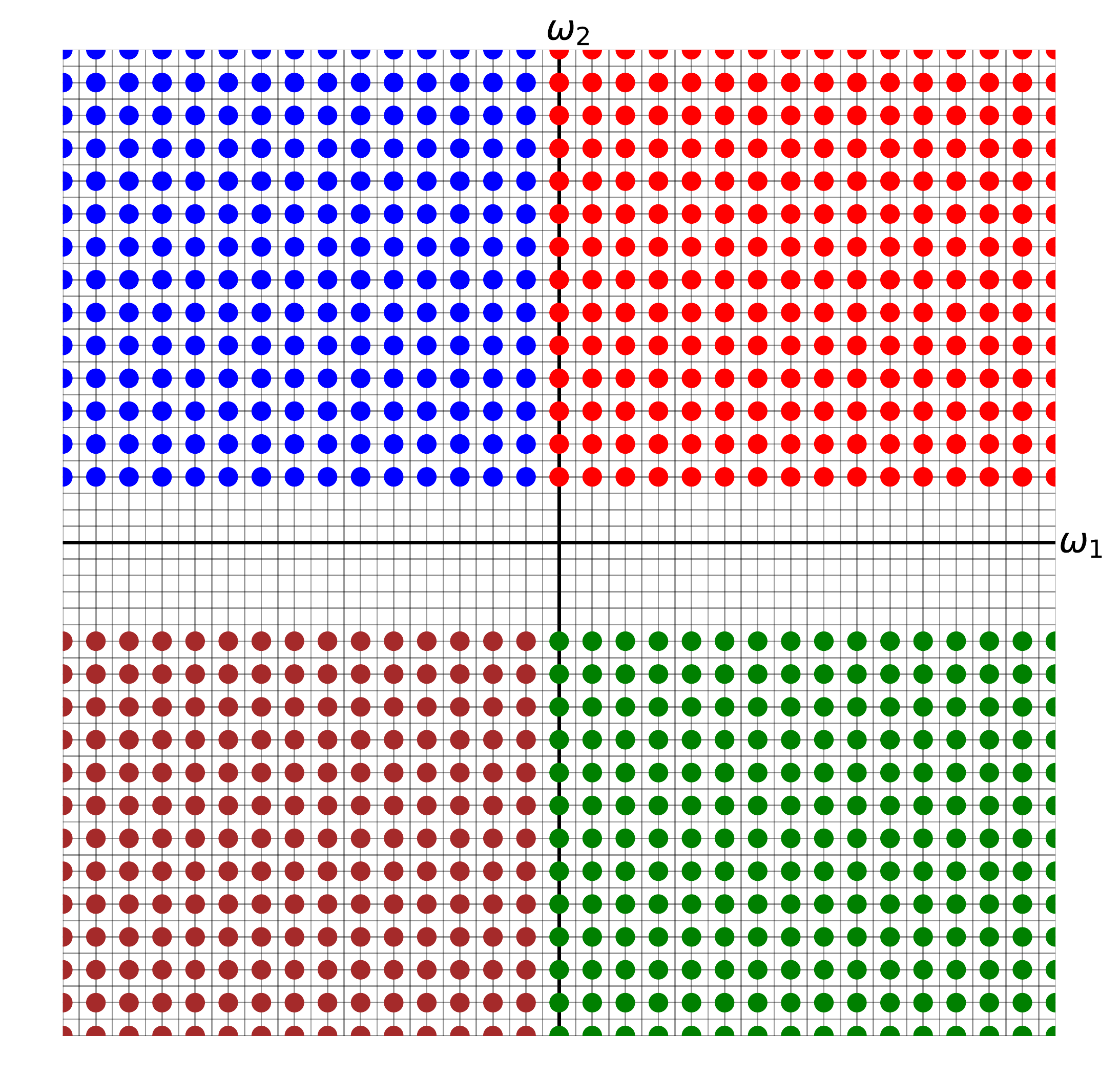} }
    \label{subfig:d2_m4}
    }
    \hfill
    \subfloat[$\mu = 6\w_2$]{{\includegraphics[width=1.5in]{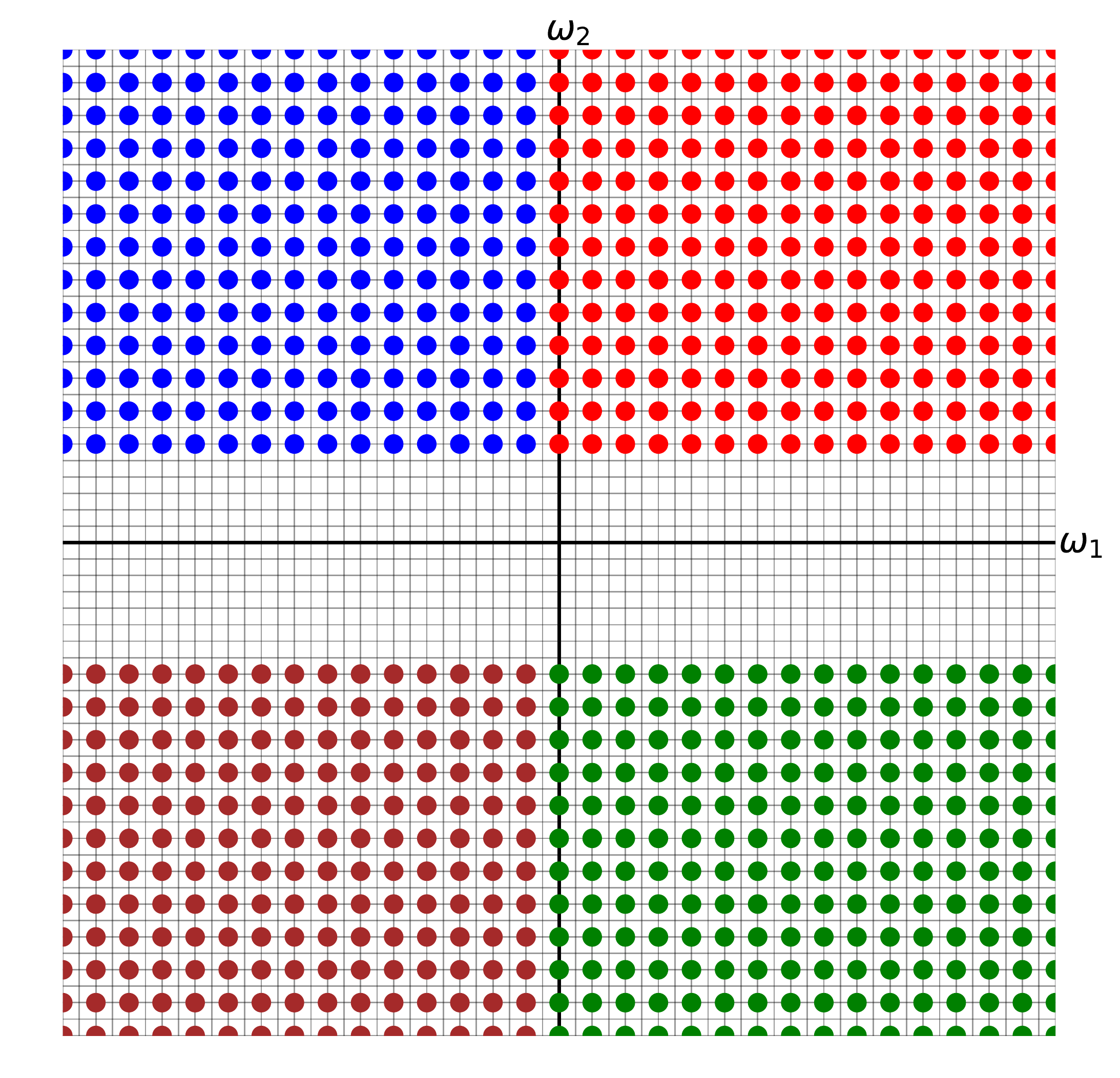} }
    \label{subfig:d2_m6}
    }
    \hfill
    \subfloat[$\mu = 8\w_2$]{{\includegraphics[width=1.5in]{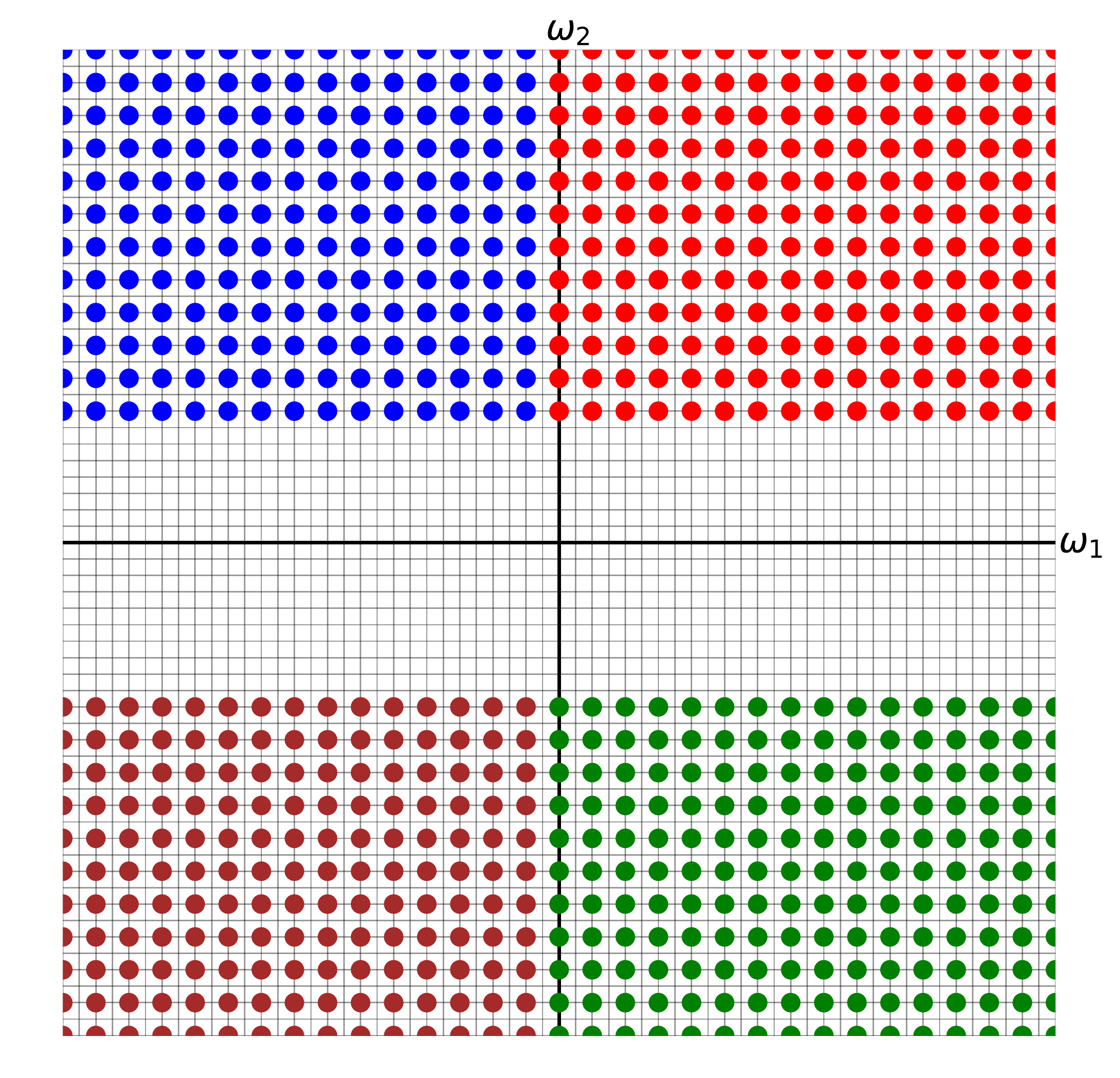} }
    \label{subfig:d2_m8}
    }
    \caption{Weyl alternation diagrams for the Lie algebra of type $D_2$ with $\mu=m\w_2$.}
    \label{fig:d2_mu_m}
\end{figure}
\subsubsection{\normalfont{\textbf{Case}} \texorpdfstring{$\mu = n\w_1 + m\w_2$}{mu equals n omega 1 plus omega 2}}
Figures \ref{subfig:d2_n2m2}-\ref{subfig:d2_n2m4} illustrate the Weyl alternation diagrams for $\mu = n\w_1 + m\w_2$. We observe that the empty region is in the shape of a cross that extends infinitely. 
\begin{figure}[H]%
    \centering
    \subfloat[$\mu = 2\w_1 + 2\w_2$]{{\includegraphics[width=1.5in]{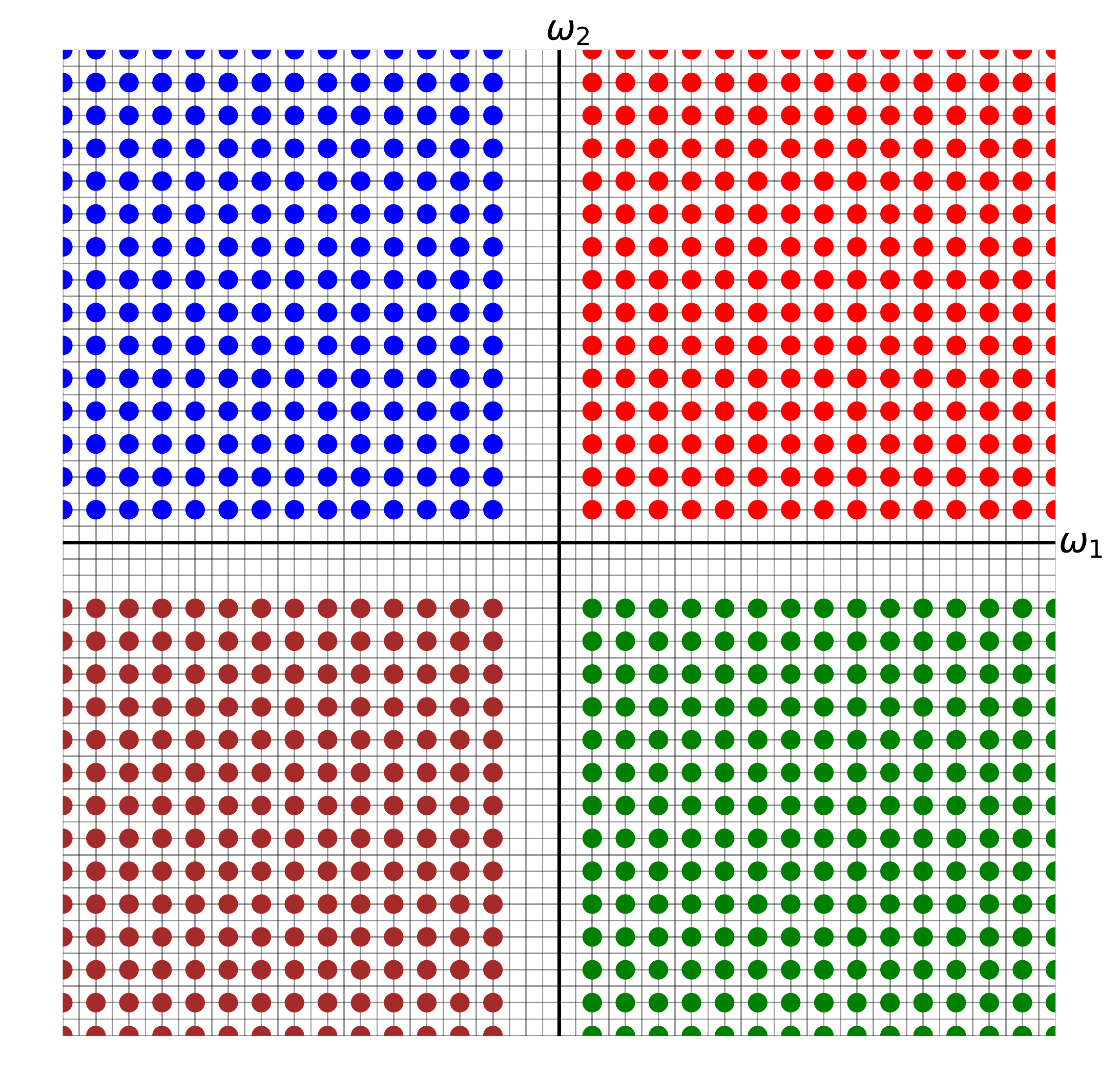}}
    \label{subfig:d2_n2m2}
    }
    \hfill
    \subfloat[$\mu = 4\w_1 + 4\w_2$]{{\includegraphics[width=1.5in]{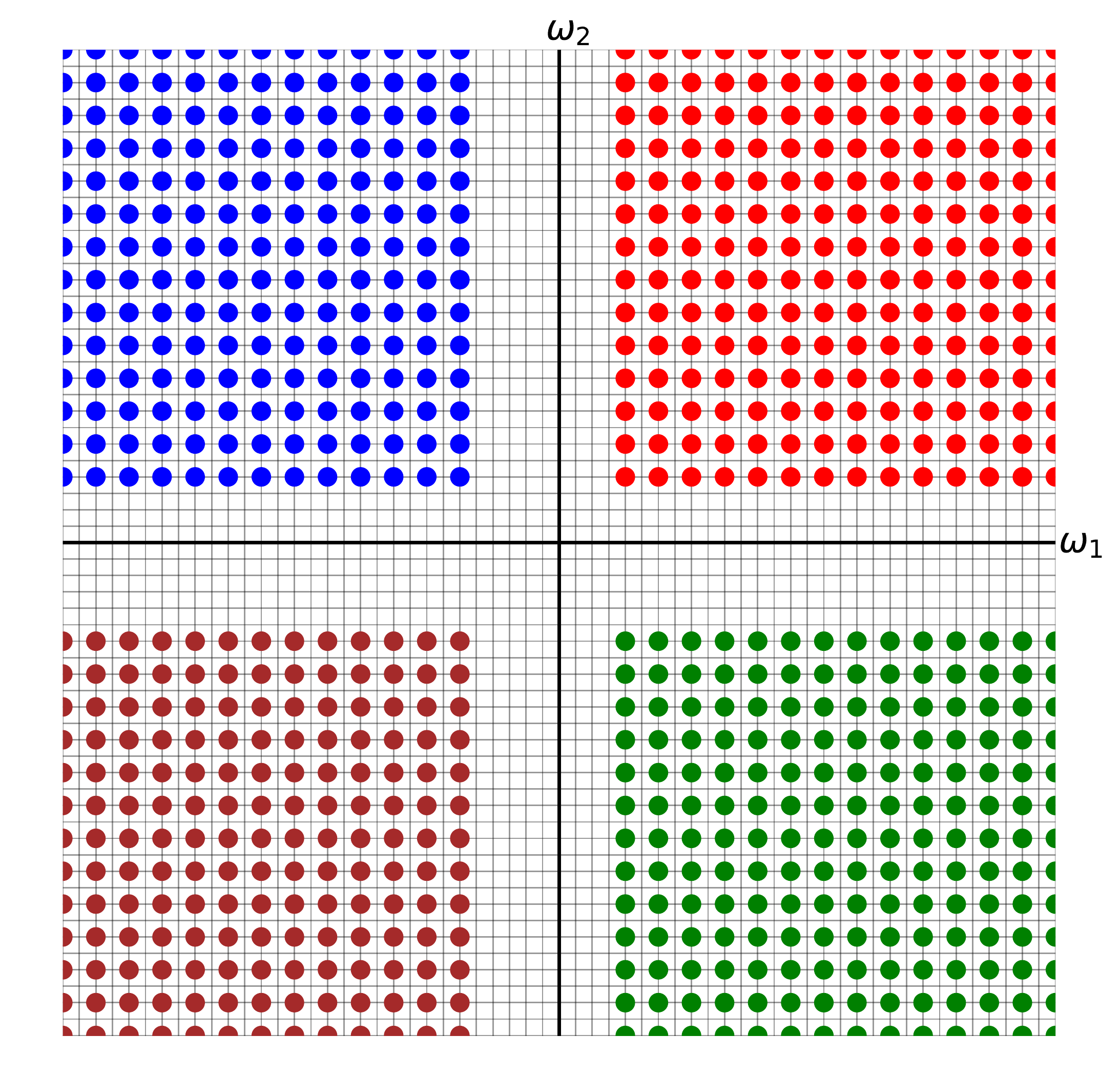} }
    \label{subfig:d2_n4m4}
    }
    \hfill
    \subfloat[$\mu = 4\w_1+2\w_2$]{{\includegraphics[width=1.5in]{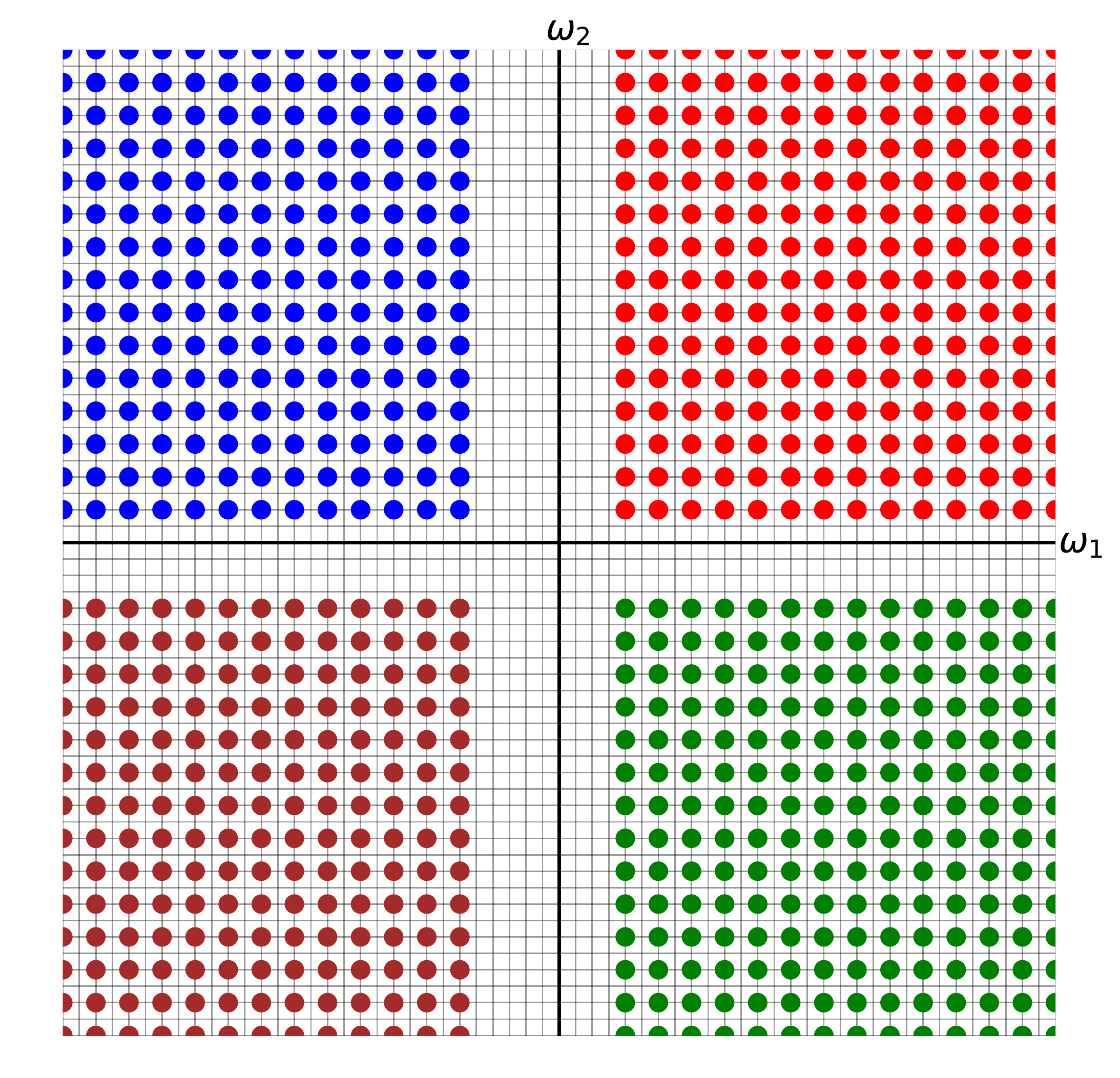}}
    \label{subfig:d2_n4m2}
    }%This one is done
    \hfill
    \subfloat[$\mu = 2\w_1+4\w_2$]{{\includegraphics[width=1.5in]{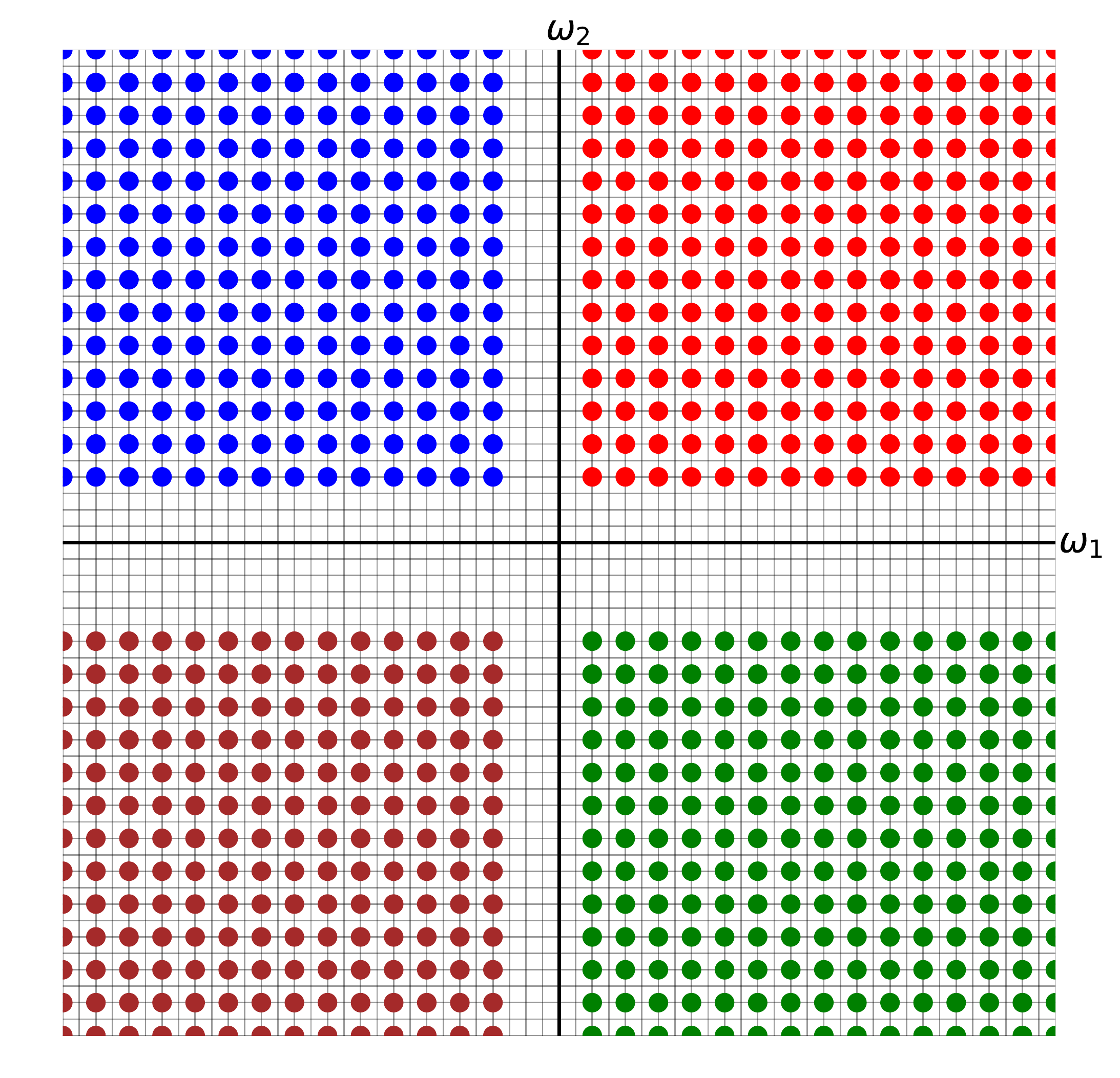}}
    \label{subfig:d2_n2m4}
    }
    \caption{Weyl alternation diagrams for the Lie algebra of type $D_2$ with \mbox{$\mu=n\w_1+m\w_2$.}}
    \label{fig:d2_mu_nm}
\end{figure}
% Inequalities for d2
\begin{center}
\begin{figure}[H]%
\resizebox{1.75in}{!}{
\begin{tikzpicture}[scale=.5]
\draw [thick, <->] (-1, -4) -- (-1, 4);
\draw [thick, <->] (1, -4) -- (1, 4);
\draw [thick, <->] (-4, -1) -- (4, -1);
\draw [thick, <->] (-4, 1) -- (4, 1);
\node at (1.9, 4.5) {$\tfrac{c_1 - n}{2} \geq 0$};
\node at (-2.0, 4.5) {$\tfrac{-c_1 - n - 2}{2} \geq 0$};
\node at (3.5, 1.6) {$\tfrac{c_2 - m}{2} \geq 0$};
\node at (3.5, -0.4) {$\tfrac{-c_2 - m - 2}{2} \geq 0$};
\end{tikzpicture}
}
\caption{Set of linear inequalities for determining the boundaries of the Weyl alternation sets for the Lie algebra of type $D_2$.}
\label{fig:d2_ineqs}
\end{figure}
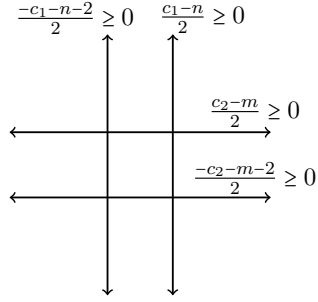
\end{center}
To explain the shapes that form in the empty region of each Weyl alternation diagram for the Lie algebra of type $D_2$ we turn to Figure \ref{fig:d2_ineqs}. From Theorem \ref{thm:maind2}, we notice that all inequalities depend on $n$ and $m$ which simply shift the inequalities, but never change the direction of the line. This means that Figure \ref{fig:d2_ineqs} is a good representation of the inequalities formed. 
\begin{center}
\begin{figure}[H]%
\subfloat[Vertical Strip]{
\begin{tikzpicture}[scale=.4]
\draw [thick, <->] (-1, -4) -- (-1, 4);
\draw [thick, <->] (1, -4) -- (1, 4);
\draw [thick, <-] (-4, 0) -- (-1, 0);
\draw [thick, <-] (4, 0) -- (1, 0);
\draw [dashed, -] (-1, 0) -- (1, 0);
\node at (2.5, 4.75) {$\tfrac{c_1 - n}{2} \geq 0$};
\node at (-3.0, 4.75) {$\tfrac{-c_1 - n - 2}{2} \geq 0$};
\node at (5.5, 1) {$\tfrac{c_2 - m}{2} = \tfrac{-c_2 - m - 2}{2} \geq 0$};
\end{tikzpicture}
    \label{subfig:d2_vertical}
    }
    \quad
    \subfloat[Horizontal Strip]{
\begin{tikzpicture}[scale=.4]
\draw [thick, <-] (0, -4) -- (0, -1);
\draw [thick, <-] (0, 4) -- (0, 1);
\draw [dashed, -] (0, -1) -- (0, 1);
%\draw [thick, <->] (1, -4) -- (1, 4);
\draw [thick, <->] (-4, -1) -- (4, -1);
\draw [thick, <->] (-4, 1) -- (4, 1);
\node at (0, 4.75) {$\tfrac{-c_1 - n - 2}{2} = \tfrac{c_1 - n}{2} \geq 0$};
\node at (3.5, 1.65) {$\tfrac{c_2 - m}{2} \geq 0$};
\node at (3.5, -0.25) {$\tfrac{-c_2 - m - 2}{2} \geq 0$};
\end{tikzpicture}
    \label{subfig:d2_horizontal}
    }
    \quad
    \subfloat[Cross]{
 \begin{tikzpicture}[scale=.4]
\draw [thick, <-] (-1, -4) -- (-1, -1);
\draw [dashed, -] (-1,-1) -- (-1,1);
\draw[thick, ->] (-1, 1) -- (-1,4);
%------
\draw [thick, <-] (1, -4) -- (1, -1);
\draw [dashed, -] (1,-1) -- (1,1);
\draw[thick, ->] (1, 1) -- (1,4);
%-----
\draw[thick, <-] (-4,1) -- (-1, 1);
\draw[dashed, -] (-1,1) -- (1,1);
\draw [thick, ->] (1, 1) -- (4, 1);
%-----
\draw[thick, <-] (-4,-1) -- (-1, -1);
\draw[dashed, -] (-1,-1) -- (1,-1);
\draw [thick, ->] (1,-1) -- (4,-1);
\node at (3, 4.75) {$\tfrac{c_1 - n}{2} \geq 0$};
\node at (-3.0, 4.75) {$\tfrac{-c_1 - n - 2}{2} \geq 0$};
\node at (3.5, 1.65) {$\tfrac{c_2 - m}{2} \geq 0$};
\node at (3.5, -0.25) {$\tfrac{-c_2 - m - 2}{2} \geq 0$};
\end{tikzpicture}
\label{subfig:d2_cross}
}
\caption{Different formation of the empty region for the Lie algebra of type $D_2$.}
\label{fig:d2_cases}
\end{figure}
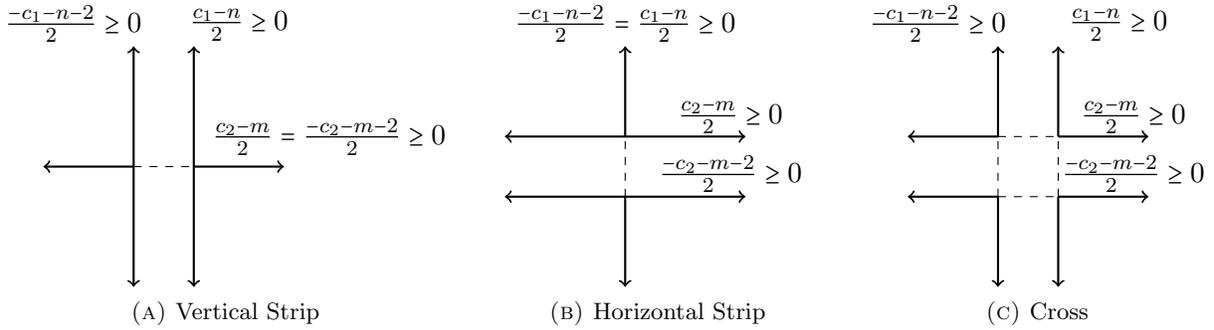
\end{center}
From Figure \ref{subfig:d2_vertical}, observe that the empty region becomes a vertical strip when the inequalities $\tfrac{-c_1-n-2}{2} \geq 0$ and $\tfrac{c_1-n}{2} \geq 0$ do not intersect and the inequalities $\tfrac{c_2-m}{2} \geq 0$ and $\tfrac{-c_2-m-2}{2} \geq 0$ intersect. This occurs exactly when $\mu = n\w_1$ such that $n \in 2\NN.$ Similarly, depicted in Figure \ref{subfig:d2_horizontal}, the empty region becomes a horizontal strip when the inequalities $\tfrac{c_2-m}{2} \geq 0$ and $\tfrac{-c_2-m-2}{2} \geq 0$ do not intersect and the inequalities $\tfrac{-c_1-n-2}{2} \geq 0$ and $\tfrac{c_1-n}{2} \geq 0$ intersect. This occurs exactly when $\mu = m\w_2$ such that $m \in 2\NN$. The empty region takes the shape of a cross (Figure \ref{subfig:d2_cross}) that stretches infinitely if and only if the inequalities $\tfrac{-c_1-n-2}{2} \geq 0$ and $\tfrac{c_1-n}{2} \geq 0$ do not intersect and $\tfrac{c_2-m}{2} \geq 0$ and $\tfrac{-c_2-m-2}{2} \geq 0$ do not intersect. Namely, this occurs when $\mu = n\w_1+m\w_2$ such that $n,m\in 2\NN$.
%---------------------------------------------------------------------------------------------------
% Section G_2
\subsection{Lie algebra of type \texorpdfstring{$G_2$}{G2}}
Now let us look at the Lie algebra of type $G_2$. This time there is no divisibility condition required, as the root lattice and fundamental weight lattice are equivalent. Consequently, we take the inequalities given in Table \ref{tab:g2} and plot them individually on the root lattice, as depicted in Figure \ref{subfig:g2_grid}, and whose solution sets we highlight in Figure \ref{fig:g2_single_elements}. Note that in Figure \ref{fig:g2_single_elements} we set $\mu=0$ and present the corresponding region for each Weyl group element. We remark that changing $\mu$ only translates these solution sets. In what follows we describe how the Weyl diagrams change as we vary the weight $\mu$.
% We assign a distinct color to each nonempty Weyl alternation set (i.e. solution set). This allows us to present a multicolored diagram which provides a visual representation of the support of Kostant's partition function.
% We observe that the empty region in Figures \ref{subfig:g2_n1}-\ref{subfig:g2_n4} form a hexagon with a vertex pointing up.

\begin{figure}[H]%
    \centering
    \subfloat[$\sigma = 1$]{{\includegraphics[width=1.5in]{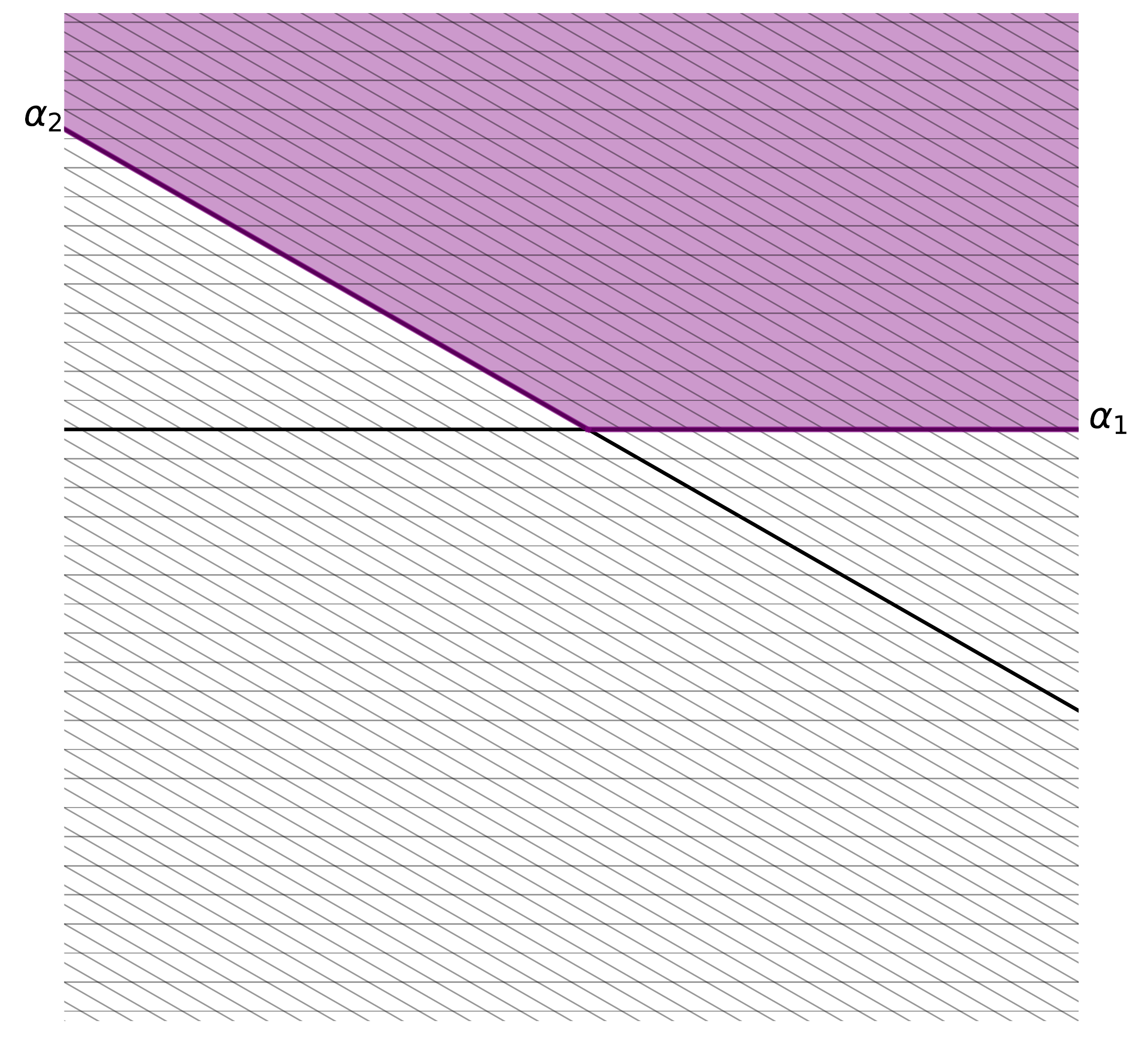}}}%
    \hfill
    \subfloat[$\sigma = s_1$]{{\includegraphics[width=1.5in]{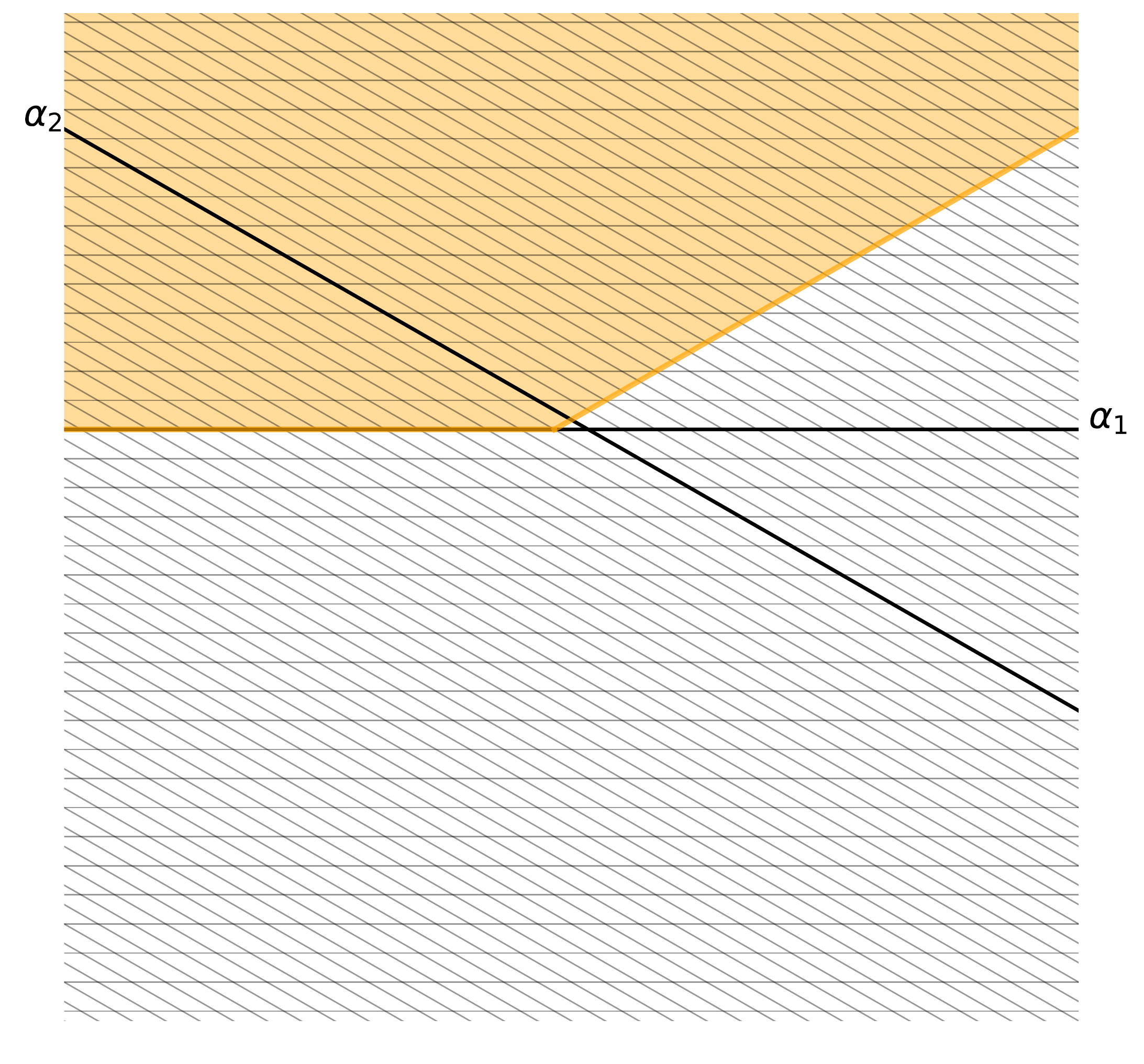} }}
    \hfill
    \subfloat[$\sigma = s_2$]{{\includegraphics[width=1.5in]{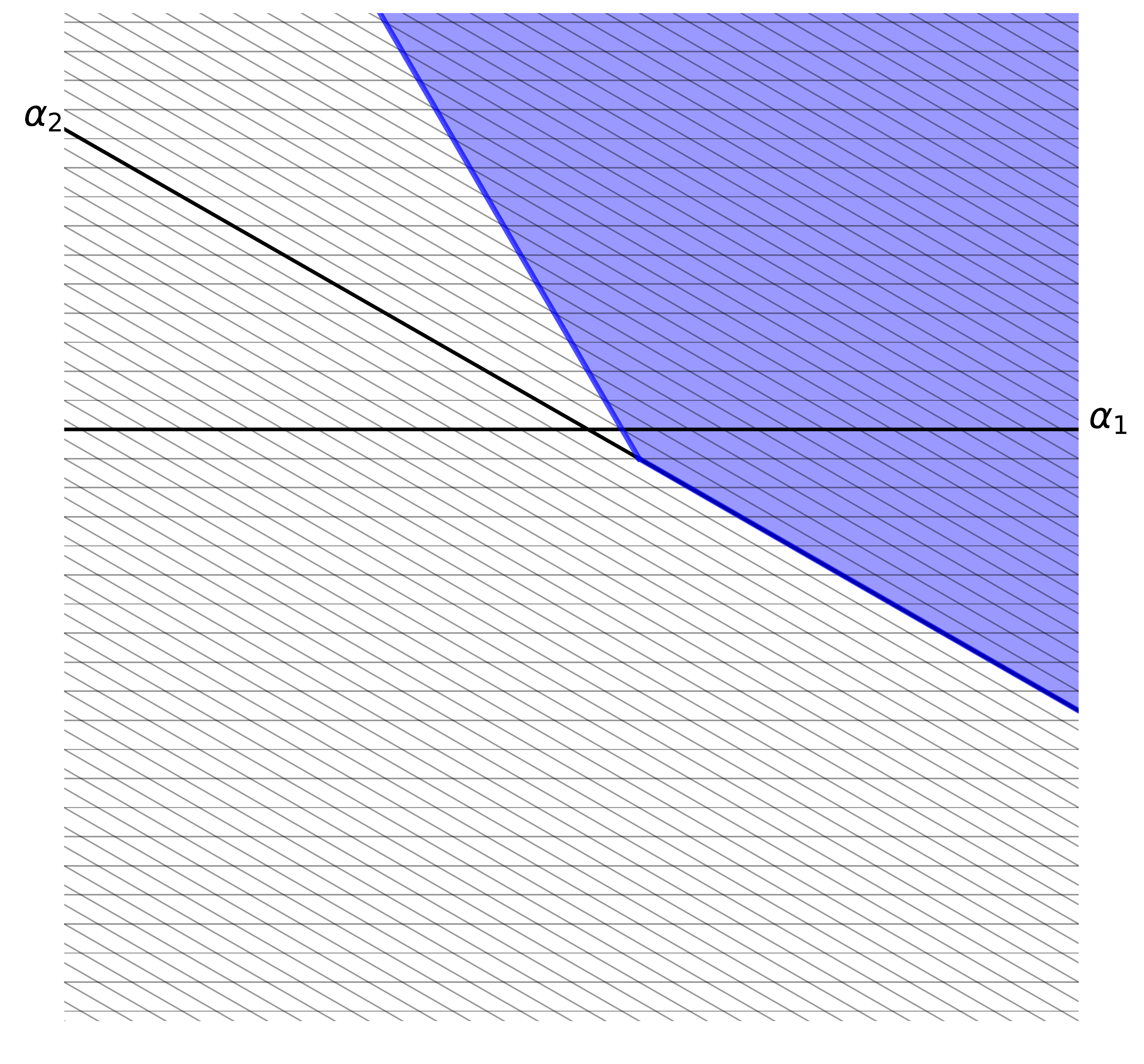} }}
    \hfill
    \subfloat[$\sigma = s_2s_1$]{{\includegraphics[width=1.5in]{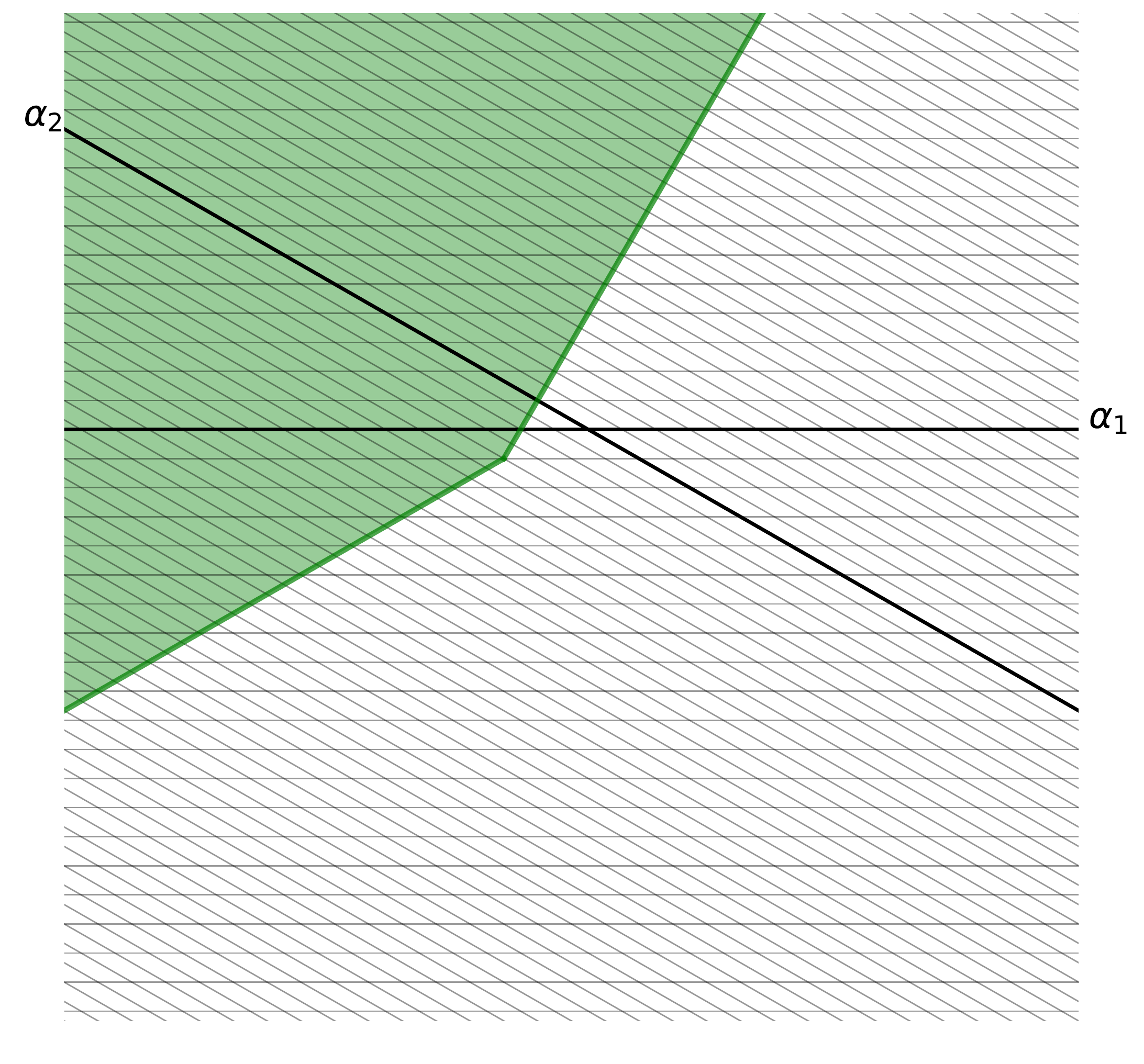} }}\\
    \subfloat[$\sigma = s_1s_2$]{{\includegraphics[width=1.5in]{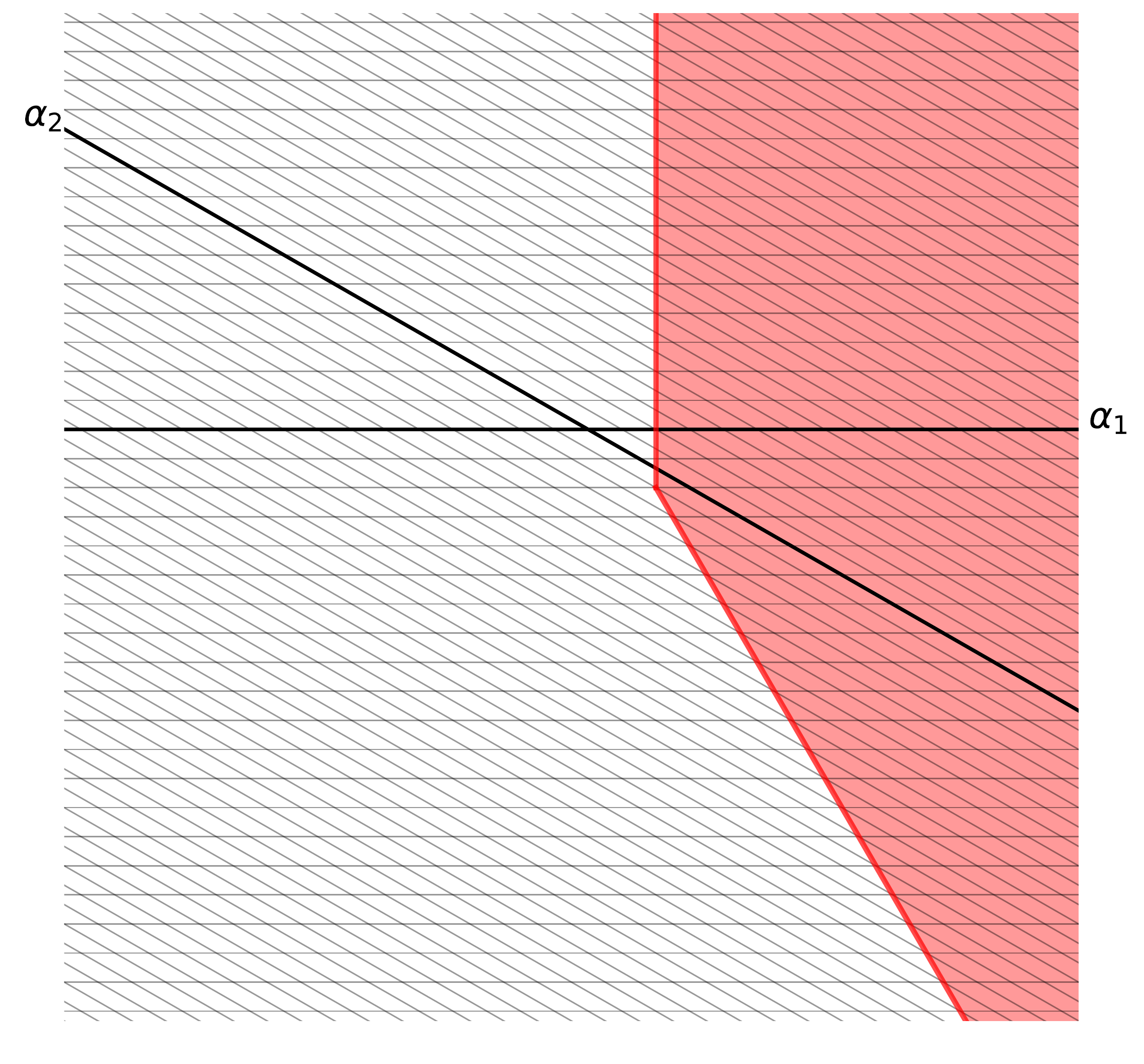} }}
    \hfill
    \subfloat[$\sigma = s_1s_2s_1$]{{\includegraphics[width=1.5in]{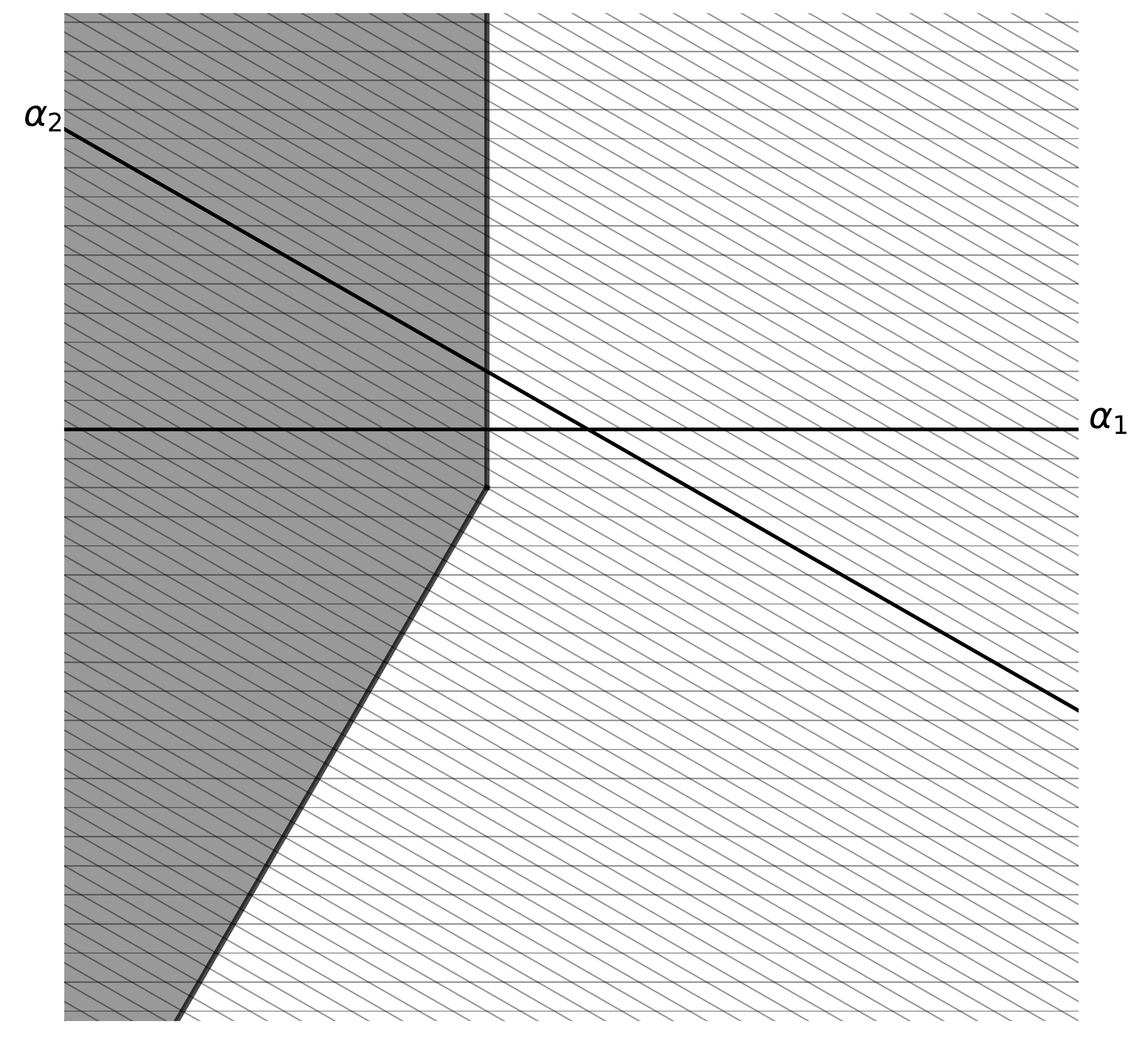} }}
    \hfill
    \subfloat[$\sigma = s_2s_1s_2$]{{\includegraphics[width=1.5in]{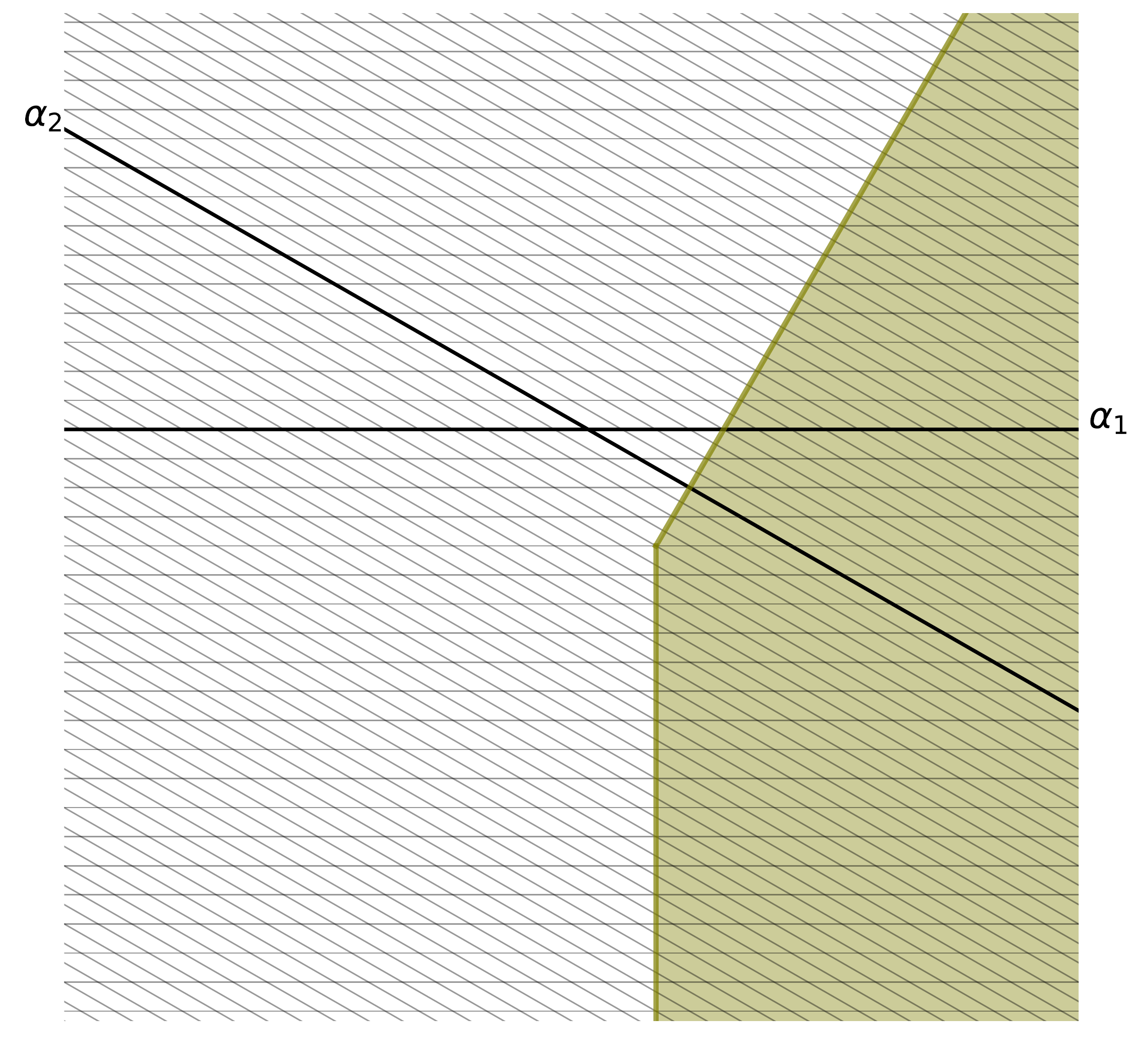} }}%
    \hfill
    \subfloat[$\sigma = (s_2s_1)^2$]{{\includegraphics[width=1.5in]{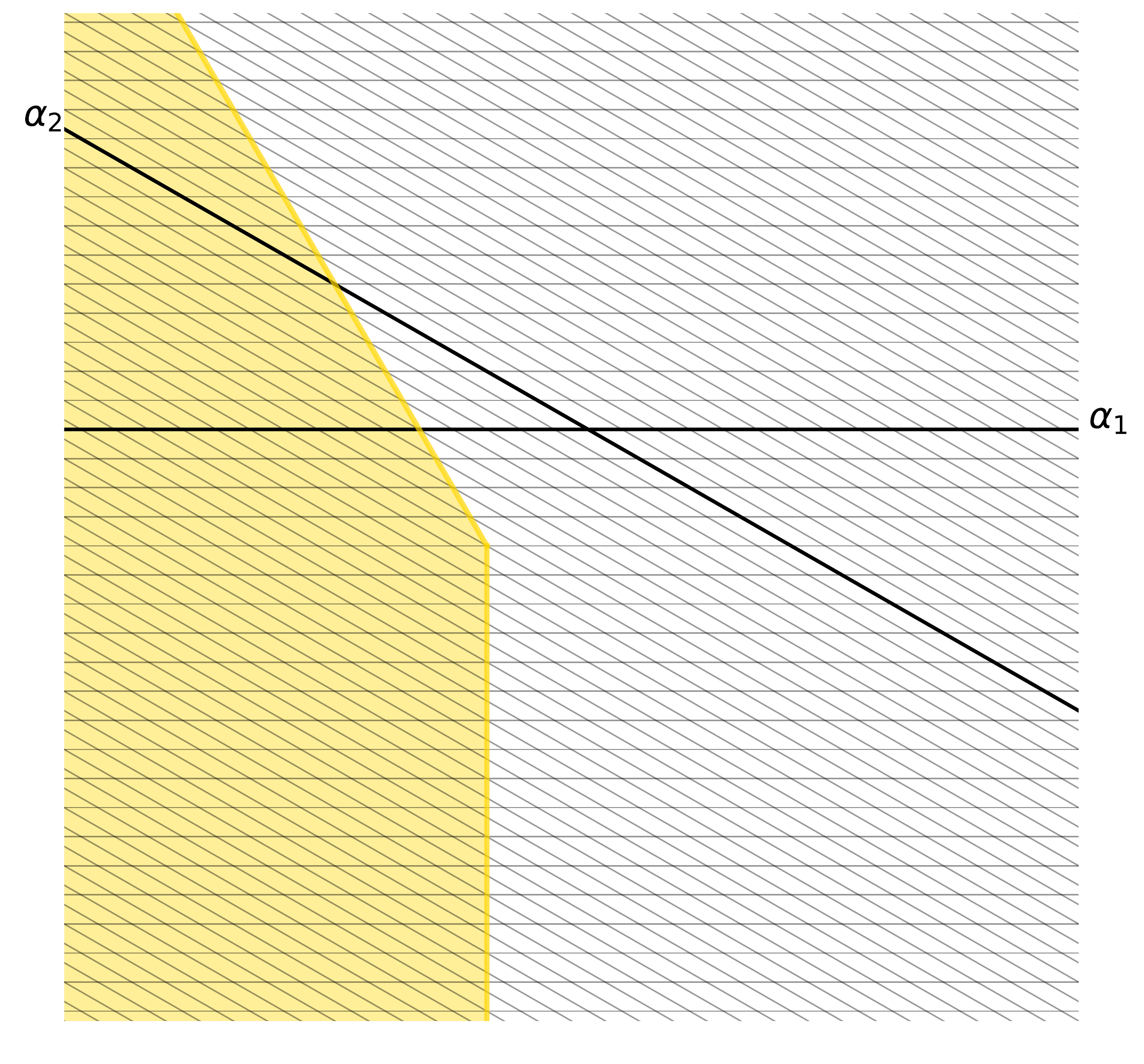} }}
    \\
    \subfloat[$\sigma = (s_1s_2)^2$]{{\includegraphics[width=1.5in]{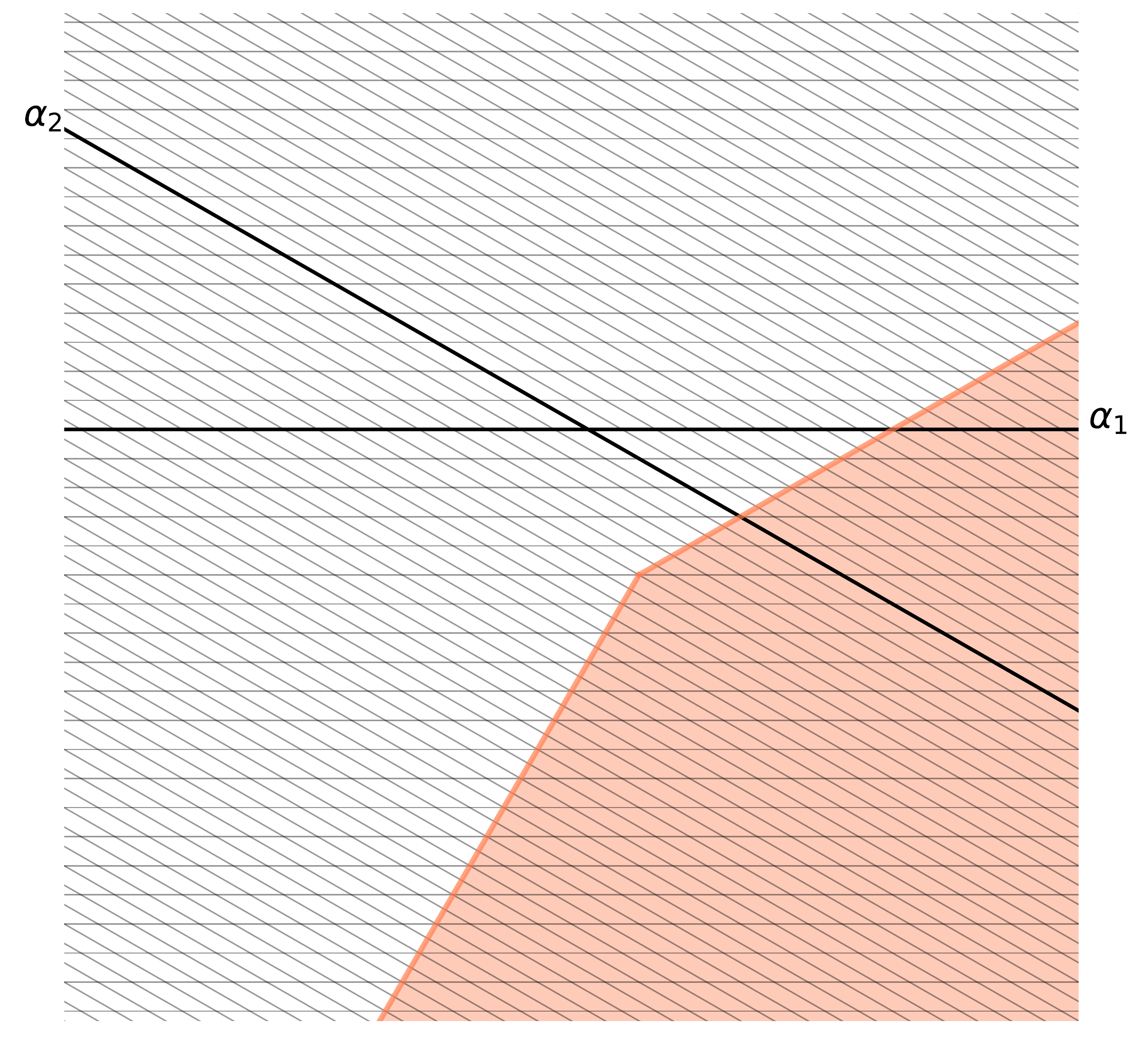} }}
    \hfill
    \subfloat[$\sigma = s_1(s_2s_1)^2$]{{\includegraphics[width=1.5in]{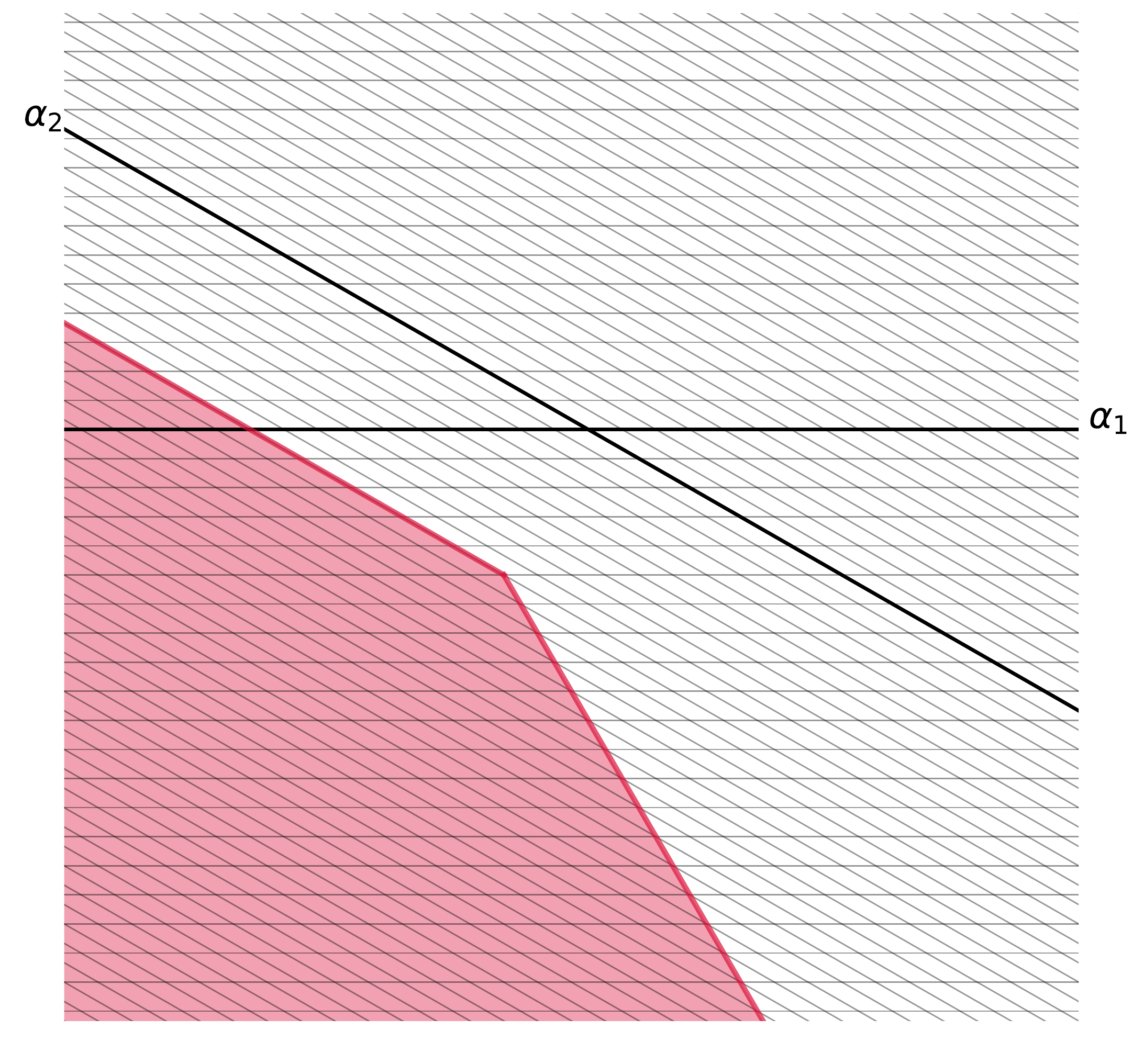} }}%
    \hfill
    \subfloat[$\sigma = s_2(s_1s_2)^2$]{{\includegraphics[width=1.5in]{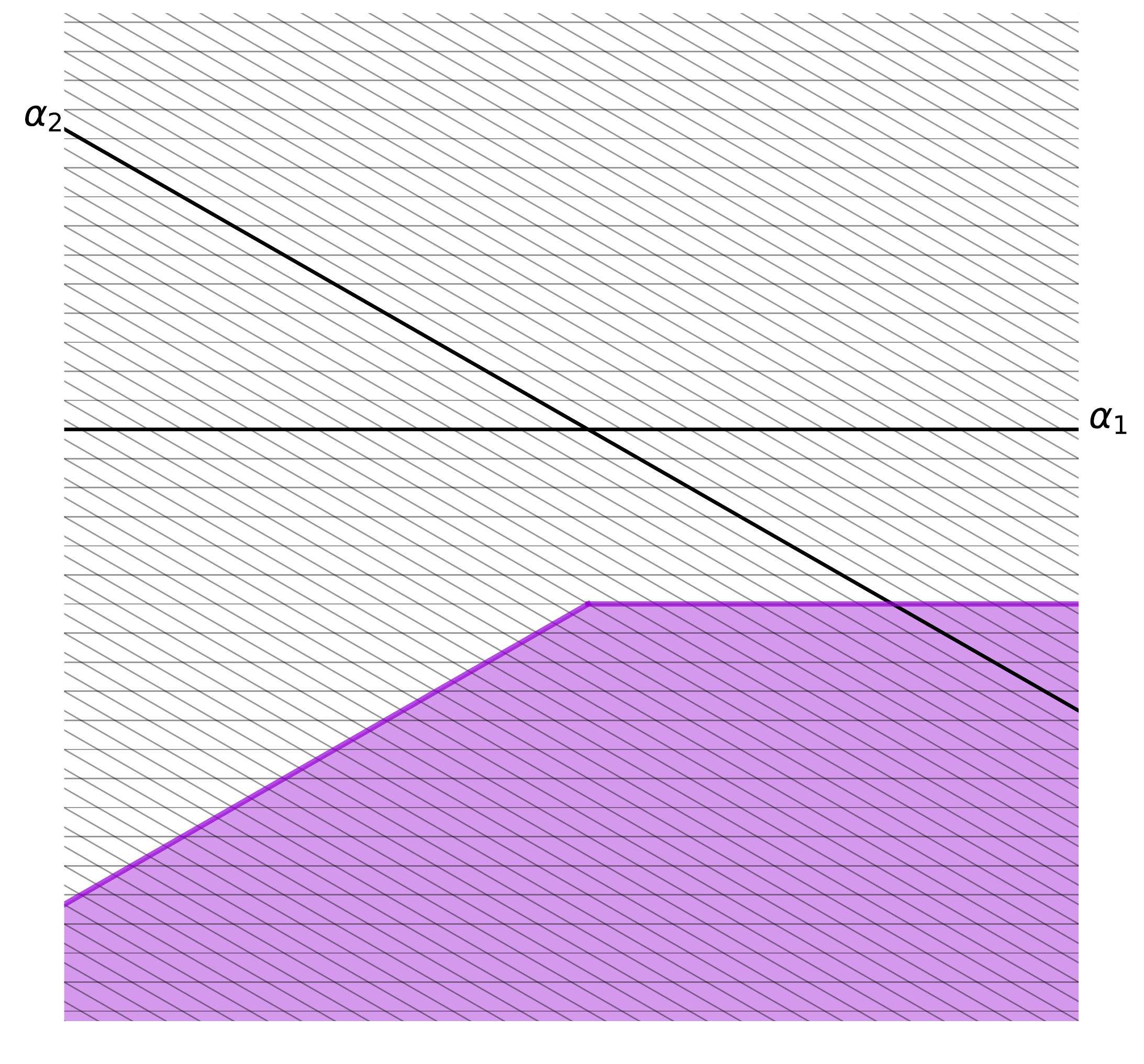} }}%
    \hfill
    \subfloat[$\sigma = (s_1s_2)^3$]{{\includegraphics[width=1.5in]{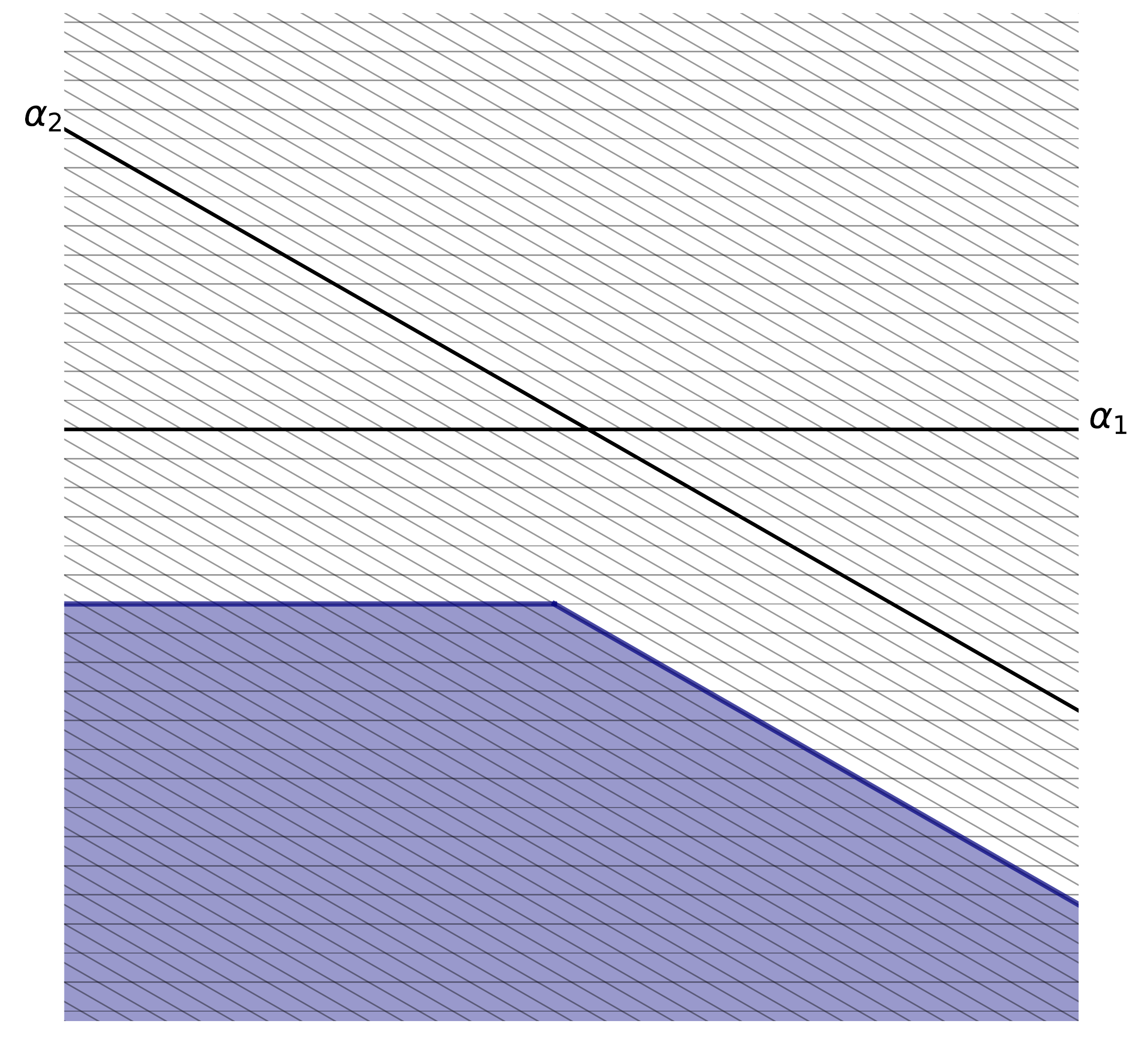} }}%
    \caption{Solution sets to linear inequalities corresponding to $\fg_2$.}
    \label{fig:g2_single_elements}
\end{figure}

In Figure \ref{fig:g2_single_elements}, we graphed the pairs of inequalities for single elements of $\A(\l,0)$ and shaded in their corresponding solution sets. Observe that changing $\mu$ would only result in translations to these solution sets. In the following sections we describe how the Weyl diagram changes as we vary the weight $\mu$.

\subsubsection{\normalfont{\textbf{Case}} \texorpdfstring{$\mu = n\al_1$}{mu equals n alpha 1}}
Figures \ref{subfig:g2_n1}-\ref{subfig:g2_n4} illustrate the Weyl alternation diagrams for $\mu = n\alpha_1$ such that $n = 1,2,3,4$. We observe that the empty region is a hexagon with a vertex pointing up. We also note that as $n$ increases from 1 to 4, the length of the edges of the hexagon in the center also increases. 
\begin{figure}[H]%
    \centering
    \subfloat[$\mu = \al_1$]{{\includegraphics[width=1.5in]{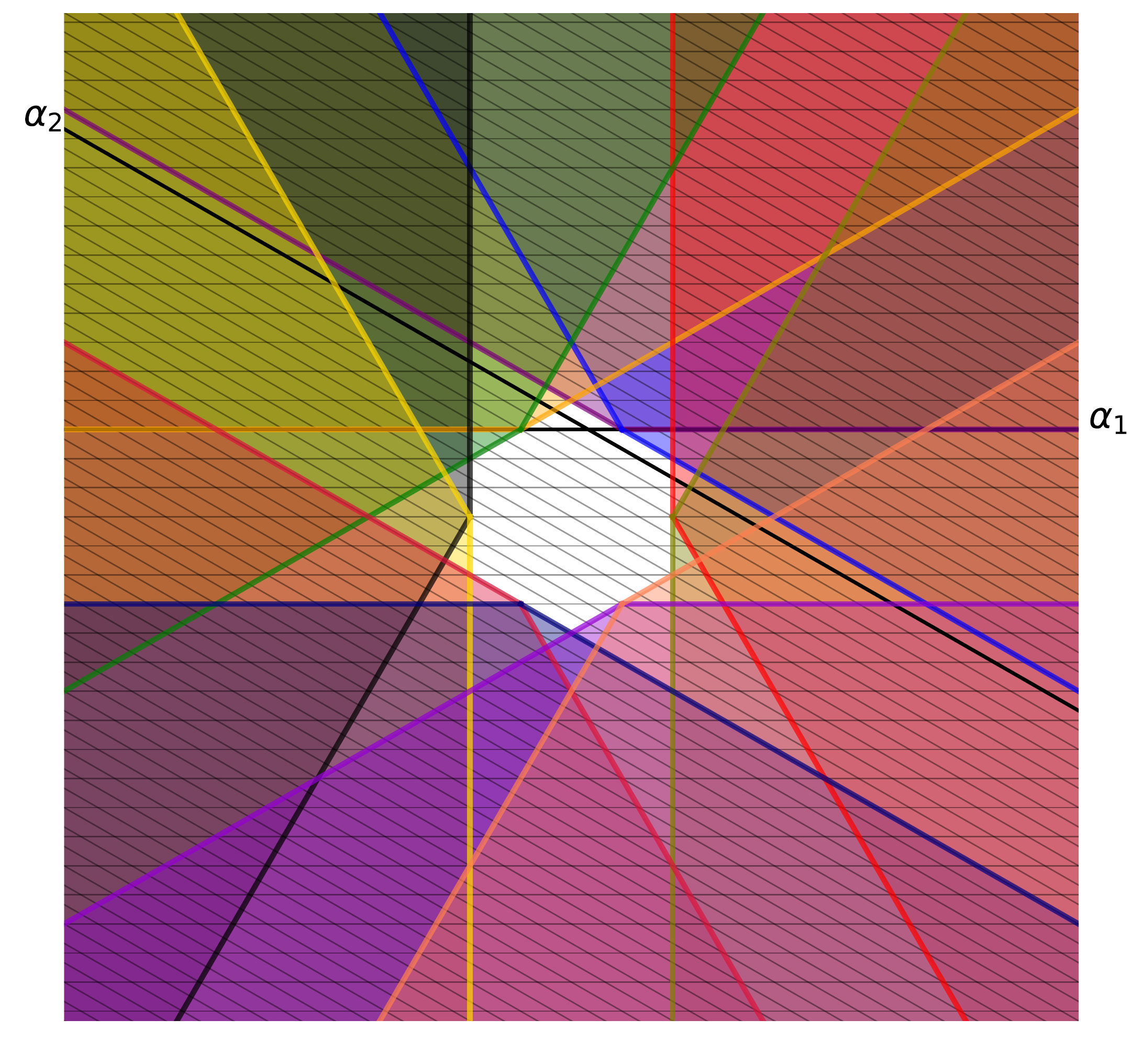}}
    \label{subfig:g2_n1}
    }%
    \hfill
    \subfloat[$\mu = 2\al_1$]{{\includegraphics[width=1.5in]{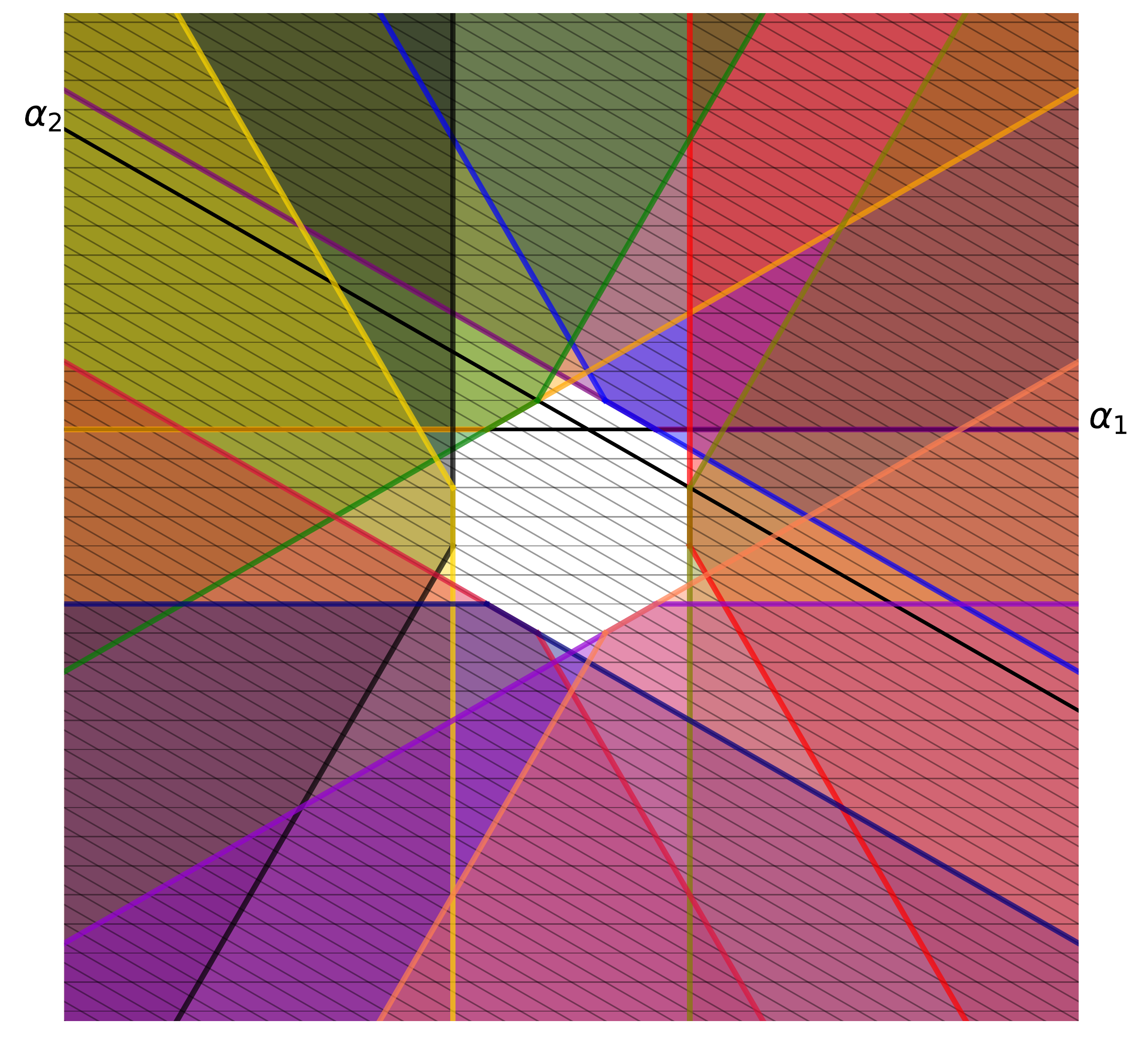} }
    \label{subfig:g2_n2}
    }
    \hfill
    \subfloat[$\mu = 3\al_1$]{{\includegraphics[width=1.5in]{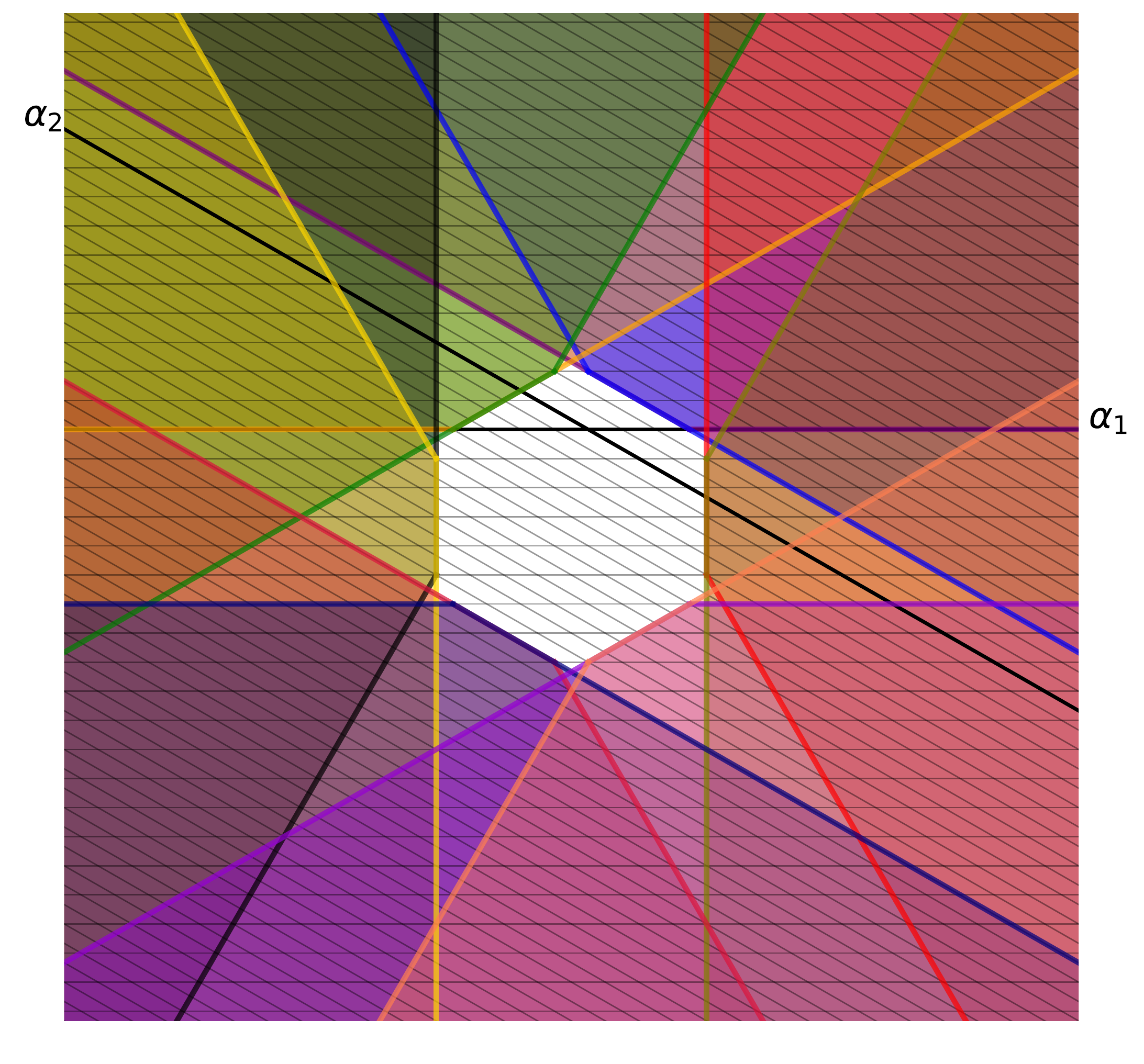}}
    \label{subfig:g2_n3}
    }%
    \hfill
    \subfloat[$\mu = 4\al_1$]{{\includegraphics[width=1.5in]{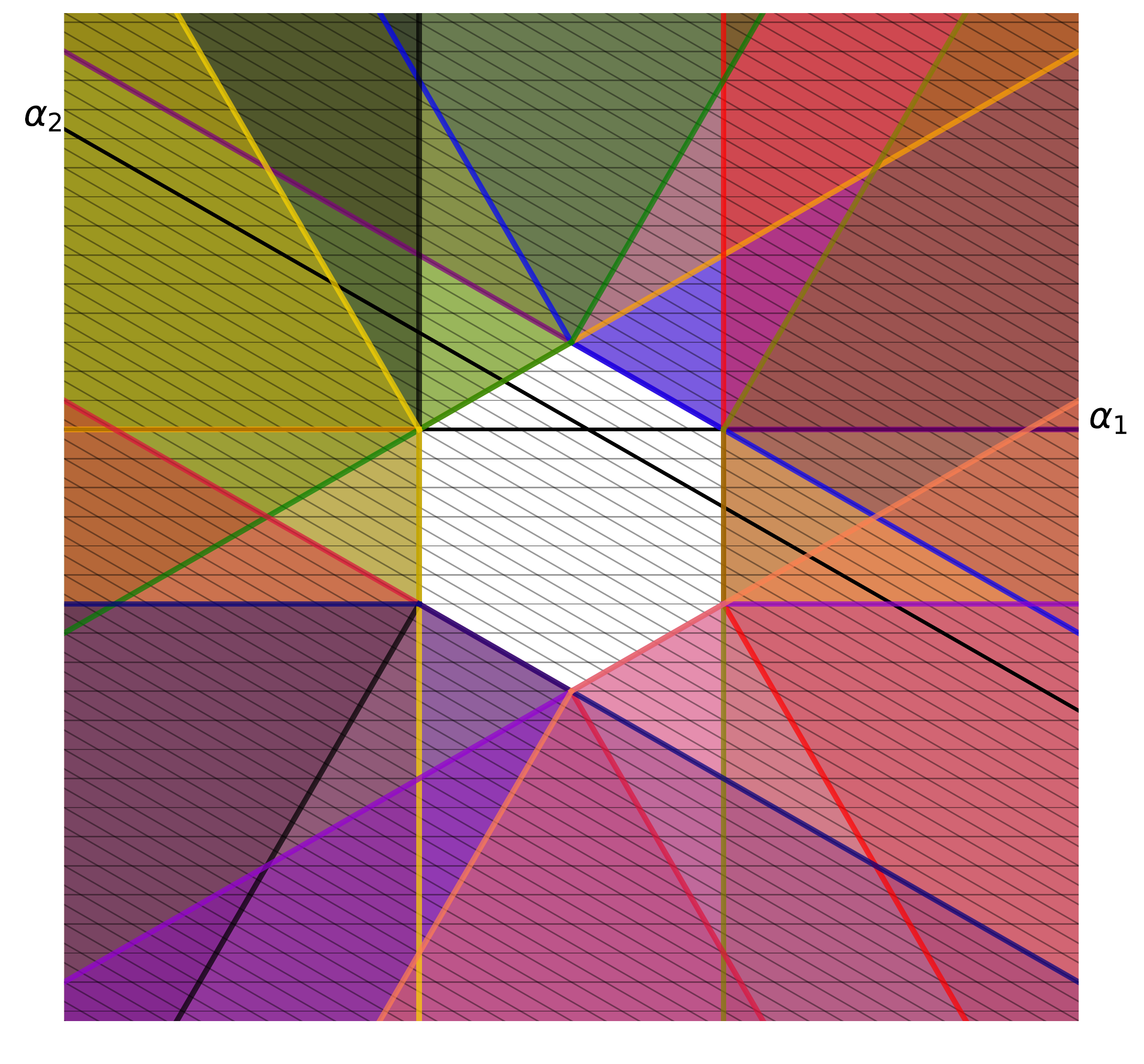} }
    \label{subfig:g2_n4}
    }
    \caption{Weyl alternation diagrams of $\fg_2$ with $\mu=n\al_1$.}
    \label{fig:g2_mu_n}
\end{figure}
\subsubsection{\normalfont{\textbf{Case}} \texorpdfstring{$\mu = m\al_2$}{mu equals m alpha 2}}
Figures \ref{subfig:g2_m1}-\ref{subfig:g2_m4} give a geometric representation of the Weyl alternation diagrams for $\mu = m\al_2$ with $m= 1,2,3,4$. Observe that the empty region takes a hexagonal shape with an edge on top. We also note that as $m$ increases from 1 to 4, the length of the edges of the hexagon also increases.
\begin{figure}[H]%
    \centering
    \subfloat[$\mu = \al_2$]{{\includegraphics[width=1.5in]{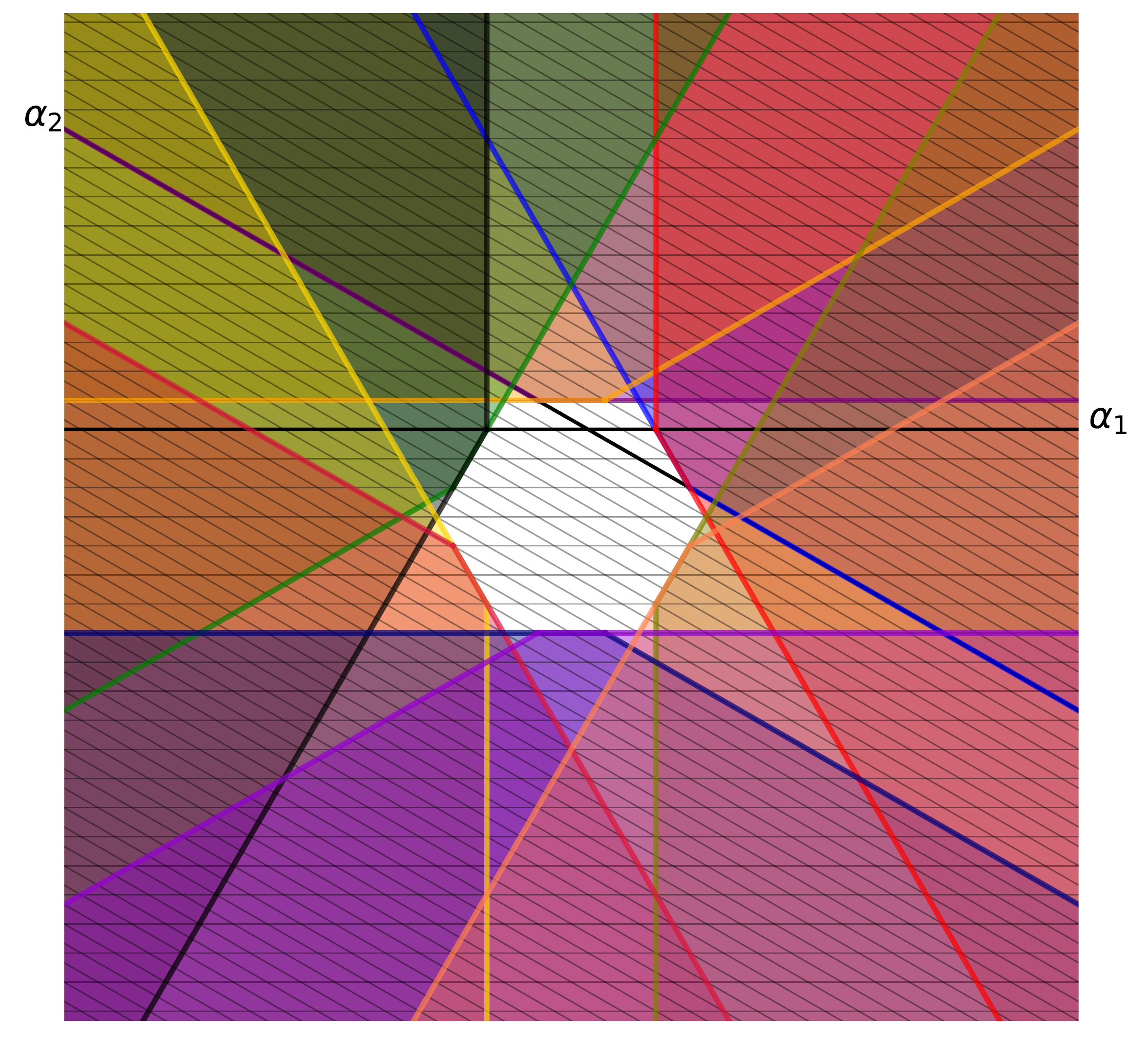}}
    \label{subfig:g2_m1}
    }%
    \hfill
    \subfloat[$\mu = 2\al_2$]{{\includegraphics[width=1.5in]{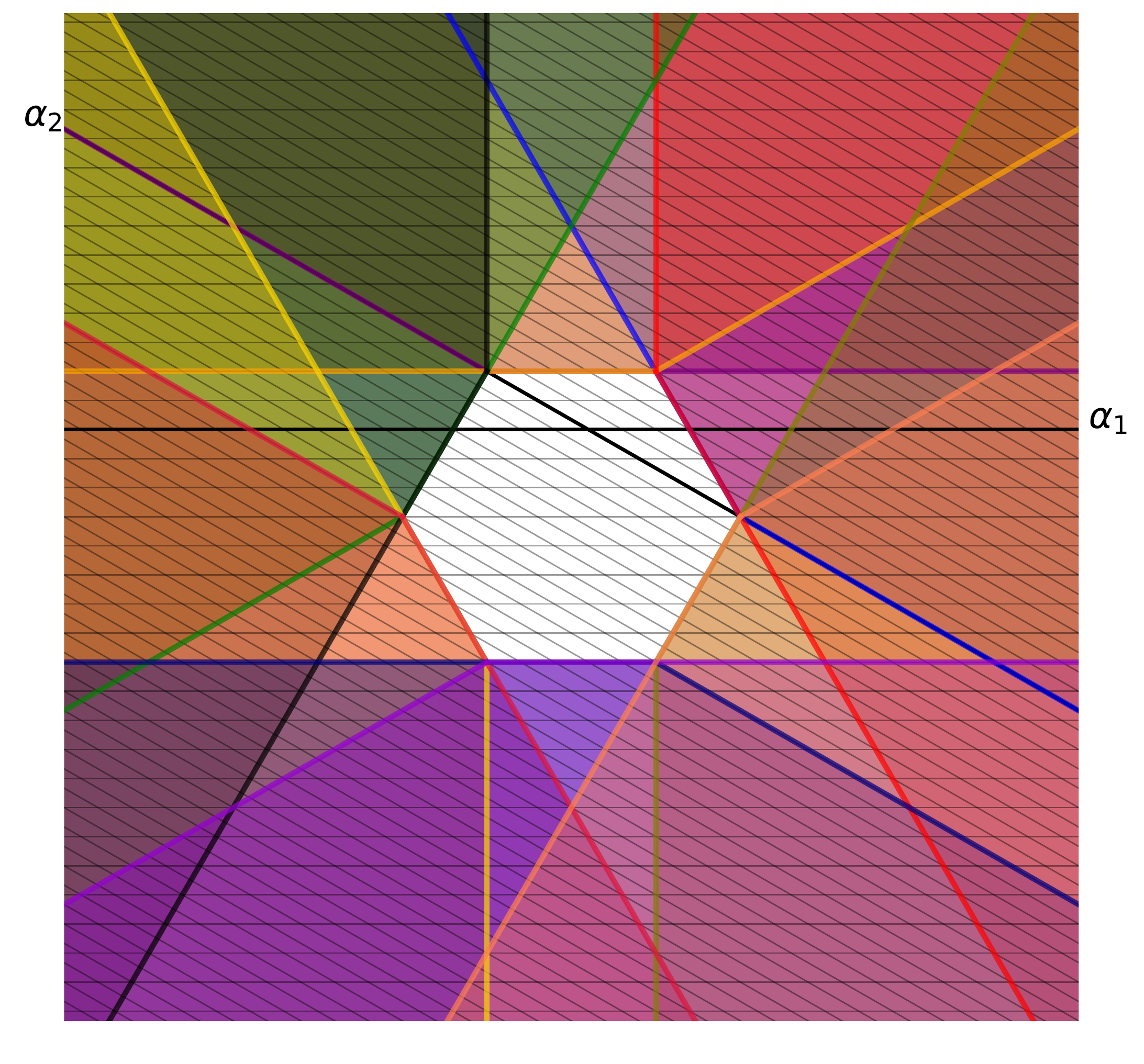} }
    \label{subfig:g2_m2}
    }
    \hfill
    \subfloat[$\mu = 3\al_2$]{{\includegraphics[width=1.5in]{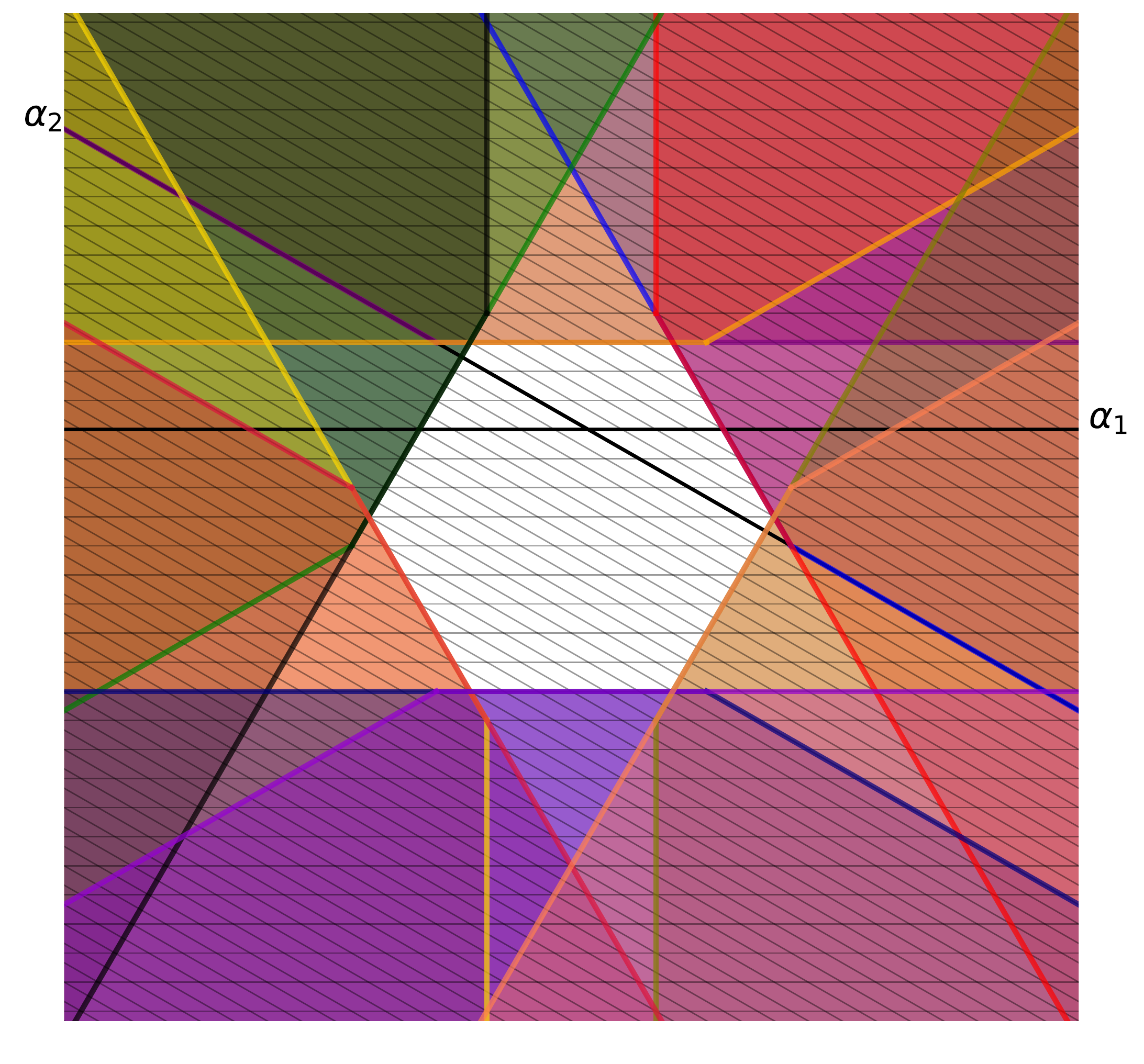}}
    \label{subfig:g2_m3}
    }%
    \hfill
    \subfloat[$\mu = 4\al_2$]{{\includegraphics[width=1.5in]{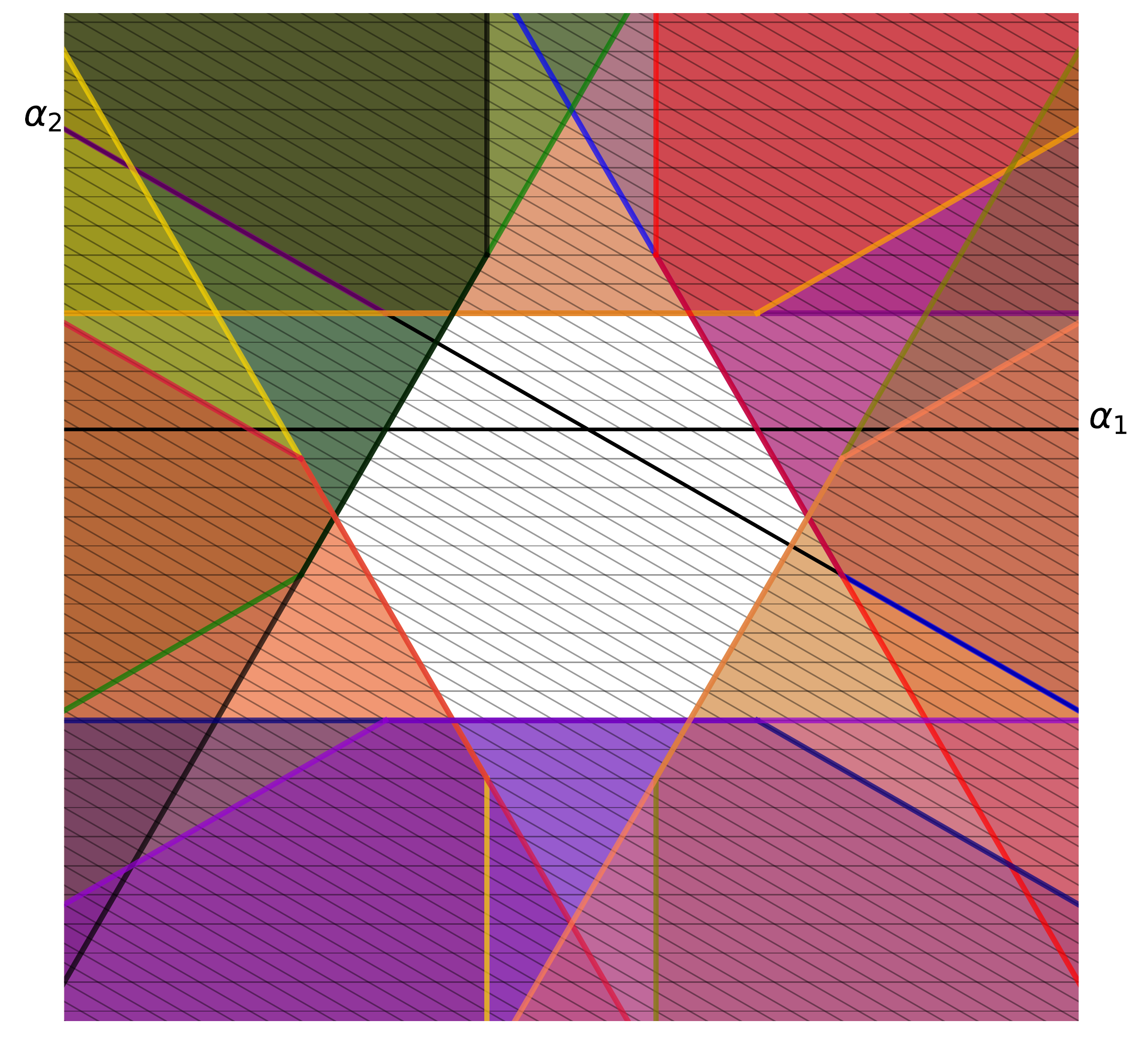} }
    \label{subfig:g2_m4}
    }
    \caption{Weyl alternation diagrams of $\fg_2$ with $\mu=m\al_2$.}
    \label{fig:g2_mu_m}
\end{figure}

\subsubsection{\normalfont{\textbf{Case}} \texorpdfstring{$\mu = n\al_1+m\al_2$}{mu equals n alpha 1 plus m alpha 2}}
First, we assume that $n$ and $m$ satisfy the inequalities $2n+1 > 3m$ and $2m+1 \leq n$ or $2n+1 \leq 3m$ and $2m+1 > n$.
Figures \ref{subfig:g2_n1m1}-\ref{subfig:g2_n3m1} give a geometric representation of the Weyl alternation diagrams. We observe that the empty region is a either a hexagon with a vertex pointing up or a hexagon with an edge on top. We also note that as $n, m$ increase, the empty region covers a larger area. 
\begin{figure}[H]%
    \centering
    \subfloat[$\mu = \al_1 + \al_2$]{{\includegraphics[width=1.5in]{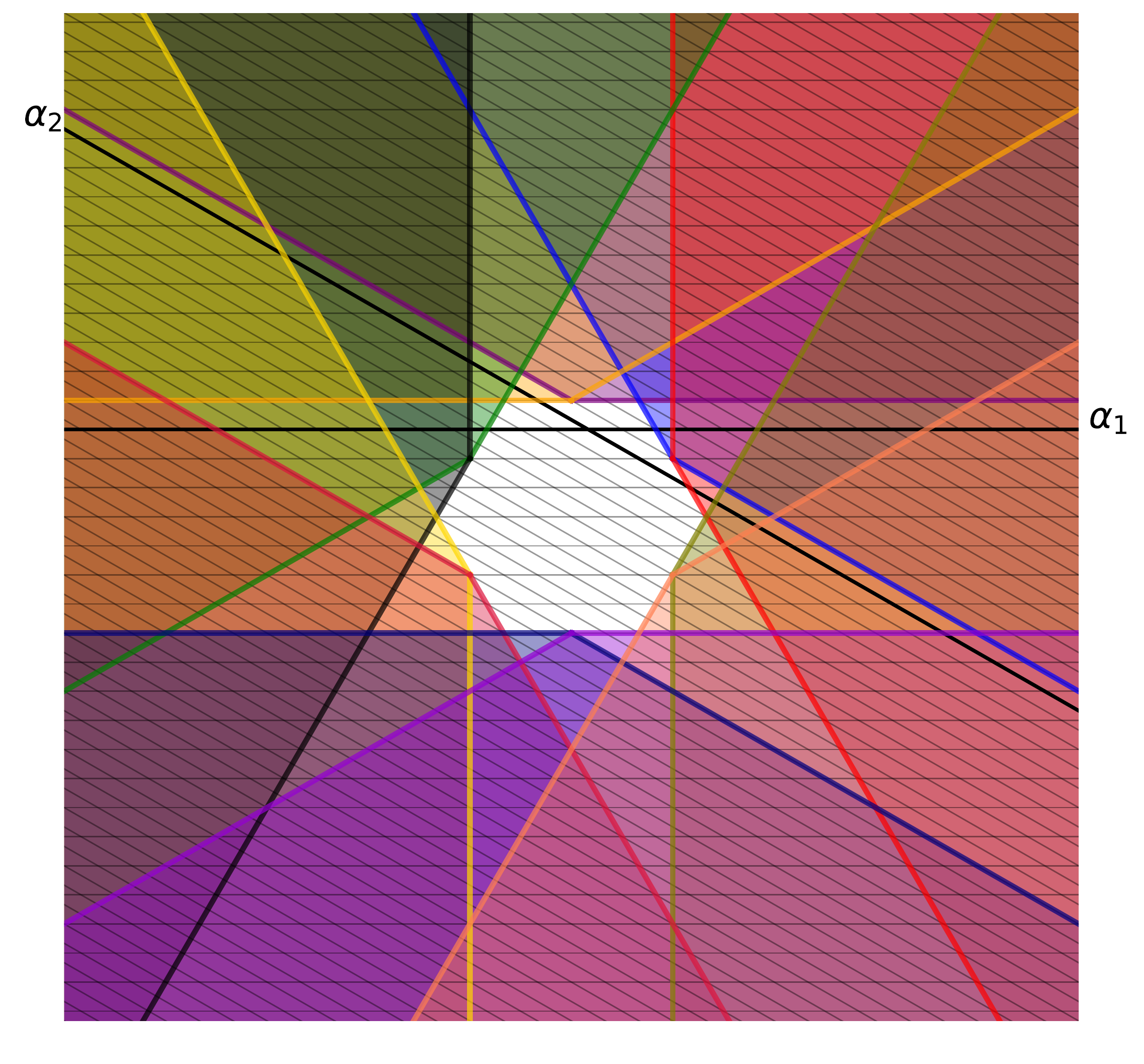}}
    \label{subfig:g2_n1m1}
    }%This one is Done
    \hfill
    \subfloat[$\mu = 2\al_1 + 2\al_2$]{{\includegraphics[width=1.5in]{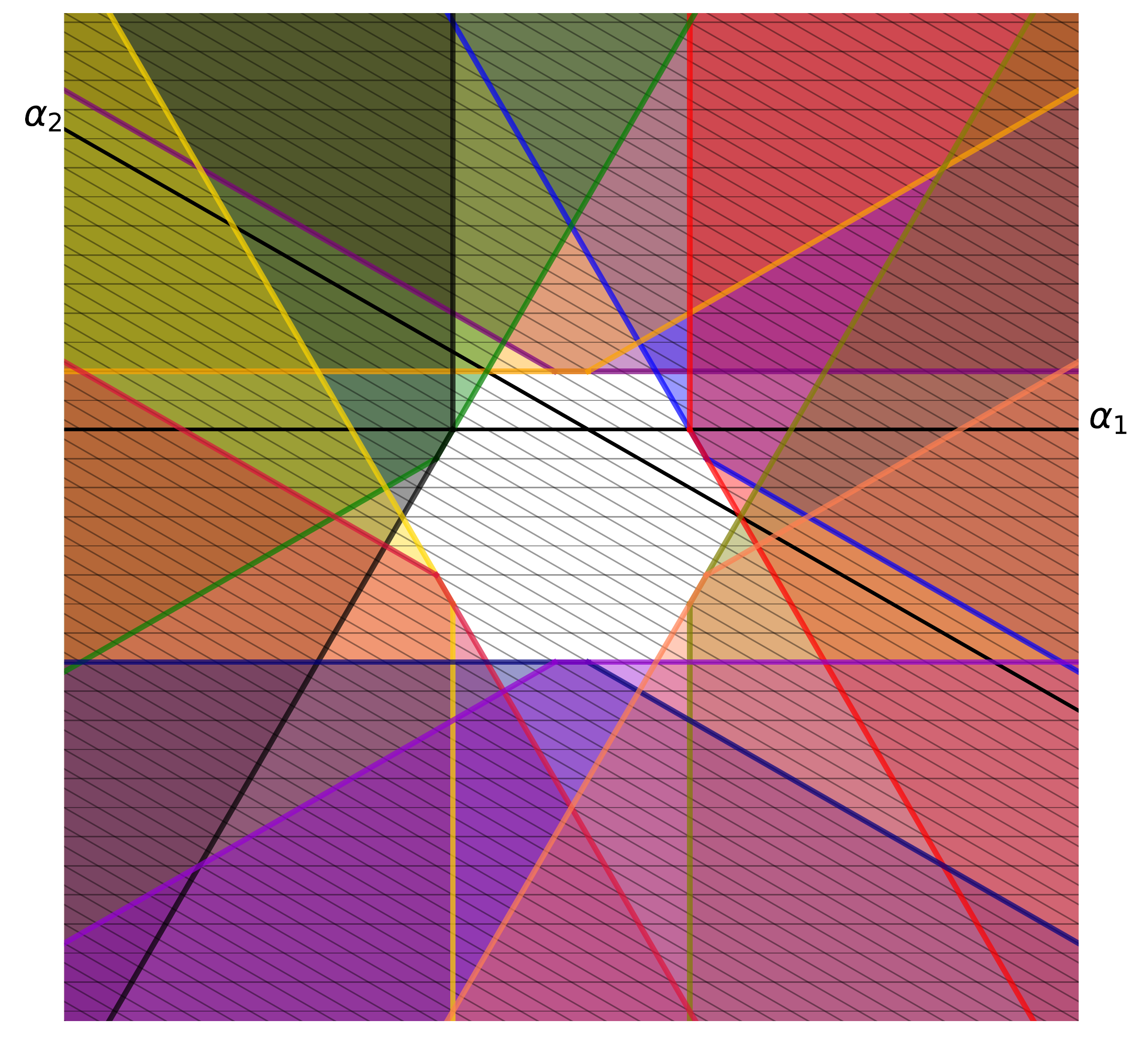} }
    \label{subfig:g2_n2m2}
    }
    \hfill
    \subfloat[$\mu = 1\al_1+2\al_2$]{{\includegraphics[width=1.5in]{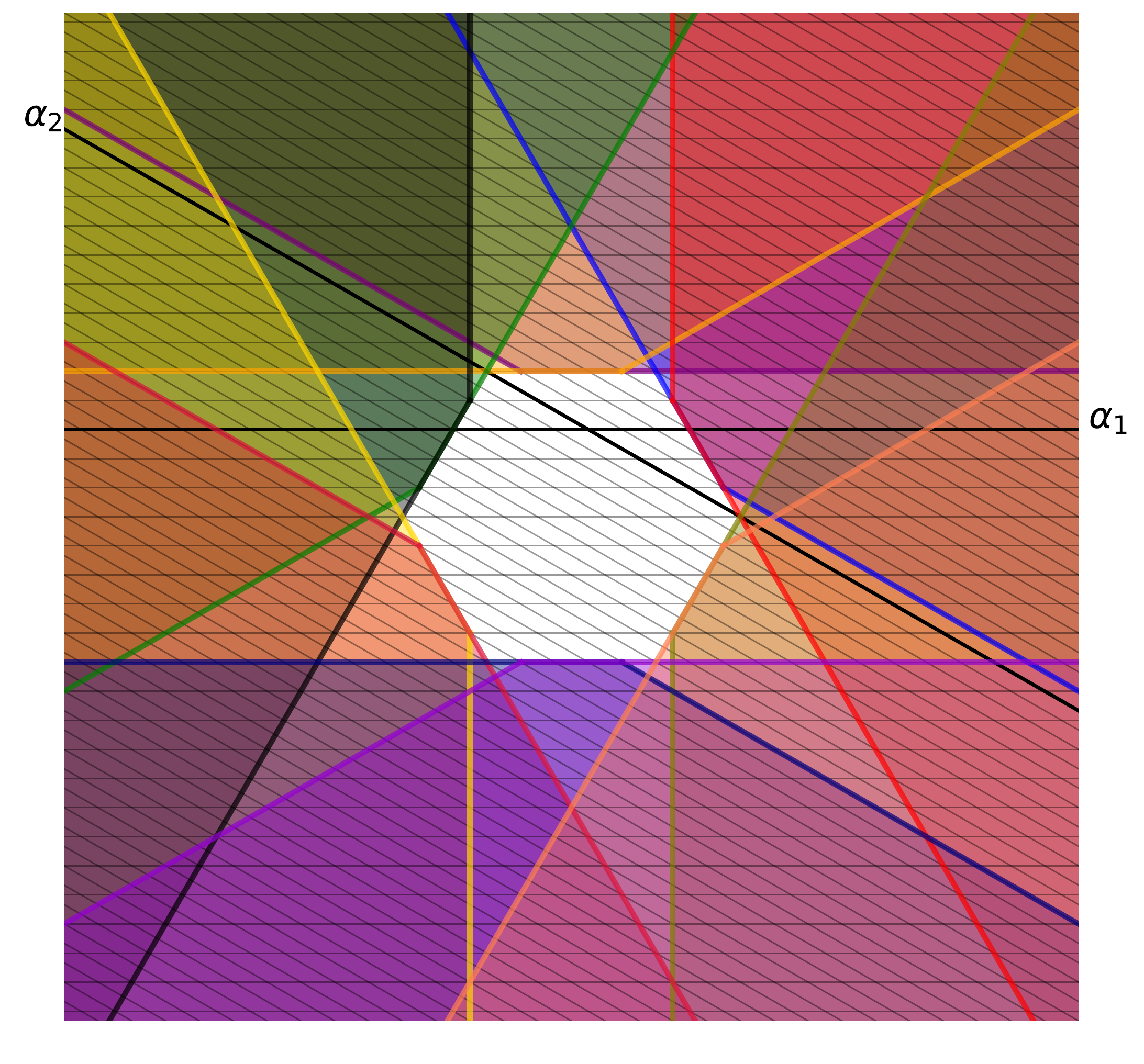}}
    \label{subfig:g2_n1m2}
    }%This one is done
    \hfill
    \subfloat[$\mu = 3\al_1+1\al_2$]{{\includegraphics[width=1.5in]{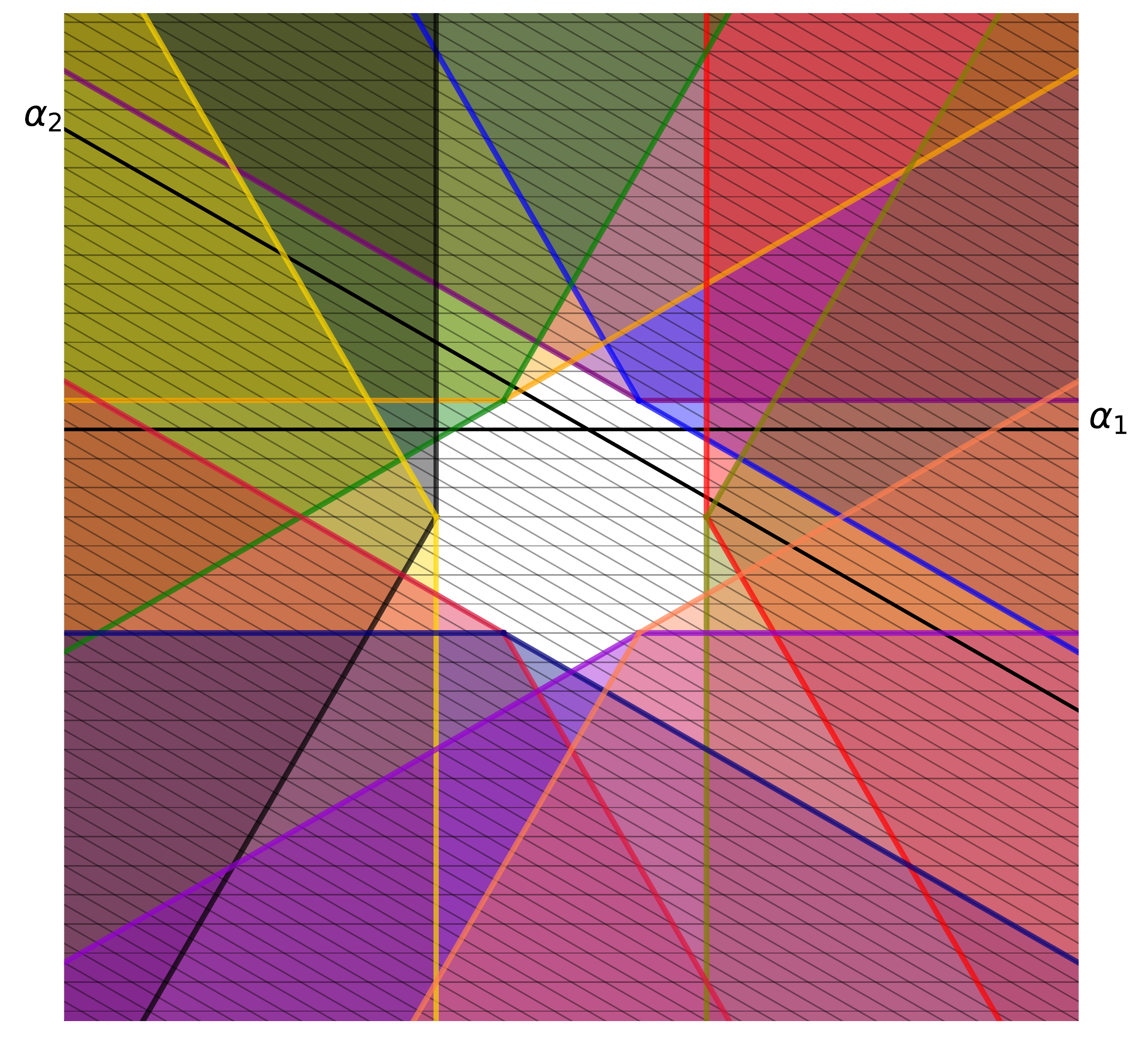} }
    \label{subfig:g2_n3m1}
    }%Did this one already
    \caption{Weyl alternation diagrams of $\fg_2$ where $\mu=n\al_1+m\al_2$ with $2n+1 > 3m$ and $2m+1 \leq n$ or $2n+1 \leq 3m$ and $2m+1 > n$.}
    \label{fig:g2_mu_nm}
\end{figure}
Now assume that $n$ and $m$ must satisfy the inequalities $2n+1 > 3m$ and $2m+1 > n$.
Figures \ref{subfig:g2_n2m1}-\ref{subfig:g2_n4m2} give a geometric representation of the Weyl alternation diagrams. We observe that the empty region is a 12-pointed star. We also note that as $n, m$ increase, the empty region covers a larger area. 
\begin{figure}[H]%
    \centering
    \subfloat[$\mu = 2\al_1 + \al_2$
    ]{{\includegraphics[width=1.5in]{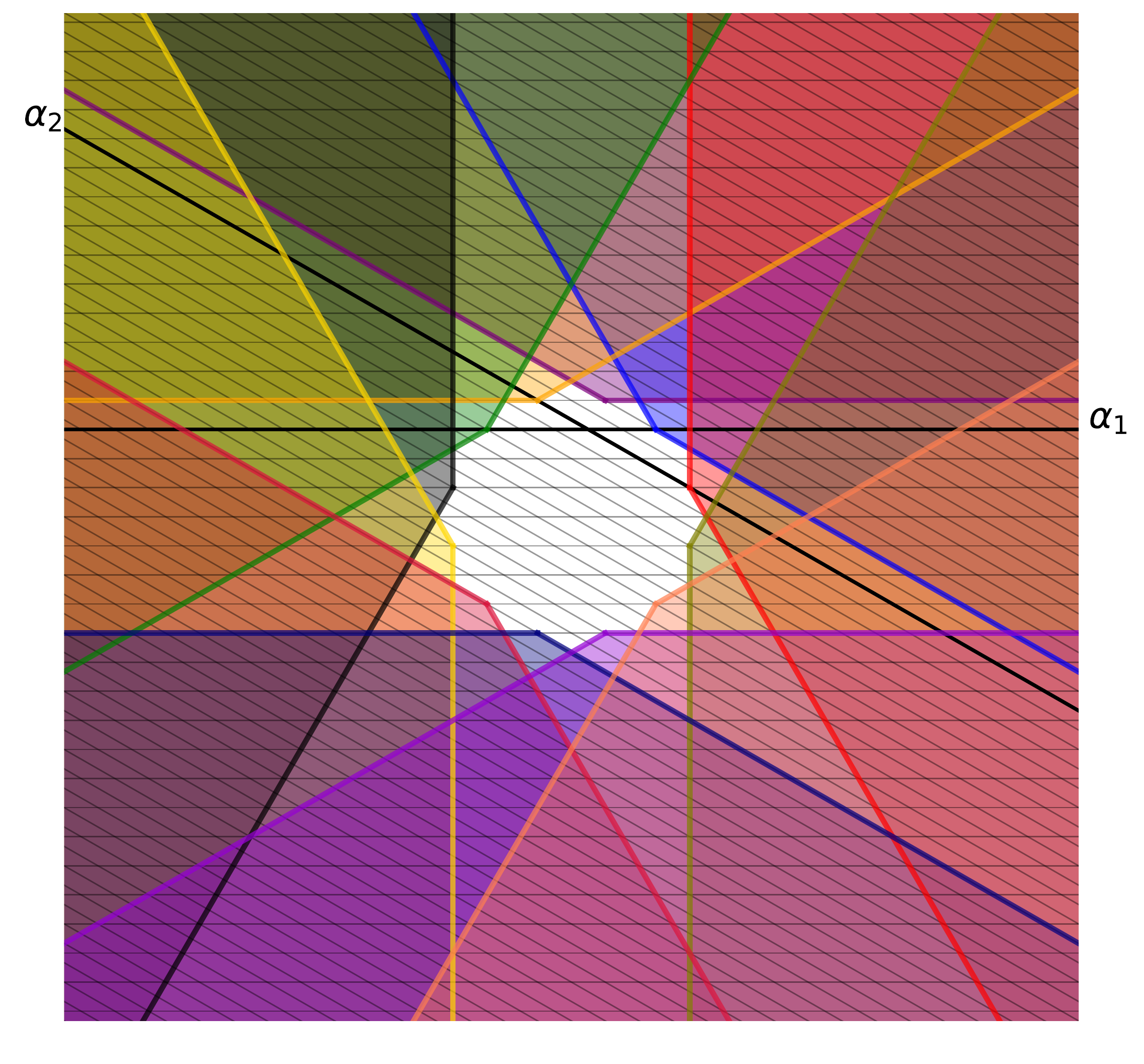}}
    \label{subfig:g2_n2m1}}
    \hfill
    \subfloat[$\mu = 3\al_1 + 2\al_2$
    ]{{\includegraphics[width=1.5in]{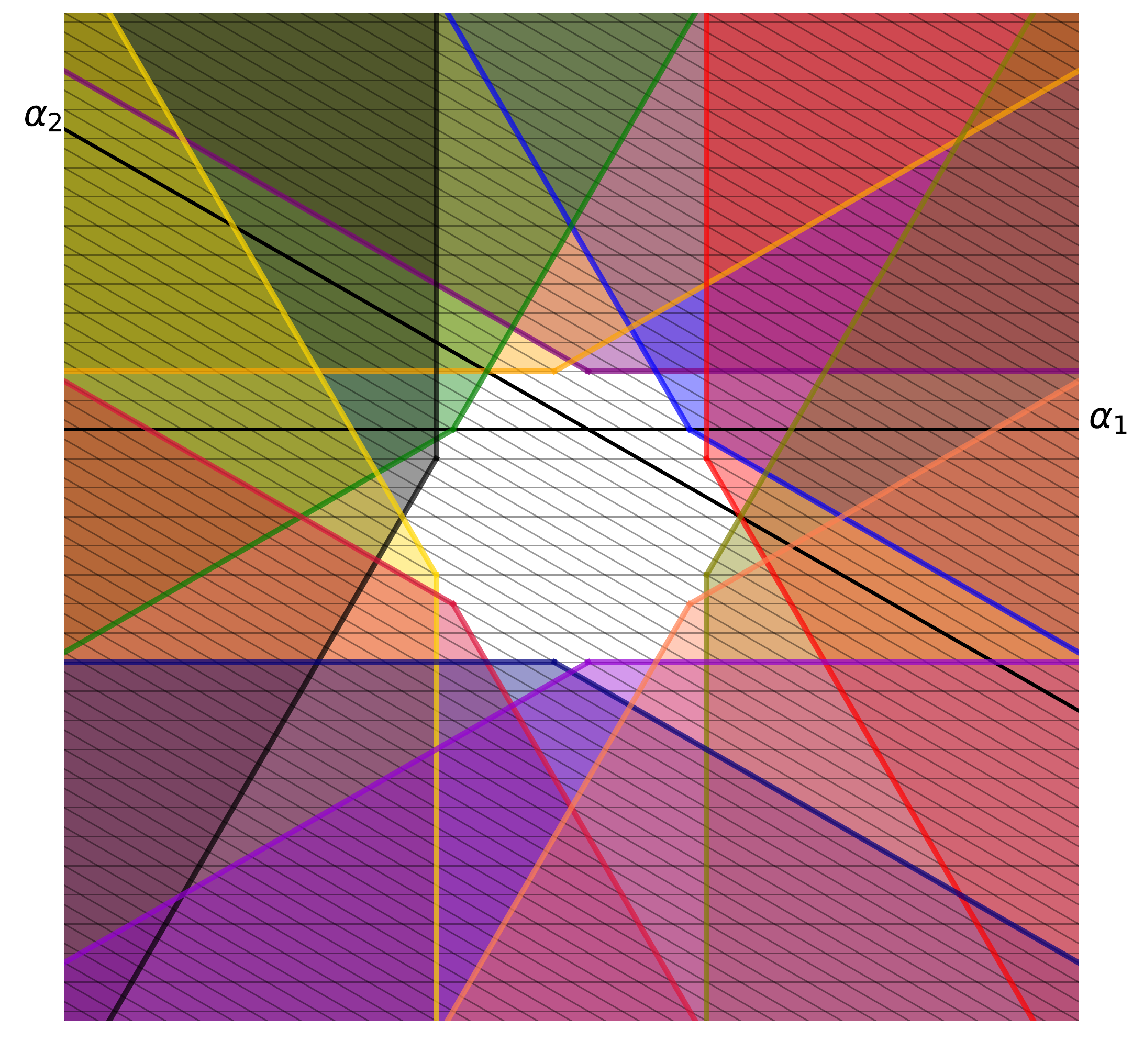} }}
    \hfill
    \subfloat[$\mu = 5\al_1 + 3\al_2$
    ]{{\includegraphics[width=1.5in]{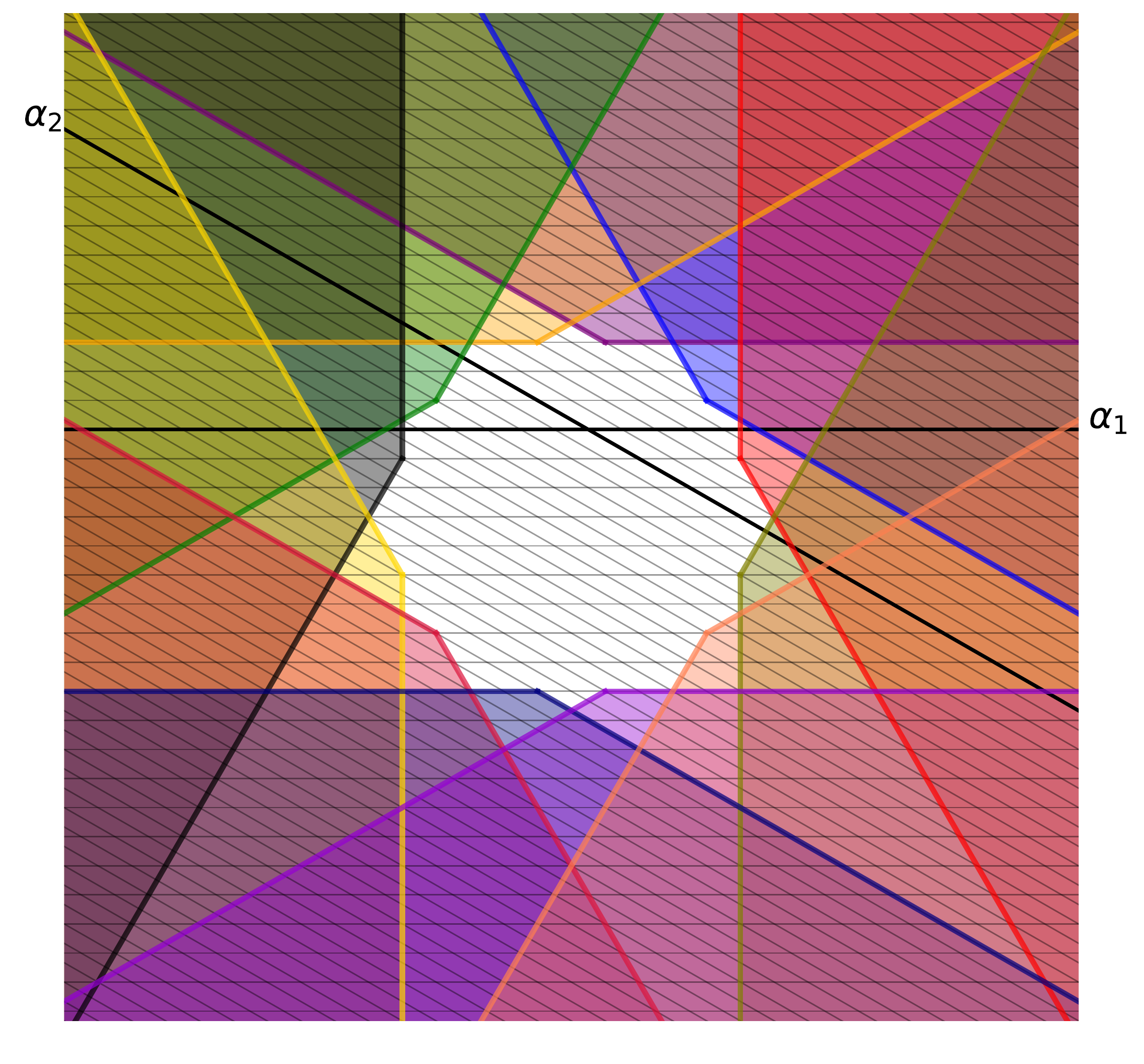} }}
    \hfill
    \subfloat[$\mu = 4\al_1 + 2\al_2$
    ]{{\includegraphics[width=1.5in]{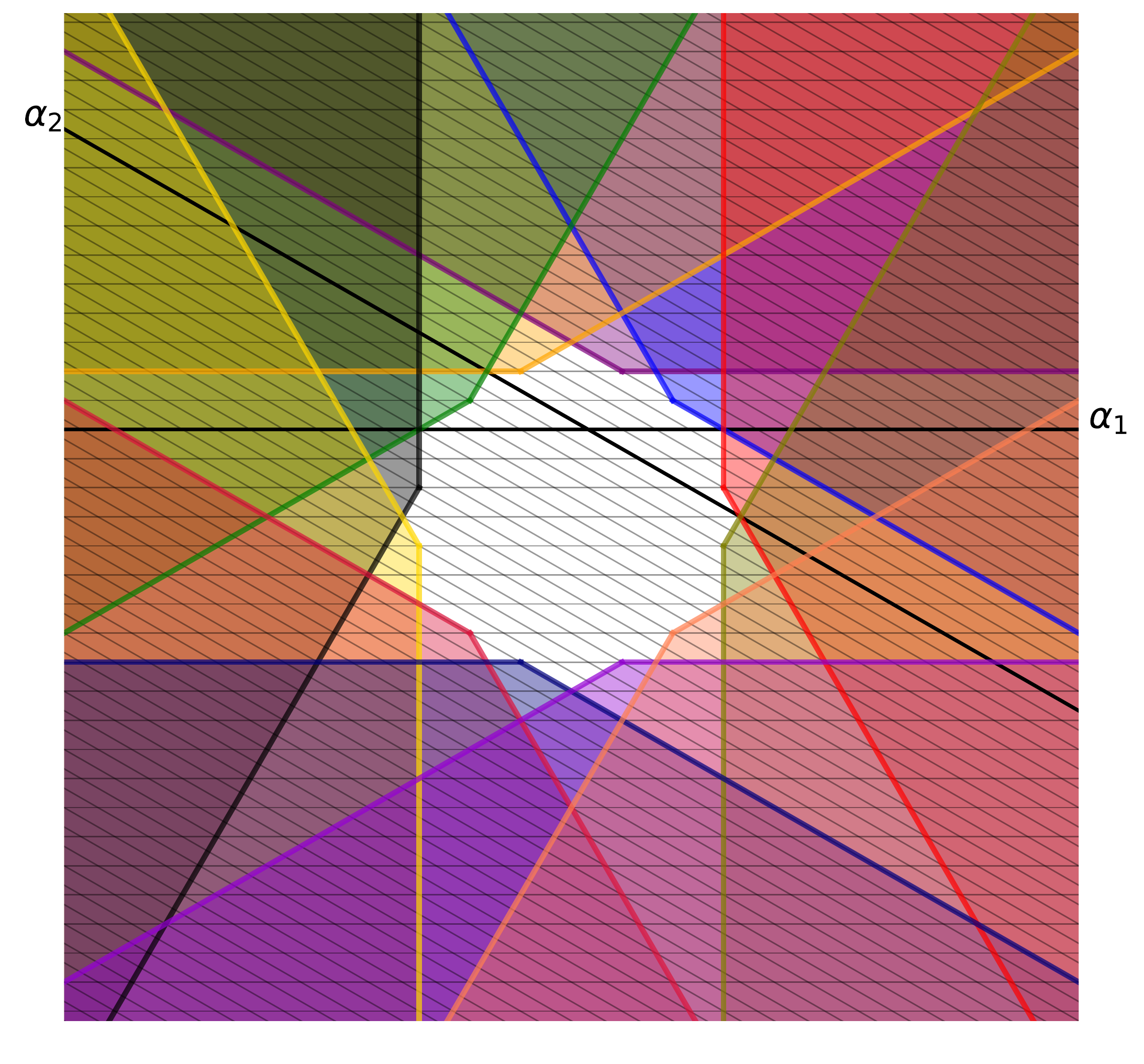} }
    \label{subfig:g2_n4m2}}
    \caption{Weyl alternation diagrams of $\fg_2$ where $\mu=n\a_1+m\a_2$ where $2n+1 > 3m$ and $2m+1 > n$.} 
    \label{fig:g2_mu_w}
\end{figure}

\begin{figure}[H]
    \centering
    \includegraphics[scale=.2]{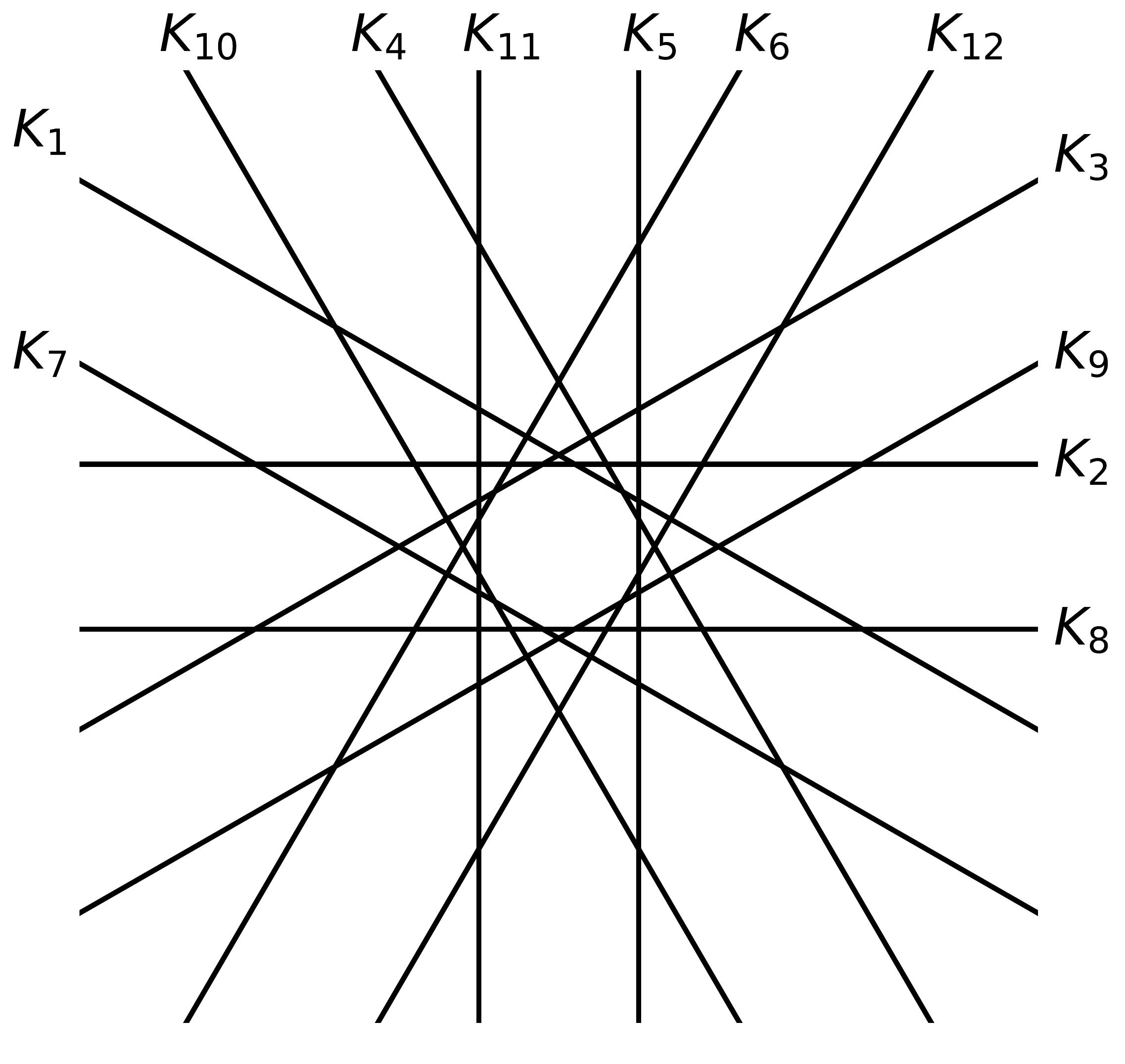}
    \caption{Set of linear inequalities for determining the boundaries of the Weyl alternation sets for the Lie algebra of type $G_2$.}
    \label{fig:g2_ineqs}
\end{figure}

To explain the shapes that form in the empty region of each Weyl alternation diagram for the Lie algebra of type $G_2$ we turn to Figure \ref{fig:g2_ineqs}. From Table \ref{tab:g2}, we notice that all inequalities depend on $n$ and $m$ which simply shift the inequalities. However, changing $n$ and $m$ never changes the direction of the line. This means that Figure \ref{fig:g2_ineqs} is a good representation of the inequalities formed. 
\begin{figure}[H]%
    \centering
    \subfloat[Edge on top.]{{\includegraphics[width=5cm]{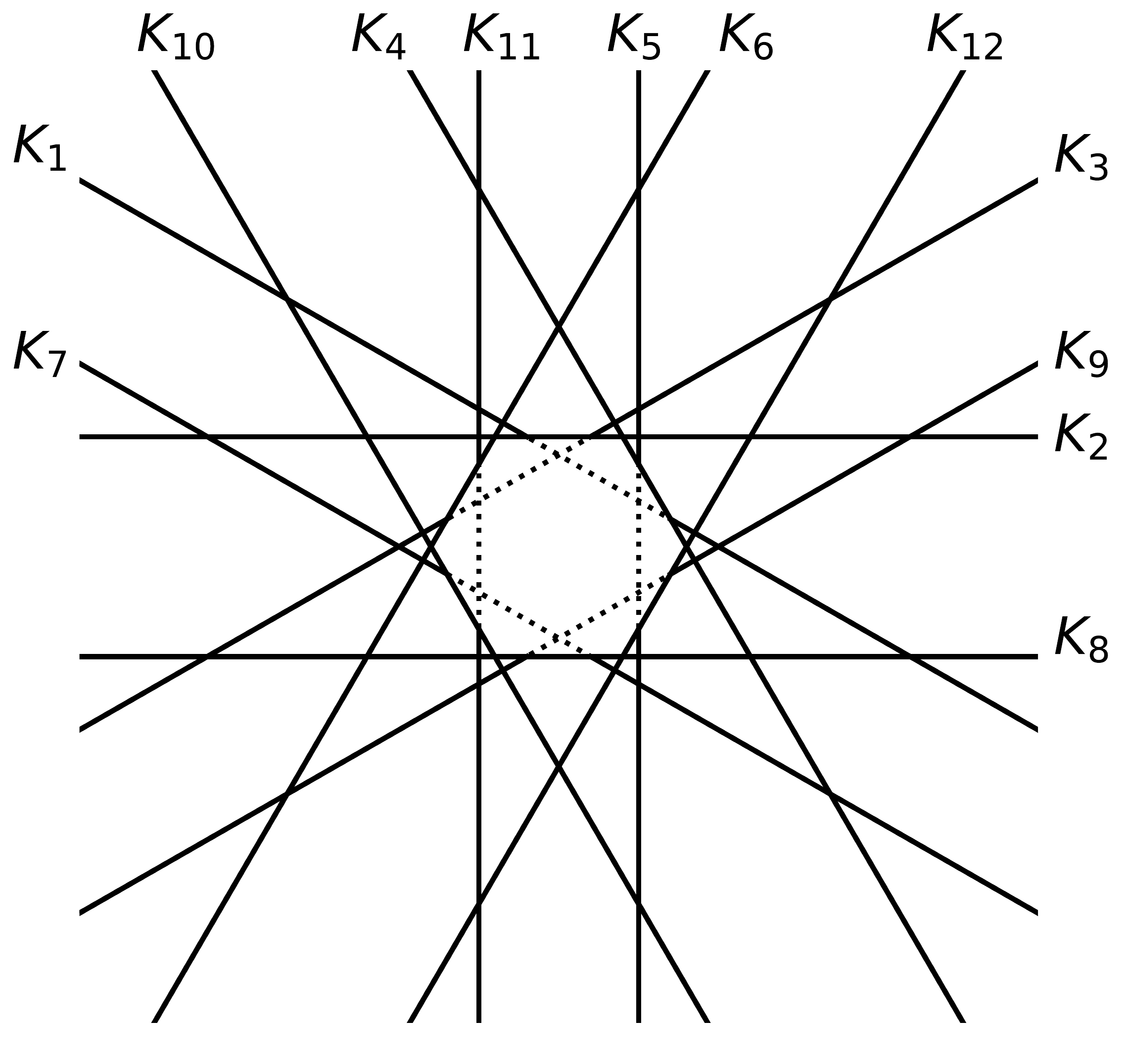}}
    \label{subfig:g2_edge}
    }%This one is Done
    \quad
    \subfloat[Vertex on top.]{{\includegraphics[width=5cm]{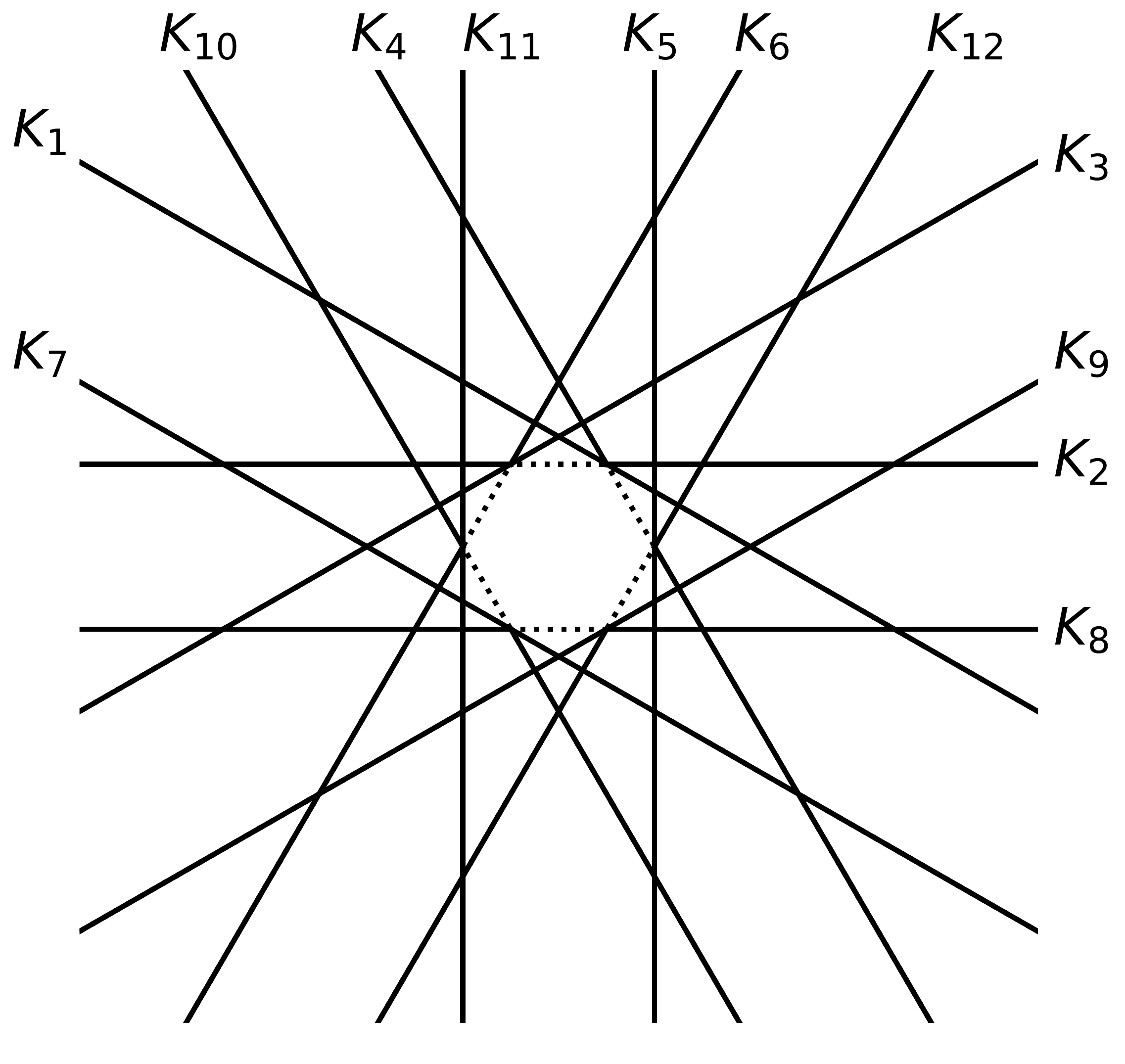} }
    \label{subfig:g2_vertex}
    }%This is done
    \quad
    \subfloat[12-pointed star.]{{\includegraphics[width=5cm]{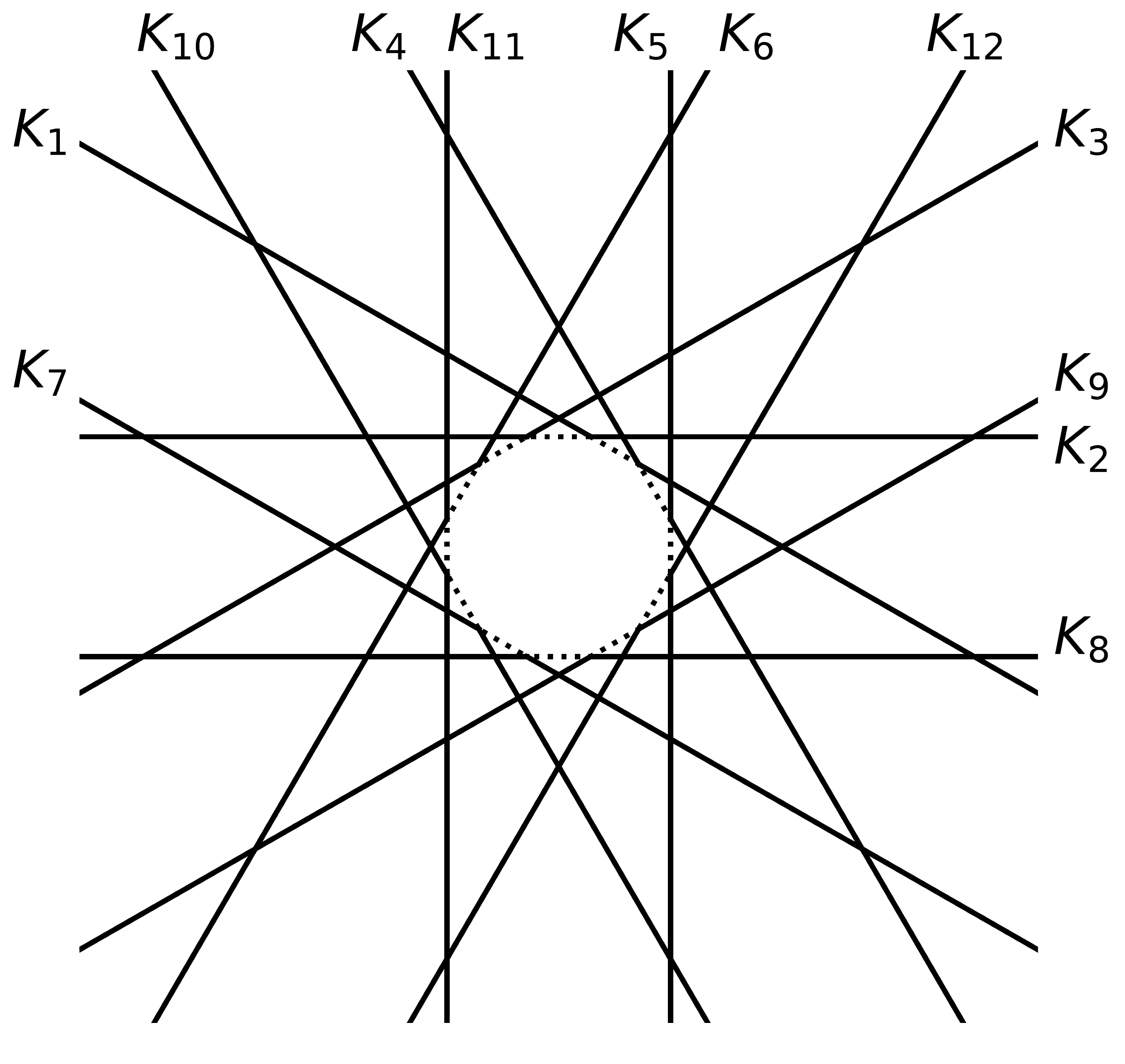}}
    \label{subfig:g2_star}
    }%This one is done
    \caption{Different formations of the center empty region for the Lie algebra of type $G_2$.}
    \label{fig:g2_cases}
\end{figure}
 From Figure \ref{subfig:g2_edge}, we note that the hexagon with an edge on top will occur when the inequalities $K_1$ and $K_3$ intersect at or below $K_2$ and the inequalities $K_4$ and $K_2$ intersect strictly above $K_1$. The dashed lines indicate the continuation of the inequalities into the empty region. Similarly, the empty region takes the shape of a hexagon with a vertex on top when inequalities $K_1$ and $K_3$ intersect strictly above $K_2$ and the inequalities $K_4$ and $K_2$ intersect at or below $K_1$. This is depicted in Figure \ref{subfig:g2_vertex}. The dashed lines indicate the continuation of the inequalities into the empty region. In addition, the empty region takes the shape of a 12-pointed star if and only if inequalities $K_1$ and $K_3$ are strictly above $K_2$ and the inequalities $K_2$ and $K_4$ intersect strictly above $K_1$. This is depicted in Figure \ref{subfig:g2_star}. This occurs exactly when $\mu = n\w_1+m\w_2$ with $n,m \in \NN$.
 
\section{Future work}\label{sec:open}
One possible direction for extending the work presented in this manuscript is to study the Weyl alternation sets of the rank 3 Lie algebras along with their Weyl alternation diagrams. For example, one can ask the following.

\begin{question}
What are the Weyl alternation diagrams for the Lie algebra $\mathfrak{sl}_4(\mathbb{C})$? How does the empty region change as the weight $\mu$ changes?
\end{question}

\section*{Acknowledgements}
This research was supported in part by the Alfred P. Sloan Foundation, the Mathematical Sciences Research Institute, and the National Science Foundation. We would like to thank Prof. Rebecca Garcia (Sam Houston State University) for her feedback on earlier drafts of this manuscript.

\end{document}